\documentclass{tm-amsart-latex}

\begin{document}

\title{
  Revisiting derived crystalline cohomology
}
  \tmnote{{\citetexmacs{TeXmacs:website}}}

\author{Zhouhang Mao}

\date{\tmdate}

\selectlanguage{english}


\maketitle

\selectlanguage{french}
\begin{abstract} (Revisiter la cohomologie cristalline dérivée)
  Nous prouvons que la $\infty$-catégorie des surjections d'anneaux animés est projectivement générée, introduisons et étudions la notion de PD-paires animées --- des surjections d'anneaux animés avec une PD-structure ``dérivée''. Cela nous permet de généraliser des résultats classiques à des situations non plates et non de type fini.

  En utilisant les PD-paires animées, nous développons plusieurs approches de la cohomologie cristalline dérivée et établissons des théorèmes de comparaison. En tant qu'application, nous généralisons la comparaison entre la cohomologie cristalline dérivée et classique à partir de schémas syntomiques (affines) (due à Bhatt) à des schémas quasisyntomiques.

  Nous développons également un analogue animé non complété des prismes et des enveloppes prismatiques. Nous prouvons une variante de la comparaison de Hodge--Tate pour les enveloppes prismatiques animées, à partir de laquelle nous déduisons un résultat sur le recouvrement plat de l'objet final pour les schémas quasisyntomiques, qui généralise plusieurs résultats connus sous des conditions de lissité et de finitude.
\end{abstract}

\selectlanguage{english}
\begin{abstract}
  We prove that the $\infty$-category of surjections of animated rings is
  projectively generated, introduce and study the notion of animated PD-pairs
  --- surjections of animated rings with a ``derived'' PD-structure. This
  allows us to generalize classical results to non-flat and
  non-finitely-generated situations.
  
  Using animated PD-pairs, we develop several approaches to derived
  crystalline cohomology and establish comparison theorems. As an application,
  we generalize the comparison between derived and classical crystalline
  cohomology from syntomic (affine) schemes (due to Bhatt) to quasisyntomic
  schemes.
  
  We also develop a non-completed animated analogue of prisms and prismatic
  envelopes. We prove a variant of the Hodge--Tate comparison for animated
  prismatic envelopes from which we deduce a result about flat cover of the
  final object for quasisyntomic schemes, which generalizes several known
  results under smoothness and finiteness conditions.
\end{abstract}

{\tableofcontents}

\section{Introduction}\label{sec:intro}

In this introductory section, we start with a non-technical discussion of
background, stating the main results in simplified forms. Then we explain the
main techniques used in this article. After that, we present the main
definitions and constructions in this article.

\subsection{Background and main results}In this subsection, we discuss the
background and the main results of the current work in a simplified form.

Regular sequences and local complete intersections play an important role in
the study of Noetherian rings. However, in arithmetic geometry, Noetherianness
is not preserved by operations related to perfectoids. Various generalizations
to the non-Noetherian case are available. In {\cite{Bhatt2018}}, it has been
shown that, the {\tmdfn{quasiregularity}} (à la Quillen) is a particularly
good candidate to replace (Koszul) regularity in classical algebraic geometry:
an ideal $I$ of a ring $A$ is called {\tmdfn{quasiregular}}
(\Cref{def:quasiregular}) if the $A / I$-module $I / I^2$ is flat and the
homotopy groups $\pi_i (L_{(A / I) / A})$ of the cotangent complex vanish for
$i > 1$, or equivalently put, $L_{(A / I) / A} \simeq (I / I^2) [1]$. In
particular, if an ideal is generated by a Koszul-regular sequence, then it is
also quasiregular.

Let us briefly review some details in the simple case of characteristic $p$
(instead of mixed characteristic). An $\mathbb{F}_p$-algebra $R$ is called
{\tmdfn{perfect}} if the Frobenius map $R \rightarrow R, x \mapsto x^p$ is
bijective. An $\mathbb{F}_p$-algebra $S$ is called {\tmdfn{quasiregular
semiperfect}} if there exists a perfect $\mathbb{F}_p$-algebra $R$ along with
a surjective map $R \twoheadrightarrow S$ of rings of which the kernel $I
\subseteq R$ is quasiregular. In this case, {\cite[Thm~8.12]{Bhatt2018}} shows
that the derived de Rham cohomology of $R$ with respect to the base
$\mathbb{F}_p$ is concentrated in degree $0$, and as a ring, it is equivalent
to the PD-envelope of $(R, I)$. Since the cotangent complex $L_{R
/\mathbb{F}_p}$ vanishes, the base $\mathbb{F}_p$ of the derived de Rham
cohomology could be replaced by $R$.

This result was already known {\cite[Thm~3.27]{Bhatt2012a}} when the kernel
$I$ of the map $R \twoheadrightarrow S$ in question is Koszul regular. In
other words, {\cite{Bhatt2018}} generalizes the classical results about
Koszul-regular ideals to quasiregular ideals.

In this article, we develop a different approach which works in greater
generality: we do not need the base to be perfect, of characteristic $p$ or
even ``$p$-local'' such as $\mathbb{Z}_p$ or a perfectoid ring. We build a
machinery to extend results about Koszul-regular ideals to quasiregular ideals
in a systematic fashion. We say that a map $R \rightarrow S$ of animated rings
{\cite[§5.1]{Cesnavicius2019}} is {\tmdfn{surjective}} if the induced map
$\pi_0 (R) \rightarrow \pi_0 (S)$ is surjective
(\Cref{def:surj-map-ani-ring}).

\begin{theorem}[\Cref{thm:ani-smith-eq}]
  The $\infty$-category of surjective maps of animated rings is projectively
  generated. The set $\{ \mathbb{Z} [x_1, \ldots, x_m, y_1, \ldots, y_n]
  \twoheadrightarrow \mathbb{Z} [x_1, \ldots, x_m] \barsuchthat m, n \in
  \mathbb{N} \}$ of objects forms a set of compact projective generators.
\end{theorem}

For technical reasons, we will introduce the $\infty$-category of
{\tmdfn{animated pairs}}, which is equivalent to the $\infty$-category of
surjective maps of animated rings. By the formalism of left derived functors
(\Cref{prop:left-deriv-n-fun}), given a functor defined for ``standard''
Koszul-regular pairs $(\mathbb{Z} [X, Y], (Y))$ where $X = \{ x_1, \ldots, x_m
\}$ and $Y = \{ y_1, \ldots, y_m \}$\footnote{In this article, the
multivariable notations $X$ and $Y$ are used from time to time.}, we get a
functor defined on {\tmem{all}} animated pairs, and in particular, on
classical ring-ideal pairs $(A, I)$, and any comparison map between such
functors is determined by the restriction to these Koszul-regular pairs. We
learned the importance of such standard pairs from the proof of
{\cite[Cor~4.14]{Bhatt2012}}.

In order to formulate a reasonable generalization of
{\cite[Thm~8.12]{Bhatt2018}}, just as we need animated pairs, we also need
{\tmdfn{animated PD-pairs}} (\Cref{def:anipair-anipdpair}), denoted by $(A
\twoheadrightarrow A'', \gamma)$ (\Cref{nota:ani-pdpair}). There is a
canonical forgetful functor from the $\infty$-category of animated PD-pairs to
the $\infty$-category of animated pairs, which preserves small colimits
(\Cref{prop:forget-PD-small-colim}). This is remarkable since the forgetful
functor from the $1$-category of PD-pairs to the $1$-category of ring-ideal
pairs does not preserve small colimits
(\Cref{rem:pdpair-pair-forget-not-preserve-colim}). Moreover, the forgetful
functor admits a left adjoint, called the {\tmdfn{animated PD-envelope
functor}}.

In general, the animated PD-envelope, considered as a kind of derived functor,
is different from the PD-envelope. We will show that, there is a canonical
filtration on the animated PD-envelope of $\mathbb{F}_p$-pairs\footnote{Or
more generally, of animated $\mathbb{F}_p$-pairs.} (i.e. pairs $(A, I)$ where
$A$ is an $\mathbb{F}_p$-algebra), called the {\tmdfn{conjugate filtration}}
(\Cref{def:conj-fil-pd-env}), of which we can control the associated graded
pieces:

\begin{theorem}[\Cref{cor:conjfil-gr,cor:LAdFil-symm-cot-cx}]
  Let $A$ be an $\mathbb{F}_p$-algebra and $I \subseteq A$ an
  ideal\footnote{In the introduction, for sake of simplicity, we usually
  replace the occurrences of animated pairs (resp. animated PD-pairs) by
  ring-ideal pairs (resp. PD-pairs) as input data.}. Then
  \begin{enumerate}
    \item the animated PD-envelope of $(A, I)$ admits a natural animated
    $\varphi_A^{\ast} (A / I)$-algebra structure.
    
    \item for every $i \in \mathbb{N}$, the $(- i)$-th associated graded piece
    of the animated PD-envelope of $A$ is, as a $\varphi_A^{\ast} (A /
    I)$-module spectrum, naturally equivalent to $\varphi_A^{\ast} (\Gamma_{A
    / I}^i (L_{(A / I) / A} [- 1]))$, where $\Gamma_{A / I}^i$ is the $i$-th
    derived divided power.
  \end{enumerate}
\end{theorem}

As a corollary, the notion of quasiregularity provides an important
{\tmem{acyclicity condition}}: along with a mild assumption, the animated
PD-envelope coincides with the classical PD-envelope:

\begin{theorem}[\Cref{cor:Fp-qreg-frob-ani-pd-env}]
  Let $A$ be an $\mathbb{F}_p$-algebra, $I \subseteq A$ a quasiregular ideal.
  Suppose that the (derived) Frobenius twist $(A / I) \otimes_{A,
  \varphi_A}^{\mathbb{L}} A$ is concentrated in degree $0$, i.e.,
  $\tmop{Tor}_A^i (A / I, A) \cong 0$ (where the last $A$ is viewed as an
  $A$-module via the Frobenius $\varphi_A \of A \rightarrow A$) for all $i \in
  \mathbb{N}_{> 0}$. Then the animated PD-envelope of $(A, I)$ coincides with
  the classical PD-envelope.
\end{theorem}

We want to point out that $(A / I) \otimes_{A, \varphi_A}^{\mathbb{L}} A$
being concentrated in degree $0$ is a very mild assumption. For example, when
$I \subseteq A$ is generated by a Koszul-regular sequence, then this holds
automatically {\cite[Lem~3.41]{Bhatt2012a}}. This also happens when $(A, I)$
comes from a ``good'' PD-envelope, see \Cref{rem:pd-env-frob-flat}. Using
this, we show that

\begin{theorem}[\Cref{prop:koszul-regular-ani-pd-env}]
  Let $A$ be a ring and $I \subseteq A$ an ideal generated by a Koszul-regular
  sequence. Then the animated PD-envelope of $(A, I)$ coincides with the
  classical PD-envelope.
\end{theorem}

Moreover, this mild assumption is not needed if we are only interested in
associated graded pieces of the PD-filtration, which answers a question of
Illusie {\cite[VIII.~Ques~2.2.4.2]{Illusie1972}}:

\begin{theorem}[\Cref{prop:comp-PDFil-equiv-qreg,prop:qreg-assoc-gr-pd}]
  Let $A$ be an $\mathbb{F}_p$-algebra, $I \subseteq A$ a quasiregular ideal.
  Then there is a canonical comparison map from the animated PD-envelope to
  the classical PD-envelope $(B, J, \gamma)$ of $(A, I)$ compatible with
  PD-filtrations which induces equivalences on associated graded pieces.
  Furthermore, these associated graded pieces $\tmop{gr}_{\tmop{PD}}^{\ast}
  B$, as a graded commutative ring, are given by the free divided power $A /
  I$-algebra $\Gamma_{A / I}^{\ast} (I / I^2)$ generated by the $A / I$-module
  $I / I^2$.
\end{theorem}

The key point is that animated PD-envelopes admit natural PD-filtrations of
which we can control the associated graded pieces
(\Cref{prop:comp-PDFil-equiv-qreg}).

Based on animated PD-pairs, we develop a theory of {\tmdfn{derived crystalline
cohomology}} (\Cref{def:deriv-crys-cohomol}) based on a technical construction
called {\tmdfn{derived de Rham cohomology of a map of animated PD-pairs}}
(\Cref{def:deriv-dR-pdpair}) which generalizes the derived de Rham cohomology
of a map of animated rings. In other words, our derived crystalline cohomology
should be understood as a variant of derived de Rham cohomology, not
site-theoretic cohomology. These functors preserve small colimits by
\Cref{prop:deriv-crys-coh-preserv-colim,lem:dR-preserv-colim}, therefore
formal properties such as base change compatibility and ``Künneth'' formula
hold
(\Cref{cor:crys-coh-base-chg,cor:crys-coh-symm-mon,cor:crys-coh-transitive}).

In fact, the animated PD-envelope is, roughly speaking, a special case of
derived crystalline cohomology:

\begin{theorem}[\Cref{prop:crys-PD-env-equiv}]
  Let $(A, I, \gamma_A)$ be a PD-pair and $J \subseteq A$ be an ideal
  containing $I$. Let $(B \twoheadrightarrow A / J, \gamma_B)$ be the relative
  animated PD-envelope of $(A, J)$ with respect to the PD-pair $(A, I,
  \gamma_A)$. Then the underlying $\mathbb{E}_{\infty}$-$\mathbb{Z}$-algebra
  of $B$ is equivalent to the derived crystalline cohomology of $A / J$ with
  respect to $(A, I, \gamma_A)$.
\end{theorem}

From this we deduce a generalization of {\cite[Thm~8.12]{Bhatt2018}} under
quasiregularity and the $\tmop{Tor}$-independent assumption mentioned above.
To see this, similar to the animated PD-envelope, we introduce the
{\tmdfn{conjugate filtration}} on the derived crystalline cohomology
(\Cref{def:conj-fil-crys-coh}) and on the relative animated PD-envelope
(\Cref{def:conj-fil-rel-pd-env}) in characteristic $p$, and we have a similar
control of associated graded pieces for the conjugate filtration on relative
animated PD-envelopes (\Cref{cor:conjfil-gr-rel}) and also on the derived
crystalline cohomology, which is a crystalline variant of the {\tmdfn{Cartier
isomorphism}} (cf. {\cite[Prop~3.5]{Bhatt2012a}} over the base $(A, I, \gamma)
= (\mathbb{F}_p, 0, 0)$):

\begin{theorem}[\Cref{prop:crys-Cartier-isom}]
  Let $(A, I, \gamma)$ be a PD-pair where $A$ is an $\mathbb{F}_p$-algebra.
  Note that the Frobenius map $\varphi_A \of A \rightarrow A$ factors through
  $A \twoheadrightarrow A / I$, giving rise to a natural map $\varphi_{(A, I)}
  \of A / I \rightarrow A$ (cf. \Cref{lem:PD-frob}). Then for every animated
  $A / I$-algebra $R$ and $n \in \mathbb{N}$, the $(- i)$-th associated graded
  piece of the conjugate filtration on the derived crystalline cohomology of
  $R$ relative to $(A, I, \gamma)$ is, as a $\varphi_{(A, I)}^{\ast}
  (R)$-module spectrum, equivalent to $\varphi_{(A, I)}^{\ast} \left(
  \bigwedgestar_R^i L_{R / (A / I)} \right) [- i]$.
\end{theorem}

On the other hand, similar to {\cite{Berthelot1974}}, we develop an
{\tmdfn{affine crystalline site}} (\Cref{def:aff-crys-site}) based on animated
PD-pairs (Bhatt had already indicated such a possibility, see the paragraph
before {\cite[Ex~3.21]{Bhatt2012a}}). Recall that a map $A \rightarrow R$ of
rings is called {\tmdfn{quasisyntomic}} (\Cref{def:qsyn}) if it is flat and
the cotangent complex $L_{R / A}$, as an $R$-module spectrum, has
$\tmop{Tor}$-amplitude in $[0, 1]$. We could also compare the derived
crystalline cohomology to the site-theoretic cohomology:

\begin{theorem}[\Cref{prop:deriv-crys-coh-site-comp,prop:site-comp-qsyn-integral,prop:site-static-comp}]
  Let $(A, I, \gamma_A)$ be a PD-pair and $R$ an $A / I$-algebra.
  \begin{enumerate}
    \item There is a comparison map from the derived crystalline cohomology of
    $R$ with respect to $(A, I, \gamma_A)$ to the cohomology of the affine
    crystalline site, which is an equivalence when as an $A / I$-algebra, $R$
    is either of finite type, or quasisyntomic.
    
    \item There is a comparison map from the cohomology of the affine
    crystalline site to the (classical) crystalline cohomology of $R$ with
    respect to $(A, I, \gamma_A)$. When $R$ is a quasisyntomic $A /
    I$-algebra,
    \begin{enumerate}
      \item Supposing that $p$ is nilpotent in $A$, then the comparison map is
      an equivalence.
      
      \item Supposing that $A$ is $p$-torsion free, then the comparison map
      becomes an equivalence after derived $p$-completion, or equivalently,
      after derived modulo $p$.
    \end{enumerate}
  \end{enumerate}
\end{theorem}

The theorem above generalizes {\cite[Prop~3.25]{Bhatt2012a}} which is
established for syntomic algebras.

We do not know whether the derived crystalline cohomology and the cohomology
of the affine crystalline site are equivalent without any assumption, we
reduced this equivalence to a descent property of the derived crystalline
cohomology ``with respect to the base animated PD-pair''
(\Cref{prop:crys-site-comp-desc}).

In addition to PD-pairs and the crystalline cohomology, we also introduce
{\tmdfn{animated $\delta$-rings}} and {\tmdfn{animated $\delta$-pairs}}, and a
non-complete but animated version of prisms, the static version of which was
introduced in {\cite{Bhatt2019}}. Similar to animated PD-envelopes, the
non-completed animated prismatic envelope, which generalizes\footnote{More
precisely, it is a non-completed version.} the prismatic envelope for local
complete intersections {\cite[Prop~3.13]{Bhatt2019}}, admits the
{\tmdfn{conjugate filtration}} of which the associated graded pieces are
easily determined by a variant of the {\tmdfn{Hodge--Tate comparison}}:

\begin{theorem}[\Cref{thm:Hdg-Tate}]
  Let $(A, d)$ be a prism and $J \subseteq A / d$ an ideal. Then for every $i
  \in \mathbb{N}$, the $(- i)$-th associated graded piece of non-completed
  prismatic envelope, as an $A / (d, J)$-module spectrum, is
  equivalent\footnote{Here we suppress the Breuil--Kisin twists.} to
  $\Gamma_{A / (d, J)}^i (L_{(A / (d, J)) / (A, d)} [- 1])$.
\end{theorem}

As a corollary, similar to animated PD-envelopes, when the ideal $J$ is
$p$-completely quasiregular, roughly speaking, the $(p, d)$-completed animated
prismatic envelope satisfies the universal property of the prismatic envelope
in {\cite[Prop~3.13]{Bhatt2019}} (\Cref{rem:p-compl-prism-env}). Furthermore,
the non-completed prismatic envelope satisfies a faithful flatness
(\Cref{prop:quasireg-prism-env}), which leads to a technical result which is
essentially about the flat cover of the final object
(\Cref{prop:prism-flat-cov-final-obj}), and a similar argument shows the $(p,
d)$-completed variant:

\begin{theorem}[\Cref{prop:compl-prism-flat-cover}]
  Let $(B, d)$ be a bounded oriented prism, $R$ a derived $p$-complete and
  $p$-completely quasisyntomic $B / d$-algebra. Let $P$ be a derived $(p,
  d)$-complete animated $\delta$-$B$-algebra which is $(p, d)$-completely
  quasismooth over $B$, equipped with a surjection $P \twoheadrightarrow R$ of
  $B$-algebras. Then the (completed) prismatic envelope of $P
  \twoheadrightarrow R$ exists and is a flat cover of the final object in the
  prismatic site.
\end{theorem}

We stress that our theory is non-completed. Technically, it is easier to deal
with non-completed version than with $p$-completed version because the
$\infty$-category of $p$-completed objects is usually not projectively
generated. For example, $\mathbb{Z}_p \in D_{\tmop{comp}} (\mathbb{Z}_p)$ is
not a compact object. We could overcome this issue by applying the techniques
developed in \Cref{subsec:refl-subcat}, but it would make the theory
inconvenient.

However, thanks to Clausen--Scholze's condensed mathematics, the non-completed
version could serve a cornerstone of an analytic version which allows us to
put ``topologies'' and ``analytic structures'' on our animated rings.

\subsection{Main techniques}\label{subsec:main-tech}We systematically adopt
two techniques in this article: the {\tmdfn{animation}} and a kind of
local-global principle for $\mathbb{Z}$. We briefly summarize them as follows:

There is a procedure to associate to $1$-projectively generated $1$-categories
projectively generated $\infty$-categories, called the {\tmdfn{animation}}, a
concept abstracted out in {\cite[§5.1]{Cesnavicius2019}}, and defined by the
{\tmdfn{non-abelian derived category}} {\cite[§5.5.8]{Lurie2009}} of a
{\tmdfn{set of compact $1$-projective generators}}.

\begin{example*}
  The abelian category of $R$-modules admits a set of compact $1$-projective
  generators given by free $R$-modules of finite rank. The animation of this
  category is the connective part $D_{\geq 0} (R)$ of the derived category $D
  (R)$.
\end{example*}

\begin{example*}
  The $1$-category of rings admits a set of compact $1$-projective generators
  given by polynomial rings on finitely many variables.
\end{example*}

\begin{remark*}
  It is not a coincidence that the sets of compact $1$-projective generators
  above are given by ``finite free objects''. Indeed, it is a corollary of
  \Cref{prop:adjoint-n-proj-gen}, applied to the pairs $\tmop{Set}
  \rightleftarrows \tmop{Mod}_R$ and $\tmop{Set} \rightleftarrows \tmop{Ring}$
  of adjoint functors.
\end{remark*}

We review the definition of animation and summarize its main properties in
\Cref{subsec:animation}. When applying this construction to the $1$-category
of rings, we get {\tmdfn{the $\infty$-category of animated rings}}. We apply
this construction to the $1$-category of $\delta$-rings, obtaining the
{\tmdfn{$\infty$-category of animated $\delta$-rings}}
(\Cref{def:ani-delta-ring}).

The technical advantage of this construction is that, to produce a
sifted-colimit-preserving functor from a projectively generated
$\infty$-category, it suffices to produce a functor from the full subcategory
spanned by a set of compact projective generators which, as we have seen, is
given by ``finite free objects''.

Now we want to apply this procedure to the $1$-category of ring-ideal pairs.
Unfortunately, the $1$-category of ring-ideal pairs might not be
$1$-projectively generated. However, it is reasonable to say that ``standard''
Koszul-regular pairs $(\mathbb{Z} [x_1, \ldots, x_m, y_1, \ldots, y_n], (y_1,
\ldots, y_n))$ are ``finite free objects''. We pick the non-abelian derived
category of the full category spanned by these pairs, and the $1$-category of
ring-ideal pairs embeds fully faithfully into it
(\Cref{prop:infty-cat-forget-embedding}). This $\infty$-category is equivalent
to the $\infty$-category of surjections of animated rings
(\Cref{thm:ani-smith-eq}). Similarly, we apply this ``modified animation'' to
the $1$-category of PD-pairs, obtaining the {\tmdfn{$\infty$-category of
animated PD-pairs}}. The PD-envelope functor gives rise to the
{\tmdfn{animated PD-envelope}} (\Cref{def:anipair-anipdpair}): a ``good
enough'' pair of adjoint functors between $1$-projectively generated
$1$-categories give rise to a pair of adjoint functors between animations
(\Cref{cor:ani-adjoint-funs}). However, here the story is slightly more
complicated due to our ``modification'' of the animation.

In a similar fashion, we apply these animation techniques to $\delta$-pairs,
obtaining {\tmdfn{animated $\delta$-pairs}} (\Cref{def:ani-delta-pair}), and
we use similar techniques to define and analyze non-completed animated
prismatic envelopes. We also use the animation techniques to define the ``de
Rham context'' $\tmop{dRCon}$, the ``crystalline context'' $\tmop{CrysCon}$,
the derived de Rham cohomology and the derived crystalline cohomology in
\Cref{subsec:deriv-dR}.

Now we describe the second main technique that we use: the local-global
principle for $\mathbb{Z}$. Some techniques are only valid in characteristic
$p$. For example, we do not know how to define the conjugate filtration on the
derived crystalline cohomology beyond characteristic $p$. However, these
arithmetic objects, such as PD-structures, usually degenerate in
characteristic $0$. In view of these, we can usually then glue the results for
each prime number $p \in \mathbb{N}$ and the result after rationalization. The
simplest example of this technique is the following: Let $X \in \tmop{Sp}$ be
a spectrum. Suppose that the spectrum $X /^{\mathbb{L}} p$ is equivalent to
$0$ for every prime number $p \in \mathbb{N}$, and that $X$ is also
contractible after rationalization. Then the spectrum $X$ itself is
contractible. We establish similar results
(\Cref{lem:static-loc-global,lem:flat-loc-global}) under connectivity
assumptions. These results allow us to deduce integral results.

\subsection{Main definitions and constructions}In this subsection, we present
main definitions and constructions appearing in the current work, explaining
the techniques mentioned above in more details. As explained above, we
introduce the following concept:

\begin{definition}[\Cref{subsec:ani-pairs-PD-pairs}]
  The {\tmdfn{$\infty$-category of animated pairs}} is the non-abelian derived
  category of the $1$-category of ring-ideal pairs $(\mathbb{Z} [x_1, \ldots,
  x_m, y_1, \ldots, y_n], (y_1, \ldots, y_n))$. The {\tmdfn{$\infty$-category
  of animated PD-pairs}} is the non-abelian derived category of the
  $1$-category of PD-pairs $\Gamma_{\mathbb{Z} [x_1, \ldots, x_m]} (y_1,
  \ldots, y_n) \twoheadrightarrow \mathbb{Z} [x_1, \ldots, x_m]$.
\end{definition}

Here, as mentioned before, later we will write these pairs as $(\mathbb{Z} [X,
Y], (Y))$ for finite sets $X = \{ x_1, \ldots, x_m \}$ and $Y = \{ y_1,
\ldots, y_n \}$. We will systematically identify the ring-ideal pairs $(A, I)$
with the surjections $A \twoheadrightarrow A / I$ of rings. The notation
$\Gamma_R (y_1, \ldots, y_n)$ denotes the free PD-$R$-algebra\footnote{This is
sometimes denoted by $R \langle y_1, \ldots, y_n \rangle$ in the literature.
We choose our notation to avoid confusion with that of topologically free
algebras also usually denoted by $R \langle y_1, \ldots, y_n \rangle$.}
generated by $y_1, \ldots, y_n$. One can show that the $1$-category of
ring-ideal pairs (resp. PD-pairs) embeds fully faithfully into the
$\infty$-category of animated pairs (resp. animated PD-pairs), see
\Cref{prop:infty-cat-forget-embedding}. Then the first main theorem is the
following

\begin{theorem}[\Cref{thm:ani-smith-eq}]
  The left derived functor of sending the animated pair $(\mathbb{Z} [X, Y],
  (Y))$ to the surjection $\mathbb{Z} [X, Y] \twoheadrightarrow \mathbb{Z}
  [X]$ of animated rings identifies the $\infty$-category of animated pairs
  with the $\infty$-category of surjections of animated rings.
\end{theorem}

One can prove this directly. The proof presented in this article follows an
indirect approach where one first shows an ``linear analogue'' for $D
(\mathbb{Z})_{\geq 0}$ instead of $\tmop{CAlg}^{\tmop{an}}$ and then proves
that this equivalence is ``compatible with the multiplicative structure''.

As explained before, this ``modified animation'' allows us to extend
constructions for ring-ideal pairs $(\mathbb{Z} [X, Y], (Y))$ to animated
pairs (or equivalently, surjections of animated rings). In particular, the
$(Y)$-adic filtration gives rise to the {\tmdfn{adic filtration}}
(\Cref{subsec:quasiregular}), the equivalence $(Y^n) / (Y^{n + 1}) \cong
\tmop{Sym}_{\mathbb{Z} [X]} ((Y) / (Y^2))$ gives rise to its ``derived
version'', and the computation of the cotangent complex $L_{\mathbb{Z} [X]
/\mathbb{Z} [X, Y]} \simeq ((Y) / (Y^2)) [1]$ gives rise to the similar
property for arbitrary surjective maps of animated rings
(\Cref{cor:LAdFil-symm-cot-cx}). The PD-envelope functor, sending $(\mathbb{Z}
[X, Y], (Y))$ to $\Gamma_{\mathbb{Z} [X]} (Y) \twoheadrightarrow \mathbb{Z}
[X]$, gives rise to the {\tmdfn{animated PD-envelope functor}}
(\Cref{subsec:ani-pairs-PD-pairs}). Similarly, the ``modified animation''
allows us to extend the classical PD-filtration on $\Gamma_{\mathbb{Z} [X]}
(Y)$ to the {\tmdfn{PD-filtration}} on animated PD-pairs (\Cref{def:pd-fil}),
and the {\tmdfn{conjugate filtration}} on animated PD-envelopes
(\Cref{def:conj-fil-pd-env}). The result about associated graded pieces of the
PD-filtration (resp. the conjugate filtration) for the ``standard case''
$\Gamma_{\mathbb{Z} [X]} (Y) \twoheadrightarrow \mathbb{Z} [X]$ (resp.
$(\mathbb{Z} [X, Y], (Y))$) extends in a direct manner to animated PD-pairs
(resp. animated pairs), see \Cref{lem:pdfil-free-pdalg} (resp.
\Cref{cor:conjfil-gr}). Furthermore, we know how to detect whether an animated
pair (resp. animated PD-pair) is a classical ring-ideal pair (resp. PD-pair),
see \Cref{prop:characterize-pair} (resp. \Cref{prop:characterize-pdpair}).
Consequently, when the input is a quasiregular ring-ideal pair with some mild
conditions, we can deduce that the animated version coincides with the
classical version
(\Cref{cor:Fp-qreg-rel-pd-env-flat,prop:koszul-regular-ani-pd-env,prop:qreg-assoc-gr-pd}).

Based on animated pairs and animated PD-pairs, we define the {\tmdfn{derived
de Rham cohomology}} associated to a map of animated PD-pairs. More precisely,
we have

\begin{lemma}[\Cref{subsec:deriv-dR}]
  The $\infty$-category of maps of animated PD-pairs is projectively
  generated. A set of compact projective generators is given by
  $(\Gamma_{\mathbb{Z} [X]} (Y) \twoheadrightarrow \mathbb{Z} [X]) \rightarrow
  (\Gamma_{\mathbb{Z} [X, X']} (Y, Y') \twoheadrightarrow \mathbb{Z} [X,
  X'])$.
\end{lemma}

Therefore we can extend the de Rham cohomology defined on these ``standard''
maps of ``standard'' PD-pairs to arbitrary maps of animated PD-pairs. It turns
out that the derived de Rham cohomology of a map $(A \twoheadrightarrow A'',
\gamma_A) \rightarrow (B \twoheadrightarrow B'', \gamma_B)$ of animated
PD-pairs ``does not depend on $B$'' (\Cref{prop:dR-crys-inv}), which is a
generalization of the fact that the crystalline cohomology computed by the de
Rham complex does not depend on the choice of the lift. This leads to the
definition of the derived crystalline cohomology
(\Cref{def:deriv-crys-cohomol}). The Hodge-filtration and the conjugate
filtration also extends from ``standard'' maps of ``standard'' PD-pairs to
maps of animated PD-pairs (\Cref{subsec:dR-Hdg-conj-fil}). The associated
graded pieces of the conjugate filtration is determined by a crystalline
variant of the Cartier isomorphism (\Cref{prop:crys-Cartier-isom}).

Similar to the classical case, there is also a relative concept of animated
PD-envelopes (\Cref{def:rel-PD-env}), and the relative animated PD-envelope
could be reduced to the special case associated to the datum $((A
\twoheadrightarrow A'', \gamma_A), A'' \twoheadrightarrow R)$ where $(A
\twoheadrightarrow A'', \gamma_A)$ is an animated PD-pair and $A''
\twoheadrightarrow R$ is a surjection of animated rings
(\Cref{lem:rel-PD-env-crys-con}). Furthermore, we have

\begin{lemma}[\Cref{lem:cryscon-comp-proj-gen}]
  The $\infty$-category of data $((A \twoheadrightarrow A'', \gamma_A), A''
  \twoheadrightarrow R)$ is projectively generated.
\end{lemma}

Again, we can apply the techniques of non-abelian derived categories to this
category, and along with the local-global principle for $\mathbb{Z}$ mentioned
before, we show that the derived crystalline cohomology ``coincides'' with the
relative animated PD-envelope (\Cref{prop:crys-PD-env-equiv}), which
generalizes Bhatt's computation {\cite[Thm~3.27]{Bhatt2012a}} of the derived
de Rham cohomology $\tmop{dR}_{\mathbb{F}_p /\mathbb{F}_p [x]}$.

In order to compare the derived crystalline cohomology with the classical
crystalline cohomology, we introduce an animated variant of the crystalline
site (\Cref{def:aff-crys-site}). Our main tools are the Čech--Alexander
computation, the Katz--Oda filtration (\Cref{def:Katz-Oda-fil}) introduced in
{\cite{Guo2020}} and the conjugate filtration. More precisely, the Katz--Oda
filtration allows us to identify graded pieces associated to the Hodge
filtration (\Cref{lem:crys-coh-gr-desc}), and the conjugate filtration allows
us to establish a descent-type result in characteristic $p$
(\Cref{lem:crys-coh-char-p-desc}). Along with the local-global principle for
$\mathbb{Z}$ mentioned before, we prove the comparison theorem for
quasisyntomic schemes (\Cref{prop:site-comp-qsyn-integral}).

Finally, we introduce the {\tmdfn{$\infty$-category of animated
$\delta$-rings}} (\Cref{def:ani-delta-ring}), which is simply the animation of
the $1$-category of $\delta$-rings, and

\begin{definition}[\Cref{def:ani-delta-pair}]
  An {\tmdfn{animated $\delta$-pair}} is an animated $\delta$-ring $A$ along
  with a surjection $A \twoheadrightarrow A''$ of animated rings.
\end{definition}

Using similar techniques, we construct the non-complete animated prismatic
envelope (\Cref{cor:prism-env}), the conjugate filtration
(\Cref{lem:conj-fil-prism-env}) and the Hodge--Tate comparison
(\Cref{thm:Hdg-Tate}). Similar to ``animated PD-envelope being classical under
quasiregularity with some mild conditions'', we deduce a flatness result of
prismatic envelopes under quasiregularity (\Cref{prop:quasireg-prism-env}),
which is sufficient for flat cover results mentioned before
(\Cref{prop:prism-flat-cov-final-obj,prop:compl-prism-flat-cover}).

\subsection{Structure of article}Here is a Leitfaden of the article:
\Cref{sec:cat-prep} is devoted to technical preparations. We suggest the
readers skip it in the first reading. \Cref{sec:ani-pairs-pd-pairs} is devoted
to the theory of animated pairs and animated PD-pairs, and to the study of the
animated PD-envelope. \Cref{sec:deriv-crys-coh} is devoted to relative
animated PD-envelopes, derived crystalline cohomology, cohomology of the
affine crystalline site and their comparisons. \Cref{sec:prism} is devoted to
animated $\delta$-rings, animated $\delta$-pairs, ``non-complete animated
prisms'', non-completed animated prismatic envelope and a variant of the
Hodge--Tate comparison. \Cref{app:animation} is a collection of basic facts
about animations and projectively generated categories (which we suggest the
reader read first if they have not seen this concept before).

\subsection{Notations and terminology}In this article, since we often work in
the $\infty$\mbox{-}category of certain ``derived'' categories, we try to
distinguish the ``ordinary'' objects and ``derived'' objects by choosing
different words.

Given an $\infty$\mbox{-}category $\mathcal{C}$ and a diagram $Y \leftarrow X
\rightarrow Z$ in $\mathcal{C}$, the pushout of the diagram is denoted by $Y
\amalg_X Z$. In particular, if $\mathcal{C}$ admits an initial object, the
coproduct of two objects $Y, Z$ is denoted by $Y \amalg Z$.

We will denote by $\tmop{An}$ the $\infty$\mbox{-}category of (small) animæ,
that is, the simplicial nerve of the simplicial category of (small) Kan
complexes {\cite[Def~1.2.16.1]{Lurie2009}}.

We say that an anima $X \in \tmop{An}$ or a spectrum $X \in \tmop{Sp}$ is
{\tmdfn{static}}\footnote{This is usually called {\tmdfn{discrete}} in
homotopy theory. We follow Clausen--Scholze's terminology in condensed
mathematics to call them {\tmdfn{static}} to distinguish from the point-set
topological discreteness. In particular, the {\tmdfn{static}} object
$\mathbb{Z}_p$ might be equipped with the $p$-adic topology which is different
from the {\tmdfn{discrete}} topology.} if $\pi_i (X) \cong 0$ for all $i \neq
0$. For two spectra $X, Y \in \tmop{Sp}$, we will denote by $X
\otimes^{\mathbb{L}} Y$ the smash product. {\tmdfn{Rings}} are always static
and commutative, while {\tmdfn{$\mathbb{E}_n$-rings}} are
$\mathbb{E}_n$-algebras in the symmetric monoidal $\infty$\mbox{-}category
$(\tmop{Sp}, \otimes^{\mathbb{L}})$.

Given a ring $A$, we will refer to a ``classical'' $A$-module a {\tmdfn{static
$A$-module}}. The category of static $A$-modules will be denoted by
$\tmop{Mod}_A$. The category of {\tmdfn{ring-module pairs}} $(A, M)$ where $M
\in \tmop{Mod}_A$ is denoted by $\tmop{Mod}$. An object in the derived
$\infty$\mbox{-}category $D (A)$ an {\tmdfn{$A$-module spectrum}}.

Given an $\mathbb{E}_1$-ring $A$, the $\infty$\mbox{-}category of left (resp.
right) $A$-module spectra will be denoted by $\tmop{LMod}_A$ (resp.
$\tmop{RMod}_A$). Given a right $A$-module spectrum $M$ and a left $A$-module
spectrum $N$, their relative tensor product is denoted by $M
\otimes_A^{\mathbb{L}} N$, to avoid confusion with the ordinary tensor product
of static modules.

Given an $\mathbb{E}_{\infty}$-ring $A$, the $\infty$\mbox{-}category of
$A$-module spectra is denoted by $D (A)$. In particular, we have $\tmop{Sp} =
D (\mathbb{S})$. An {\tmdfn{$\mathbb{E}_n$-$A$-algebra}} is an
$\mathbb{E}_n$-algebra in the symmetric monoidal $\infty$\mbox{-}category $(D
(A), \otimes_A^{\mathbb{L}})$.

Following {\cite{Bhatt2022a}}, we will denote by $\tmop{CAlg}^{\tmop{an}}$ the
$\infty$-category of animated rings, and $\tmop{CAlg}_R^{\tmop{an}}$ the
$\infty$-category of animated $R$-algebras for an animated ring $R$.

\begin{acknowledgments*}
  The author thanks their thesis advisor Matthew {\tmname{Morrow}} for various
  suggestions and patient readings during the construction of this article
  (and more). We also thank Denis {\tmname{Nardin}} for discussions about
  $\infty$-categories and in particular, of simplicial homotopy theory in
  $\infty$-categories, and Yu {\tmname{Min}} for several discussions. We thank
  Lenny {\tmname{Taelman}} for pointing out a mistake in an early version. The
  author would also like to thank Bhargav {\tmname{Bhatt}}, Kęstutis
  {\tmname{Česnavičius}} for reading this paper, and Lukas
  {\tmname{Brantner}}, Ofer {\tmname{Gabber}}, Kirill {\tmname{Magidson}},
  Maxime {\tmname{Ramzi}}, Wouter {\tmname{Rienks}} for discussions. This
  project has received funding from the European Research Council (ERC) under
  the European Union's Horizon 2020 research and innovation programme (grant
  agreement No.~851146).
\end{acknowledgments*}

\section{Categorical preparations}\label{sec:cat-prep}

In this section, we will do some technical preparations of
$\infty$\mbox{-}categories which will be used throughout this article. We try
our best to refer to this section explicitly so that the reader could first
skip this section and read back when needed.

\subsection{Animation of adjoint functors}This subsection is devoted to
proving that animation behaves well for certain ``monadic'' pairs of adjoint
functors. Here is a general lemma.

\begin{lemma}
  \label{lem:left-deriv-fun-adjoint}Let $n \in \mathbb{N}_{> 0} \cup \{ \infty
  \}$. Let $\mathcal{C}$ be a small $n$\mbox{-}category which admits finite
  coproducts and $\mathcal{D}$ a locally small $n$\mbox{-}category which
  admits small colimits. Let $f \of \mathcal{C} \rightarrow \mathcal{D}$ be a
  functor which preserves finite coproducts. Then
  \begin{enumerate}
    \item There is a pair of adjoint functors $\mathcal{P}_{\Sigma, n}
    (\mathcal{C}) \underset{G}{\overset{F}{\longrightleftarrows}} \mathcal{D}$
    (\Cref{nota:Psigma-n}) where $F$ is the left derived functor
    (\Cref{prop:left-deriv-n-fun}) of $f$ and $G$ is the functor given by $D
    \mapsto \tmop{Map}_{\mathcal{D}} (f (\cdummy), D) \in \mathcal{P}
    (\mathcal{C})$.
    
    \item \label{pt:l2.5p2}Suppose that for all objects $C \in \mathcal{C}$,
    the object $f (C) \in \mathcal{D}$ is compact and $n$\mbox{-}projective.
    Then the functor $G$ preserves filtered colimits and geometric
    realizations. Under this assumption, if $f$ is further assumed to be fully
    faithful, then so is $F$.
    
    \item Suppose that the set $\{ f (C) \barsuchthat C \in \mathcal{C} \}
    \subseteq \mathcal{D}$ generates $\mathcal{D}$ under small colimits. Then
    the functor $G$ is conservative.
  \end{enumerate}
\end{lemma}

\begin{proof}
  We exhibit the proof for $n = \infty$. First, the functor $f \of \mathcal{C}
  \rightarrow \mathcal{D}$ extends uniquely to a functor $\tilde{F} \of
  \mathcal{P} (\mathcal{C}) \rightarrow \mathcal{D}$ which preserves small
  colimits by {\cite[Thm~5.1.5.6]{Lurie2009}}. Since $\mathcal{P}_{\Sigma}
  (\mathcal{C}) \subseteq \mathcal{P} (\mathcal{C})$ is stable under sifted
  colimits, it follows that the functor $F$ is equivalent to the composite
  $\mathcal{P}_{\Sigma} (\mathcal{C}) \hookrightarrow \mathcal{P}
  (\mathcal{C}) \xrightarrow{\tilde{F}} \mathcal{D}$. The functor $\tilde{F}$
  admits a right adjoint by {\cite[Cor~5.2.6.5]{Lurie2009}} which is
  equivalent to the composite $\mathcal{D} \xrightarrow{G}
  \mathcal{P}_{\Sigma} (\mathcal{C}) \hookrightarrow \mathcal{P}
  (\mathcal{C})$, therefore $(F, G)$ is a pair of adjoint functors.
  
  Part \ref{pt:l2.5p2} follows from the fact that $\mathcal{P}_{\Sigma}
  (\mathcal{C}) \subseteq \mathcal{P} (\mathcal{C})$ is stable under sifted
  colimits (\Cref{prop:Psigma-n}). The later statement follows from
  \Cref{prop:left-deriv-n-full}.
  
  Suppose that $\{ f (C) \barsuchthat C \in \mathcal{C} \}$ generates
  $\mathcal{D}$ under small colimits, then for any map $X \rightarrow Y$ in
  $\mathcal{D}$, if the induced map $G (X) \rightarrow G (Y)$ is an
  equivalence in $\mathcal{P}_{\Sigma} (\mathcal{C})$, then for all objects $C
  \in \mathcal{C}$, the induced map $\tmop{Map}_{\mathcal{D}} (f (C), X)
  \rightarrow \tmop{Map}_{\mathcal{D}} (f (C), Y)$ is an equivalence. Let
  $\mathcal{D}' \subseteq \mathcal{D}$ be the full subcategory spanned by
  those $D \in \mathcal{D}$ such that the induced map
  $\tmop{Map}_{\mathcal{D}} (D, X) \rightarrow \tmop{Map}_{\mathcal{D}} (D,
  Y)$ is an equivalence. Then $\mathcal{D}'$ is stable under colimits, and $f
  (C) \in \mathcal{D}'$ for all $C \in \mathcal{C}$. The result follows.
\end{proof}

It then follows from \Cref{lem:nonab-deriv-cat-n-proj-gen} and
\Cref{cor:proj-gen-n-sifted} that

\begin{corollary}
  \label{cor:nonab-deriv-cat-adjoint-fun}Let $\mathcal{C}, \mathcal{D}$ be two
  small $n$\mbox{-}categories which admit finite coproducts and $f \of
  \mathcal{C} \rightarrow \mathcal{D}$ a functor which preserves finite
  coproducts. Then
  \begin{enumerate}
    \item There is a pair of adjoint functors $\mathcal{P}_{\Sigma, n}
    (\mathcal{C}) \underset{G}{\overset{F}{\longrightleftarrows}}
    \mathcal{P}_{\Sigma, n} (\mathcal{D})$ where $G$ is given by
    $\mathcal{P}_{\Sigma, n} (\mathcal{D}) \ni H \mapsto H \circ F \in
    \mathcal{P}_{\Sigma, n} (\mathcal{C})$ and $F$ is the left derived functor
    of the composite functor $\mathcal{C} \xrightarrow{f} \mathcal{D}
    \hookrightarrow \mathcal{P}_{\Sigma, n} (\mathcal{D})$.
    
    \item The functor $G$ preserves sifted colimits, and the canonical map
    $\tau_{\leq m} \circ G \rightarrow G \circ \tau_{\leq m}$ of functors is
    an equivalence for all $m \in \mathbb{N}$ (cf.
    {\cite[Rem~5.5.8.26]{Lurie2009}} and the discussion before
    \Cref{lem:ani-trunc}).
    
    \item If $f$ is fully faithful, then so is the functor $F$.
    
    \item If $f$ is essentially surjective, then the functor $G$ is
    conservative.
  \end{enumerate}
\end{corollary}

Now we apply this to animations:

\begin{corollary}
  \label{cor:ani-adjoint-funs}Let $\mathcal{C}
  \underset{G}{\overset{F}{\longrightleftarrows}} \mathcal{D}$ be a pair of
  adjoint functors between $1$\mbox{-}categories such that
  \begin{enumerate}
    \item The $1$\mbox{-}category $\mathcal{D}$ admits filtered colimits and
    reflexive coequalizers (or equivalently, geometric realizations, by
    \Cref{rem:n-geom-real}), and $G$ preserves filtered colimits and reflexive
    coequalizers.
    
    \item The $1$\mbox{-}category $\mathcal{C}$ is projectively generated.
    
    \item The functor $G$ is conservative.
  \end{enumerate}
  Then $\mathcal{D}$ is $1$\mbox{-}projectively generated, and we have a pair
  $\tmop{Ani} (\mathcal{C}) \underset{\tmop{Ani} (G)}{\overset{\tmop{Ani}
  (F)}{\longrightleftarrows}} \tmop{Ani} (\mathcal{D})$ of adjoint functors
  between $\infty$\mbox{-}categories after animation. Furthermore, the functor
  $\tmop{Ani} (G)$ is conservative, preserves sifted colimits, and the
  canonical map $\tau_{\leq 0} \circ \tmop{Ani} (G) \rightarrow G \circ
  \tau_{\leq 0}$ of functors is an equivalence. If $G$ preserves small
  colimits, then so does $\tmop{Ani} (G)$.
\end{corollary}

\begin{proof}
  It follows from \Cref{prop:adjoint-n-proj-gen} that the $1$\mbox{-}category
  $\mathcal{D}$ is $1$\mbox{-}projectively generated, therefore $\mathcal{C},
  \mathcal{D}$ admit small colimits which are preserved by $F$. Furthermore,
  let $\mathcal{C}^0 \subseteq \mathcal{C}$ be the full subcategory spanned by
  finite coproducts of a chosen set of compact $1$\mbox{-}projective
  generators for $\mathcal{C}$, and $\mathcal{D}^0 \subseteq \mathcal{D}$ the
  full subcategory spanned by the images of objects of $\mathcal{C}$ under
  $F$, then there are equivalences $\mathcal{C} \simeq \mathcal{P}_{\Sigma, 1}
  (\mathcal{C}^0)$ and $\mathcal{D} \simeq \mathcal{P}_{\Sigma, 1}
  (\mathcal{D}^0)$ of $1$\mbox{-}categories by
  \Cref{prop:struct-n-proj-gen-cats} (note that $F$ preserves finite
  coproducts).
  
  Let $f \of \mathcal{C}^0 \rightarrow \mathcal{D}^0$ be the functor induced
  by $F$, which preserves finite coproducts and is essentially surjective. It
  follows from \Cref{cor:nonab-deriv-cat-adjoint-fun} with $n = 1$ and the
  uniqueness of the right adjoint functor that the functor $G \of \mathcal{D}
  \rightarrow \mathcal{C}$ is equivalent to $\mathcal{P}_{\Sigma, 1}
  (\mathcal{D}^0) \rightarrow \mathcal{P}_{\Sigma, 1} (\mathcal{C}^0), H
  \mapsto H \circ f$.
  
  We invoke again \Cref{cor:nonab-deriv-cat-adjoint-fun} with $n = \infty$ to
  obtain a pair of adjoint functors $\mathcal{P}_{\Sigma} (\mathcal{C}^0)
  \rightleftarrows \mathcal{P}_{\Sigma} (\mathcal{D}^0)$ induced by $f$. It
  follows from the definitions that $\tmop{Ani} (\mathcal{C}) \simeq
  \mathcal{P}_{\Sigma} (\mathcal{C}^0)$, $\tmop{Ani} (\mathcal{D}) \simeq
  \mathcal{P}_{\Sigma} (\mathcal{D}^0)$ and that the functor
  $\mathcal{P}_{\Sigma} (\mathcal{C}^0) \rightarrow \mathcal{P}_{\Sigma}
  (\mathcal{D}^0)$ obtained above is equivalent to $\tmop{Ani} (F)$. Let $G'
  \of \tmop{Ani} (\mathcal{D}) \rightarrow \tmop{Ani} (\mathcal{C})$ be the
  right adjoint to $\tmop{Ani} (F)$. Since $f$ is essentially surjective, $G'$
  is conservative. It remains to show that $G'$ is equivalent to $\tmop{Ani}
  (G)$.
  
  Indeed, both $G'$ and $\tmop{Ani} (G)$ preserve sifted colimits. Since the
  functor $G \of \mathcal{D} \rightarrow \mathcal{C}$ is equivalent to
  $\mathcal{P}_{\Sigma, 1} (\mathcal{D}^0) \rightarrow \mathcal{P}_{\Sigma, 1}
  (\mathcal{C}^0), H \mapsto H \circ f$, the restrictions of $G'$ and
  $\tmop{Ani} (G)$ to the full subcategory $\mathcal{D}^0 \subseteq \tmop{Ani}
  (\mathcal{D})$ are equivalent. It then follows from
  \Cref{prop:left-deriv-n-fun} that $G'$ and $\tmop{Ani} (G)$ are equivalent.
  The colimit preserving properties follow from \Cref{cor:ani-preserve-colim}.
\end{proof}

Now we look at two simple examples:

\begin{example}
  Let $R \rightarrow S$ be a map of rings. Then there is a pair $\tmop{Mod}_R
  \overset{\cdummy \otimes_R S}{\longrightleftarrows} \tmop{Mod}_S$ of adjoint
  functors between the categories of static modules. Since the forgetful
  functor $\tmop{Mod}_S \rightarrow \tmop{Mod}_R$ is conservative, and
  preserves small colimits, we have the pair of adjoint functors $\tmop{Ani}
  (\tmop{Mod}_R) \overset{\tmop{Ani} (\cdummy \otimes_R
  S)}{\longrightleftarrows} \tmop{Ani} (\tmop{Mod}_S)$. Under the equivalences
  $\tmop{Ani} (\tmop{Mod}_R) \simeq D_{\geq 0} (R)$ and $\tmop{Ani}
  (\tmop{Mod}_S) \simeq D_{\geq 0} (S)$, the functor $\tmop{Ani} (\cdummy
  \otimes_R S)$ is equivalent to the functor $\cdummy \otimes_R^{\mathbb{L}}
  S$.
\end{example}

\begin{example}
  Let $\tmop{Ring}$ be the $1$\mbox{-}category of rings and $\tmop{Ab}$ the
  $1$\mbox{-}category of abelian groups. Then we have a pair $\tmop{Ab}
  \overset{\tmop{Sym}_{\mathbb{Z}}}{\longrightleftarrows} \tmop{Ring}$ of
  adjoint functors. Since the forgetful functor $\tmop{Ring} \rightarrow
  \tmop{Ab}$ is conservative, and preserves filtered colimits and reflexive
  coequalizers, we get a pair $D_{\geq 0} (\mathbb{Z}) \overset{\mathbb{L}
  \tmop{Sym}_{\mathbb{Z}}}{\longrightleftarrows} \tmop{CAlg}^{\tmop{an}}$ of
  adjoint functors.
\end{example}

In \Cref{cor:ani-adjoint-funs}, the functor $G$ (resp. $\tmop{Ani} (G)$)
exhibits $\mathcal{D}$ (resp. $\tmop{Ani} (\mathcal{D})$) as monadic over
$\mathcal{C}$ (resp. $\tmop{Ani} (\mathcal{C})$). The associated endomorphism
monad is given by $G \circ F$ (resp. $\tmop{Ani} (G) \circ \tmop{Ani} (F)
\simeq \tmop{Ani} (G \circ F)$ by \Cref{prop:ani-composite}).

\begin{lemma}
  Let $\mathcal{C} \underset{G}{\overset{F}{\longrightleftarrows}}
  \mathcal{D}$ be a pair of adjoint functors between
  $\infty$\mbox{-}categories. Let $K$ be a small simplicial set. Then $G \circ
  F$ preserves $K$-indexed colimits if $G$ preserves $K$-indexed colimits. The
  converse is true if $G$ exhibits $\mathcal{D}$ as monadic over
  $\mathcal{C}$.
\end{lemma}

\begin{proof}
  If $G$ preserves $K$-indexed colimits, since $F$ is a left adjoint, it
  follows that so does $T \assign G \circ F$. Conversely, if $G$ exhibits
  $\mathcal{D}$ as monadic over $\mathcal{C}$, then $\mathcal{D} \simeq
  \tmop{LMod}_T (\mathcal{C})$ and the result follows from
  {\cite[Cor~4.2.3.5]{Lurie2017}}.
\end{proof}

\subsection{Diagram categories and undercategories}In this subsection, we will
show that diagram $n$\mbox{-}categories and undercategories of
$n$\mbox{-}projectively generated categories are $n$\mbox{-}projectively
generated, for which we give an explicit choice of $n$\mbox{-}projective
generators. We first show the version for $\infty$\mbox{-}categories, then
list the analogues for $n$\mbox{-}categories for which the proof is nearly
verbatim. We start with diagram categories.

\begin{lemma}
  \label{lem:prod-proj-gen}Let $(\mathcal{C}_{\alpha})_{\alpha \in T}$ be a
  small collection of projectively generated $\infty$\mbox{-}category. Then
  the $\infty$\mbox{-}category $\prod_{\alpha \in T} \mathcal{C}_{\alpha}$ is
  projectively generated. More precisely, let $1_{\alpha}$ denote the initial
  objects of $\mathcal{C}_{\alpha}$. If the collections $S_{\alpha} \subseteq
  \mathcal{C}_{\alpha}$ of objects are sets of compact projective generators
  for $\mathcal{C}_{\alpha}$, then the collection $\{ i_{s, \beta}
  \barsuchthat s \in S_{\beta}, \beta \in T \} \subseteq \prod_{\alpha \in T}
  \mathcal{C}_{\alpha}$ is a set of compact projective generators for
  $\prod_{\alpha \in T} \mathcal{C}_{\alpha}$, where $i_{s, \beta} \in
  \prod_{\alpha \in T} \mathcal{C}_{\alpha}$ is given by $\left(
  \left\{\begin{array}{ll}
    s & \beta' = \beta\\
    1_{\beta'} & \beta' \neq \beta
  \end{array}\right. \right)_{\beta' \in T}$.
\end{lemma}

\begin{proof}
  Since the small colimits in $\prod \mathcal{C}_{\alpha}$ are computed
  pointwise, it follows that $\prod \mathcal{C}_{\alpha}$ is cocomplete. Now
  given $S_{\alpha}$ and $i_{s, \beta}$, let $\mathcal{D} \subseteq \prod
  \mathcal{C}_{\alpha}$ be the full subcategory generated by $\{ i_{s, \beta}
  \}$ under colimits. For all $\beta \in T$, the fully faithful embedding
  $j_{\beta} \of \mathcal{C}_{\beta} \rightarrow \prod C_{\alpha}$ given by $C
  \mapsto \left( \left\{\begin{array}{ll}
    C & \beta' = \beta\\
    1_{\beta'} & \beta' \neq \beta
  \end{array}\right. \right)_{\beta' \in T}$ preserves small colimits, and
  $j_{\beta} (s) = i_{s, t}$. Thus the ``skyscraper'' functor $j_{\beta} (C)$
  is an object of $\mathcal{D}$ for $C \in \mathcal{C}_{\beta}$.
  
  Finally, we can write any object $F \in \prod \mathcal{C}_{\alpha}$ as a
  small colimit $\tmop{colim}_{\beta \in T} j_{\beta} (F_{\beta})$, therefore
  $\mathcal{D}= \prod \mathcal{C}_{\alpha}$.
\end{proof}

Now let $\mathcal{C}$ be a cocomplete $\infty$\mbox{-}category, $K \in
\tmop{Set}_{\Delta}$ a small simplicial set and $K_0 \subseteq K$ the set of
vertices. Then we have a pair of adjoint functors $\tmop{Fun} (K_0,
\mathcal{C}) \underset{(K_0 \rightarrow K)^{\ast}}{\overset{\tmop{Lan}_{K_0
\rightarrow K}}{\longrightleftarrows}} \tmop{Fun} (K, \mathcal{C})$ where
$\tmop{Lan}_{K_0 \hookrightarrow K}$ is the functor of left Kan extension
along the map $K_0 \rightarrow K$, and $(K_0 \rightarrow K)^{\ast}$ denotes
the restriction along $K_0 \rightarrow K$.

\begin{warning}
  In an early draft, we called $K_0 \rightarrow K$ an ``inclusion''. However,
  any map of simplicial sets is equivalent to a cofibration up to a trivial
  fibration in Joyal model structure. That is to say, the concept of
  ``non-full subcategory'' is not model-independent. We decided to suppress
  such model-dependent expressions.
\end{warning}

It then follows from \Cref{prop:adjoint-n-proj-gen} and
\Cref{lem:prod-proj-gen} that

\begin{corollary}
  \label{cor:fun-cat-proj-gen}Let $\mathcal{C}$ be a projectively generated
  $\infty$\mbox{-}category and $K \in \tmop{Set}_{\Delta}$ a small simplicial
  set. Then the $\infty$\mbox{-}category $\tmop{Fun} (K, \mathcal{C})$ of
  functors is projectively generated.
\end{corollary}

Next, we study undercategories.

\begin{lemma}
  \label{lem:undercat-proj-gen}Let $\mathcal{C}$ be a projectively generated
  $\infty$\mbox{-}category and $Z \in \mathcal{C}$ an object. Then the
  undercategory $\mathcal{C}_{Z \mathord{/}}$ is projectively generated. More
  precisely, letting $S \subseteq \mathcal{C}$ be a set of projective
  generators for $\mathcal{C}$, then the set $\{ Z \rightarrow X \amalg Z
  \barsuchthat X \in S \}$ is a set of compact projective generators for the
  undercategory $\mathcal{C}_{Z \mathord{/}}$.
\end{lemma}

\begin{proof}
  Consider the pair $\mathcal{C} \underset{Y \mapsfrom (Z \rightarrow
  Y)}{\overset{X \mapsto (Z \rightarrow X \amalg Z)}{\longrightleftarrows}}
  \mathcal{C}_{Z \mathord{/}}$ of adjoint functors. The forgetful functor
  $\mathcal{C}_{Z \mathord{/}} \rightarrow \mathcal{C}$
  \begin{itemize}
    \item is conservative, since an object in $\mathcal{C}_{Z \mathord{/}}$
    could be identified with a map $\Delta^1 \rightarrow \mathcal{C}, 0
    \mapsto Z$, and a map in $\mathcal{C}_{Z \mathord{/}}$ between two objects
    could be identified with a homotopy between two maps $\Delta^1
    \rightrightarrows \mathcal{C}$, then we invoke
    {\cite[\href{https://kerodon.net/tag/01DK}{Tag 01DK}]{kerodon}} to
    conclude.
    
    \item preserves sifted colimits, as it is a left fibration
    {\cite[\href{https://kerodon.net/tag/018F}{Tag 018F}]{kerodon}}, thus
    preserves weakly contractible colimits
    {\cite[\href{https://kerodon.net/tag/02KT}{Tag 02KT}]{kerodon}}, and
    sifted diagrams are weakly contractible
    {\cite[\href{https://kerodon.net/tag/02QL}{Tag 02QL}]{kerodon}}.
  \end{itemize}
  We then invoke \Cref{prop:adjoint-n-proj-gen} to conclude\footnote{We
  believe that our argument could be vastly simplified. However, we point out
  that the map $K \hookrightarrow \{ \ast \} \star K$ is not necessarily
  cofinal if the simplicial set $K$ is not sifted. For example, take $K$ to be
  a discrete set with at least two elements.}.
\end{proof}

Now we list the $n$-categorical analogues:

\begin{lemma}
  \label{lem:fun-n-cat-proj-gen}Let $\mathcal{C}$ be an
  $n$\mbox{-}projectively generated $n$\mbox{-}category and $K \in
  \tmop{Set}_{\Delta}$ a small simplicial set. Then the $n$\mbox{-}category
  $\tmop{Fun} (K, \mathcal{C})$ of functors is $n$\mbox{-}projectively
  generated.
\end{lemma}

\begin{lemma}
  Let $\mathcal{C}$ be an $n$\mbox{-}projectively generated
  $n$\mbox{-}category and $Z \in \mathcal{C}$ an object. Then the
  undercategory $\mathcal{C}_{Z \mathord{/}}$ is $n$\mbox{-}projectively
  generated. More precisely, let $S \subseteq \mathcal{C}$ be a set of
  $n$\mbox{-}projective generators for $\mathcal{C}$, then the set $\{ Z
  \rightarrow X \amalg Z \barsuchthat X \in S \}$ is a set of compact
  $n$\mbox{-}projective generators for the undercategory $\mathcal{C}_{Z
  \mathord{/}}$.
\end{lemma}

Now we deduce the corollaries for animations.

\begin{corollary}
  \label{cor:ani-arrow-cat}Let $\mathcal{C}$ be an $n$\mbox{-}projectively
  generated $n$\mbox{-}category. Then there is a canonical equivalence
  $\tmop{Ani} (\tmop{Fun} ((\Delta^1)^{\tmop{op}}, \mathcal{C})) \rightarrow
  \tmop{Fun} ((\Delta^1)^{\tmop{op}}, \tmop{Ani} (\mathcal{C}))$ of
  $\infty$-categories, or equivalently, a canonical equivalence $\tmop{Ani}
  (\tmop{Fun} (\Delta^1, \mathcal{C})) \rightarrow \tmop{Fun} (\Delta^1,
  \tmop{Ani} (\mathcal{C}))$ of $\infty$-categories.
\end{corollary}

\begin{proof}
  Let $S \subseteq \mathcal{C}$ be a set of compact $n$\mbox{-}projective
  generators for $\mathcal{C}$. Spelling out the proof of
  \Cref{cor:fun-cat-proj-gen} (more precisely, its analogue
  \Cref{lem:fun-n-cat-proj-gen}), we extract an explicit set of compact
  $n$\mbox{-}projective generators for $\tmop{Fun} ((\Delta^1)^{\tmop{op}},
  \mathcal{C})$, namely, $T \assign \{ X \leftarrow 0 \barsuchthat X \in S \}
  \cup \left\{ X \xleftarrow{\tmop{id}_X} X \barsuchthat X \in S \right\}$.
  Note that $\tmop{Fun} ((\Delta^1)^{\tmop{op}}, \mathcal{C}) \subseteq
  \tmop{Fun} ((\Delta^1)^{\tmop{op}}, \tmop{Ani} (\mathcal{C}))$ is a full
  subcategory, and again by the proof of \Cref{cor:fun-cat-proj-gen}, it
  follows that $T$ is a set of compact projective generators for $\tmop{Fun}
  ((\Delta^1)^{\tmop{op}}, \tmop{Ani} (\mathcal{C}))$. The result follows.
\end{proof}

The same proof leads to the following (compare with
{\cite[Cons~4.3.4]{Raksit2020}}).

\begin{corollary}
  \label{cor:ani-undercat}Let $\mathcal{C}$ be an $n$\mbox{-}projectively
  generated $n$\mbox{-}category. Then there are canonical equivalences
  \begin{eqnarray*}
    \tmop{Ani} (\tmop{Fun} ((\mathbb{Z}, \geq), \mathcal{C})) &
    \longrightarrow & \tmop{Fun} ((\mathbb{Z}, \geq), \tmop{Ani}
    (\mathcal{C}))\\
    \tmop{Ani} (\tmop{Fun} (\mathbb{Z}, \mathcal{C})) & \longrightarrow &
    \tmop{Fun} (\mathbb{Z}, \tmop{Ani} (\mathcal{C}))\\
    \tmop{Ani} (\tmop{Fun} (\{ 0, 1 \}, \mathcal{C})) & \longrightarrow &
    \tmop{Fun} (\{ 0, 1 \}, \tmop{Ani} (\mathcal{C}))\\
    \tmop{Ani} \left( \mathcal{C}_{Z \mathord{/}} \right) & \longrightarrow &
    \tmop{Ani} (\mathcal{C})_{Z \mathord{/}}
  \end{eqnarray*}
  of $\infty$-categories. The same for replacing $\mathbb{Z}$'s by
  $\mathbb{N}$'s.
\end{corollary}

\subsection{Comma categories and base change}\label{subsec:comma-cat}In this
subsection, we will discuss comma categories, which serves as our basic
language to discuss various base changes.

\begin{definition}
  Let $\mathcal{C}, \mathcal{D}$ be $\infty$\mbox{-}categories and $F \of
  \mathcal{C} \rightarrow \mathcal{D}$ a functor. The {\tmdfn{comma
  category}}, sometimes denoted by $F \mathrel{\downarrow} \mathcal{D}$, is
  given by the simplicial set $\mathcal{C} \times_{\tmop{Fun} (\{ 0 \},
  \mathcal{D})} \tmop{Fun} (\Delta^1, \mathcal{D})$, where the map
  $\mathcal{C} \rightarrow \tmop{Fun} (\{ 0 \}, \mathcal{D})$ is given by $F$
  and the map $\tmop{Fun} (\Delta^1, \mathcal{D}) \rightarrow \tmop{Fun} (\{ 0
  \}, \mathcal{D})$ is induced by the vertex $\{ 0 \} \rightarrow \Delta^1$.
\end{definition}

\begin{example}
  \label{ex:comma-ani-ring}Consider the functor
  $\tmop{id}_{\tmop{CAlg}^{\tmop{an}}} \of \tmop{CAlg}^{\tmop{an}} \rightarrow
  \tmop{CAlg}^{\tmop{an}}$. The comma category
  \[ \tmop{CAlg}^{\tmop{an}} \times_{\tmop{Fun} (\{ 0 \},
     \tmop{CAlg}^{\tmop{an}})} \tmop{Fun} (\Delta^1, \tmop{CAlg}^{\tmop{an}})
  \]
  is equivalent to $\tmop{Fun} (\Delta^1, \tmop{CAlg}^{\tmop{an}})$. An object
  is simply given by a base $A \in \tmop{CAlg}^{\tmop{an}}$ and an $A$-algebra
  $A \rightarrow R$.
\end{example}

\begin{example}
  \label{ex:comma-PD-pair}Consider the functor $\tmop{Pair} \rightarrow
  \tmop{Ring}, (A, I) \mapsto A / I$ and the composite functor
  $\tmop{Pair}^{\gamma} \rightarrow \tmop{Pair} \rightarrow \tmop{Ring}$.
  Concretely, the objects in the comma category $\tmop{Pair}^{\gamma}
  \times_{\tmop{Fun} (\{ 0 \}, \tmop{Ring})} \tmop{Fun} (\Delta^1,
  \tmop{Ring})$ are given by a PD-pair $(A, I, \gamma)$ along with an $A /
  I$-algebra $A / I \rightarrow R$. This is the non-animated version of
  $\tmop{CrysCon}$ that will be introduced in \Cref{subsec:deriv-dR}.
\end{example}

\begin{remark}
  \label{rem:prism-comma-cat}A similar comma category plays an role for
  prismatic cohomology. We will study a non-complete version in
  \Cref{subsec:prism-conj-fil}.
\end{remark}

\begin{lemma}
  Let $\mathcal{C}, \mathcal{D}$ be $\infty$\mbox{-}categories and $F \of
  \mathcal{C} \rightarrow \mathcal{D}$ a functor. Then the simplicial set
  $\mathcal{C} \times_{\tmop{Fun} (\{ 0 \}, \mathcal{D})} \tmop{Fun}
  (\Delta^1, \mathcal{D})$ is an $\infty$\mbox{-}category.
\end{lemma}

\begin{proof}
  It follows from {\cite[Corollary~2.3.2.5]{Lurie2009}} applied to the inner
  fibration $\mathcal{D} \rightarrow \{ \ast \}$ that $\tmop{Fun} (\Delta^1,
  \mathcal{D}) \rightarrow \tmop{Fun} (\{ 0 \}, \mathcal{D})$ is an inner
  fibration. Then it follows {\cite[Corollary~2.4.6.5]{Lurie2009}} that
  $\tmop{Fun} (\Delta^1, \mathcal{D}) \rightarrow \tmop{Fun} (\{ 0 \},
  \mathcal{D})$ is a categorical fibration. The result follows.
\end{proof}

\begin{remark}
  \label{rem:proj-comma-sect}The canonical projection $\mathcal{C}
  \times_{\tmop{Fun} (\{ 0 \}, \mathcal{D})} \tmop{Fun} (\Delta^1,
  \mathcal{D}) \rightarrow \mathcal{C}$ admits a fully faithful section
  induced by $\mathcal{D} \rightarrow \tmop{Fun} (\Delta^1, \mathcal{D}), D
  \mapsto \tmop{id}_D$ which is also a left adjoint of the projection in
  question.
\end{remark}

\begin{lemma}
  Let $\mathcal{C}, \mathcal{D}$ be $\infty$\mbox{-}categories and $F \of
  \mathcal{C} \rightarrow \mathcal{D}$ a functor. Suppose that $\mathcal{D}$
  admits finite coproducts. Then the functor $\mathcal{C} \times_{\tmop{Fun}
  (\{ 0 \}, \mathcal{D})} \tmop{Fun} (\Delta^1, \mathcal{D}) \rightarrow
  \mathcal{C} \times \mathcal{D}$ induced by $\tmop{Fun} (\Delta^1,
  \mathcal{D}) \rightarrow \tmop{Fun} (\{ 1 \}, \mathcal{D}) \simeq
  \mathcal{D}$ admits a left adjoint informally given by $(C, D) \mapsto (C, F
  (C) \rightarrow F (C) \amalg D)$.
\end{lemma}

\begin{proof}
  We need the concept of relative adjunctions {\cite[§7.3.2]{Lurie2017}}. In
  fact, the adjunction is relative to $\mathcal{C}$.
  
  To see this, we start with the special case that $\mathcal{C}=\mathcal{D}$
  and $F = \tmop{id}_{\mathcal{D}}$. The point is that, the pair $\mathcal{D}
  \times \mathcal{D} \rightleftarrows \tmop{Fun} (\Delta^1, \mathcal{D})$ of
  adjoint functors satisfies {\cite[Prop~7.3.2.1]{Lurie2017}}, where the
  functor $\mathcal{D} \times \mathcal{D} \rightarrow \tmop{Fun} (\Delta^1,
  \mathcal{D})$ is given by left Kan extension along the functor $\{ 0, 1 \}
  \rightarrow \Delta^1$, and the functor $\tmop{Fun} (\Delta^1, \mathcal{D})
  \rightarrow \mathcal{D} \times \mathcal{D}$ is simply given by the
  restriction along $\{ 0, 1 \} \rightarrow \Delta^1$.
  
  The general case follows from {\cite[Prop~7.3.2.5]{Lurie2017}} by base
  change along $F \of \mathcal{C} \rightarrow \mathcal{D}$.
\end{proof}

It follows from \Cref{prop:adjoint-n-proj-gen} that

\begin{corollary}
  \label{cor:comma-proj-gen}Let $\mathcal{C}, \mathcal{D}$ be projectively
  generated $\infty$\mbox{-}categories and $F \of \mathcal{C} \rightarrow
  \mathcal{D}$ a functor. Then the comma category $\mathcal{C}
  \times_{\tmop{Fun} (\{ 0 \}, \mathcal{D})} \tmop{Fun} (\Delta^1,
  \mathcal{D})$ is projectively generated. More precisely, let $S \subseteq
  \mathcal{C}$ and $T \subseteq \mathcal{D}$ be sets of compact projective
  generators. Then $\{ (C, F (C) \rightarrow F (C) \amalg D) \barsuchthat C
  \in S, D \in T \}$ is a set of compact projective generators for
  $\mathcal{C} \times_{\tmop{Fun} (\{ 0 \}, \mathcal{D})} \tmop{Fun}
  (\Delta^1, \mathcal{D})$.
\end{corollary}

It follows from {\cite[Lem~5.4.5.5]{Lurie2009}} that the colimits in comma
categories exist and are easy to describe under the assumption that the
functor in question preserves colimits:

\begin{lemma}
  \label{lem:comma-colim}Let $\mathcal{C}, \mathcal{D}$ be
  $\infty$\mbox{-}categories and $F \of \mathcal{C} \rightarrow \mathcal{D}$ a
  functor. Let $K$ be a simplicial set. Suppose that $\mathcal{C},
  \mathcal{D}$ admits $K$-indexed colimits which are preserved by $F$. Then
  the comma category $\mathcal{C} \times_{\tmop{Fun} (\{ 0 \}, \mathcal{D})}
  \tmop{Fun} (\Delta^1, \mathcal{D})$ admits $K$-indexed colimits which are
  preserved by projection to either factor.
\end{lemma}

\begin{remark}[Base change]
  \label{rem:comma-base-chg}Let $\mathcal{C}, \mathcal{D}$ be
  $\infty$\mbox{-}categories which admit finite colimits and $F \of
  \mathcal{C} \rightarrow \mathcal{D}$ a functor which preserves finite
  colimits. Given an object $(C, F (C) \rightarrow D) \in \mathcal{C}
  \times_{\tmop{Fun} (\{ 0 \}, \mathcal{D})} \tmop{Fun} (\Delta^1,
  \mathcal{D})$, there is a unique map $(C, \tmop{id}_{F (C)}) \rightarrow (C,
  F (C) \rightarrow D)$ (which is in fact the unit map for the adjunction in
  \Cref{rem:proj-comma-sect}). For all maps $C \rightarrow C'$ in
  $\mathcal{C}$, we have the pushout of the diagram $(C', \tmop{id}_{F (C')})
  \leftarrow (C, \tmop{id}_{F (C)}) \rightarrow (F (C) \rightarrow D)$ in
  $\mathcal{C}$, which is $(C', D \amalg_{F (C)} F (C'))$ by
  \Cref{lem:comma-colim}. At the beginning of this section, we said that the
  objects in $\mathcal{C}$ are considered as ``bases''. Thus we understand
  this pushout as ``base change''.
\end{remark}

\begin{example}
  \label{ex:cot-cx-base-chg}In \Cref{ex:comma-ani-ring}, given a map $A
  \rightarrow B$ of animated rings, the base change of $A \rightarrow R$ along
  $A \rightarrow B$ is $B \rightarrow R \otimes_A^{\mathbb{L}} B$. Since the
  cotangent complex functor $L_{\cdummy / \cdummy} \of \tmop{Fun} (\Delta^1,
  \tmop{CAlg}^{\tmop{an}}) \rightarrow \tmop{Ani} (\tmop{Mod})$ preserves
  small colimits (\Cref{lem:cot-cx-preserv-colim}), we get the base change
  property: the natural map $L_{R / A} \otimes_A^{\mathbb{L}} B \rightarrow
  L_{R \otimes_A^{\mathbb{L}} B / B}$ is an equivalence (here we implicitly
  used \Cref{lem:cot-cx-id-trivial}). Similarly, we get the base change
  property $\tmop{HH} (R / A) \otimes_A^{\mathbb{L}} B \simeq \tmop{HH} (R
  \otimes_A^{\mathbb{L}} B / B)$ for Hochschild homology (the reader should
  feel free to ignore this since it will not be used in this article).
\end{example}

\begin{example}
  In \Cref{ex:comma-PD-pair}, given a map $(A, I, \gamma) \rightarrow (B, J,
  \delta)$ of PD-pairs, the base change of $((A, I, \gamma), A / I \rightarrow
  R)$ along $(A, I, \gamma) \rightarrow (B, J, \delta)$ is $((B, J, \delta), B
  / J \rightarrow R \otimes_{A / I} (B / J))$.
\end{example}

\begin{remark}
  We have a prismatic version of base change by \Cref{rem:prism-comma-cat}.
\end{remark}

\begin{remark}[Colimits over a fixed base]
  \label{rem:comma-colim-fix-base}Let $\mathcal{C}, \mathcal{D}$ be cocomplete
  $\infty$\mbox{-}categories and $F \of \mathcal{C} \rightarrow \mathcal{D}$ a
  functor which preserves small colimits. Given an object $C \in \mathcal{C}$,
  a small simplicial set $K$ and a diagram $q \of K \rightarrow \mathcal{D}_{F
  (C) \mathord{/}}$, we associate a diagram $K \rightarrow \mathcal{C}
  \times_{\tmop{Fun} (\{ 0 \}, \mathcal{D})} \tmop{Fun} (\Delta^1,
  \mathcal{D})$ informally given by $k \mapsto (C, F (C) \rightarrow q (k))$
  (the formal description necessitates a discussion of ``fat'' overcategories
  {\cite[§4.2.1]{Lurie2009}}). By \Cref{lem:comma-colim}, the colimit of this
  diagram is given by $(C, \tmop{colim} q)$. We understand this colimit as
  taking colimits over a fixed base.
\end{remark}

\begin{example}
  In \Cref{ex:comma-ani-ring}, given an animated ring $A$ and two $A$-algebras
  $R, S$, the map $(A \rightarrow R \otimes_A^{\mathbb{L}} S)$, seen as an
  object of $\tmop{Fun} (\Delta^1, \tmop{CAlg}^{\tmop{an}})$, is the pushout
  of the diagram $(A \rightarrow R) \leftarrow (A, \tmop{id}_A) \rightarrow (A
  \rightarrow S)$. Since the cotangent complex functor $L_{\cdummy / \cdummy}
  \of \tmop{Fun} (\Delta^1, \tmop{CAlg}^{\tmop{an}}) \rightarrow \tmop{Ani}
  (\tmop{Mod})$ (which we will review in \Cref{def:cot-cx-fun}) preserves
  small colimits (\Cref{lem:cot-cx-preserv-colim}), we get the ``Künneth
  formula'': the natural map $L_{R / A} \otimes_R^{\mathbb{L}} (R
  \otimes_A^{\mathbb{L}} S) \oplus L_{S / A} \otimes_S^{\mathbb{L}} (R
  \otimes_A^{\mathbb{L}} S) \rightarrow L_{(R \otimes_A^{\mathbb{L}} S) / A}$
  is an equivalence (again we used \Cref{lem:cot-cx-id-trivial}, and also the
  form of colimits in $\tmop{Ani} (\tmop{Mod})$). Similarly, we have
  $\tmop{HH} (R / A) \otimes_A^{\mathbb{L}} \tmop{HH} (S / A) \simeq \tmop{HH}
  (R \otimes_A^{\mathbb{L}} S / A)$ for Hochschild homology (again, Hochschild
  homology is not needed in this article).
\end{example}

\begin{remark}
  In view of \Cref{rem:prism-comma-cat}, prismatic cohomology has a similar
  ``Künneth formula'' {\cite[Prop~3.5.1]{Anschuetz2019a}}.
\end{remark}

\begin{remark}[Transitivity]
  \label{rem:comma-transtive}Let $\mathcal{C}, \mathcal{D}$ be
  $\infty$\mbox{-}categories which admit finite colimits and $F \of
  \mathcal{C} \rightarrow \mathcal{D}$ a functor which preserves finite
  colimits. Given a map $C \rightarrow C'$ in $\mathcal{C}$, any object $(C',
  F (C') \rightarrow D) \in \mathcal{C} \times_{\tmop{Fun} (\{ 0 \},
  \mathcal{D})} \tmop{Fun} (\Delta^1, \mathcal{D})$ could be written as the
  pushout of the diagram $(C, F (C) \rightarrow D) \leftarrow (C, F (C)
  \rightarrow F (C')) \rightarrow (C', \tmop{id}_{F (C')})$. This is closely
  related to transitivity sequence in the cohomology theory, as shown in
  examples below.
\end{remark}

\begin{example}
  In \Cref{ex:comma-ani-ring}, for any maps $A \rightarrow B \rightarrow R$ of
  animated rings, the ``relative'' map $B \rightarrow R$, viewed as an object
  of $\tmop{Fun} (\Delta^1, \tmop{CAlg}^{\tmop{an}})$, is the pushout of the
  diagram $(A \rightarrow R) \leftarrow (A \rightarrow B) \rightarrow
  (\tmop{id}_B \of B \rightarrow B)$. Since the cotangent complex functor
  $L_{\cdummy / \cdummy} \of \tmop{Fun} (\Delta^1, \tmop{CAlg}^{\tmop{an}})
  \rightarrow \tmop{Ani} (\tmop{Mod})$ preserves small colimits
  (\Cref{lem:cot-cx-preserv-colim}), we get the transitivity sequence
  \[ L_{B / A} \otimes_B^{\mathbb{L}} R \longrightarrow L_{R / A}
     \longrightarrow L_{R / B} \]
  (\Cref{lem:cot-cx-id-trivial} was used) Similarly, we have $\tmop{HH} (R /
  A) \otimes_{\tmop{HH} (B / A)}^{\mathbb{L}} B \simeq \tmop{HH} (R / B)$ for
  Hochschild homology.
\end{example}

Finally, we briefly review the theory of the cotangent complex of maps of
animated rings, meanwhile we explain how this ``coincides'' with the theory of
cotangent complex of maps of animated $A$-algebra for some ring $A$. By
\Cref{cor:ani-arrow-cat}, the $\infty$-category $\tmop{AniArr} \assign
\tmop{Fun} (\Delta^1, \tmop{CAlg}^{\tmop{an}})$ is projectively generated, and
the proof leads to a set $\{ \mathbb{Z} [X] \rightarrow \mathbb{Z} [X, Y]
\barsuchthat X, Y \in \tmop{Fin} \}$ of compact projective generators. Let
$\tmop{AniArr}^0 \subseteq \tmop{AniArr}$ denote the full subcategory spanned
by those compact projective generators.

\begin{definition}
  \label{def:cot-cx-fun}The {\tmdfn{cotangent complex functor}} $\tmop{AniArr}
  \rightarrow \tmop{Ani} (\tmop{Mod})$ is defined to be the left derived
  functor (\Cref{prop:left-deriv-n-fun}) of the functor $\tmop{AniArr}^0
  \rightarrow \tmop{Ani} (\tmop{Mod}), (A \rightarrow B) \mapsto (B, \Omega_{B
  / A}^1)$. The image of an object $(A \rightarrow B) \in \tmop{AniArr}$ is
  denoted by $(B, L_{B / A})$.
\end{definition}

\begin{remark}
  In fact, this functor is also the animation of the functor $\tmop{Fun}
  (\Delta^1, \tmop{Ring}) \rightarrow \tmop{Mod}, (A \rightarrow B) \mapsto
  (B, \Omega_{B / A}^1)$. We do not take this as the definition since later we
  will apply the same idea to functors which are not defined by the animation
  of a functor.
\end{remark}

Since the functor $\tmop{AniArr}^0 \rightarrow \tmop{Ani} (\tmop{Mod}), B
\mapsto (B, \Omega_{B / A}^1)$ preserves finite coproducts, by
\Cref{prop:left-deriv-n-fun}, we get

\begin{lemma}
  \label{lem:cot-cx-preserv-colim}The cotangent complex functor $\tmop{AniArr}
  \rightarrow \tmop{Ani} (\tmop{Mod})$ preserves small colimits.
\end{lemma}

Now we consider the functor $\tmop{CAlg}^{\tmop{an}} \rightarrow
\tmop{AniArr}, A \mapsto (\tmop{id}_A \of A \rightarrow A)$. This functors
preserves small colimits\footnote{In fact, this is fully faithful. However, in
order to apply the same idea to later contexts, we only abstract out the
colimit-preserving property.}, thus so does the composite functor
$\tmop{CAlg}^{\tmop{an}} \rightarrow \tmop{AniArr} \xrightarrow{L_{\cdummy /
\cdummy}} \tmop{Ani} (\tmop{Mod})$, concretely given by $A \mapsto (A, L_{A /
A})$. The next simple\footnote{We warn the reader that this lemma is not
tautological.} lemma is key to the ``independence of the choice of the base'',
which was already used in examples before:

\begin{lemma}
  \label{lem:cot-cx-id-trivial}The composite functor $\tmop{CAlg}^{\tmop{an}}
  \rightarrow \tmop{AniArr} \xrightarrow{L_{\cdummy / \cdummy}} \tmop{Ani}
  (\tmop{Mod})$ above coincides with the functor $\tmop{CAlg}^{\tmop{an}}
  \rightarrow \tmop{Ani} (\tmop{Mod}), A \mapsto (A, 0)$.
\end{lemma}

\begin{proof}
  By the colimit-preserving property above and \Cref{prop:left-deriv-n-fun},
  it suffices to check this for polynomial rings $A =\mathbb{Z} [x_1, \ldots,
  x_n]$, but this follows directly from the definitions.
\end{proof}

We now consider the full subcategory $\mathcal{P}^0$ of $\tmop{Fun} (\Delta^1,
\tmop{Ring})$ spanned by maps $A [X] \rightarrow A [X, Y]$ with $A \in
\tmop{Ring}$ and $X, Y \in \tmop{Fin}$. The functor $\tmop{Fun} (\Delta^1,
\tmop{Ring}) \rightarrow \tmop{Ani} (\tmop{Mod}), (A \rightarrow B) \mapsto
(B, \Omega_{B / A}^1)$ restricts to a functor $G \of \mathcal{P}^0 \rightarrow
\tmop{Ani} (\tmop{Mod})$. By \Cref{prop:left-deriv-n-fun}, the restriction $F$
of the cotangent complex functor $\tmop{AniArr} \rightarrow \tmop{Ani}
(\tmop{Mod})$ to the full subcategory $\mathcal{P}^0$ is left Kan extended
from $\tmop{AniArr}^0 \subseteq \mathcal{P}^0$, therefore we have a comparison
map $F \rightarrow G$, which becomes an equivalence after restricting to the
full subcategory $\tmop{AniArr}^0$. By \Cref{ex:cot-cx-base-chg}, this
comparison map is an equivalence since $G$ also has the ``base change
property''.

Now we fix a ring $A$, and let $\tmop{AniArr}_A$ denote the $\infty$-category
$\tmop{Fun} (\Delta^1, \tmop{CAlg}_A^{\tmop{an}})$. As before, by
\Cref{cor:ani-arrow-cat}, it is projectively generated with a set $\{ A [X]
\rightarrow A [X, Y] \barsuchthat X, Y \in \tmop{Fin} \}$ of compact
projective generators, which spans a full subcategory $\tmop{AniArr}_A^0
\subseteq \tmop{AniArr}_A$. Note that the functor $\tmop{AniArr}_A^0
\rightarrow \tmop{Ani} (\tmop{Mod}), (B \rightarrow C) \mapsto (C, \Omega_{C /
B}^1)$ coincides with the composite functor $\tmop{AniArr}_A^0 \rightarrow
\mathcal{P}^0 \xrightarrow{G} \tmop{Ani} (\tmop{Mod})$, and since $F \simeq
G$, this composite functor is just the cotangent complex functor applied to
the underlying map of animated rings. It follows from
\Cref{prop:left-deriv-n-fun} that

\begin{lemma}
  \label{lem:cot-cx-indep-base}The composite functor $\tmop{AniArr}_A
  \rightarrow \tmop{AniArr} \rightarrow \tmop{Ani} (\tmop{Mod}), (B
  \rightarrow C) \mapsto (C, L_{C / B})$ is equivalent to the left derived
  functor of $\tmop{AniArr}_A^0 \rightarrow \tmop{Ani} (\tmop{Mod}), (A [X]
  \rightarrow A [X, Y]) \mapsto \Omega_{A [X, Y] / A [X]}^1$.
\end{lemma}

That is to say, the definition of the cotangent complex does not depend on the
choice of the base. This argument applies to similar situations, such as
animated PD-envelope, and such phenomenon will appear frequently in this
article.

\subsection{$\infty$\mbox{-}category of graded and filtered
objects}\label{subsec:graded-fil-objs}In this section, we recollect basic
properties of the $\infty$\mbox{-}category of graded and filtered objects. Our
main reference is {\cite[§3]{Raksit2020}}.

The $\infty$\mbox{-}category of ($\mathbb{Z}$-)graded objects in an
$\infty$\mbox{-}category $\mathcal{C}$ is the $\infty$\mbox{-}category
$\tmop{Gr} (\mathcal{C}) \assign \tmop{Fun} (\mathbb{Z}, \mathcal{C})$ of
functors, where $\mathbb{Z}$ is the set of integers as an
$\infty$\mbox{-}category. Given a graded object $G \in \tmop{Gr}
(\mathcal{C})$, we will denote the value of $G$ at $i \in \mathbb{Z}$ by
$X^i$. This defines a functor $(\cdummy)^i \of \tmop{Gr} (\mathcal{C})
\rightarrow \mathcal{C}$.

When the $\infty$\mbox{-}category $\mathcal{C}$ is presentable, for all $i \in
\mathbb{Z}$, the functor $(\cdummy)^i$ admits a fully faithful left adjoint
$\tmop{ins}^i \of \mathcal{C} \rightarrow \tmop{Gr} (\mathcal{C})$ simply
given by $X \mapsto G$ where $G^j = \left\{\begin{array}{ll}
  X & j = i\\
  0_{\mathcal{C}} & \tmop{otherwise}
\end{array}\right.$ where $0_{\mathcal{C}} \in \mathcal{C}$ is the initial
object.

We say that a graded object $G \in \tmop{Gr} (\mathcal{C})$ is
{\tmdfn{nonnegatively graded}} (resp. {\tmdfn{nonpositively graded}}) if the
restriction $F|_{\mathbb{Z}_{< 0}}$ (resp. $F|_{\mathbb{Z}_{> 0}}$) is
constantly $0_{\mathcal{C}}$. The full subcategory spanned by nonnegatively
graded (resp. nonpositively graded) objects is denoted by $\tmop{Gr}^{\geq 0}
(\mathcal{C})$ (resp. $\tmop{Gr}^{\leq 0} (\mathcal{C})$), which is
canonically equivalent to $\tmop{Fun} (\mathbb{Z}_{\geq 0}, \mathcal{C})$
(resp. $\tmop{Fun} (\mathbb{Z}_{\leq 0}, \mathcal{C})$).

Similarly, the {\tmdfn{$\infty$\mbox{-}category of ($\mathbb{Z}$-)filtered
objects}} in an $\infty$\mbox{-}category $\mathcal{C}$ is the
$\infty$\mbox{-}category $\tmop{Fil} (\mathcal{C}) \assign \tmop{Fun}
((\mathbb{Z}, \geq), \mathcal{C})$ of functors. Given a filtered object $F \in
\tmop{Fil} (\mathcal{C})$, we will systematically denote the value of $F$ at
$i \in \mathbb{Z}$ by $\tmop{Fil}^i F$ instead of $F (i)$ to indicate that we
consider it as a filtered object. This defines a functor $\tmop{Fil}^i \of
\tmop{Fil} (\mathcal{C}) \rightarrow \mathcal{C}$.

When the $\infty$\mbox{-}category $\mathcal{C}$ is presentable, for all $i \in
\mathbb{Z}$, the functor $\tmop{Fil}^i$ admits a fully faithful left adjoint
$\tmop{ins}^i \of \mathcal{C} \rightarrow \tmop{Fil} (\mathcal{C})$ given by
the left Kan extension along $\{ i \} \rightarrow (\mathbb{Z}, \geq)$. Given
$X \in \mathcal{C}$, $\tmop{Fil}^j (\tmop{ins}^i (X)) =
\left\{\begin{array}{ll}
  X & j \leq i\\
  0_{\mathcal{C}} & j > i
\end{array}\right.$ where $0_{\mathcal{C}} \in \mathcal{C}$ is the initial
object.

We say that a filtered object $F \in \tmop{Fil} (\mathcal{C})$ is
{\tmdfn{nonnegatively filtered}} if the restriction $F|_{\mathbb{Z}_{\leq 0}}$
is a constant functor $(\mathbb{Z}_{\leq 0}, \geq) \rightarrow \mathcal{C}$.
We denote by $\tmop{Fil}^{\geq 0} (\mathcal{C}) \subseteq \tmop{Fil}
(\mathcal{C})$ the full subcategory spanned by nonnegatively filtered objects,
which is canonically equivalent to $\tmop{Fun} ((\mathbb{Z}_{\geq 0}, \geq),
\mathcal{C})$. Similarly, we say that a filtered object $F \in \tmop{Fil}
(\mathcal{C})$ is {\tmdfn{nonpositively filtered}} if the restriction
$F|_{\mathbb{Z}_{> 0}}$ is constantly $0_{\mathcal{C}}$. We denote by
$\tmop{Fil}^{\leq 0} (\mathcal{C}) \subseteq \tmop{Fil} (\mathcal{C})$ the
full subcategory spanned by nonpositively filtered objects, which is
canonically equivalent to $\tmop{Fun} ((\mathbb{Z}_{\leq 0}, \geq),
\mathcal{C})$.

Given a filtered object $F \in \tmop{Fil} (\mathcal{C})$, the {\tmdfn{union}}
$\tmop{Fil}^{- \infty}$ is defined to be the colimit
$\tmop{colim}_{(\mathbb{Z}, \geq)} F$ (when it exists). When $\mathcal{C}$
admits all sequential colimits, this defines a functor $\tmop{Fil}^{- \infty}
\of \tmop{Fil} (\mathcal{C}) \rightarrow \mathcal{C}$. Furthermore, when the
$\infty$-category $\mathcal{C}$ is stable, a filtered object $F \in \tmop{Fil}
(\mathcal{C})$ is called {\tmdfn{complete}} {\cite[Def~3.2.12]{Raksit2020}},
or a {\tmdfn{completely filtered object}}, if $\lim F \simeq 0$ in
$\mathcal{C}$. We will denote by $\tmop{Fil}^{\wedge} (\mathcal{C}) \subseteq
\tmop{Fil} (\mathcal{C})$ the full subcategory spanned by completely filtered
objects.

\begin{remark}
  To avoid confusions, our filtrations are always decreasing. When we need
  increasing filtrations, we invert the sign to get a decreasing filtration.
\end{remark}

Now let $(\mathcal{C}, \otimes)$ be a presentable symmetric monoidal
$\infty$\mbox{-}category. Note that $\mathbb{Z}$ (resp. $(\mathbb{Z}, \geq)$)
has a symmetric monoidal structure given by the addition $+$, so the
$\infty$\mbox{-}category $\tmop{Gr} (\mathcal{C})$ (resp. $\tmop{Fil}
(\mathcal{C})$) admits a presentable symmetric monoidal structure given by the
{\tmdfn{Day convolution}} $\otimes^{\tmop{Day}}$ {\cite[§3]{Nikolaus2016}}.
Informally, given two graded (resp. filtered) objects $F, G$, we have $(F
\otimes^{\tmop{Day}} G)^i = \bigoplus_{j + k = i} F^j \otimes G^k$ (resp.
$\tmop{Fil}^i  (F \otimes^{\tmop{Day}} G) = \tmop{colim}_{j + k \geq i}
\tmop{Fil}^j F \otimes \tmop{Fil}^k G$). Under this symmetric monoidal
structure, $(\cdummy)^0 \of \tmop{Gr} (\mathcal{C}) \rightarrow \mathcal{C}$
(resp. $\tmop{Fil}^0 \of \tmop{Fil} (\mathcal{C}) \rightarrow \mathcal{C}$) is
lax symmetric monoidal, while the fully faithful left adjoint $\tmop{ins}^0
\of \mathcal{C} \rightarrow \tmop{Gr} (\mathcal{C})$ (resp. $\mathcal{C}
\rightarrow \tmop{Fil} (\mathcal{C})$) is symmetric monoidal.

The stable subcategory $\tmop{Gr}^{\geq 0} (\mathcal{C}) \subseteq \tmop{Gr}
(\mathcal{C})$ (resp. $\tmop{Gr}^{\leq 0} (\mathcal{C}) \subseteq \tmop{Gr}
(\mathcal{C})$) inherits a presentable symmetric monoidal structure, and the
$0$th piece $(\cdummy)^0 \of \tmop{Gr}^{\geq 0} (\mathcal{C}) \rightarrow
\mathcal{C}$ (resp. $\tmop{Gr}^{\leq 0} (\mathcal{C}) \rightarrow
\mathcal{C}$) is symmetric monoidal.

Similarly, the stable subcategory $\tmop{Fil}^{\geq 0} (\mathcal{C}) \subseteq
\tmop{Fil} (\mathcal{C})$ (resp. $\tmop{Fil}^{\leq 0} (\mathcal{C}) \subseteq
\tmop{Fil} (\mathcal{C})$) inherits a presentable symmetric monoidal
structure, and the $0$th piece $\tmop{Fil}^0 \of \tmop{Fil}^{\geq 0}
(\mathcal{C}) \rightarrow \mathcal{C}$ (resp. $\tmop{Fil}^{\leq 0}
(\mathcal{C}) \rightarrow \mathcal{C}$) is symmetric monoidal.

Now we study the relation between graded objects and filtered objects. First,
the symmetric monoidal functor $\mathbb{Z} \rightarrow (\mathbb{Z}, \geq)$
induces a lax symmetric monoidal functor $\tmop{Fil} (\mathcal{C}) \rightarrow
\tmop{Gr} (\mathcal{C})$, which admits a symmetric monoidal left adjoint $I
\of \tmop{Gr} (\mathcal{C}) \rightarrow \tmop{Fil} (\mathcal{C})$, the left
Kan extension along $\mathbb{Z} \rightarrow (\mathbb{Z}, \geq)$. Concretely,
it is given by $G \mapsto F$ where $\tmop{Fil}^i F = \coprod_{j \geq i} G^j$.

All of the functors mentioned above preserve small colimits. From now on, let
$\mathcal{C}$ be a presentable stable symmetric monoidal
$\infty$\mbox{-}category. Then these functors are exact. Now we consider the
associated graded functor $\tmop{gr} \of \tmop{Fil} (\mathcal{C}) \rightarrow
\tmop{Gr} (\mathcal{C}), F \mapsto G$ where $G^i = \tmop{cofib} (F^{i + 1} F
\rightarrow \tmop{Fil}^i F)$. It turns out that the functor $\tmop{gr}$
behaves well:

\begin{proposition}[{\cite[Prop~3.2.1]{Lurie2015}}
{\cite[Prop~2.26]{Gwilliam2018}}]
  Let $\mathcal{C}$ be a presentable stable symmetric monoidal
  $\infty$\mbox{-}category. Then there exists a symmetric monoidal structure
  on the functor $\tmop{gr}^{\ast} \of \tmop{Fil} (\mathcal{C}) \rightarrow
  \tmop{Gr} (\mathcal{C})$. Moreover, this symmetric monoidal structure can be
  chosen so that the composite functor $\tmop{Gr} (\mathcal{C})
  \xrightarrow{I} \tmop{Fil} (\mathcal{C}) \xrightarrow{\tmop{gr}^{\ast}}
  \tmop{Gr} (\mathcal{C})$ is homotopic to the identity as a symmetric
  monoidal functor.
\end{proposition}

We also need the {\tmdfn{Beilinson t-structure}} on the
$\infty$\mbox{-}category $\tmop{Fil} (\mathcal{C})$ of filtered objects. As
before, let $(\mathcal{C}, \otimes, 1_{\mathcal{C}})$ be a presentable stable
symmetric monoidal $\infty$\mbox{-}category. Moreover, we assume that
$\mathcal{C}$ admits an accessible $t$-structure $(\mathcal{C}_{\geq 0},
\mathcal{C}_{\leq 0})$ compatible with filtered colimits such that
$1_{\mathcal{C}} \in \mathcal{C}_{\geq 0}$ and $\mathcal{C}_{\geq 0}$ is
closed under $\otimes$.

\begin{lemma}
  \label{lem:heart-symm-mon}Under the assumptions above, the heart
  $\mathcal{C}^{\heartsuit} \assign \mathcal{C}_{\geq 0} \cap
  \mathcal{C}_{\leq 0}$ admits a canonical symmetric monoidal structure
  $\otimes^{\heartsuit}$ given by $X \otimes^{\heartsuit} Y \assign \tau_{\leq
  0} (X \otimes Y)$, and the embedding $\mathcal{C}^{\heartsuit} \rightarrow
  \mathcal{C}$ is then lax symmetric monoidal.
\end{lemma}

The following is the $\infty$-categorical enrichment of
{\cite[App]{Beilinson1987}}.

\begin{proposition}[{\cite[Prop~3.3.11]{Raksit2020}}]
  \label{prop:Beilinson-t-struct}Let $\tmop{Fil} (\mathcal{C})_{\geq 0}^B
  \subseteq \tmop{Fil} (\mathcal{C})$ be the full subcategory spanned by $X
  \in \tmop{Fil} (\mathcal{C})$ such that $\tmop{gr}^i (X) \in
  \mathcal{C}_{\geq - i}$ for all $i \in \mathbb{Z}$. Then $1_{\tmop{ins}^0
  (1_{\mathcal{C}})} \in \tmop{Fil} (\mathcal{C})_{\geq 0}^B$, $\tmop{Fil}
  (\mathcal{C})_{\geq 0}^B$ is closed under $\otimes^{\tmop{Day}}$ and is the
  connective part of an accessible $t$-structure, called the {\tmdfn{Beilinson
  $t$-structure}}, whose heart is equivalent as symmetric monoidal
  $1$\mbox{-}categories to the $1$\mbox{-}category $\tmop{Ch}
  (\mathcal{C}^{\heartsuit})$ of chain complexes with ``stupid'' truncation
  $\tmop{Fil}^i K = K_{\leq - i}$ for all $i \in \mathbb{Z}$ and $K \in
  \tmop{Ch} (\mathcal{C}^{\heartsuit})$.
\end{proposition}

In particular, when $\mathcal{C}$ is the derived $\infty$\mbox{-}category of a
ring $R$, the {\tmdfn{filtered derived category}} $\tmop{DF} (R)$ is the
$\infty$\mbox{-}category $\tmop{Fil} (D (R))$ of filtered objects in the
derived $\infty$\mbox{-}category $D (R)$ with the symmetric monoidal structure
given by the derived tensor product $\cdummy \otimes_R^{\mathbb{L}} \cdummy$,
and $\tmop{DF}^{\geq 0} (R)$ is the $\infty$\mbox{-}category $\tmop{Fil}^{\geq
0} (D (R))$ of nonnegatively filtered objects in $D (R)$. In this case, we
will still denote by $\cdummy \otimes_R^{\mathbb{L}} \cdummy$ the Day
convolution.

\begin{remark}[{\cite[Cons~4.3.4]{Raksit2020}}]
  \label{rem:fil-deriv-alg}Let $R$ be a ring. The $\infty$\mbox{-}category
  $\tmop{DF} (R)$ admits a structure of {\tmdfn{derived algebraic context}}
  {\cite[Def~4.2.1]{Raksit2020}}, of which the derived commutative algebras
  are called {\tmdfn{filtered derived (commutative) $R$-algebras}}. When $R
  =\mathbb{Z}$, they are also called {\tmdfn{filtered derived rings}}.
  Although we will not use this fact, we might comment when a filtered
  $\mathbb{E}_{\infty}$-$\mathbb{Z}$-algebra admits such a structure.
\end{remark}

We need the following lemma, which follows from the fact that left Kan
extensions are pointwise colimits which preserve cofibers and filtered
colimits:

\begin{lemma}
  \label{lem:assoc-graded-union-Lan}Let $\mathcal{C}$ be an
  $\infty$\mbox{-}category, $\mathcal{C}^0 \subseteq \mathcal{C}$ a full
  subcategory, $\mathcal{E}$ a stable $\infty$\mbox{-}category which admits
  filtered colimits, and $\tilde{F} \of \mathcal{C} \rightarrow \tmop{Fil}
  (\mathcal{E})$ a functor left Kan extended along the fully faithful
  embedding $\mathcal{C}^0 \hookrightarrow \mathcal{C}$. Then
  \begin{enumerate}
    \item The composite functor $\tmop{gr}^{\ast} \circ \tilde{F} \of
    \mathcal{C} \rightarrow \tmop{Fil} (\mathcal{E})
    \xrightarrow{\tmop{gr}^{\ast}} \tmop{Gr} (\mathcal{E})$ is left Kan
    extended along $\mathcal{C}^0 \hookrightarrow \mathcal{C}$.
    
    \item The composite functor $\tmop{Fil}^{- \infty} \circ \tilde{F} \of
    \mathcal{C} \rightarrow \tmop{Fil} (\mathcal{E})
    \xrightarrow{\tmop{Fil}^{- \infty}} \mathcal{E}$ is left Kan extended
    along $\mathcal{C}^0 \hookrightarrow \mathcal{C}$.
  \end{enumerate}
\end{lemma}

\subsection{Reflective subcategories}\label{subsec:refl-subcat}In this
subsection, we will develop the necessary machinery to deal with the (derived)
$p$-complete or more generally $I$-complete situations. We start with the
general formalism of reflective subcategories.

\begin{definition}[{\cite[Rem~5.2.7.9 \& Def~5.2.7.2]{Lurie2009}}]
  \label{def:reflexive-cat}Let $\mathcal{C}$ be an $\infty$\mbox{-}category.
  We say that a full subcategory $\mathcal{D} \subseteq \mathcal{C}$ is
  {\tmdfn{reflective}} if the inclusion $\mathcal{D} \hookrightarrow
  \mathcal{C}$ admits a left adjoint $L \of \mathcal{C} \rightarrow
  \mathcal{D}$. In such case, we call the left adjoint $L \of \mathcal{C}
  \rightarrow \mathcal{D}$ a {\tmdfn{localization}}.
\end{definition}

\begin{proposition}[{\cite[Prop~5.2.7.8]{Lurie2009}}]
  Let $\mathcal{C}$ be an $\infty$-category. A full subcategory $\mathcal{C}^0
  \subseteq \mathcal{C}$ is reflective if and only if for every object $C \in
  \mathcal{C}$, there exists an object $D \in \mathcal{C}^0$ along with a map
  $f \of C \rightarrow D$ which induces an equivalence
  $\tmop{Map}_{\mathcal{C}} (D, E) \rightarrow \tmop{Map}_{\mathcal{C}} (C,
  E)$ for each object $E \in \mathcal{C}^0$ (in this case, $L C \simeq D$
  where $L \of \mathcal{C} \rightarrow \mathcal{C}^0$ is the localization).
\end{proposition}

\begin{example}
  Let $D_{\tmop{comp}} (\mathbb{Z}_p) \subseteq D (\mathbb{Z})$ be the
  $p$-complete derived category of $\mathbb{Z}$, consisting of (derived)
  $p$-complete $\mathbb{Z}_p$-module spectra. Then $D_{\tmop{comp}}
  (\mathbb{Z}_p) \subseteq D (\mathbb{Z})$ is reflective. The localization is
  the (derived) $p$-completion functor $D (\mathbb{Z}) \rightarrow
  D_{\tmop{comp}} (\mathbb{Z}_p)$. Similarly, $D_{\tmop{comp}, \geq 0}
  (\mathbb{Z}_p) \subseteq D_{\geq 0} (\mathbb{Z})$ is the reflective
  subcategory of connective $p$-complete $\mathbb{Z}_p$-module spectra.
\end{example}

\begin{example}
  More generally, let $A$ be an animated ring and $I \subseteq \pi_0 (A)$ a
  finitely generated ideal. Then the $I$-complete derived category
  $D_{\tmop{comp}} (A)$ is a reflective subcategory of the derived category $D
  (A)$. The same for $D_{\tmop{comp}, \geq 0} (A) \subseteq D_{\geq 0} (A)$.
\end{example}

Now we study the left derived functors. Unfortunately, the localization does
not in general map compact projective objects to compact projective objects.
For example, $\mathbb{Z} \in D (\mathbb{Z})$ is compact and projective but
$\mathbb{Z}_p \in D_{\tmop{comp}} (\mathbb{Z}_p)$ is not. We suspect that
$D_{\tmop{comp}} (\mathbb{Z}_p)$ is not projectively generated, therefore we
are probably unable to left derive ``arbitrary'' functors as in the
projectively generated case. However, most functors in practice are good
enough to have a reasonable theory of left derived functors\footnote{This
approach is essentially depicted in the special case of $p$-completed rings in
Bhatt's Eilenberg Lectures notes {\cite[Lecture~VII]{Bhatt2018b}}. We are
informed by Yu {\tmname{Min}} of this approach in private discussions.}. We
start with a general discussion about the interaction between localization and
left Kan extension {\cite[Def~4.3.2.2]{Lurie2009}}.

\begin{notation}
  \label{setup:localization-extension}Let $\mathcal{C}$ be an
  $\infty$\mbox{-}category and $\mathcal{D} \subseteq \mathcal{C}$ a
  reflective (full) subcategory with the localization $L \of \mathcal{C}
  \rightarrow \mathcal{D}$. Let $\mathcal{C}^0 \subseteq \mathcal{C}$ be a
  full subcategory, $\mathcal{D}^0 \subseteq \mathcal{D}$ the full subcategory
  spanned by objects $L C$ where $C$ runs through all objects in
  $\mathcal{C}^0$. Let $\mathcal{C}^1 \subseteq \mathcal{C}$ be the full
  subcategory spanned by vertices of both $\mathcal{C}^0$ and $\mathcal{D}^0$.
\end{notation}

It follows from definitions that

\begin{lemma}
  In Setup~\ref{setup:localization-extension}, $\mathcal{D}^0 \subseteq
  \mathcal{C}^1$ is a reflective subcategory with localization $\nobracket L
  |_{\mathcal{C}^1} \of \mathcal{C}^1 \rightarrow \mathcal{D}^0$ being the
  restriction of $L \of \mathcal{C} \rightarrow \mathcal{D}$.
\end{lemma}

\begin{example}
  \label{ex:setup-p-completion}One of the crucial example for the setup above:
  $\mathcal{C}$ is the $\infty$\mbox{-}category of animated rings,
  $\mathcal{D}$ is the full subcategory of $p$-complete animated
  $\mathbb{Z}_p$-algebras, and $\mathcal{C}^0 \subseteq \mathcal{C}$ is the
  full subcategory spanned by polynomial rings $\mathbb{Z} [X_1, \ldots,
  X_n]$. More generally, let $A$ be an animated ring and $I \subseteq \pi_0
  (A)$ a finitely generated ideal. Then we can consider the case that
  $\mathcal{C}$ is the $\infty$\mbox{-}category of animated $A$-algebras and
  $\mathcal{D} \subseteq \mathcal{C}$ is the full subcategory of $I$-complete
  animated $A$-algebras, and $\mathcal{C}^0 \subseteq \mathcal{C}$ is the full
  subcategory spanned by polynomial $A$-algebras $A [X_1, \ldots, X_n] \assign
  \mathbb{Z} [X_1, \ldots, X_n] \otimes_{\mathbb{Z}}^{\mathbb{L}} A$.
\end{example}

\begin{lemma}
  \label{lem:Lan-localization}In Setup~\ref{setup:localization-extension}, let
  $\mathcal{E}$ be an $\infty$\mbox{-}category and $\tilde{F} \of \mathcal{C}
  \rightarrow \mathcal{E}$ a functor left Kan extended from the fully faithful
  embedding $\mathcal{C}^0 \hookrightarrow \mathcal{C}$. Then the restriction
  $\widetilde{\nobracket F |}_{\mathcal{D}}$ is left Kan extended from the
  fully faithful embedding $\mathcal{D}^0 \hookrightarrow \mathcal{D}$.
\end{lemma}

\begin{proof}
  It follows from {\cite[Lem~5.2.6.6]{Lurie2009}} that the restriction
  $\nobracket \tilde{F} |_{\mathcal{D}}$ is a left Kan extension of
  $\tilde{F}$ along $L \of \mathcal{C} \rightarrow \mathcal{D}$, therefore is
  left Kan extended from the composite functor $\mathcal{C}^0 \hookrightarrow
  \mathcal{C} \xrightarrow{L} \mathcal{D}$. The composite functor
  $\mathcal{C}^0 \hookrightarrow \mathcal{C} \rightarrow \mathcal{D}$ could be
  rewritten as the composite $\mathcal{C}^0 \xrightarrow{L} \mathcal{D}^0
  \hookrightarrow \mathcal{D}$, therefore $\nobracket \tilde{F}
  |_{\mathcal{D}}$ is left Kan extended from $\mathcal{D}^0 \hookrightarrow
  \mathcal{D}$.
\end{proof}

\begin{example}
  \label{ex:setup-p-compl-cot-cx}In Example~\ref{ex:setup-p-completion}, the
  cotangent complex $L_{\cdummy /\mathbb{Z}} \of \mathcal{C}=
  \tmop{CAlg}^{\tmop{an}} \rightarrow D (\mathbb{Z})$ is left Kan extended
  from $\tmop{Poly}_{\mathbb{Z}} \subseteq \tmop{Ring}$. Consequently, the
  restriction $\nobracket L_{\cdummy /\mathbb{Z}} |_{\mathcal{D}} \suchthat
  \mathcal{D} \rightarrow D (\mathbb{Z})$ is left Kan extended from
  $p$-completed polynomial rings. Similarly, the $p$-completed cotangent
  complex $(L_{\cdummy /\mathbb{Z}})_p^{\wedge} \of \tmop{CAlg}^{\tmop{an}}
  \rightarrow D_{\tmop{comp}} (\mathbb{Z}_p)$ is left extended from
  $\tmop{Poly}_{\mathbb{Z}} \subseteq \tmop{Ring}$, therefore the restriction
  $\nobracket (L_{\cdummy /\mathbb{Z}})_p^{\wedge} |_{\mathcal{D}} \of
  \mathcal{D} \rightarrow D_{\tmop{comp}} (\mathbb{Z}_p)$ is left extended
  from $p$-completed polynomial rings.
\end{example}

\begin{notation}
  \label{setup:Lan}In Setup~\ref{setup:localization-extension}, let $F \of
  \mathcal{D}^0 \rightarrow \mathcal{E}$ be a functor equipped with a left Kan
  extension $\tilde{F} \of \mathcal{C} \rightarrow \mathcal{E}$ along the
  fully faithful inclusion $\mathcal{C}^0 \hookrightarrow \mathcal{C}$ of the
  composite functor $\mathcal{C}^0 \xrightarrow{L} \mathcal{D}^0
  \xrightarrow{F} \mathcal{E}$.
\end{notation}

\begin{remark}
  \label{rem:setup-proj-gen}In our applications, $\mathcal{C}$ will be a
  projectively generated $\infty$\mbox{-}category (\Cref{def:n-proj-gen}) with
  a set $S$ of compact projective generators. We will choose $\mathcal{C}^0
  \subseteq \mathcal{C}$ to be the full subcategory spanned by finite
  coproducts of objects in $S$, and $\mathcal{E}$ will be a cocomplete
  $\infty$\mbox{-}category. In this case, the left Kan extension in question
  always exists (\Cref{prop:left-deriv-n-fun,prop:struct-n-proj-gen-cats}).
  More generally, if $\mathcal{C}^0$ is a small subcategory and that
  $\mathcal{C}$ is assumed to be locally small, then the left Kan extension
  exists.
\end{remark}

In Setup~\ref{setup:Lan}, we first assume without loss of generality that
$\nobracket L |_{\mathcal{D}} = \tmop{id}_{\mathcal{D}}$ by
{\cite[Prop~5.2.7.4]{Lurie2009}}, then $L^2 = L$. Now we let $F_1 \of
\mathcal{C}^1 \rightarrow \mathcal{E}$ denote the composite $\mathcal{C}^1
\rightarrow \mathcal{D}^0 \xrightarrow{F} \mathcal{E}$, which is an extension
of the composite $\mathcal{C}^0 \rightarrow \mathcal{D}^0 \xrightarrow{F}
\mathcal{E}$ along $\mathcal{C}^0 \rightarrow \mathcal{C}^1$. Since
$\nobracket \tilde{F} |_{\mathcal{C}^1} \of \mathcal{C}^1 \rightarrow
\mathcal{E}$ is, by definition, a left Kan extension of $\mathcal{C}^0
\rightarrow \mathcal{D}^0 \xrightarrow{F} \mathcal{E}$ along $\mathcal{C}^0
\rightarrow \mathcal{C}^1$, there exists an essentially unique comparison map
$\nobracket \tilde{F} |_{\mathcal{C}^1} \rightarrow F_1$ of functors
$\mathcal{C}^1 \rightrightarrows \mathcal{E}$. Restricting to the full
subcategory $\mathcal{D}^0 \subseteq \mathcal{C}^1$, we get a comparison map
$\nobracket \tilde{F} |_{\mathcal{D}^0} \rightarrow F$. It follows from
\Cref{lem:Lan-localization} that

\begin{corollary}
  \label{cor:Lan-localization}In Setup~\ref{setup:Lan}, if we assume that the
  comparison map $\nobracket \tilde{F} |_{\mathcal{D}^0} \rightarrow F$ is an
  equivalence, then $\tilde{F} |_{\mathcal{D}}$ is the left Kan extension of
  $F$ along the fully faithful embedding $\mathcal{D}^0 \hookrightarrow
  \mathcal{D}$.
\end{corollary}

We need the following concept:

\begin{proposition}[{\cite[Prop~5.2.7.12]{Lurie2009}}]
  \label{prop:L-inv}Let $\mathcal{C}$ be an $\infty$\mbox{-}category and let
  $L \of \mathcal{C} \rightarrow L\mathcal{C} \subseteq \mathcal{C}$ be a
  localization functor. Let $S$ denote the collection of all morphisms $f$ in
  $\mathcal{C}$ such that $L f$ is an equivalence. Then for every
  $\infty$\mbox{-}category $\mathcal{D}$, composition with $L$ induces a fully
  faithful functor $\psi \of \tmop{Fun} (L\mathcal{C}, \mathcal{D})
  \rightarrow \tmop{Fun} (\mathcal{C}, \mathcal{D})$. Moreover, the essential
  image of $\psi$ consists of those functors $F \of \mathcal{C} \rightarrow
  \mathcal{D}$ such that $F (f)$ is an equivalence in $\mathcal{D}$ for each
  $f \in S$.
\end{proposition}

\begin{definition}
  \label{def:L-inv}Let $\mathcal{C}$ be an $\infty$\mbox{-}category, $L \of
  \mathcal{C} \rightarrow L\mathcal{C} \subseteq \mathcal{C}$ a localization
  functor and $\mathcal{D}$ an $\infty$\mbox{-}category. We say that a functor
  $F \of \mathcal{C} \rightarrow \mathcal{D}$ is {\tmdfn{$L$-invariant}} if
  for every morphism $f$ in $\mathcal{C}$ such that $L f$ is an equivalence,
  then so is $F (f)$ in $\mathcal{D}$.
\end{definition}

Now we come back to our previous discussion.

\begin{lemma}
  \label{lem:L-inv-Lan}Under the above discussion, consider the following
  conditions:
  \begin{enumerateroman}
    \item \label{pt:l10p1}The left Kan extension $\tilde{F} \of \mathcal{C}
    \rightarrow \mathcal{E}$ is $L$-invariant.
    
    \item \label{pt:l10p2}The comparison map $\nobracket \tilde{F}
    |_{\mathcal{C}^1} \rightarrow F_1$ constructed above is an equivalence.
    
    \item \label{pt:l10p3}The comparison map $\nobracket \tilde{F}
    |_{\mathcal{D}^1} \rightarrow F$ constructed above is an equivalence.
  \end{enumerateroman}
  We have
  \begin{enumerate}
    \item Conditions \ref{pt:l10p2} and \ref{pt:l10p3} are equivalent.
    
    \item Condition \ref{pt:l10p1} implies condition \ref{pt:l10p3}.
    
    \item Under the assumptions in \Cref{rem:setup-proj-gen}, condition
    \ref{pt:l10p2} implies condition \ref{pt:l10p1}.
  \end{enumerate}
\end{lemma}

\begin{proof}
  First, restricting the comparison map $\nobracket \tilde{F}
  |_{\mathcal{C}^1} \rightarrow F_1$ to $\mathcal{C}^0$, we get the identity,
  so conditions \ref{pt:l10p2} and \ref{pt:l10p3} are equivalent.
  
  If $\tilde{F}$ is $L$-invariant, then for all $X \in \mathcal{C}^1$, the
  unit map $X \rightarrow L X$ induces a commutative diagram
  \[ \begin{array}{ccc}
       \tilde{F} (X) & \longrightarrow & \tilde{F} (L X)\\
       \longdownarrow &  & \longdownarrow\\
       F_1 (X) & \longrightarrow & F_1 (L X)
     \end{array} \]
  with the horizontal maps being equivalences. In particular, for all $Y \in
  \mathcal{D}^0$, there exists $X \in \mathcal{C}^0$ such that $Y \simeq L X$.
  Then $\tilde{F} (X) \rightarrow F_1 (X)$ is an equivalence, therefore so are
  $\tilde{F} (L X) \rightarrow F_1 (L X)$ and $\tilde{F} (Y) \rightarrow F_1
  (Y)$, which proves condition \ref{pt:l10p2}.
  
  We now assume that we are in the special case described in
  \Cref{rem:setup-proj-gen}. Suppose that condition \ref{pt:l10p2} holds. Note
  that $F_1$ is, by definition, $L$-invariant, therefore for all $X \in
  \mathcal{C}^0$, $\tilde{F}$ maps the unit map $X \rightarrow L X$ to an
  equivalence. Let $\mathcal{C}' \subseteq \mathcal{C}$ be the full
  subcategory spanned by those $X \in \mathcal{C}$ such that $\tilde{F}$ maps
  $X \rightarrow L X$ to an equivalence. Then $\mathcal{C}^0 \subseteq
  \mathcal{C}'$. It follows from
  \Cref{prop:left-deriv-n-fun,prop:struct-n-proj-gen-cats} that $\tilde{F}$
  preserves sifted colimits. Since $L$ preserves small colimits,
  $\mathcal{C}'$ is closed under sifted colimits, therefore $\mathcal{C}'
  =\mathcal{C}$ by \Cref{lem:nonab-deriv-cat-n-proj-gen}.
\end{proof}

\begin{remark}
  We conjecture that all conditions in \Cref{lem:L-inv-Lan} are equivalent
  under Setup~\ref{setup:Lan} without the assumptions in
  \Cref{rem:setup-proj-gen}.
\end{remark}

Now we describe how the setups above give rise to derived prismatic cohomology
in {\cite{Bhatt2019}}. Let $(A, I)$ be a bounded prism
{\cite[Def~3.2]{Bhatt2019}}. Let $\mathcal{C}= \tmop{Ani} (\tmop{Alg}_{A /
I})$ be the $\infty$\mbox{-}category of $A / I$-algebras and $\mathcal{D}
\subseteq \mathcal{C}$ the full subcategory of $p$-completed $A / I$-algebras.
In this case, the localization functor $\mathcal{C} \rightarrow \mathcal{D}$
is simply given by the $p$-completion $(-)_p^{\wedge}$. Let $\mathcal{C}^0
\subseteq \mathcal{C}$ be the full subcategory of polynomial $A / I$-algebras.
Then $\mathcal{D}^0 \subseteq \mathcal{D}$ is the full subcategory of
$p$-completed polynomial $A / I$-algebras. {\cite[§4.2]{Bhatt2019}} defines
the functors $F \assign \Prism_{\cdummy / A} \of \mathcal{D}^0 \rightarrow
D_{\tmop{comp}} (A)$ and $G \assign \overline{\Prism}_{\cdummy / A} \of
\mathcal{D}^0 \rightarrow D_{\tmop{comp}} (A / I)$, where $D_{\tmop{comp}}
(A)$ is the $\infty$\mbox{-}category of $(p, I)$-complete $A$-module spectra,
and $D_{\tmop{comp}} (A / I)$ is the $\infty$\mbox{-}category of $p$-complete
$A / I$-module spectra. In Setup~\ref{setup:Lan} and
\Cref{rem:setup-proj-gen}, we claim that the functor $\tilde{F}$ and
$\tilde{G}$ are left Kan extended from $\mathcal{D}^0$ after restriction to
$\mathcal{D}$. That is to say, $\tilde{F}$ and $\tilde{G}$ are left derived
functors $\mathbb{L} \Prism_{\cdummy / A}$ and $\mathbb{L}
\overline{\Prism}_{\cdummy / A}$ defined in {\cite[Cons~7.6]{Bhatt2019}}.
Thanks to \Cref{lem:L-inv-Lan}, it suffices to show that $\tilde{F}$ and
$\tilde{G}$ are $(-)_p^{\wedge}$-invariant. We will first describe our proof,
then we offer the lemmas used in the proof.

We start with $\tilde{G}$. Composing $G$ with the Postnikov tower
$D_{\tmop{comp}} (A / I) \rightarrow \tmop{DF}_{\tmop{comp}} (A / I), X
\mapsto (\tau_{\geq n} X)_{n \in (\mathbb{Z}, \geq)}$ where
$\tmop{DF}_{\tmop{comp}} (A / I) \assign \tmop{Fil} (D_{\tmop{comp}} (A / I))$
is the filtered derived $\infty$\mbox{-}category of $p$-completed $A /
I$-module spectra, we get a functor $G^P \of \mathcal{D}^0 \rightarrow
\tmop{DF}_{\tmop{comp}} (A / I)$ such that the {\tmdfn{union}} (see
\Cref{cor:L-inv-union}) $\tmop{Fil}^{- \infty} G^P \of \mathcal{D}^0
\rightarrow D_{\tmop{comp}} (A / I)$ is equivalent to $G$. It follows from the
Hodge--Tate comparison {\cite[Prop~6.2]{Bhatt2019}} that the functorial
comparison map $\left( \bigwedge^i L_{\cdummy / (A / I)} \{ - i \} [- i]
\right)_p^{\wedge} \rightarrow \tmop{gr}^{- i} \circ G^P$ is an equivalence).
Now \Cref{rem:setup-proj-gen} shows that $G^P \of \mathcal{D}^0 \rightarrow
\tmop{DF}_{\tmop{comp}} (A / I)$ gives rise to $\widetilde{G^P} \of
\mathcal{C} \rightarrow \tmop{DF}_{\tmop{comp}} (A / I)$ and the functor
$\left( \bigwedge^i L_{\cdummy / (A / I)} \{ - i \} [- i] \right)_p^{\wedge}
\of \mathcal{D}^0 \rightarrow \tmop{DF}_{\tmop{comp}} (A / I)$ gives rise to
some $\mathcal{C} \rightarrow \tmop{DF}_{\tmop{comp}} (A / I)$, which is
$\left( \bigwedge^i L_{\cdummy / (A / I)} \{ - i \} [- i] \right)_p^{\wedge}$
by \Cref{ex:setup-p-compl-cot-cx}, and in particular,
$(-)_p^{\wedge}$-invariant. It follows from \Cref{lem:assoc-graded-union-Lan}
that the associated graded pieces $\tmop{gr}^{- i} \circ \widetilde{G^P}$ are
$(-)_p^{\wedge}$-invariant and therefore the $(-)_p^{\wedge}$-invariance of
$\tilde{G}$ follows from \Cref{cor:L-inv-union}.

Note that $\tilde{F}$ coincides with $\tilde{G}$ composed with the derived
modulo $I$, that is, the composite functor $\tmop{Ani} (\tmop{Alg}_{A / I})
\xrightarrow{\tilde{F}} D_{\tmop{comp}} (A) \xrightarrow{\cdummy
\hat{\otimes}_A^{\mathbb{L}} (A / I)} D_{\tmop{comp}} (A / I)$. We deduce by
derived Nakayama {\cite[\href{https://stacks.math.columbia.edu/tag/0G1U}{Tag
0G1U}]{stacks-project}} that $\tilde{F}$ is also $(-)_p^{\wedge}$-invariant.

Here are the lemmas that we used in the argument above:

\begin{lemma}
  \label{lem:L-inv-gr}Let $\mathcal{C}$ be an $\infty$\mbox{-}category and
  $\mathcal{D} \subseteq \mathcal{C}$ a reflective subcategory with
  localization $L \of \mathcal{C} \rightarrow \mathcal{D}$. Let $\mathcal{E}$
  be a stable $\infty$\mbox{-}category. Let $F \of \mathcal{C} \rightarrow
  \tmop{Fil}^{\leq 0} (\mathcal{E})$ be a functor. If the {\tmdfn{associated
  graded pieces}} $\tmop{gr}^i \circ \tilde{F}$ are $L$-invariant for all $i
  \in \mathbb{Z}$, then so is $\tilde{F}$.
\end{lemma}

\begin{proof}
  For all $C \in \mathcal{C}$, we inductively show that the unit map $C
  \rightarrow L C$ induces an equivalence $\tmop{Fil}^i (\tilde{F} (C))
  \rightarrow \tmop{Fil}^i (\tilde{F} (L C))$. By assumption, this is true for
  all $i > 0$. Now consider the commutative diagram
  \[ \begin{array}{ccccc}
       \tmop{Fil}^{i + 1} (\tilde{F} (C)) & \longrightarrow & \tmop{Fil}^i
       (\tilde{F} (C)) & \longrightarrow & \tmop{gr}^i (\tilde{F} (C))\\
       \longdownarrow &  & \longdownarrow &  & \longdownarrow\\
       \tmop{Fil}^{i + 1} (\tilde{F} (L C)) & \longrightarrow & \tmop{Fil}^i
       (\tilde{F} (L C)) & \longrightarrow & \tmop{gr}^i (\tilde{F} (L C))
     \end{array} \]
  where the horizontal maps are fiber sequences. Suppose that the result is
  true for $i + 1$. Then the leftmost and the rightmost vertical maps are
  equivalences, therefore so is the middle vertical maps, which shows that the
  result is true for $i$.
\end{proof}

It then follows from definitions that

\begin{corollary}
  \label{cor:L-inv-union}Under the assumptions in \Cref{lem:L-inv-gr}, if we
  further assume that $\mathcal{E}$ admits filtered colimits, then the
  {\tmdfn{union}} $\tmop{Fil}^{- \infty} \circ \tilde{F} \of \mathcal{C}
  \rightarrow \tmop{Fil} (\mathcal{E}) \xrightarrow{\tmop{Fil}^{- \infty}}
  \mathcal{E}$ is also $L$-invariant.
\end{corollary}

\section{Animated ideals and PD-pairs}\label{sec:ani-pairs-pd-pairs}

In this section, we will first give an informal exposition of Smith ideals
introduced in {\cite{Hovey2014}} in terms of $\infty$-categories. See also
{\cite{White2017,White2017a}} for various generalizations. Then we will show
how to apply these ideas to define and study ``ideals'' of animated rings and
animated PD-pairs, which are the cornerstones of the animated theory of
crystalline cohomology.

\subsection{Smith ideals}\label{subsec:Smith-ideals}We fix a presentable
stable symmetric monoidal $\infty$\mbox{-}category $(\mathcal{C}, \otimes)$.
The reader should feel free to take the special case that $\mathcal{C}=
\tmop{Sp}$ is the $\infty$\mbox{-}category of spectra and $\otimes$ is the
smash product of spectra.

Consider the $1$\mbox{-}simplex $\Delta^1$, which is simply the
$1$\mbox{-}category associated to the ordinal $[1] = \{ 0 < 1 \}$. The
opposite category $(\Delta^1)^{\tmop{op}}$ has a symmetric monoidal structure
given by $\max \{ \cdummy, \cdummy \}$\footnote{This is informed to us by
Denis {\tmname{Nardin}}.}.

Thus the presentable stable $\infty$\mbox{-}category $\tmop{Fun}
((\Delta^1)^{\tmop{op}}, \mathcal{C})$ admits a presentable symmetric monoidal
structure given by the {\tmdfn{Day convolution}} $\otimes^{\tmop{Day}}$
{\cite[§3]{Nikolaus2016}}.

Informally, the unit object $\tmmathbf{1}_{\tmop{Fun} ((\Delta^1)^{\tmop{op}},
\mathcal{C})}$ is given by $(\tmmathbf{1}_{\mathcal{C}} \leftarrow 0) \in
\tmop{Fun} ((\Delta^1)^{\tmop{op}}, \mathcal{C})$, and given $n$ functors
$F_1, \ldots, F_n \in \tmop{Fun} ((\Delta^1)^{\tmop{op}}, \mathcal{C})$, the
Day convolution $F_1 \otimes^{\tmop{Day}} \cdots \otimes^{\tmop{Day}} F_n$ is
given as follows: $F_1, \ldots, F_n$ determines an $n$-cube $F \of
(\Delta^1)^{\tmop{op}} \times \cdots \times (\Delta^1)^{\tmop{op}} \rightarrow
\mathcal{C}, (e_1, \ldots, e_n) \mapsto F_1 (e_1) \otimes \cdots \otimes F_n
(e_n)$. This cube, except the final vertex, determines a ``cubical pushout''
mapping to the final vertex: $(F (0, \ldots, 0) \leftarrow
\tmop{colim}_{(\Delta^1)^{\tmop{op}} \times \cdots \times
(\Delta^1)^{\tmop{op}} \setminus (0, \ldots, 0)} F)$, which is $F_1
\otimes^{\tmop{Day}} \cdots \otimes^{\tmop{Day}} F_n$.

In particular, when $n = 2$, the Day convolution of $(X_0 \leftarrow X_1)$ and
$(Y_0 \leftarrow Y_1)$ is given by $(X_0 \otimes Y_0 \leftarrow (X_0 \otimes
Y_1) \amalg_{X_1 \otimes Y_1} (X_1 \otimes Y_0))$. This is essentially
equivalent to the {\tmdfn{pushout product monoidal structure}} in
{\cite[Thm~1.2]{Hovey2014}}.

On the other hand, there exists a pointwise symmetric monoidal structure
$\otimes$ on the stable $\infty$\mbox{-}category $\tmop{Fun} (\Delta^1,
\mathcal{C})$ where $F_1 \otimes \cdots \otimes F_n$ is given by the functor
$e \mapsto F_1 (e) \otimes \cdots \otimes F_n (e)$.

There is a comparison between these two stable symmetric monoidal
$\infty$\mbox{-}categories:

\begin{proposition}[M.~{\tmname{Ramzi}}]
  \label{prop:eq-day-point}Let $\mathcal{C}$ be a presentable stable symmetric
  monoidal $\infty$-category. Then there is an equivalence $\tmop{Fun}
  ((\Delta^1)^{\tmop{op}}, \mathcal{C}) \simeq \tmop{Fun} (\Delta^1,
  \mathcal{C})$ in $\tmop{CAlg}_{\mathcal{C}} (\Pr^L)$. On the level of
  underlying categories, the equivalence is given by $\tmop{Fun}
  ((\Delta^1)^{\tmop{op}}, \mathcal{C}) \ni F \mapsto (F (0) \rightarrow
  \tmop{cofib} (F (1) \rightarrow F (0))) \in \tmop{Fun} (\Delta^1,
  \mathcal{C})$ of which the inverse is given by $\tmop{Fun} (\Delta^1,
  \mathcal{C}) \ni G \mapsto (G (0) \leftarrow \tmop{fib} (G (0) \rightarrow G
  (1))) \in \tmop{Fun} ((\Delta^1)^{\tmop{op}}, \mathcal{C})$.
\end{proposition}

\begin{proof}
  The pair of inverse functors are clearly well-defined. It suffices to endow
  the functor
  \begin{eqnarray*}
    K \of (\tmop{Fun} ((\Delta^1)^{\tmop{op}}, \mathcal{C}),
    \otimes^{\tmop{Day}}) & \longrightarrow & (\tmop{Fun} (\Delta^1,
    \mathcal{C}), \otimes)\\
    (F (1) \rightarrow F (0)) & \longmapsto & (F (0) \rightarrow \tmop{cofib}
    (F (1) \rightarrow F (0)))
  \end{eqnarray*}
  a symmetric monoidal structure. Such a structure exists for every pointed
  presentable symmetric monoidal $\infty$-categories. Indeed, since
  $\tmop{Fun} ((\Delta^1)^{\tmop{op}}, \mathcal{C}) \simeq \tmop{Fun}
  ((\Delta^1)^{\tmop{op}}, \tmop{An}_{\ast}) \otimes \mathcal{C}$, and the
  same for $\tmop{Fun} (\Delta^1, \mathcal{C})$, without loss of generality,
  we may assume that $\mathcal{C}= \tmop{An}_{\ast}$. By the universal
  property of Day convolution, it suffices to endow the composite functor
  \[ \Delta^1 \longrightarrow \tmop{Fun} ((\Delta^1)^{\tmop{op}},
     \tmop{An}_{\ast}) \xrightarrow{K} \tmop{Fun} (\Delta^1, \tmop{An}_{\ast})
  \]
  a symmetric monoidal structure, where the first functor is the Yoneda
  embedding combined with adjoining a point $\tmop{An} \rightarrow
  \tmop{An}_{\ast}$, but since $\Delta^1$ is a $1$-category, it is equivalent
  to endowing the composite functor
  \[ \Delta^1 \longrightarrow \tmop{Fun} ((\Delta^1)^{\tmop{op}},
     \tmop{Set}_{\ast}) \xrightarrow{K} \tmop{Fun} (\Delta^1,
     \tmop{Set}_{\ast}) \]
  a symmetric monoidal structure. Note that a symmetric monoidal functor out
  of $(\Delta^1, \max)$ is the same as a map of idempotent algebras, and this
  can be checked directly.
\end{proof}

Now we assume that $\mathcal{C}$ admits a symmetric monoidal $t$-structure
$(\mathcal{C}_{\geq 0}, \mathcal{C}_{\leq 0})$ which is compatible with
filtered colimits. This is the case when $\mathcal{C}= \tmop{Sp}$ and
$(\mathcal{C}_{\geq 0}, \mathcal{C}_{\leq 0})$ is the canonical $t$-structure
for spectra. Then so does $\tmop{Fun} ((\Delta^1)^{\tmop{op}}, \mathcal{C})$,
that is to say, $\tmop{Fun} ((\Delta^1)^{\tmop{op}}, \mathcal{C})_{\geq 0}
\assign \tmop{Fun} ((\Delta^1)^{\tmop{op}}, \mathcal{C}_{\geq 0})$ and
$\tmop{Fun} ((\Delta^1)^{\tmop{op}}, \mathcal{C})_{\leq 0} \assign \tmop{Fun}
((\Delta^1)^{\tmop{op}}, \mathcal{C}_{\leq 0})$. Transferring this
$t$-structure along the equivalence in \Cref{prop:eq-day-point}, we get a
$t$-structure on $\tmop{Fun} (\Delta^1, \mathcal{C})$ where $\tmop{Fun}
(\Delta^1, \mathcal{C})_{\geq 0} \subseteq \tmop{Fun} (\Delta^1, \mathcal{C})$
is spanned by edges $f \of X \rightarrow Y$ in $\mathcal{C}$ such that $X \in
\mathcal{C}_{\geq 0}$ and $\tmop{fib} (Y \rightarrow X) \in \mathcal{C}_{\geq
0}$, or equivalently, $X, Y \in \mathcal{C}_{\geq 0}$ and $f$ is
$1$-connective (that is to say, $\pi_0 (f)$ is surjective). In summary,

\begin{corollary}
  \label{cor:conn-eq-day-point}The equivalence in \Cref{prop:eq-day-point}
  induces an equivalence of presentable symmetric monoidal full subcategories
  $\tmop{Fun} ((\Delta^1)^{\tmop{op}}, \mathcal{C}_{\geq 0}) \simeq \tmop{Fun}
  (\Delta^1, \mathcal{C})_{\geq 0}$, where the full subcategory $\tmop{Fun}
  ((\Delta^1)^{\tmop{op}}, \mathcal{C})_{\geq 0}$ is spanned by maps $Y
  \leftarrow X$ in $\mathcal{C}_{\geq 0}$.
\end{corollary}

Passing to $\mathbb{E}_n$-algebras for any $n \in \mathbb{N} \cup \{ \infty
\}$, we get

\begin{corollary}
  \label{cor:smith-eq}There is an equivalence between the
  $\infty$\mbox{-}category of $\mathbb{E}_n$-algebras in $(\tmop{Fun}
  ((\Delta^1)^{\tmop{op}}, \mathcal{C}), \otimes^{\tmop{Day}})$ and the
  $\infty$\mbox{-}category of $\mathbb{E}_n$-maps between
  $\mathbb{E}_n$-algebras in $(\mathcal{C}, \otimes)$. This equivalence
  induces an equivalence between the full subcategory spanned by connective
  $\mathbb{E}_n$-algebras in $(\tmop{Fun} ((\Delta^1)^{\tmop{op}},
  \mathcal{C}), \otimes^{\tmop{Day}})$ and the full subcategory spanned by
  $1$-connective $\mathbb{E}_n$-maps between connective
  $\mathbb{E}_n$-algebras in $(\mathcal{C}, \otimes)$.
\end{corollary}

Explicitly, for any $\mathbb{E}_n$-algebra in $((\tmop{Fun}
(\Delta^1)^{\tmop{op}}, \mathcal{C}), \otimes^{\tmop{Day}})$ of which the
underlying object is $(A \leftarrow I)$, the object $A \in \mathcal{C}$, the
cofiber $\tmop{cofib} (I \rightarrow A)$ and the map $A \rightarrow
\tmop{cofib} (I \rightarrow A)$ admit canonical $\mathbb{E}_n$-algebra
structures. We can then understand $I$ as an ``ideal'' of
$\mathbb{E}_n$-algebra $A$. When $A$ is connective, the previous
identification also describes connective ``ideals'' of $A$. This is the Smith
ideal in {\cite{Hovey2014}}, which gives rise to a reasonable theory of ideals
(resp. connective ideals) of $\mathbb{E}_n$-rings (resp. connective
$\mathbb{E}_n$-rings) when $\mathcal{C}$ is the presentable stable symmetric
monoidal $\infty$\mbox{-}category $\tmop{Sp}$ of spectra.

In the rest of this section, we will study the animated analogue of the
preceding equivalence, that is to say, ``ideals'' and ``PD-ideals'' of an
animated ring. To do so, we need to exploit more structures of $D
(\mathbb{Z})$.

\subsection{Animated (PD-)pairs}\label{subsec:ani-pairs-PD-pairs}In this
subsection, we introduce the central object of this section: animated pairs
and (absolute) animated PD-pairs.

\begin{notation}
  Let $\tmop{Pair}$ denote the $1$\mbox{-}category of ring-ideal pairs $(A,
  I)$, that is, a (commutative) ring $A$ along with an ideal $I \subseteq A$.
  Let $\tmop{Pair}^{\gamma}$ denote the $1$\mbox{-}category of divided power
  rings $(A, I, \gamma)$
  {\cite[\href{https://stacks.math.columbia.edu/tag/07GU}{Tag
  07GU}]{stacks-project}}. The {\tmdfn{(absolute) PD-envelope functor}}
  {\cite[\href{https://stacks.math.columbia.edu/tag/07H9}{Tag
  07H9}]{stacks-project}} $\tmop{Pair} \rightarrow \tmop{Pair}^{\gamma}$,
  being the left adjoint to the forgetful functor $\tmop{Pair}^{\gamma}
  \rightarrow \tmop{Pair}$, is denoted by $(A, I) \mapsto D_A (I)$.
\end{notation}

\begin{notation}
  Let $\tmop{Inj} \subseteq \tmop{Fun} ((\Delta^1)^{\tmop{op}}, \tmop{Ab})$ be
  the full subcategory spanned by injective maps $M \leftarrowtail M'$.
\end{notation}

We note that there is a pair $\tmop{Inj} \rightleftarrows \tmop{Pair}$ of
adjoint functors where $\tmop{Pair} \rightarrow \tmop{Fun}
((\Delta^1)^{\tmop{op}}, \tmop{Ab})$ is the forgetful functor $(A, I) \mapsto
(A \leftarrow I)$, and $\tmop{Fun} ((\Delta^1)^{\tmop{op}}, \tmop{Ab})
\rightarrow \tmop{Pair}$ is the ``symmetric product'' $(M \leftarrowtail M')
\mapsto (\tmop{Sym}_{\mathbb{Z}} (M), M' \tmop{Sym}_{\mathbb{Z}} (M))$ where
$\tmop{Sym}_{\mathbb{Z}} (M) \leftarrowtail M' \tmop{Sym}_{\mathbb{Z}} (M)$ is
the ideal generated by elements in $M'$.

Unfortunately, the category $\tmop{Inj}$ might not be $1$\mbox{-}projectively
generated. In particular, we cannot apply \Cref{cor:ani-adjoint-funs} to
deduce that the category $\tmop{Pair}$ is $1$\mbox{-}projectively generated
(we believe that it is not), and to construct ``$\tmop{Ani} (\tmop{Pair})$''.
In fact, we need to embed $\tmop{Pair}$ as a full subcategory of a
$1$\mbox{-}projectively generated $1$\mbox{-}category and then the
$\infty$\mbox{-}category of animated pairs coincides with the animation of
that larger $1$\mbox{-}category.

We begin by analyzing the full subcategory $\tmop{Inj} \subseteq \tmop{Fun}
((\Delta^1)^{\tmop{op}}, \tmop{Ab})$. Note that $\{ \mathbb{Z} \leftarrow 0,
\tmop{id}_{\mathbb{Z}} \of \mathbb{Z} \leftarrow \mathbb{Z} \} \subseteq
\tmop{Inj}$ is a set of compact $1$\mbox{-}projective generators for
$\tmop{Fun} ((\Delta^1)^{\tmop{op}}, \tmop{Ab})$ by
\Cref{lem:fun-n-cat-proj-gen}.

\begin{notation}
  Let $\tmop{Inj}^{\tmop{st}} \subseteq \tmop{Fun} ((\Delta^1)^{\tmop{op}},
  \tmop{Ab})$ denote the full subcategory generated by $\{ \mathbb{Z}
  \leftarrow 0, \tmop{id}_{\mathbb{Z}} \of \mathbb{Z} \leftarrow \mathbb{Z}
  \}$ under finite coproducts, which is in fact a full subcategory of
  $\tmop{Inj}$.
\end{notation}

It follows from \Cref{prop:struct-n-proj-gen-cats} that there is an
equivalence $\mathcal{P}_{\Sigma, 1} (\tmop{Inj}^{\tmop{st}})
\xrightarrow{\simeq} \tmop{Fun} ((\Delta^1)^{\tmop{op}}, \tmop{Ab})$ of
$\infty$\mbox{-}categories. It then follows from
\Cref{lem:left-deriv-fun-adjoint} that the fully faithful embedding
$\tmop{Inj} \hookrightarrow \mathcal{P}_{\Sigma, 1} (\tmop{Inj}^{\tmop{st}})$
admits a left adjoint given by the left derived functor of the inclusion
$\tmop{Inj}^{\tmop{st}} \hookrightarrow \tmop{Inj}$. We claim that

\begin{lemma}
  \label{lem:ess-img-inj-ab}The essential image of $\tmop{Inj} \hookrightarrow
  \mathcal{P}_{\Sigma, 1} (\tmop{Inj}^{\tmop{st}})$ is spanned by those
  finite-product-preserving functors $F \of
  (\tmop{Inj}^{\tmop{st}})^{\tmop{op}} \rightarrow \tmop{Set}$ which maps the
  edge $(\mathbb{Z} \leftarrow 0) \rightarrow (\tmop{id}_{\mathbb{Z}} \of
  \mathbb{Z} \leftarrow \mathbb{Z})$ in $\tmop{Inj}^{\tmop{st}}$ to an
  injective map of sets.
\end{lemma}

\begin{proof}
  The functors $\tmop{Fun} ((\Delta^1)^{\tmop{op}}, \tmop{Ab})
  \rightrightarrows \tmop{Set}$ corepresented by $\tmop{id}_{\mathbb{Z}} \in
  \tmop{Inj}^{\tmop{st}}$ and $(\mathbb{Z} \leftarrow 0) \in
  \tmop{Inj}^{\tmop{st}}$ are given by $(A \leftarrow A') \mapsto A'$ and $(A
  \leftarrow A') \mapsto A$ respectively, and the edge $(\mathbb{Z} \leftarrow
  0) \rightarrow (\tmop{id}_{\mathbb{Z}} \of \mathbb{Z} \leftarrow
  \mathbb{Z})$ gives rise to the natural map $A \leftarrow A'$ of the two
  functors. It follows that an object $F \in \tmop{Fun}
  ((\Delta^1)^{\tmop{op}}, \tmop{Ab})$ lies in $\tmop{Inj}$ if and only if the
  value of the natural map on $F$ is an injection. The result then follows
  from the equivalence $\mathcal{P}_{\Sigma, 1} (\tmop{Inj}^{\tmop{st}})
  \xrightarrow{\simeq} \tmop{Fun} ((\Delta^1)^{\tmop{op}}, \tmop{Ab})$.
\end{proof}

\begin{notation}
  Let $\tmop{Pair}^{\tmop{st}} \subseteq \tmop{Pair}$ denote the full
  subcategory spanned by {\tmdfn{standard pairs}}, being the images of
  $\tmop{Inj}^{\tmop{st}}$ under the functor $\tmop{Inj} \rightarrow
  \tmop{Pair}$. In other words, a standard pair is a pair of form $(\mathbb{Z}
  [X, Y], (Y))$ for finite sets $X$ and $Y$.
\end{notation}

Then by \Cref{cor:nonab-deriv-cat-adjoint-fun}, we have

\begin{lemma}
  \label{lem:1-forget-pair}The free-pair functor $\tmop{Inj}^{\tmop{st}}
  \rightarrow \tmop{Pair}^{\tmop{st}}$, being essentially surjective, gives
  rise to the {\tmdfn{forgetful functor}} $\mathcal{P}_{\Sigma, 1}
  (\tmop{Pair}^{\tmop{st}}) \rightarrow \tmop{Fun} ((\Delta^1)^{\tmop{op}},
  \tmop{Ab})$ which is conservative and preserves sifted colimits.
\end{lemma}

{\construction{\label{cons:pair-st-pair-adjunct}\Cref{lem:left-deriv-fun-adjoint}
gives us a canonical pair of adjoint functors $\mathcal{P}_{\Sigma, 1}
(\tmop{Pair}^{\tmop{st}}) \rightleftarrows \tmop{Pair}$, where
$\mathcal{P}_{\Sigma, 1} (\tmop{Pair}^{\tmop{st}}) \rightarrow \tmop{Pair}$ is
the left derived $1$\mbox{-}functor (\Cref{prop:left-deriv-n-fun}) of the
inclusion $\tmop{Pair}^{\tmop{st}} \hookrightarrow \tmop{Pair}$, and
$\tmop{Pair} \rightarrow \mathcal{P}_{\Sigma, 1} (\tmop{Pair}^{\tmop{st}})$ is
the given by the restricted Yoneda embedding $(A, I) \mapsto
\tmop{Hom}_{\tmop{Pair}} (\cdummy, (A, I))$.}}

We first note that the forgetful functors are compatible:

\begin{lemma}
  \label{lem:pair-ab-forget-embedding}There is a commutative diagram
  \[ \begin{array}{ccc}
       \tmop{Pair} & \longrightarrow & \mathcal{P}_{\Sigma, 1}
       (\tmop{Pair}^{\tmop{st}})\\
       \longdownarrow &  & \longdownarrow\\
       \tmop{Inj} & \longhookrightarrow & \tmop{Fun} (\Delta^{1, \tmop{op}},
       \tmop{Ab})
     \end{array} \]
  of $1$\mbox{-}categories, where the vertical arrows are forgetful functors,
  and the top horizontal arrow is described in
  Construction~\ref{cons:pair-st-pair-adjunct}.
\end{lemma}

\begin{proof}
  Given a pair $(A, I) \in \tmop{Pair}$, the image in $\mathcal{P}_{\Sigma, 1}
  (\tmop{Pair}^{\tmop{st}})$ is given by $\tmop{Pair}^{\tmop{st}} \ni (B, J)
  \mapsto \tmop{Hom}_{\tmop{Pair}} ((B, J), (A, I))$, subsequently mapped to
  $\tmop{Inj}^{\tmop{st}} \ni (M \leftarrowtail M') \mapsto
  \tmop{Hom}_{\tmop{Pair}} ((\tmop{Sym}_{\mathbb{Z}} (M), M'
  \tmop{Sym}_{\mathbb{Z}} (M)), (A, I)) \cong \tmop{Hom}_{\tmop{Fun}
  ((\Delta^1)^{\tmop{op}}, \tmop{Ab})} (M \leftarrowtail M', A \leftarrowtail
  I)$. The other composite is the same. This identification is functorial in
  $(A, I)$.
\end{proof}

Now we show that $\tmop{Pair} \rightarrow \mathcal{P}_{\Sigma, 1}
(\tmop{Pair}^{\tmop{st}})$ is an embedding to a $1$\mbox{-}projectively
generated $1$\mbox{-}category that we want. The trick is to talk about maps
$(\mathbb{Z} [X], 0) \rightarrow (A, I)$ and $(\mathbb{Z} [X], (X))
\rightarrow (A, I)$ in place of elements in $A$ and $I$ respectively to do
certain ``element chasing''. We remind the reader that
$\tmop{Poly}_{\mathbb{Z}}$ is a set of compact projective objects for
$\tmop{Ring}$, which gives rise to an equivalence $\mathcal{P}_{\Sigma, 1}
(\tmop{Poly}_{\mathbb{Z}}) \simeq \tmop{Ring}$ of $1$\mbox{-}categories, where
$\tmop{Poly}_{\mathbb{Z}}$ is the $1$\mbox{-}category of polynomial rings.

\begin{lemma}
  \label{lem:pair-Psigma1-full-faith}The functor $\tmop{Pair} \rightarrow
  \mathcal{P}_{\Sigma, 1} (\tmop{Pair}^{\tmop{st}})$ described in
  Construction~\ref{cons:pair-st-pair-adjunct} is fully faithful.
\end{lemma}

\begin{proof}
  The faithfulness follows from \Cref{lem:pair-ab-forget-embedding} and the
  faithfulness of the forgetful functor $\tmop{Pair} \rightarrow \tmop{Inj}$.
  Given two pairs $(A, I), (B, J)$ in $\tmop{Pair}$ and a natural map
  \[ \nobracket \tmop{Hom}_{\tmop{Pair}} (\cdummy, (A, I))
     |_{(\tmop{Pair}^{\tmop{st}})^{\tmop{op}}} \rightarrow \nobracket
     \tmop{Hom}_{\tmop{Pair}} (\cdummy, (B, J))
     |_{(\tmop{Pair}^{\tmop{st}})^{\tmop{op}}} \]
  of finite-product-preserving functors $(\tmop{Pair}^{\tmop{st}})^{\tmop{op}}
  \rightrightarrows \tmop{Set}$, we need to show that this is induced by some
  map $(A, I) \rightarrow (B, J)$ of pairs.
  
  By \Cref{lem:pair-ab-forget-embedding}, there exists a unique map $(A
  \leftarrowtail I) \rightarrow (B \leftarrowtail J)$ in $\tmop{Inj}$ which
  corresponds to the natural transform after composition
  $(\tmop{Inj}^{\tmop{st}})^{\tmop{op}} \rightarrow
  (\tmop{Pair}^{\tmop{st}})^{\tmop{op}} \rightrightarrows \tmop{Set}$.
  
  Similarly, since $\mathcal{P}_{\Sigma, 1} (\tmop{Poly}_{\mathbb{Z}}) \simeq
  \tmop{Ring}$, there exists a unique map $A \rightarrow B$ of rings which
  corresponds to the natural transform after composition
  $\tmop{Poly}_{\mathbb{Z}}^{\tmop{op}} \rightarrow
  (\tmop{Pair}^{\tmop{st}})^{\tmop{op}} \rightrightarrows \tmop{Set}$ where
  $\tmop{Poly}_{\mathbb{Z}} \rightarrow \tmop{Pair}^{\tmop{st}}$ is given by
  $R \mapsto (R, 0)$.
  
  It then follows from the commutativity of the diagram
  \begin{equation}
    \begin{array}{ccc}
      \tmop{Free}_{\mathbb{Z}}^{\tmop{fin}} & \xrightarrow{\tmop{Sym}} &
      \tmop{Poly}_{\mathbb{Z}}\\
      \longdownarrow &  & \longdownarrow\\
      \tmop{Inj}^{\tmop{st}} & \longrightarrow & \tmop{Pair}^{\tmop{st}}
    \end{array} \label{eq:free-ab-poly-C0-D0}
  \end{equation}
  of $1$\mbox{-}categories, where $\tmop{Free}_{\mathbb{Z}}^{\tmop{fin}}$ is
  the $1$\mbox{-}category of finite free abelian groups, with
  finite-coproduct-preserving functors that the two maps $(A \leftarrowtail I)
  \rightarrow (B \leftarrowtail J)$ in $\tmop{Inj}$ and $A \rightarrow B$ in
  $\tmop{Ring}$ are compatible, which gives rise to a map $(A, I) \rightarrow
  (B, J)$ in $\tmop{Pair}$.
\end{proof}

Now we characterize the image of this embedding:

\begin{lemma}
  \label{lem:ess-img-pair}The square in \Cref{lem:pair-ab-forget-embedding} is
  Cartesian. Equivalently by \Cref{lem:ess-img-inj-ab}, the essential image of
  the fully faithful functor $\tmop{Pair} \hookrightarrow \mathcal{P}_{\Sigma,
  1} (\tmop{Pair}^{\tmop{st}})$ is spanned by those finite-product-preserving
  functors $F \of (\tmop{Pair}^{\tmop{st}})^{\tmop{op}} \rightarrow
  \tmop{Set}$ which maps the edge $(\mathbb{Z} [X], 0) \rightarrow (\mathbb{Z}
  [X], (X))$ in $\tmop{Pair}^{\tmop{st}}$ to an injective map of sets.
\end{lemma}

\begin{proof}
  Let $F \of (\tmop{Pair}^{\tmop{st}})^{\tmop{op}} \rightarrow \tmop{Set}$ be
  a functor which preserves finite products such that the composite
  $(\tmop{Inj}^{\tmop{st}})^{\tmop{op}} \rightarrow
  (\tmop{Pair}^{\tmop{st}})^{\tmop{op}} \xrightarrow{F} \tmop{Set}$ belongs to
  the full subcategory $\tmop{Inj}$ under the identification
  $\mathcal{P}_{\Sigma} (\tmop{Inj}^{\tmop{st}}) \simeq \tmop{Fun}
  ((\Delta^1)^{\tmop{op}}, \tmop{Ab})$. The goal is to show that there exists
  a pair $(A, I) \in \tmop{Pair}$ which represents $F$.
  
  Let $(A \leftarrowtail I) \in \tmop{Inj}$ correspond to the composite
  functor $(\tmop{Inj}^{\tmop{st}})^{\tmop{op}} \rightarrow
  (\tmop{Pair}^{\tmop{st}})^{\tmop{op}} \xrightarrow{F} \tmop{Set}$, and the
  map $A \leftarrowtail I$ of underlying sets is precisely induced by the map
  $(\mathbb{Z} [X], 0) \rightarrow (\mathbb{Z} [X], (X))$ in
  $\tmop{Pair}^{\tmop{st}}$.
  
  The ring structure is given as follows: the functor
  $\tmop{Poly}_{\mathbb{Z}} \rightarrow \tmop{Pair}^{\tmop{st}}$ given by $R
  \mapsto (R, 0)$ preserves finite coproducts, thus the composite functor
  $\tmop{Poly}_{\mathbb{Z}}^{\tmop{op}} \rightarrow
  (\tmop{Pair}^{\tmop{st}})^{\tmop{op}} \xrightarrow{F} \tmop{Set}$ preserves
  finite products, which corresponds to a ring structure on $A$. The
  compatibility follows from \eqref{eq:free-ab-poly-C0-D0}.
  
  Now we show that $A \leftarrowtail I$ is an ideal, that is to say, the ring
  multiplication $A \times A \rightarrow A$ restricts to a map $A \times I
  \rightarrow I$. By the above construction, $A = F (\mathbb{Z} [Y], 0)$ and
  $I = F (\mathbb{Z} [X], (X))$, and since $F$ preserves finite products, $A
  \times I = F (\mathbb{Z} [X, Y], (X))$. Consider $(\mathbb{Z} [T], (T)) \in
  \tmop{Pair}^{\tmop{st}}$. The map $(\mathbb{Z} [T], (T)) \rightarrow
  (\mathbb{Z} [X, Y], (X)), T \mapsto XY$ in $\tmop{Pair}^{\tmop{st}}$ induces
  a map $A \times I \rightarrow I$. The commutative diagram
  \[ \begin{array}{ccc}
       (\mathbb{Z} [X, Y], 0) & \longrightarrow & (\mathbb{Z} [T], 0)\\
       \longdownarrow &  & \longdownarrow\\
       (\mathbb{Z} [X, Y], (X)) & \longrightarrow & (\mathbb{Z} [T], (T))
     \end{array} \]
  in $\tmop{Pair}^{\tmop{st}}$ shows that the preceding map $A \times I
  \rightarrow I$ is compatible with the ring structure and the inclusion $I
  \rightarrow A$.
  
  It remains to construct an isomorphism $F \rightarrow \nobracket
  \tmop{Fun}_{\tmop{Pair}} (\cdummy, (A, I))
  |_{(\tmop{Pair}^{\tmop{st}})^{\tmop{op}}}$ of finite-product-preserving
  functors $(\tmop{Pair}^{\tmop{st}})^{\tmop{op}} \rightrightarrows
  \tmop{Set}$. Composing with the functor
  $(\tmop{Inj}^{\tmop{st}})^{\tmop{op}} \rightarrow
  (\tmop{Pair}^{\tmop{st}})^{\tmop{op}}$ denoted by $j$, we get a map $F \circ
  j \rightarrow \tmop{Fun}_{\tmop{Pair}} (\cdummy, (A, I)) \circ j$ of
  functors $(\tmop{Inj}^{\tmop{st}})^{\tmop{op}} \rightrightarrows \tmop{Set}$
  which is an equivalence by construction (and the adjunction $\tmop{Fun}
  (\Delta^1, \tmop{Ab})_{\tmop{inj}} \rightleftarrows \tmop{Pair}$). We need
  to show that this equivalence descends along the essentially surjective
  functor $j$.
  
  First, for any $(B, J) \in \tmop{Pair}^{\tmop{st}}$, by picking any lift
  under $j$, the map $F (B, J) \rightarrow \tmop{Fun}_{\tmop{Pair}} ((B, J),
  (A, I))$ could be described as follows: for any $f \in F (B, J)$ and any $b
  \in B$, the element $b$ corresponds uniquely to a map $\overline{b} \of
  (\mathbb{Z} [t], 0) \rightarrow (B, J)$ of pairs. Note that
  $\overline{b}^{\ast} (f) \in F (\mathbb{Z} [t], 0) \cong A$. The image of
  $f$, as a map $(B, J) \rightarrow (A, I)$ of pairs, is concretely given by
  $b \mapsto \overline{b}^{\ast} (f)$, which is independent of the choice of
  the lift of $(B, J)$.
  
  Now it remains to show that, for any map $\varphi \of (B, J) \rightarrow (C,
  K)$ of pairs, the diagram
  \[ \begin{array}{ccc}
       F (B, J) & \longrightarrow & \tmop{Fun}_{\tmop{Pair}} ((B, J), (A,
       I))\\
       \longuparrow &  & \longuparrow\\
       F (C, K) & \longrightarrow & \tmop{Fun}_{\tmop{Pair}} ((C, K), (A, I))
     \end{array} \]
  is commutative. Indeed, for any $f \in F (C, K)$, the image in
  $\tmop{Fun}_{\tmop{Pair}} ((C, K), (A, I))$ is given by $c \mapsto
  \overline{c}^{\ast} (f)$, and the image in $\tmop{Fun}_{\tmop{Pair}} ((B,
  J), (A, I))$ is given by $b \mapsto \overline{\varphi (b)}^{\ast} (f)$. On
  the other hand, the image of $f$ in $F (B, J)$ is $\varphi^{\ast} (f)$, and
  the image in $\tmop{Fun}_{\tmop{Pair}} ((B, J), (A, I))$ is given by $b
  \mapsto \overline{b}^{\ast} (\varphi^{\ast} (f))$. The result follows from
  the fact that $\varphi \circ \overline{b} = \overline{\varphi (b)}$ as maps
  $(\mathbb{Z} [t], 0) \rightrightarrows (C, K)$ of pairs.
\end{proof}

\begin{remark}
  The $1$\mbox{-}category $\mathcal{P}_{\Sigma, 1} (\tmop{Pair}^{\tmop{st}})$
  contains more objects than $\tmop{Pair}$. They might be of independent
  interest. For example, let $A$ be a ring and $I$ an invertible $A$-module
  along with a map $j \of I \rightarrow A$ of $A$-modules. If the map $j$ in
  question is not injective, then it does not ``faithfully'' correspond to a
  ring-ideal pair such as $(A, \tmop{im} (j))$, that is to say, it represents
  an object in $\mathcal{P}_{\Sigma, 1} (\tmop{Pair}^{\tmop{st}})$ which is
  different from $(A, \tmop{im} (j))$. In fact, the $1$\mbox{-}category
  $\mathcal{P}_{\Sigma, 1} (\tmop{Pair}^{\tmop{st}})$ could be identified with
  the $1$\mbox{-}category of commutative algebra objects in $\tmop{Fun}
  (\Delta^1, \tmop{Ab})_{\tmop{surj}}$ with pushout product monoidal
  structure, $1$\mbox{-}categorical version of \Cref{subsec:Smith-ideals}, or
  equivalently, the category of {\tmdfn{quasi-ideals}}\footnote{We thank Ofer
  {\tmname{Gabber}} for informing us this concept.} in
  {\cite[§3.3]{Drinfeld2021}}.
\end{remark}

We now develop a PD analogue as follows:

\begin{notation}
  Let $\tmop{Pair}^{\gamma, \tmop{st}} \subseteq \tmop{Pair}^{\gamma}$ denote
  the full subcategory spanned by the images of $(A, I) \in
  \tmop{Pair}^{\tmop{st}}$ under the functor of PD-envelope
  {\cite[\href{https://stacks.math.columbia.edu/tag/07H9}{Tag
  07H9}]{stacks-project}}, denoted by $(D_A (I) \twoheadrightarrow A / I,
  \gamma)$ instead of the cumbersome notation $(D_A (I), \ker (D_A (I)
  \twoheadrightarrow A / I), \gamma)$.
\end{notation}

{\construction{\label{cons:pdpair-st-pdpair-adjunct}By
\Cref{lem:left-deriv-fun-adjoint}, we get a pair $\mathcal{P}_{\Sigma, 1}
(\tmop{Pair}^{\gamma, \tmop{st}}) \rightleftarrows \tmop{Pair}^{\gamma}$ of
adjoint functors. Explicitly, the objects in $\tmop{Pair}^{\gamma, \tmop{st}}$
are of the form $D_{\mathbb{Z} [X, Y]} (Y) \cong \Gamma_{\mathbb{Z} [X]} (Y)$
for finite sets $X$ and $Y$.}}

On the other hand, it follows from \Cref{cor:nonab-deriv-cat-adjoint-fun} that

\begin{lemma}
  \label{lem:1-forget-pdpair}The PD-envelope functor $\tmop{Pair}^{\tmop{st}}
  \rightarrow \tmop{Pair}^{\gamma, \tmop{st}}$, being essentially surjective,
  gives rise to the {\tmdfn{forgetful functor}} $\mathcal{P}_{\Sigma, 1}
  (\tmop{Pair}^{\gamma, \tmop{st}}) \rightarrow \mathcal{P}_{\Sigma, 1}
  (\tmop{Pair}^{\tmop{st}})$ which is conservative and preserves sifted
  colimits.
\end{lemma}

There is another {\tmdfn{forgetful functor}} $\tmop{Pair}^{\gamma}
\rightarrow \tmop{Pair}$. These functors are compatible:

\begin{lemma}
  \label{lem:1-pdpair-pair-forget-embedding}The diagram
  \[ \begin{array}{ccc}
       \tmop{Pair}^{\gamma} & \longrightarrow & \mathcal{P}_{\Sigma, 1}
       (\tmop{Pair}^{\gamma, \tmop{st}})\\
       \longdownarrow &  & \longdownarrow\\
       \tmop{Pair} & \longhookrightarrow & \mathcal{P}_{\Sigma, 1}
       (\tmop{Pair}^{\tmop{st}})
     \end{array} \]
  is a commutative diagram of $1$\mbox{-}categories, where vertical arrows are
  forgetful functors, and the top horizontal arrow is described in
  Construction~\ref{cons:pdpair-st-pdpair-adjunct}.
\end{lemma}

\begin{proof}
  For any PD-pair $(A, I, \gamma) \in \tmop{Pair}^{\gamma}$, the image in
  $\mathcal{P}_{\Sigma, 1} (\tmop{Pair}^{\gamma, \tmop{st}})$ is given by
  $\tmop{Pair}^{\gamma, \tmop{st}} \ni (B, J, \delta) \mapsto
  \tmop{Hom}_{\tmop{Pair}^{\gamma}} ((B, J, \delta), (A, I, \gamma))$, which
  is subsequently mapped to an object in $\mathcal{P}_{\Sigma, 1}
  (\tmop{Pair}^{\tmop{st}})$ given by $\tmop{Pair}^{\tmop{st}} \ni (B, J)
  \mapsto \tmop{Hom}_{\tmop{Pair}^{\gamma}} (D_J (B), (A, I, \gamma))$. On the
  other hand, the image of $(A, I, \gamma)$ in $\tmop{Pair}$ is $(A, I)$,
  which is subsequently mapped to an object in $\mathcal{P}_{\Sigma, 1}
  (\tmop{Pair}^{\tmop{st}})$ given by $\tmop{Pair}^{\tmop{st}} \ni (B, J)
  \mapsto \tmop{Hom}_{\tmop{Pair}} ((B, J), (A, I))$. It then follows from the
  functorial isomorphism $\tmop{Hom}_{\tmop{Pair}^{\gamma}} (D_J (B), (A, I,
  \gamma)) \cong \tmop{Hom}_{\tmop{Pair}} ((B, J), (A, I))$ by adjunction.
\end{proof}

Similarly, we have the embedding:

\begin{lemma}
  \label{lem:pdpair-Psigma1-full-faith}The functor $\tmop{Pair}^{\gamma}
  \rightarrow \mathcal{P}_{\Sigma, 1} (\tmop{Pair}^{\gamma, \tmop{st}})$
  described in Construction~\ref{cons:pdpair-st-pdpair-adjunct} is fully
  faithful.
\end{lemma}

\begin{proof}
  The proof is similar to that of \Cref{lem:pair-Psigma1-full-faith}. The
  faithfulness follows from \Cref{lem:1-pdpair-pair-forget-embedding} and the
  faithfulness of the forgetful functor $\tmop{Pair}^{\gamma} \rightarrow
  \tmop{Pair}$. Given two PD-pairs $(A, I, \gamma)$ and $(B, J, \delta)$ in
  $\tmop{Pair}^{\gamma}$ and a map
  \[ F \assign \nobracket \tmop{Hom}_{\tmop{Pair}^{\gamma}} (\cdummy, (A, I,
     \gamma)) |_{(\tmop{Pair}^{\gamma, \tmop{st}})^{\tmop{op}}} \rightarrow
     \nobracket \tmop{Hom}_{\tmop{Pair}^{\gamma}} (\cdummy, (B, J, \delta))
     |_{(\tmop{Pair}^{\gamma, \tmop{st}})^{\tmop{op}}} \backassign G \]
  of finite-product-preserving functors $(\tmop{Pair}^{\gamma,
  \tmop{st}})^{\tmop{op}} \rightrightarrows \tmop{Set}$, we need to show that
  this is induced by some map $(A, I, \gamma) \rightarrow (B, J, \delta)$ of
  PD-pairs.
  
  By \Cref{lem:pair-Psigma1-full-faith}, there exists a unique map $(A, I)
  \rightarrow (B, J)$ of pairs which correspond to the natural transform after
  composition $(\tmop{Pair}^{\tmop{st}})^{\tmop{op}} \rightarrow
  (\tmop{Pair}^{\gamma, \tmop{st}})^{\tmop{op}} \rightrightarrows \tmop{Set}$.
  It remains to show that this map preserves the PD-structure.
  
  Indeed, any $x \in I$ corresponds to a map $(\Gamma_{\mathbb{Z}} (t)
  \twoheadrightarrow \mathbb{Z}, \gamma_0) \rightarrow (A, I, \gamma)$ of
  PD-pairs, i.e., an element $\overline{x} \in F (\Gamma_{\mathbb{Z}} (t)
  \twoheadrightarrow \mathbb{Z}, \gamma_0)$, and the image $y$ of $x \in I$ in
  $J$ is given by the image $\overline{y} \in G (\Gamma_{\mathbb{Z}} (t)
  \twoheadrightarrow \mathbb{Z}, \gamma_0)$ under the map $F \rightarrow G$.
  For any $n \in \mathbb{N}_{> 0}$, there is a canonical endomorphism
  $(\Gamma_{\mathbb{Z}} (t) \twoheadrightarrow \mathbb{Z}, \gamma_0)
  \rightarrow (\Gamma_{\mathbb{Z}} (t) \twoheadrightarrow \mathbb{Z},
  \gamma_0), t \mapsto \gamma_n (t)$ of PD-pairs which induces endomorphisms
  $F (\Gamma_{\mathbb{Z}} (t) \twoheadrightarrow \mathbb{Z}, \gamma_0)
  \rightarrow F (\Gamma_{\mathbb{Z}} (t) \twoheadrightarrow \mathbb{Z},
  \gamma_0)$ and $G (\Gamma_{\mathbb{Z}} (t) \twoheadrightarrow \mathbb{Z},
  \gamma_0) \rightarrow G (\Gamma_{\mathbb{Z}} (t) \twoheadrightarrow
  \mathbb{Z}, \gamma_0)$ compatible with the map $F \rightarrow G$. In
  particular, the image, denoted by $\overline{x}_n$, of $\overline{x}$ under
  the first endomorphism maps to the image, denoted by $\overline{y}_n$, of
  $\overline{y}$ under the second endomorphism. We note that $\overline{x}_n$
  corresponds to $\gamma_n (x)$ and $\overline{y}_n$ corresponds to $\gamma_n
  (y)$. Thus the map $(A, I) \rightarrow (B, J)$ maps $\gamma_n (x)$ to
  $\gamma_n (y)$.
\end{proof}

With the following description of the essential image (cf.
\Cref{lem:ess-img-pair}):

\begin{lemma}
  \label{lem:ess-img-pdpair}The square in
  \Cref{lem:1-pdpair-pair-forget-embedding} is Cartesian. Equivalently by
  \Cref{lem:ess-img-pair}, the essential image of the fully faithful functor
  $\tmop{Pair}^{\gamma} \rightarrow \mathcal{P}_{\Sigma, 1}
  (\tmop{Pair}^{\gamma, \tmop{st}})$ (\Cref{lem:pdpair-Psigma1-full-faith}) is
  spanned by those finite-product-preserving functors $F \of
  (\tmop{Pair}^{\gamma, \tmop{st}})^{\tmop{op}} \rightarrow \tmop{Set}$ which
  maps the edge $(\mathbb{Z} [X], 0, 0) \rightarrow (\Gamma_{\mathbb{Z}} (X)
  \twoheadrightarrow \mathbb{Z}, \gamma)$ in $\tmop{Pair}^{\gamma, \tmop{st}}$
  to an injective map of sets.
\end{lemma}

\begin{proof}
  The proof is similar to that of \Cref{lem:ess-img-pair}. Let $F \of
  (\tmop{Pair}^{\gamma, \tmop{st}})^{\tmop{op}} \rightarrow \tmop{Set}$ be a
  functor such that the composite $(\tmop{Pair}^{\tmop{st}})^{\tmop{op}}
  \rightarrow (\tmop{Pair}^{\gamma, \tmop{st}})^{\tmop{op}} \xrightarrow{F}
  \tmop{Set}$ lies in the essential image of the fully faithful functor
  $\tmop{Pair} \rightarrow \mathcal{P}_{\Sigma, 1} (\tmop{Pair}^{\tmop{st}})$.
  We need to construct a PD-pair $(A, I, \gamma) \in \tmop{Pair}^{\gamma}$
  which represents the functor $F$.
  
  Let $(A, I)$ represent the composite functor
  $(\tmop{Pair}^{\tmop{st}})^{\tmop{op}} \rightarrow (\tmop{Pair}^{\gamma,
  \tmop{st}})^{\tmop{op}} \xrightarrow{F} \tmop{Set}$. Unrolling the
  definitions, we see that $A = F (\mathbb{Z} [t], 0, 0)$, $I = F
  (\Gamma_{\mathbb{Z}} (t) \twoheadrightarrow \mathbb{Z}, \gamma)$ and the map
  $I \rightarrow A$ is induced by the map $(\mathbb{Z} [t], 0, 0) \rightarrow
  (\Gamma_{\mathbb{Z}} (t) \twoheadrightarrow \mathbb{Z}, \gamma)$ of
  PD-pairs. We endow a PD-structure $(A, I)$ as follows: there exists a
  canonical endomorphism $\gamma_n \of (\Gamma_{\mathbb{Z}} (t)
  \twoheadrightarrow \mathbb{Z}, \gamma) \rightarrow (\Gamma_{\mathbb{Z}} (t)
  \twoheadrightarrow \mathbb{Z}, \gamma), t \mapsto \gamma_n (t)$ of PD-pair,
  which induces a map $\gamma_n \of I \rightarrow I$ for all $n \in
  \mathbb{N}_{> 0}$. We need to check that $(\gamma_n)_{n \in \mathbb{N}_{>
  0}}$ satisfies the axioms of divided power structure
  {\cite[\href{https://stacks.math.columbia.edu/tag/07GL}{Tag
  07GL}]{stacks-project}}, setting $\gamma_0 = \tmop{id}$. We spell out the
  verification of two of them:
  \begin{description}
    \item[$\gamma_n (x + y) = \sum_i \gamma_i (x) \gamma_{n - i} (y)$ for $(x,
    y) \in I^2$] First, in the PD-pair $(\Gamma_{\mathbb{Z}} (X, Y)
    \twoheadrightarrow \mathbb{Z}, \gamma)$, the identity $\gamma_n (X + Y) =
    \sum_i \gamma_i (X) \gamma_{n - i} (Y)$ holds. This implies that the
    composite
    \[ (\Gamma_{\mathbb{Z}} (T) \twoheadrightarrow \mathbb{Z}, \gamma)
       \rightarrow (\Gamma_{\mathbb{Z}} (T_0, \ldots, T_n) \twoheadrightarrow
       \mathbb{Z}, \gamma) \rightarrow (\Gamma_{\mathbb{Z}} (X, Y)
       \twoheadrightarrow \mathbb{Z}, \gamma) \]
    where the first map is induced by $T \mapsto \sum_i T_i$, and the second
    map is induced by $T_i \mapsto \gamma_i (X) \gamma_{n - i} (Y)$, coincides
    with the composite
    \[ (\Gamma_{\mathbb{Z}} (T) \twoheadrightarrow \mathbb{Z}, \gamma)
       \xrightarrow{\gamma_n} (\Gamma_{\mathbb{Z}} (T) \twoheadrightarrow
       \mathbb{Z}, \gamma) \rightarrow (\Gamma_{\mathbb{Z}} (X, Y)
       \twoheadrightarrow \mathbb{Z}, \gamma) \]
    where the second map is induced by $T \mapsto X + Y$. Applying $F$ to the
    two compositions, using the fact that $F$ preserves finite products, and
    that $(x, y) \in I^2$ corresponds to an element in $F (\Gamma_{\mathbb{Z}}
    (X, Y) \twoheadrightarrow \mathbb{Z}, \gamma)$, we get the result (for the
    part $T_i \mapsto \gamma_i (X) \gamma_{n - i} (Y)$, one need to separate
    $i = 0$ and $i > 0$).
    
    \item[$\gamma_n (ax) = a^n \gamma_n (x)$ for $(a, x) \in A \times I$] In
    the PD-pair $(\Gamma_{\mathbb{Z} [Y]} (X) \twoheadrightarrow \mathbb{Z}
    [Y], \gamma)$, the identity $\gamma_n (YX) = Y^n \gamma_n (X)$ holds. This
    implies that the composite
    \[ (\Gamma_{\mathbb{Z}} (T) \twoheadrightarrow \mathbb{Z}, \gamma)
       \rightarrow (\Gamma_{\mathbb{Z} [T_1, \ldots, T_n]} (t)
       \twoheadrightarrow \mathbb{Z} [T_1, \ldots, T_n], \gamma) \rightarrow
       (\Gamma_{\mathbb{Z} [Y]} (X) \twoheadrightarrow \mathbb{Z} [Y], \gamma)
    \]
    where the first map is induced by $T \mapsto T_1 \cdots T_n t$, and the
    second map is induced by $T_i \mapsto Y$ and $t \mapsto X$, coincides with
    the composite
    \[ (\Gamma_{\mathbb{Z}} (T) \twoheadrightarrow \mathbb{Z}, \gamma)
       \xrightarrow{\gamma_n} (\Gamma_{\mathbb{Z}} (T) \twoheadrightarrow
       \mathbb{Z}, \gamma) \rightarrow (\Gamma_{\mathbb{Z} [Y]} (X)
       \twoheadrightarrow \mathbb{Z} [Y], \gamma) \]
    where the second map is induced by $T \mapsto XY$. Applying $F$ to the two
    compositions, using the fact that $F$ preserves finite products, and that
    $(a, x) \in A \times I$ corresponds to an element in $F
    (\Gamma_{\mathbb{Z} [Y]} (X) \twoheadrightarrow \mathbb{Z} [Y], \gamma)$,
    we get the result.
  \end{description}
  Finally, the proof of the fact that $(A, I, \gamma)$ represents $F$ is
  parallel to the corresponding part of the proof of \Cref{lem:ess-img-pair}.
\end{proof}

Now we arrive at the definition of animated pairs and animated PD-pairs:

\begin{definition}
  \label{def:anipair-anipdpair}The $\infty$\mbox{-}category
  $\tmop{Pair}^{\tmop{an}}$ of {\tmdfn{animated pairs}} is defined to be the
  $\infty$\mbox{-}category $\mathcal{P}_{\Sigma} (\tmop{Pair}^{\tmop{st}})$,
  and the $\infty$\mbox{-}category $\tmop{Pair}^{\gamma, \tmop{an}}$ of
  {\tmdfn{animated PD-pairs}} is defined to be the $\infty$\mbox{-}category
  $\mathcal{P}_{\Sigma} (\tmop{Pair}^{\gamma, \tmop{st}})$.
  
  The {\tmdfn{forgetful functor}} $\tmop{Pair}^{\tmop{an}} \rightarrow
  \tmop{Fun} (\Delta^1, D (\mathbb{Z})_{\geq 0})$ is given by the pair
  $\mathcal{P}_{\Sigma} (\tmop{Inj}^{\tmop{st}}) \rightleftarrows
  \mathcal{P}_{\Sigma} (\tmop{Pair}^{\tmop{st}})$ obtained by applying
  \Cref{cor:nonab-deriv-cat-adjoint-fun} to the free-pair functor
  $\tmop{Inj}^{\tmop{st}} \rightarrow \tmop{Pair}^{\tmop{st}}$ (which is
  essentially surjective).
  
  The {\tmdfn{forgetful functor}} $\tmop{Pair}^{\gamma, \tmop{an}} \rightarrow
  \tmop{Pair}^{\tmop{an}}$ and the {\tmdfn{animated PD-envelope functor}}
  $\tmop{Env}^{\gamma, \tmop{an}} \of \tmop{Pair}^{\tmop{an}} \rightarrow
  \tmop{Pair}^{\gamma, \tmop{an}}$ are given by the pair $\mathcal{P}_{\Sigma}
  (\tmop{Pair}^{\tmop{st}}) \rightleftarrows \mathcal{P}_{\Sigma}
  (\tmop{Pair}^{\gamma, \tmop{st}})$ obtained by applying
  \Cref{cor:nonab-deriv-cat-adjoint-fun} to the PD-envelope functor
  $\tmop{Pair}^{\tmop{st}} \rightarrow \tmop{Pair}^{\gamma, \tmop{st}}$ being
  essentially surjective.
\end{definition}

It follows from \Cref{cor:nonab-deriv-cat-adjoint-fun} that

\begin{corollary}
  \label{cor:forget-pair-conserv-sifted-colim}The forgetful functors
  $\tmop{Pair}^{\tmop{an}} \rightarrow \tmop{Fun} ((\Delta^1)^{\tmop{op}}, D
  (\mathbb{Z})_{\geq 0})$ and $\tmop{Pair}^{\gamma, \tmop{an}} \rightarrow
  \tmop{Pair}^{\tmop{an}}$ are conservative and preserve sifted colimits.
\end{corollary}

These forgetful functors are compatible with canonical embeddings $\tmop{Pair}
\hookrightarrow \mathcal{P}_{\Sigma, 1} (\tmop{Pair}^{\tmop{st}})
\hookrightarrow \tmop{Pair}^{\tmop{an}}$ and $\tmop{Pair}^{\gamma}
\hookrightarrow \mathcal{P}_{\Sigma, 1} (\tmop{Pair}^{\gamma, \tmop{st}})
\hookrightarrow \tmop{Pair}^{\gamma, \tmop{an}}$:

\begin{proposition}
  \label{prop:infty-cat-forget-embedding}The diagram
  \[ \begin{array}{ccc}
       \tmop{Pair}^{\gamma} & \longhookrightarrow & \tmop{Pair}^{\gamma,
       \tmop{an}}\\
       \longdownarrow &  & \longdownarrow\\
       \tmop{Pair} & \longhookrightarrow & \tmop{Pair}^{\tmop{an}}\\
       \longdownarrow &  & \longdownarrow\\
       \tmop{Inj} & \longhookrightarrow & \tmop{Fun} (\Delta^{1, \tmop{op}}, D
       (\mathbb{Z})_{\geq 0})
     \end{array} \]
  is a commutative diagram of $\infty$\mbox{-}categories, where the vertical
  arrows are forgetful functors. Moreover, the squares are Cartesian.
\end{proposition}

\begin{proof}
  The commutativity follows from
  \Cref{rem:ani-trunc,lem:pair-ab-forget-embedding,lem:1-pdpair-pair-forget-embedding}.
  The last claim follows from \Cref{lem:ess-img-pair,lem:ess-img-pdpair}.
\end{proof}

\begin{remark}
  \label{rem:localization-pair-pdpair}The embeddings $\tmop{Pair}
  \hookrightarrow \tmop{Pair}^{\tmop{an}}$ and $\tmop{Pair}^{\gamma}
  \hookrightarrow \tmop{Pair}^{\gamma, \tmop{an}}$ admits left adjoints given
  by the composite functors $\tmop{Pair}^{\tmop{an}} \xrightarrow{\tau_{\leq
  0}} \mathcal{P}_{\Sigma, 1} (\tmop{Pair}^{\tmop{st}}) \rightarrow
  \tmop{Pair}$ and $\tmop{Pair}^{\gamma, \tmop{an}} \xrightarrow{\tau_{\leq
  0}} \mathcal{P}_{\Sigma, 1} (\tmop{Pair}^{\gamma, \tmop{st}}) \rightarrow
  \tmop{Pair}^{\gamma}$ (see \Cref{rem:ani-trunc} for $\tau_{\leq 0}$). We
  will give an explicit description of the functor $\tmop{Pair}^{\tmop{an}}
  \rightarrow \tmop{Pair}$ in \Cref{prop:characterize-pair}.
\end{remark}

Taking the left adjoints to the upper square of the diagram in
\Cref{prop:infty-cat-forget-embedding}, we get

\begin{corollary}
  The diagram
  \[ \begin{array}{ccc}
       \tmop{Pair}^{\tmop{an}} & \longrightarrow & \tmop{Pair}\\
       \longdownarrow &  & \longdownarrow\\
       \tmop{Pair}^{\gamma, \tmop{an}} & \longrightarrow &
       \tmop{Pair}^{\gamma}
     \end{array} \]
  is a commutative diagram of $\infty$\mbox{-}categories, where the horizontal
  arrows are described in \Cref{rem:localization-pair-pdpair}.
\end{corollary}

In particular, we rewrite the PD-envelope in terms of animated PD-envelope:

\begin{corollary}
  \label{cor:pd-env-animated-pd-env}The PD-envelope functor $\tmop{Pair}
  \rightarrow \tmop{Pair}^{\gamma}, (A, I) \mapsto D_A (I)$ coincides with the
  composite functor $\tmop{Pair} \hookrightarrow \tmop{Pair}^{\tmop{an}}
  \xrightarrow{\tmop{Env}^{\gamma, \tmop{an}}} \tmop{Pair}^{\gamma, \tmop{an}}
  \rightarrow \tmop{Pair}^{\gamma}$, where the last functor is described in
  \Cref{rem:localization-pair-pdpair}.
\end{corollary}

In fact, there is a more concrete description of $\tmop{Pair}^{\tmop{an}}$,
given by the following:

\begin{definition}
  \label{def:surj-map-ani-ring}The {\tmdfn{$\infty$\mbox{-}category of
  surjective maps of animated rings}} is the full subcategory $\tmop{Fun}
  (\Delta^1, \tmop{CAlg}^{\tmop{an}})_{\geq 0} \subseteq \tmop{Fun} (\Delta^1,
  \tmop{CAlg}^{\tmop{an}})$ of maps $A \rightarrow A''$ such that the induced
  map $\pi_0 (A) \rightarrow \pi_0 (A'')$ on the 0th homotopy groups is
  surjective.
\end{definition}

We now show that the strategy to prove \Cref{cor:smith-eq} adapts to our case.
Indeed, by \Cref{cor:conn-eq-day-point}, we have the equivalence $\tmop{Fun}
((\Delta^1)^{\tmop{op}}, D (\mathbb{Z})_{\geq 0}) \simeq \tmop{Fun} (\Delta^1,
D (\mathbb{Z}))_{\geq 0}$ of $\infty$\mbox{-}categories, therefore a set of
compact projective generators for $\tmop{Fun} ((\Delta^1)^{\tmop{op}}, D
(\mathbb{Z})_{\geq 0})$ gives rise to a set of compact projective generators
for $\tmop{Fun} (\Delta^1, D (\mathbb{Z}))_{\geq 0}$: $\left\{ \mathbb{Z}
\xrightarrow{\tmop{id}} \mathbb{Z}, \mathbb{Z} \rightarrow 0 \right\}$. Now we
study two adjunctions over these $\infty$\mbox{-}categories.

We have a pair $\tmop{Fun} (\Delta^1, D (\mathbb{Z})_{\geq 0})
\rightleftarrows \tmop{Fun} (\Delta^1, \tmop{CAlg}^{\tmop{an}})$ of adjoint
functors induced by the pair $D (\mathbb{Z})_{\geq 0} \overset{\mathbb{L}
\tmop{Sym}_{\mathbb{Z}}}{\longrightleftarrows} \tmop{CAlg}^{\tmop{an}}$ of
adjoint functors. Restricting to full subcategories, we get a pair $\tmop{Fun}
(\Delta^1, D (\mathbb{Z}))_{\geq 0} \rightleftarrows \tmop{Fun} (\Delta^1,
\tmop{CAlg}^{\tmop{an}})_{\geq 0}$, the later is defined before
\Cref{cor:conn-eq-day-point}. We summarize the preceding discussion by the
diagram
\begin{equation}
  \begin{array}{ccccc}
    \mathcal{P}_{\Sigma} (\tmop{Pair}^{\tmop{st}}) &  & \tmop{Fun} (\Delta^1,
    \tmop{CAlg}^{\tmop{an}})_{\geq 0} & \subseteq & \tmop{Fun} (\Delta^1,
    \tmop{CAlg}^{\tmop{an}})\\
    \longuparrow \nocomma \longdownarrow &  & \longuparrow \nocomma
    \longdownarrow &  & \longuparrow \nocomma \longdownarrow\\
    \tmop{Fun} ((\Delta^1)^{\tmop{op}}, D (\mathbb{Z})_{\geq 0}) &
    \xrightarrow{\simeq} & \tmop{Fun} (\Delta^1, D (\mathbb{Z}))_{\geq 0} &
    \subseteq & \tmop{Fun} (\Delta^1, D (\mathbb{Z})_{\geq 0})
  \end{array} \label{diag:ani-smith-eq}
\end{equation}
We note that both full subcategories are stable under small colimits,
therefore the forgetful functor $\tmop{Fun} (\Delta^1,
\tmop{CAlg}^{\tmop{an}})_{\geq 0} \rightarrow \tmop{Fun} (\Delta^1, D
(\mathbb{Z}))_{\geq 0}$ preserves sifted colimits. Since the forgetful functor
is also conservative, it follows by \Cref{prop:adjoint-n-proj-gen} that
$\tmop{Fun} (\Delta^1, \tmop{CAlg}^{\tmop{an}})_{\geq 0}$ is projectively
generated, for which $\left\{ \mathbb{Z} [t] \xrightarrow{\tmop{id}}
\mathbb{Z} [t], \mathbb{Z} [t] \rightarrow \mathbb{Z} \right\}$ is a set of
compact projective generators. Let $\mathcal{Z} \subseteq \tmop{Fun}
(\Delta^1, \tmop{CAlg}^{\tmop{an}})_{\geq 0}$ denote the full subcategory
spanned by finite coproducts of these objects, which is effectively a full
subcategory of $\tmop{Fun} (\Delta^1, \tmop{Ring})$. The following lemma is
then obvious:

\begin{lemma}
  There is an equivalence $\tmop{Pair}^{\tmop{st}} \simeq \mathcal{Z}$ of
  $1$\mbox{-}categories given by $\tmop{Pair}^{\tmop{st}} \rightarrow
  \mathcal{Z}, (A, I) \mapsto (A \twoheadrightarrow A / I)$ and $\mathcal{Z}
  \rightarrow \tmop{Pair}^{\tmop{st}}, (A \twoheadrightarrow A'') \mapsto (A,
  \ker (A \twoheadrightarrow A''))$.
\end{lemma}

It follows from previous discussion that

\begin{theorem}
  \label{thm:ani-smith-eq}There is an equivalence $\tmop{Pair}^{\tmop{an}}
  =\mathcal{P}_{\Sigma} (\tmop{Pair}^{\tmop{st}}) \xrightarrow{\simeq}
  \tmop{Fun} (\Delta^1, \tmop{CAlg}^{\tmop{an}})_{\geq 0}$ of
  $\infty$\mbox{-}categories which fits into \eqref{diag:ani-smith-eq}, making
  the left square a commutative square\footnote{More precisely, there are two
  possible left squares in \eqref{diag:ani-smith-eq}. However, by uniqueness
  of left/right adjoint, roughly speaking, one commutes if and only if the
  other commutes.}.
\end{theorem}

In the proof of {\cite[Lem~3.16]{Antonio2020}}, Jorge {\tmname{Antonío}}
sketched a slightly different set of compact projective generators. It would
be nice to compare the two choices.

\begin{remark}
  \Cref{cor:smith-eq} says that the $\infty$\mbox{-}category of
  $\mathbb{E}_{\infty}$-algebras in the symmetric monoidal
  $\infty$\mbox{-}category $\tmop{Fun} ((\Delta^1)^{\tmop{op}}, D
  (\mathbb{Z})_{\geq 0})$ is equivalent to that of
  $\mathbb{E}_{\infty}$-algebras in the symmetric monoidal
  $\infty$\mbox{-}category $\tmop{Fun} (\Delta^1, D (\mathbb{Z}))_{\geq 0}$
  since two symmetric monoidal $\infty$\mbox{-}categories are equivalent. Our
  result essentially says that both $\infty$\mbox{-}categories admits
  endomorphism monads which is also preserved under this equivalence,
  therefore the module categories over these monads are equivalent.
\end{remark}

\begin{notation}
  \label{nota:ani-pdpair}Given the equivalence in \Cref{thm:ani-smith-eq}, we
  will symbolically denote an object in $\tmop{Pair}^{\gamma, \tmop{an}}$ by
  $(A \twoheadrightarrow A'', \gamma)$ where $A \twoheadrightarrow A''$ is the
  image under the forgetful functor $\tmop{Pair}^{\gamma, \tmop{an}}
  \rightarrow \tmop{Pair}^{\tmop{an}} \xrightarrow{\simeq} \tmop{Fun}
  (\Delta^1, \tmop{CAlg}^{\tmop{an}})_{\geq 0}$. When the PD-structure is the
  ``obvious'' one (like $\Gamma_{\mathbb{Z} [X]} (Y) \twoheadrightarrow
  \mathbb{Z} [X]$), by abuse of notation, we will omit the $\gamma$ in
  question. Under this notation, objects in $\tmop{Pair}^{\tmop{st}}$ could be
  identified with $\mathbb{Z} [X, Y] \twoheadrightarrow \mathbb{Z} [X]$, and
  objects in $\tmop{Pair}^{\gamma, \tmop{st}}$ could be identified with
  $\Gamma_{\mathbb{Z} [X]} (Y) \twoheadrightarrow \mathbb{Z} [X]$.
\end{notation}

\begin{remark}
  In \Cref{thm:ani-smith-eq}, we can replace $D (\mathbb{Z})$ by any
  {\tmdfn{derived algebraic context}} $\mathcal{C}$
  {\cite[Def~4.2.1]{Raksit2020}} and then both $\tmop{Fun}
  ((\Delta^1)^{\tmop{op}}, \mathcal{C})$ and $\tmop{Fun} (\Delta^1,
  \mathcal{C})$ admit canonical structures of derived algebraic contexts which
  are preserved under the equivalence $\tmop{Fun} ((\Delta^1)^{\tmop{op}},
  \mathcal{C}) \rightarrow \tmop{Fun} (\Delta^1, \mathcal{C})$, and
  \Cref{thm:ani-smith-eq} essentially generalizes to the equivalence between
  the $\infty$\mbox{-}categories of connective maps of {\tmdfn{derived
  commutative algebras}} {\cite[Rem~4.2.24]{Raksit2020}} (note that
  $\tmop{Pair}^{\tmop{an}} \simeq \tmop{DAlg} (\tmop{Fun}
  ((\Delta^1)^{\tmop{op}}, D (\mathbb{Z})))^{\tmop{cn}}$ and $\tmop{Fun}
  (\Delta^1, \tmop{CAlg}^{\tmop{an}})_{\geq 0} \simeq \tmop{DAlg} (\tmop{Fun}
  (\Delta^1, D (\mathbb{Z})))^{\tmop{cn}}$).
\end{remark}

\begin{remark}
  \label{rem:deriv-PD-pairs}In fact, the machinery in
  {\cite[§4]{Raksit2020}}, due to Bhatt--Mathew and {\cite{Brantner2019}},
  allows us to define the $\infty$\mbox{-}category of {\tmdfn{derived
  PD-pairs}} of which the connective objects spans a full subcategory
  equivalent to the $\infty$\mbox{-}category of animated PD-pairs.
\end{remark}

\begin{warning}
  We warn the reader that the heart $\tmop{DAlg} (\tmop{Fun}
  ((\Delta^1)^{\tmop{op}}, D (\mathbb{Z})))^{\heartsuit}$ in
  {\cite[Rem~4.2.24]{Raksit2020}} is equivalent to the $1$\mbox{-}category
  $\mathcal{P}_{\Sigma, 1} (\tmop{Pair}^{\tmop{st}})$, not the
  $1$\mbox{-}category $\tmop{Pair}$.
\end{warning}

We also identify the equivalence in \Cref{thm:ani-smith-eq} restricted to the
full subcategory $\tmop{Pair} \subseteq \tmop{Pair}^{\tmop{an}}$:

\begin{proposition}
  \label{prop:characterize-pair}Let $\tmop{Fun} (\Delta^1,
  \tmop{Ring})_{\tmop{surj}} \subseteq \tmop{Fun} (\Delta^1, \tmop{Ring})$ be
  the full subcategory spanned by those surjective maps $A \twoheadrightarrow
  A''$ of rings. Then the equivalence $\tmop{Pair} \rightarrow \tmop{Fun}
  (\Delta^1, \tmop{Ring})_{\tmop{surj}}, (A, I) \mapsto (A \twoheadrightarrow
  A / I)$ fits into a canonical commutative diagram
  \[ \begin{array}{ccc}
       \tmop{Pair} & \xrightarrow{\simeq} & \tmop{Fun} (\Delta^1,
       \tmop{Ring})_{\tmop{surj}}\\
       \longdownarrow &  & \longdownarrow\\
       \tmop{Pair}^{\tmop{an}} & \xrightarrow{\simeq} & \tmop{Fun} (\Delta^1,
       \tmop{CAlg}^{\tmop{an}})_{\geq 0}
     \end{array} \]
  of $\infty$\mbox{-}categories. Furthermore, the localization
  $\tmop{Pair}^{\tmop{an}} \rightarrow \tmop{Pair}$
  (\Cref{rem:localization-pair-pdpair}) could be identified with $\tmop{Fun}
  (\Delta^1, \tmop{CAlg}^{\tmop{an}})_{\geq 0} \rightarrow \tmop{Fun}
  (\Delta^1, \tmop{Ring})_{\tmop{surj}}, (A \twoheadrightarrow A'') \mapsto
  (\pi_0 (A) \twoheadrightarrow \pi_0 (A''))$.
\end{proposition}

\begin{proof}
  We note that $\tmop{Fun} (\Delta^1, \tmop{Ring}) \subseteq \tmop{Fun}
  (\Delta^1, \tmop{CAlg}^{\tmop{an}})$ is the reflective subcategory
  (\Cref{def:reflexive-cat}) spanned by the $1$\mbox{-}truncated objects, of
  which the localization is given by $\tmop{Fun} (\Delta^1,
  \tmop{CAlg}^{\tmop{an}}) \rightarrow \tmop{Fun} (\Delta^1, \tmop{Ring}), (A
  \rightarrow A'') \mapsto (\pi_0 (A) \rightarrow \pi_0 (A''))$ by
  \Cref{cor:ani-arrow-cat,rem:ani-trunc}. Restricting to the full subcategory
  $\tmop{Fun} (\Delta^1, \tmop{CAlg}^{\tmop{an}})_{\geq 0} \subseteq
  \tmop{Fun} (\Delta^1, \tmop{CAlg}^{\tmop{an}})$, we get a localization
  $\tmop{Fun} (\Delta^1, \tmop{CAlg}^{\tmop{an}})_{\geq 0} \rightarrow
  \tmop{Fun} (\Delta^1, \tmop{Ring})_{\tmop{surj}}$. Consider the diagram
  \[ \begin{array}{ccc}
       \tmop{Pair}^{\tmop{an}} & \xrightarrow{\simeq} & \tmop{Fun} (\Delta^1,
       \tmop{CAlg}^{\tmop{an}})_{\geq 0}\\
       \longdownarrow &  & \longdownarrow\\
       \tmop{Pair} & \xrightarrow{\simeq} & \tmop{Fun} (\Delta^1,
       \tmop{Ring})_{\tmop{surj}}
     \end{array} \]
  {\noindent}of $\infty$\mbox{-}categories, where the vertical arrows are
  localizations (\Cref{rem:localization-pair-pdpair}). We claim that this is a
  commutative diagram. Indeed, both compositions commute with filtered
  colimits and geometric realizations (in fact, all small colimits, since both
  vertical arrows are localizations in \Cref{def:reflexive-cat}), and when
  restricting to $\tmop{Pair}^{\tmop{st}} \subseteq \tmop{Pair}^{\tmop{an}}$,
  both compositions are canonically equivalent. Then the claim follows from
  \Cref{prop:left-deriv-n-fun}.
  
  Another way to show the commutativity is to show that the top right
  composition is $(\tmop{Pair}^{\tmop{an}} \rightarrow \tmop{Pair})$-invariant
  in \Cref{def:L-inv}, then invoke \Cref{prop:L-inv}.
  
  Then the result follows by taking the right adjoints to the vertical arrows.
\end{proof}

\begin{corollary}
  The lower square in \Cref{prop:infty-cat-forget-embedding} is
  left-adjointable {\cite[Def~4.7.4.13]{Lurie2017}}, which gives rise to a
  commutative diagram
  \[ \begin{array}{ccc}
       \tmop{Pair}^{\tmop{an}} & \longrightarrow & \tmop{Pair}\\
       \longdownarrow &  & \longdownarrow\\
       \tmop{Fun} (\Delta^{1, \tmop{op}}, D (\mathbb{Z})_{\geq 0}) &
       \longrightarrow & \tmop{Inj}
     \end{array} \]
  {\noindent}of $1$\mbox{-}categories, where the vertical arrows are forgetful
  functors.
\end{corollary}

\begin{warning}
  \label{warn:local-anipair-pair-incompatible-forget}The upper square in
  \Cref{prop:infty-cat-forget-embedding} is not left-adjointable. That is to
  say, the localizations $\tmop{Pair}^{\tmop{an}} \rightarrow \tmop{Pair}$ and
  $\tmop{Pair}^{\gamma, \tmop{an}} \rightarrow \tmop{Pair}^{\gamma}$ are not
  compatible with forgetful functors, otherwise the forgetful functor
  $\tmop{Pair}^{\gamma} \rightarrow \tmop{Pair}$ would commute with small
  colimits, which is false (see
  \Cref{rem:pdpair-pair-forget-not-preserve-colim}).
\end{warning}

It follows from \Cref{prop:infty-cat-forget-embedding,prop:characterize-pair}
that

\begin{proposition}
  \label{prop:characterize-pdpair}The essential image of the fully faithful
  embedding $\tmop{Pair}^{\gamma} \hookrightarrow \tmop{Pair}^{\gamma,
  \tmop{an}}$ is spanned by those animated PD-pairs $(A \twoheadrightarrow
  A'', \gamma)$ such that both $A$ and $A''$ are static.
\end{proposition}

To understand the difference between $\tmop{Pair}$ and $\mathcal{P}_{\Sigma,
1} (\tmop{Pair}^{\tmop{st}}) \simeq \tau_{\leq 0} (\tmop{Pair}^{\tmop{an}})$
better, we compute the following example:

\begin{example}
  \label{ex:anipair-self-coprod}Consider $(\mathbb{Z}/ 4\mathbb{Z}, (2)) \in
  \tmop{Pair}$ as an animated pair. By \Cref{prop:characterize-pair}, this
  corresponds to the surjective map $\mathbb{Z}/ 4\mathbb{Z}
  \twoheadrightarrow \mathbb{F}_2$ of rings. Let us study the coproduct
  $(\mathbb{Z}/ 4\mathbb{Z}, (2)) \amalg (\mathbb{Z}/ 4\mathbb{Z}, (2))$ taken
  in $\tmop{Pair}^{\tmop{an}}$. Thanks to \Cref{thm:ani-smith-eq}, this
  corresponds to the surjective map $\mathbb{Z}/ 4\mathbb{Z}
  \otimes_{\mathbb{Z}}^{\mathbb{L}} \mathbb{Z}/ 4\mathbb{Z} \twoheadrightarrow
  \mathbb{F}_2 \otimes_{\mathbb{Z}}^{\mathbb{L}} \mathbb{F}_2$ of animated
  rings. The underlying map in $\tmop{Fun} (\Delta^1,
  \tmop{CAlg}^{\tmop{an}})_{\geq 0}$ is given by
  \[ (\mathbb{Z}/ 4\mathbb{Z}) [1] \oplus \mathbb{Z}/ 4\mathbb{Z}
     \twoheadrightarrow \mathbb{F}_2 [1] \oplus \mathbb{F}_2 \]
  induced by $0 \of (\mathbb{Z}/ 4\mathbb{Z}) [1] \rightarrow \mathbb{F}_2
  [1]$ and the canonical projection $\mathbb{Z}/ 4\mathbb{Z}
  \twoheadrightarrow \mathbb{F}_2$. Under the forgetful functor
  $\tmop{Pair}^{\tmop{an}} \rightarrow \tmop{Fun} ((\Delta^1)^{\tmop{op}}, D
  (\mathbb{Z})_{\geq 0})$, the image of $(\mathbb{Z}/ 4\mathbb{Z}, (2)) \amalg
  (\mathbb{Z}/ 4\mathbb{Z}, (2))$ is thus given by
  \begin{eqnarray*}
    (\mathbb{Z}/ 4\mathbb{Z}) [1] \oplus \mathbb{Z}/ 4\mathbb{Z} &
    \longleftarrow & \tmop{fib} ((\mathbb{Z}/ 4\mathbb{Z}) [1] \oplus
    \mathbb{Z}/ 4\mathbb{Z} \twoheadrightarrow \mathbb{F}_2 [1] \oplus
    \mathbb{F}_2)\\
    & \xleftarrow{\simeq} & (\mathbb{Z}/ 4\mathbb{Z}) [1] \oplus \mathbb{F}_2
    \oplus 2\mathbb{Z}/ 4\mathbb{Z}
  \end{eqnarray*}
  induced by $\mathbb{F}_2 \rightarrow 0$ and other maps are canonical. Since
  the forgetful functor $\tmop{Pair}^{\tmop{an}} \rightarrow \tmop{Fun}
  ((\Delta^1)^{\tmop{op}}, D (\mathbb{Z})_{\geq 0})$ commutes with $\tau_{\leq
  0}$ (\Cref{rem:ani-trunc}), we can identify the underlying object of
  $\tau_{\leq 0} ((\mathbb{Z}/ 4\mathbb{Z}, (2)) \amalg (\mathbb{Z}/
  4\mathbb{Z}, (2)))$ in $\tmop{Fun} ((\Delta^1)^{\tmop{op}}, \tmop{Ab})$ with
  $(\mathbb{Z}/ 4\mathbb{Z} \leftarrow \mathbb{F}_2 \oplus 2\mathbb{Z}/
  4\mathbb{Z})$, which is not injective. Roughly speaking, the localization
  $\mathcal{P}_{\Sigma} (\tmop{Pair}^{\tmop{st}}) \rightarrow \tmop{Pair}$
  will kill the kernel $\mathbb{F}_2$.
\end{example}

We now prove a stronger colimit-preserving property of the forgetful functor
from animated PD-pairs to animated pairs, which does not seem to be obvious
without this identification:

\begin{proposition}
  \label{prop:forget-PD-small-colim}The forgetful functor
  $\tmop{Pair}^{\gamma, \tmop{an}} \rightarrow \tmop{Pair}^{\tmop{an}}$
  preserves small colimits.
\end{proposition}

\begin{proof}
  By \Cref{prop:left-deriv-n-fun}, it suffices to show that the composite
  functor $\tmop{Pair}^{\gamma, \tmop{st}} \hookrightarrow
  \tmop{Pair}^{\gamma, \tmop{an}} \rightarrow \tmop{Pair}^{\tmop{an}}$
  preserves finite coproducts. We first ``simplify'' this composition, then we
  compute the finite coproducts by hand.
  
  Since $\tmop{Pair}^{\gamma, \tmop{st}} \hookrightarrow \tmop{Pair}^{\gamma,
  \tmop{an}}$ factors as $\tmop{Pair}^{\gamma, \tmop{st}} \hookrightarrow
  \tmop{Pair}^{\gamma} \hookrightarrow \tmop{Pair}^{\gamma, \tmop{an}}$, it
  follows from \Cref{prop:infty-cat-forget-embedding} that the composite
  functor $\tmop{Pair}^{\gamma, \tmop{st}} \hookrightarrow
  \tmop{Pair}^{\gamma, \tmop{an}} \rightarrow \tmop{Pair}^{\tmop{an}}$ is
  equivalent to the composite functor $\tmop{Pair}^{\gamma, \tmop{st}}
  \hookrightarrow \tmop{Pair}^{\gamma} \rightarrow \tmop{Pair} \hookrightarrow
  \tmop{Pair}^{\tmop{an}}$. Under the equivalence in \Cref{thm:ani-smith-eq},
  this functor is concretely given by $\tmop{Pair}^{\gamma, \tmop{st}} \ni (A,
  I, \gamma) \mapsto (A \twoheadrightarrow A / I) \in \tmop{Fun} (\Delta^1,
  \tmop{CAlg}^{\tmop{an}})_{\geq 0}$. Since $\tmop{Fun} (\Delta^1,
  \tmop{CAlg}^{\tmop{an}})_{\geq 0} \subseteq \tmop{Fun} (\Delta^1,
  \tmop{CAlg}^{\tmop{an}})$ is stable under small colimits, we can take the
  finite coproducts in the larger $\infty$\mbox{-}category $\tmop{Fun}
  (\Delta^1, \tmop{CAlg}^{\tmop{an}})$.
  
  Every object in $\tmop{Pair}^{\gamma, \tmop{st}}$ is the PD-envelope of a
  pair of form $(\mathbb{Z} [X_1, \ldots, X_m, Y_1, \ldots, Y_n], (Y_1,
  \ldots, Y_n))$, which we will denote by $\Gamma_{\mathbb{Z} [X_1, \ldots,
  X_m]} (Y_1, \ldots, Y_n) \twoheadrightarrow \mathbb{Z} [X_1, \ldots, X_m]$.
  Now the result follows from the fact that
  \[ \Gamma_{\mathbb{Z} [X]} (Y) \otimes_{\mathbb{Z}}^{\mathbb{L}}
     \Gamma_{\mathbb{Z} [X']} (Y') \simeq \Gamma_{\mathbb{Z} [X, X']} (Y, Y')
  \]
  and
  \[ \mathbb{Z} [X] \otimes_{\mathbb{Z}}^{\mathbb{L}} \mathbb{Z} [X'] \simeq
     \mathbb{Z} [X, X'] \]
  where $X = (X_1, \ldots, X_m)$, $X' = (X_1', \ldots, X_{m'}')$, $Y = (Y_1,
  \ldots, Y_n)$ and $Y' = (Y_1', \ldots, Y_{n'}')$.
\end{proof}

\begin{remark}
  \label{rem:pdpair-pair-forget-not-preserve-colim}\Cref{prop:forget-PD-small-colim}
  implies that the forgetful functor $\mathcal{P}_{\Sigma, 1}
  (\tmop{Pair}^{\gamma, \tmop{st}}) \rightarrow \mathcal{P}_{\Sigma, 1}
  (\tmop{Pair}^{\tmop{st}})$ preserves small colimits, cf.
  \Cref{lem:1-forget-pdpair}. However, the forgetful functor
  $\tmop{Pair}^{\gamma} \rightarrow \tmop{Pair}$ does not preserve small
  colimits, even pushouts
  {\cite[\href{https://stacks.math.columbia.edu/tag/07GY}{Tag
  07GY}]{stacks-project}}. The counterexample there is given by two
  PD-structures on the pair $(\mathbb{Z}/ 4\mathbb{Z}, (2))$. We explain the
  incompatibility of the localizations in
  Warning~\ref{warn:local-anipair-pair-incompatible-forget} by
  \Cref{ex:anipair-self-coprod}: the localization $\tmop{Pair}^{\tmop{an}}
  \rightarrow \tmop{Pair}$ kills the kernel $\mathbb{F}_2$, while the
  localization $\tmop{Pair}^{\gamma, \tmop{an}} \rightarrow
  \tmop{Pair}^{\gamma}$ kills more, since the PD-structure does not
  necessarily pass to the quotient, so one needs to quotient out more
  relations.
\end{remark}

\begin{corollary}
  The composite functor $\tmop{Pair}^{\tmop{an}}
  \xrightarrow{\tmop{Env}^{\gamma, \tmop{an}}} \tmop{Pair}^{\gamma, \tmop{an}}
  \rightarrow \tmop{Pair}^{\tmop{an}}$, where the second functor is the
  forgetful functor, preserves small colimits.
\end{corollary}

\subsection{Basic properties}\label{subsec:basic-prop-pairs}In this
subsection, we will discuss basic properties of animated pairs (resp. animated
PD-pairs).

First, we recall that, given a pair $(A, I)$, let $(B, J, \gamma)$ be its
PD-envelope, then there is a canonical equivalence $A / I \cong B / J$
{\cite[\href{https://stacks.math.columbia.edu/tag/07H7}{Tag
07H7}]{stacks-project}}. There is an analogue for animated PD-envelope:

\begin{lemma}
  \label{lem:ani-pd-env-preserv-target}The composite functor $F \of
  \tmop{Pair}^{\tmop{an}} \xrightarrow{\tmop{Env}^{\gamma, \tmop{an}}}
  \tmop{Pair}^{\gamma, \tmop{an}} \rightarrow \tmop{Pair}^{\tmop{an}}$, where
  the second functor is the forgetful functor, is compatible with the
  evaluation $\tmop{ev}_{[1]} \of \tmop{Pair}^{\tmop{an}} \simeq \tmop{Fun}
  (\Delta^1, \tmop{CAlg}^{\tmop{an}})_{\geq 0} \rightarrow
  \tmop{CAlg}^{\tmop{an}}$ at $[1] \in \Delta^1$. That is to say, the
  composite functor $\tmop{Pair}^{\tmop{an}} \xrightarrow{F}
  \tmop{Pair}^{\tmop{an}} \xrightarrow{\tmop{ev}_{[1]}}
  \tmop{CAlg}^{\tmop{an}}$ is homotopy equivalent to the functor
  $\tmop{Pair}^{\tmop{an}} \xrightarrow{\tmop{ev}_{[1]}}
  \tmop{CAlg}^{\tmop{an}}$.
\end{lemma}

\begin{proof}
  Both functors are left derived functors, therefore it suffices to check on
  the full subcategory $\tmop{Poly}_{\mathbb{Z}} \subseteq
  \tmop{Pair}^{\tmop{an}}$, which follows from a direct identification.
\end{proof}

We note that the functor $\tmop{CAlg}^{\tmop{an}} \rightarrow \tmop{Fun}
(\Delta^1, \tmop{CAlg}^{\tmop{an}}), A \mapsto (\tmop{id}_A \of A \rightarrow
A)$ is fully faithful, admits a left adjoint $\tmop{ev}_{[1]}$ and a right
adjoint $\tmop{ev}_{[0]}$. Restricting to the fully faithful embedding
$\tmop{Pair}^{\tmop{an}} \hookrightarrow \tmop{Fun} (\Delta^1,
\tmop{CAlg}^{\tmop{an}})$, we get

\begin{lemma}
  \label{lem:ani-ring-as-pair}The functor $\tmop{CAlg}^{\tmop{an}} \rightarrow
  \tmop{Pair}^{\tmop{an}}, A \mapsto (\tmop{id}_A \of A \rightarrow A)$ is
  fully faithful and admits a left adjoint $\tmop{ev}_{[1]} \of
  \tmop{Pair}^{\tmop{an}} \rightarrow \tmop{CAlg}^{\tmop{an}}, (A
  \twoheadrightarrow A'') \mapsto A''$ and a right adjoint $\tmop{ev}_{[0]}
  \of \tmop{CAlg}^{\tmop{an}} \rightarrow \tmop{Pair}^{\tmop{an}}, (A
  \twoheadrightarrow A'') \mapsto A$.
\end{lemma}

This functor preserves small colimits, therefore by
\Cref{prop:left-deriv-n-fun}, it is the left derived functor of the composite
functor $\tmop{Poly}_{\mathbb{Z}} \rightarrow \tmop{Pair}^{\tmop{st}}
\hookrightarrow \tmop{Pair}^{\tmop{an}}, A \mapsto (A, 0)$. Apply
\Cref{cor:nonab-deriv-cat-adjoint-fun} to the composite
$\tmop{Poly}_{\mathbb{Z}} \rightarrow \tmop{Pair}^{\tmop{st}} \rightarrow
\tmop{Pair}^{\gamma, \tmop{st}}$, we get:

\begin{lemma}
  \label{lem:ani-ring-as-pdpair}The composite functor $\tmop{CAlg}^{\tmop{an}}
  \rightarrow \tmop{Pair}^{\tmop{an}} \xrightarrow{\tmop{Env}^{\gamma,
  \tmop{an}}} \tmop{Pair}^{\gamma, \tmop{an}}$ is fully faithful, where the
  first functor is $\tmop{CAlg}^{\tmop{an}} \rightarrow
  \tmop{Pair}^{\tmop{an}}, A \mapsto (\tmop{id}_A \of A \rightarrow A)$, and a
  further composition $\tmop{CAlg}^{\tmop{an}} \rightarrow
  \tmop{Pair}^{\gamma, \tmop{an}} \rightarrow \tmop{Pair}^{\tmop{an}}$, where
  the second functor is the forgetful functor, is equivalent to the fully
  faithful functor $\tmop{CAlg}^{\tmop{an}} \rightarrow
  \tmop{Pair}^{\tmop{an}}, A \mapsto (\tmop{id}_A \of A \rightarrow A)$.
\end{lemma}

Despite Warning~\ref{warn:local-anipair-pair-incompatible-forget}, the image
of an animated PD-pair $(A \twoheadrightarrow A'', \gamma)$ under the
localization $\tmop{Pair}^{\gamma, \tmop{an}} \rightarrow
\tmop{Pair}^{\gamma}$ is of the form $(-) \rightarrow \tau_{\leq 0} (A'')$:

\begin{lemma}
  \label{lem:ani-pdpair-local-target}There is a canonical commutative diagram
  \[ \begin{array}{ccccc}
       \tmop{Pair}^{\gamma, \tmop{an}} & \longrightarrow &
       \tmop{Pair}^{\tmop{an}} & \xrightarrow{\tmop{ev}_{[1]}} &
       \tmop{CAlg}^{\tmop{an}}\\
       \longdownarrow &  &  &  & \longdownarrow \nocomma \tau_{\leq 0}\\
       \tmop{Pair}^{\gamma} & \longrightarrow & \tmop{Pair} &
       \xrightarrow{\tmop{ev}_{[1]}} & \tmop{Ring}
     \end{array} \]
  {\noindent}of $\infty$-categories, where the left horizontal arrows are
  forgetful functors, and the leftmost vertical arrow is the localization
  (\Cref{rem:localization-pair-pdpair}).
\end{lemma}

\begin{proof}
  We first note that the composite functor $\tmop{Pair}^{\gamma, \tmop{an}}
  \rightarrow \tmop{Pair}^{\tmop{an}} \rightarrow \tmop{CAlg}^{\tmop{an}}
  \rightarrow \tmop{Ring}$ preserves small colimits, therefore is a left
  derived functor (\Cref{prop:left-deriv-n-fun}), hence left Kan extended
  along $\tmop{Pair}^{\gamma, \tmop{st}} \hookrightarrow \tmop{Pair}^{\gamma,
  \tmop{an}}$. The diagram is canonically commutative on the full subcategory
  $\tmop{Pair}^{\gamma, \tmop{st}} \subseteq \tmop{Pair}^{\gamma, \tmop{an}}$.
  It remains to show the existence of the extension of the equivalence in
  question.
  
  Now consider the diagram
  \[ \begin{array}{ccccc}
       \tmop{Ring} & \longhookrightarrow & \tmop{Pair} & \longrightarrow &
       \tmop{Pair}^{\gamma}\\
       \longdownarrow &  &  &  & \longdownarrow\\
       \tmop{CAlg}^{\tmop{an}} & \longhookrightarrow & \tmop{Pair}^{\tmop{an}}
       & \longrightarrow & \tmop{Pair}^{\gamma, \tmop{an}}
     \end{array} \]
  {\noindent}of $\infty$-categories where the functors $\tmop{Ring}
  \rightarrow \tmop{Pair}$ and $\tmop{CAlg}^{\tmop{an}} \rightarrow
  \tmop{Pair}^{\tmop{an}}$ are given by $A \mapsto (A, 0)$ and $A \mapsto
  (\tmop{id}_A \of A \rightarrow A)$ respectively, and the functor
  $\tmop{Pair} \rightarrow \tmop{Pair}^{\gamma}$ and the functor
  $\tmop{Pair}^{\tmop{an}} \rightarrow \tmop{Pair}^{\gamma, \tmop{an}}$ are
  the PD-envelope (resp. animated PD-envelope) functors. This is a commutative
  diagram by \Cref{lem:ani-ring-as-pdpair}. Taking the right adjoints, we get
  the commutativity by \Cref{lem:ani-ring-as-pair}.
\end{proof}

Next, we show that animated PD-envelope ``does nothing'' after
rationalization. More precisely,

\begin{lemma}
  \label{lem:ani-PD-env-char-0}Consider the unit map $\eta$ from the functor
  $\tmop{id}_{\tmop{Pair}^{\tmop{an}}}$ to the composite functor
  $\tmop{Pair}^{\tmop{an}} \rightarrow \tmop{Pair}^{\gamma, \tmop{an}}
  \rightarrow \tmop{Pair}^{\tmop{an}}$ where the first functor is the animated
  PD-envelope functor and the second is the forgetful functor. Then the
  composite of $\eta$ with the rationalization functor
  $\tmop{Pair}^{\tmop{an}} \simeq \tmop{Fun} (\Delta^1,
  \tmop{CAlg}^{\tmop{an}})_{\geq 0} \xrightarrow{\cdummy
  \otimes_{\mathbb{Z}}^{\mathbb{L}} \mathbb{Q}} \tmop{Fun} (\Delta^1,
  \tmop{Ani} (\tmop{Alg}_{\mathbb{Q}}))_{\geq 0}$ is an equivalence of
  functors.
\end{lemma}

\begin{proof}
  Since the rationalization functor preserves filtered colimits and geometric
  realizations, by \Cref{prop:left-deriv-n-fun}, it suffices to show the
  equivalence on $\tmop{Pair}^{\tmop{st}} \subseteq \tmop{Pair}^{\tmop{an}}$.
  Concretely, it is saying that the canonical map $\mathbb{Z} [X, Y]
  \rightarrow \Gamma_{\mathbb{Z} [X]} (Y)$ becomes an equivalence after
  rationalization, which follows from definitions.
\end{proof}

Now we consider the base change. Given a surjective map $(A \twoheadrightarrow
A'') \in \tmop{Fun} (\Delta^1, \tmop{CAlg}^{\tmop{an}})_{\geq 0}$ and a map $A
\rightarrow B$ of animated rings, the base changed map $B \rightarrow A''
\otimes_A^{\mathbb{L}} B$ is also surjective. The key observation is that this
base change is a pushout $(A \twoheadrightarrow A'') \amalg_{(\tmop{id}_A \of
A \rightarrow A)} (\tmop{id}_B \of B \rightarrow B)$. Since the animated
PD-envelope functor, being a left adjoint, and the forgetful functor preserve
small colimits (\Cref{prop:forget-PD-small-colim}), it follows from
\Cref{lem:ani-ring-as-pdpair} that (to compare with
\Cref{rem:comma-base-chg}).

\begin{lemma}
  \label{lem:base-chg-PD-env}The composite functor $\tmop{Pair}^{\tmop{an}}
  \rightarrow \tmop{Pair}^{\gamma, \tmop{an}} \rightarrow
  \tmop{Pair}^{\tmop{an}}$ is compatible with base change, where the first
  functor is the animated PD-envelope functor and the second is the forgetful
  functor. More precisely, there is an equivalence from $(C
  \otimes_A^{\mathbb{L}} B \twoheadrightarrow C'' \otimes_A^{\mathbb{L}} B)$
  to the animated PD-e1.2: the technique of animation also appears in Lurie's
  higher topos theorynvelope of $B \rightarrow A'' \otimes_A^{\mathbb{L}} B$
  between animated pairs, where $(C \twoheadrightarrow C'', \gamma)$ is the
  animated PD-envelope of $(A \twoheadrightarrow A'')$, which is functorial
  with respect to the diagram $\begin{array}{ccc}
    A & \nonconverted{longtwoheadrightarrow} & A''\\
    \longdownarrow &  & \\
    B &  & 
  \end{array}$in $\tmop{CAlg}^{\tmop{an}}$.
\end{lemma}

\begin{remark}[General base]
  Let $R$ be a ring. We could then replace $\mathbb{Z}$ by $R$ in the theory
  of animated pairs and PD-pairs. For example, the $1$\mbox{-}category
  $\tmop{Ab}$ is replaced by $\tmop{Mod}_R$, the $\infty$\mbox{-}category $D
  (\mathbb{Z})$ is replaced by $D (R)$, the $\infty$\mbox{-}category
  $\tmop{CAlg}^{\tmop{an}}$ is replaced by $\tmop{CAlg}_R^{\tmop{an}}$, the
  $1$\mbox{-}category $\tmop{Pair}^{\tmop{st}}$ is replaced by
  $\tmop{Pair}_R^{\tmop{st}}$ consisting the pairs of the form $(R [X, Y],
  (Y))$, and $\tmop{Pair}^{\gamma, \tmop{st}}$ is replaced by
  $\tmop{Pair}_R^{\gamma, \tmop{st}}$ consisting the PD-pairs of the form
  $\Gamma_{R [X]} (Y) \twoheadrightarrow R [X]$, etc. We get
  $\tmop{Pair}^{\tmop{an}}_R$ and $\tmop{Pair}^{\gamma, \tmop{an}}_R$. There
  are canonical base change functors $\tmop{CAlg}^{\tmop{an}} \rightarrow
  \tmop{CAlg}_R^{\tmop{an}}$, $\tmop{Pair}^{\tmop{an}} \rightarrow
  \tmop{Pair}^{\tmop{an}}_R$ and $\tmop{Pair}^{\gamma, \tmop{an}} \rightarrow
  \tmop{Pair}^{\gamma, \tmop{an}}_R$ essentially induced by the base change $D
  (\mathbb{Z}) \xrightarrow{\cdummy \otimes_{\mathbb{Z}}^{\mathbb{L}} R} D
  (R)$.
\end{remark}

It follows from \Cref{cor:ani-undercat} that

\begin{lemma}
  There are canonical equivalences of $\infty$\mbox{-}categories
  \begin{eqnarray*}
    \tmop{CAlg}_R^{\tmop{an}} & \xrightarrow{\simeq} &
    \tmop{CAlg}^{\tmop{an}}_{R \mathord{/}}\\
    \tmop{Pair}^{\tmop{an}}_R & \xrightarrow{\simeq} &
    \tmop{Pair}^{\tmop{an}}_{(\tmop{id}_R \of R \rightarrow R) \mathord{/}}\\
    \tmop{Pair}^{\gamma, \tmop{an}}_R & \xrightarrow{\simeq} &
    \tmop{Pair}^{\gamma, \tmop{an}}_{(\tmop{id}_R \of R \rightarrow R, 0)
    \mathord{/}}
  \end{eqnarray*}
\end{lemma}

By the proof of \Cref{lem:cot-cx-indep-base}, it follows from
\Cref{lem:base-chg-PD-env} that

\begin{lemma}
  \label{lem:PD-env-indep-base}Let $R$ be a ring. Then there is a canonical
  commutative diagram
  \[ \begin{array}{ccc}
       \tmop{Pair}^{\tmop{an}}_R & \longrightarrow & \tmop{Pair}^{\gamma,
       \tmop{an}}_R\\
       \longdownarrow &  & \longdownarrow\\
       \tmop{Pair}^{\tmop{an}} & \longrightarrow & \tmop{Pair}^{\gamma,
       \tmop{an}}
     \end{array} \]
  {\noindent}of $\infty$-categories where the vertical arrows are forgetful
  functors and the horizontal arrows are animated PD-envelope functors.
\end{lemma}

Moreover, again by \Cref{lem:base-chg-PD-env}, we have

\begin{lemma}
  \label{lem:PD-env-compat-base-chg}Let $R$ be a ring. Then there is a
  canonical commutative diagram
  \[ \begin{array}{ccc}
       \tmop{Pair}^{\tmop{an}} & \longrightarrow & \tmop{Pair}^{\gamma,
       \tmop{an}}\\
       \longdownarrow &  & \longdownarrow\\
       \tmop{Pair}^{\tmop{an}}_R & \longrightarrow & \tmop{Pair}^{\gamma,
       \tmop{an}}_R
     \end{array} \]
  {\noindent}of $\infty$\mbox{-}categories, where the horizontal arrows are
  animated PD-envelope functors and the vertical arrows are base change
  functors.
\end{lemma}

\subsection{Quasiregular pairs}\label{subsec:quasiregular}This subsection is
devoted to comparison of animated theory of pairs and PD-pairs with the
classical version. Quasiregularity, introduced by Quillen, play an important
role:

\begin{definition}[{\cite[Thm~6.13]{Quillen1967}}]
  \label{def:quasiregular}We say that a pair $(A, I) \in \tmop{Pair}$ is
  {\tmdfn{quasiregular}} if the shifted cotangent complex $L_{(A / I) / A} [-
  1] \in D (A / I)$ is a flat $A / I$-module. We will denote by $\tmop{We}
  \tmop{will} \tmop{denote} \tmop{by} \tmop{QReg}^{\tmop{an}} \subseteq
  \tmop{Pair}^{\tmop{an}} \tmop{the} \tmop{full} \tmop{subcategory}
  \tmop{spanned} \tmop{by} \tmop{quasiregular} \tmop{animated} \tmop{pairs} .
  \tmop{The} \tmop{same} \tmop{for} \tmop{QReg}^{\tmop{an}}_{\mathbb{F}_p}
  \subseteq \tmop{Pair}^{\tmop{an}}_{\mathbb{F}_p} . \subseteq \tmop{Pair}$
  the full subcategory spanned by quasiregular pairs.
\end{definition}

\begin{example}
  Let $A$ be a ring and $I \subseteq A$ an ideal generated by a Koszul-regular
  sequence. Then $L_{(A / I) / A} \simeq (I / I^2) [1]$
  {\cite[\href{https://stacks.math.columbia.edu/tag/08SJ}{Tag
  08SJ}]{stacks-project}}, and $I / I^2$ is a free $A / I$-module
  {\cite[\href{https://stacks.math.columbia.edu/tag/062I}{Tag
  062I}]{stacks-project}}. We warn the reader that Quillen's quasiregular is
  different from ``quasi-regular'' in
  {\cite[\href{https://stacks.math.columbia.edu/tag/061M}{Tag
  061M}]{stacks-project}}, and the later is not used in this article.
\end{example}

The first goal of this subsection is to show that there is a ``derived''
version of the adic filtration on animated pairs. Furthermore, for pairs,
there is a natural comparison map from the ``derived'' version to the
classical version (strictly speaking, our comparison is slightly more
general), which becomes an equivalence when the pair in question is
quasiregular. We refer to \Cref{subsec:graded-fil-objs} for concepts and
notations about filtrations. We need the following results, which relates the
cotangent complex to powers of ideals.

\begin{lemma}[{\cite[\href{https://stacks.math.columbia.edu/tag/08RA}{Tag
08RA}]{stacks-project}}]
  There exists a map $(I / I^2) [1] \rightarrow L_{(A / I) / A}$ in $D (A /
  I)$ which is functorial in $(A, I) \in \tmop{Pair}$, such that the composite
  map $(I / I^2) [1] \rightarrow L_{(A / I) / A} \rightarrow \tau_{\leq 1}
  L_{(A / I) / A}$ is an equivalence.
\end{lemma}

\begin{remark}
  By abuse of terminology, by a map $M_{(A, I)} \rightarrow N_{(A, I)}$ in
  $D_{\geq 0} (A / I)$ being functorial in $(A, I) \in \tmop{Pair}$, we mean
  that the map in question is a map between two functors $(A, I)
  \rightrightarrows \tmop{Ani} (\tmop{Mod})$ given by $(A, I) \mapsto (A / I,
  M_{(A, I)})$ and $(A, I) \mapsto (A / I, N_{(A, I)})$ respectively.
\end{remark}

\begin{lemma}[{\cite[\href{https://stacks.math.columbia.edu/tag/08SI}{Tag
08SI}]{stacks-project}}]
  For any $(A, I) \in \tmop{Pair}^{\tmop{st}} \subseteq \tmop{Pair}$, the
  cotangent complex $L_{(A / I) / A}$ is $1$\mbox{-}truncated.
\end{lemma}

\begin{corollary}
  \label{cor:D0-cot-cx-square-quot}There exists an equivalence $(I / I^2) [1]
  \rightarrow L_{(A / I) / A}$ in $D_{\geq 0} (A / I)$ functorial in $(A, I)
  \in \tmop{Pair}^{\tmop{st}}$.
\end{corollary}

{\construction[{Adic filtration}]{\label{cons:adic-fil}Consider the
{\tmdfn{classical adic filtration functor}} $\tmop{Fil}_{\tmop{ad},
\tmop{cl}}^{\ast} \of \tmop{Pair} \rightarrow \tmop{CAlg} (\tmop{DF}^{\geq 0}
(\mathbb{Z})), (A, I) \mapsto (I^n)_{n \in \mathbb{N}_{\geq 0}}$. Restricting
to the full subcategory $\tmop{Pair}^{\tmop{st}} \subseteq \tmop{Pair}$ and
applying \Cref{prop:left-deriv-n-fun}, we get a functor
$\tmop{Fil}_{\tmop{ad}}^{\ast} \of \tmop{Pair}^{\tmop{an}} \rightarrow
\tmop{CAlg} (\tmop{DF}^{\geq 0} (\mathbb{Z}))$, called the {\tmdfn{adic
filtration functor}}. Such a construction appears in
{\cite[Cor~4.14]{Bhatt2012}} in the language of model categories. In
{\cite[§6]{Hekking2021}}, Jeroen {\tmname{Hekking}} constructed the same
filtration via different approaches.}}

\begin{remark}
  By the same argument, there is a natural structure of filtered derived ring
  (\Cref{rem:fil-deriv-alg}) on $\tmop{Fil}_{\tmop{ad}}^{\ast} (A
  \twoheadrightarrow A'')$, which we will not use in this article.
\end{remark}

By \Cref{thm:ani-smith-eq}, we can identify $\tmop{Fil}_{\tmop{ad}}^0 \of
\tmop{Pair}^{\tmop{an}} \rightarrow \tmop{CAlg}_{\mathbb{Z}}$ with the
composite functor
\[ \tmop{Fun} (\Delta^1, \tmop{CAlg}^{\tmop{an}})
   \xrightarrow{\tmop{ev}_{[0]}} \tmop{CAlg}^{\tmop{an}} \rightarrow
   \tmop{CAlg}_{\mathbb{Z}}, (A \twoheadrightarrow A'') \mapsto A \]
and $\tmop{gr}_{\tmop{ad}}^0 \of \tmop{Pair}^{\tmop{an}} \rightarrow
\tmop{CAlg}_{\mathbb{Z}}$ with the composite functor
\[ \tmop{Fun} (\Delta^1, \tmop{CAlg}^{\tmop{an}})
   \xrightarrow{\tmop{ev}_{[1]}} \tmop{CAlg}^{\tmop{an}} \rightarrow
   \tmop{CAlg}_{\mathbb{Z}}, (A \twoheadrightarrow A'') \mapsto A'' \]
Combining \Cref{cor:D0-cot-cx-square-quot}, \Cref{prop:left-deriv-n-fun},
sifted-colimit-preserving properties of $\mathbb{L} \tmop{Sym}^{\ast}$, and
the concrete analysis of pairs in $\tmop{Pair}^{\tmop{st}} \subseteq
\tmop{Pair}^{\tmop{an}}$, we get

\begin{corollary}
  \label{cor:LAdFil-symm-cot-cx}For every $(A \twoheadrightarrow A'') \in
  \tmop{Pair}^{\tmop{an}}$, the shifted cotangent complex $L_{A'' / A} [- 1]
  \simeq \tmop{gr}_{\tmop{ad}}^1 (A \twoheadrightarrow A'')$ is connective,
  and there exists an equivalence
  \[ \mathbb{L} \tmop{Sym}_{A''}^{\ast} (\tmop{gr}_{\tmop{ad}}^1 (A
     \twoheadrightarrow A'')) \rightarrow \tmop{gr}_{\tmop{ad}}^{\ast} (A
     \twoheadrightarrow A'') \]
  of graded $\mathbb{E}_{\infty}$-$\mathbb{Z}$-algebras functorial in $(A
  \twoheadrightarrow A'') \in \tmop{Pair}^{\tmop{an}}$.
\end{corollary}

Now we construct a comparison map between the ``derived'' filtration
$\tmop{Fil}_{\tmop{ad}}^{\ast}$ and the ``non-derived'' filtration
$\tmop{Fil}_{\tmop{ad}, \tmop{cl}}^{\ast}$ on ring-ideal pairs. We apply a
trick used in the proof of \Cref{prop:characterize-pair,lem:ani-trunc}.

{\construction{\label{cons:deriv-non-deriv-adic-comp}\Cref{prop:left-deriv-n-fun}
and the universal property of left Kan extensions give us a comparison natural
transform from the functor $\tmop{Fil}_{\tmop{ad}}^{\ast} \of
\tmop{Pair}^{\tmop{an}} \rightarrow \tmop{CAlg} (\tmop{DF}^{\geq 0}
(\mathbb{Z}))$ to the composite functor $\tmop{Pair}^{\tmop{an}} \rightarrow
\tmop{Pair} \xrightarrow{\tmop{Fil}_{\tmop{ad}, \tmop{cl}}^{\ast}} \tmop{CAlg}
(\tmop{DF}^{\geq 0} (\mathbb{Z}))$ where the first functor
$\tmop{Pair}^{\tmop{an}} \rightarrow \tmop{Pair}$ is the localization
(\Cref{rem:localization-pair-pdpair}).}}

Our next goal is to show that the comparison map is an equivalence after
restriction to $\tmop{QReg} \subseteq \tmop{Pair}^{\tmop{an}}$. Since the
forgetful functor $\tmop{CAlg} (\tmop{DF}^{\geq 0} (\mathbb{Z})) \rightarrow
\tmop{DF}^{\geq 0} (\mathbb{Z})$ is conservative, we can show the equivalence
after forgetting the $\mathbb{E}_{\infty}$-structure.

The previous discussion show that the comparison map is an equivalence after
composing with $\tmop{Fil}^0 \of \tmop{DF}^{\geq 0} (\mathbb{Z}) \rightarrow D
(\mathbb{Z})$ and $\tmop{gr}^0 \of \tmop{DF}^{\geq 0} (\mathbb{Z}) \rightarrow
D (\mathbb{Z})$ on the $1$\mbox{-}category $\tmop{Pair}$ (not only for
quasiregular pairs). We define the functor $\tmop{gr}^{[0, n)} \of
\tmop{DF}^{\geq 0} (\mathbb{Z}) \rightarrow D (\mathbb{Z}), F \mapsto
\tmop{cofib} (\tmop{Fil}^n (F) \rightarrow \tmop{Fil}^0 (F))$. Thus it suffice
to prove that the comparison map is an equivalence after composing with
$\tmop{gr}^{[0, n)} \of \tmop{DF}^{\geq 0} (\mathbb{Z}) \rightarrow D
(\mathbb{Z})$ for all $n > 1$ for quasiregular pairs. Note that by definition,
the essential image of $\tmop{gr}_{\tmop{ad}}^{[0, n)}$ already lies in
$\tmop{Ab} \subseteq D (\mathbb{Z})$. We show a more general statement (cf.
the proof of \Cref{prop:characterize-pair,lem:ani-trunc}):

\begin{lemma}
  \label{lem:adfil-local-compatible}There is a commutative diagram
  \[ \begin{array}{ccc}
       \tmop{Pair}^{\tmop{an}} & \xrightarrow{\tmop{gr}_{\tmop{ad}}^{[0, n)}}
       & D (\mathbb{Z})_{\geq 0}\\
       \longdownarrow &  & \longdownarrow \nocomma \tau_{\leq 0}\\
       \tmop{Pair} & \xrightarrow{\tmop{gr}_{\tmop{ad}}^{[0, n)}} & \tmop{Ab}
     \end{array} \]
  {\noindent}of $\infty$\mbox{-}categories, where the homotopy from the
  top-right composition to the bottom-left composition is induced by the
  comparison map in Construction~\ref{cons:deriv-non-deriv-adic-comp}.
\end{lemma}

\begin{proof}
  The trick is to consider an auxiliary functor. Let $(A \twoheadrightarrow
  A'') \in \tmop{Pair}^{\tmop{an}}$ and let $I \assign \tmop{fib} (A
  \twoheadrightarrow A'') \in D (\mathbb{Z})_{\geq 0}$. We recall that, by
  \Cref{thm:ani-smith-eq}, and \eqref{diag:ani-smith-eq} in particular, the
  forgetful functor $\tmop{Pair}^{\tmop{an}} \rightarrow \tmop{Fun}
  ((\Delta^1)^{\tmop{op}}, D (\mathbb{Z})_{\geq 0})$ is just $(A
  \twoheadrightarrow A'') \mapsto (A \leftarrow I)$.
  
  Then the map $I \rightarrow A$ in $D (\mathbb{Z})_{\geq 0}$ induces a map
  $\mathbb{L} \tmop{Sym}_{\mathbb{Z}}^n I \rightarrow \mathbb{L}
  \tmop{Sym}_{\mathbb{Z}}^n A$. Composing with the multiplication $\mathbb{L}
  \tmop{Sym}_{\mathbb{Z}}^n A \rightarrow A$, we get the map $\mathbb{L}
  \tmop{Sym}_{\mathbb{Z}}^n I \rightarrow A$. We consider the functor $F \of
  \tmop{Pair}^{\tmop{an}} \rightarrow D (\mathbb{Z})_{\geq 0}, (A
  \twoheadrightarrow A'') \mapsto \tmop{cofib} (\mathbb{L}
  \tmop{Sym}_{\mathbb{Z}}^n I \rightarrow A)$.
  
  First, the functor $F$ preserves filtered colimits and geometric
  realizations, since the functor $\mathbb{L} \tmop{Sym}_{\mathbb{Z}}$ and the
  forgetful functor $\tmop{Pair}^{\tmop{an}} \rightarrow \tmop{Fun}
  ((\Delta^1)^{\tmop{op}}, D (\mathbb{Z})_{\geq 0})$ do
  (\Cref{lem:1-forget-pair}). In fact, $F$ is the left derived functor
  (\Cref{prop:left-deriv-n-fun}) of the functor $\tmop{Pair}^{\tmop{st}}
  \rightarrow D (\mathbb{Z})_{\geq 0}, (A, I) \mapsto \tmop{cofib}
  (\tmop{Sym}_{\mathbb{Z}}^n I \rightarrow A)$.
  
  Next, note that for $(A, I) \in \tmop{Pair}^{\tmop{st}}$, the map
  $\tmop{Sym}_{\mathbb{Z}}^n I \rightarrow A$ factors functorially as
  $\tmop{Sym}_{\mathbb{Z}}^n I \rightarrow I^n \rightarrow A$ and the map
  $\tmop{Sym}_{\mathbb{Z}}^n I \rightarrow I^n$ is surjective. It follows that
  there is a natural surjective map $\tmop{cofib} (\tmop{Sym}_{\mathbb{Z}}^n I
  \rightarrow A) \rightarrow A / I^n$, which gives rise to a map $F
  \rightarrow \tmop{gr}_{\tmop{ad}}^{[0, n)}$ of functors which becomes an
  equivalence after composition with $\tau_{\leq 0} \of D (\mathbb{Z})_{\geq
  0} \rightarrow \tmop{Ab}$.
  
  We now show that the functor $\tau_{\leq 0} \circ F \of
  \tmop{Pair}^{\tmop{an}} \rightarrow \tmop{Ab}$ factors through the
  localization $\tmop{Pair}^{\tmop{an}} \rightarrow \tmop{Pair}$. First, since
  $\tmop{Ab}$ is a $1$\mbox{-}category, it factors through
  $\tmop{Pair}^{\tmop{an}} \rightarrow \mathcal{P}_{\Sigma, 1}
  (\tmop{Pair}^{\tmop{st}})$. Given $(A \twoheadrightarrow A'') \in
  \mathcal{P}_{\Sigma, 1} (\tmop{Pair}^{\tmop{st}})$, let $I = \tmop{fib} (A
  \twoheadrightarrow A'')$ as before. Then $(A \leftarrow I) \in
  \mathcal{P}_{\Sigma, 1} (\tmop{Inj}^{\tmop{st}}) \simeq \tmop{Fun}
  ((\Delta^1)^{\tmop{op}}, \tmop{Ab})$, therefore $A, I$ are static. Let $I' =
  \tmop{im} (I \rightarrow A)$. It follows that the localization
  $\mathcal{P}_{\Sigma, 1} (\tmop{Pair}^{\tmop{st}}) \rightarrow \tmop{Pair}$
  maps $(A \twoheadrightarrow A'')$ to $(A, I') \in \tmop{Pair}$. By
  \Cref{prop:L-inv}, it suffices to show that $F$ maps $(A \twoheadrightarrow
  A'') \rightarrow (A, I')$ to an equivalence. This simply follows from the
  fact that $\mathbb{L} \tmop{Sym}_{\mathbb{Z}}^n I \rightarrow \mathbb{L}
  \tmop{Sym}^n_{\mathbb{Z}} I'$ is a surjection on $\pi_0$, and the
  ``multiplication'' map $\mathbb{L} \tmop{Sym}_{\mathbb{Z}}^n I \rightarrow
  A$ factors as $\mathbb{L} \tmop{Sym}_{\mathbb{Z}}^n I \rightarrow \mathbb{L}
  \tmop{Sym}_{\mathbb{Z}}^n I' \rightarrow A$.
  
  In conclusion, we have already shown that there exists an equivalence of two
  compositions in the diagram that we need to prove. To show that this
  equivalence is the equivalence that we want, we note that the top right
  composition preserves filtered colimits and geometric realizations, then the
  first paragraph of the proof of \Cref{lem:ani-pdpair-local-target} applies.
\end{proof}

In particular, when $(A, I)$ is quasiregular, it follows from
\Cref{cor:LAdFil-symm-cot-cx} that $\tmop{gr}_{\tmop{ad}}^n (A
\twoheadrightarrow A / I) \in D (\mathbb{Z})_{\geq 0}$ is static for all $n
\in \mathbb{N}$, which implies that $\tmop{gr}_{\tmop{ad}}^{[0, n)} (A
\twoheadrightarrow A / I)$ is static for all $n \in \mathbb{N}$. Consequently,
we have

\begin{proposition}
  The comparison natural transformation in
  Construction~\ref{cons:deriv-non-deriv-adic-comp} becomes an equivalence
  after restricting to the full subcategory $\tmop{QReg} \subseteq
  \tmop{Pair}^{\tmop{an}}$.
\end{proposition}

\begin{corollary}[{\cite[6.11]{Quillen1967}}]
  \label{cor:symm-assoc-gr-equiv-qreg}For every quasiregular pair $(A, I)$,
  the canonical map $\tmop{Sym}_{A / I}^{\ast} (I / I^2) \rightarrow \bigoplus
  I^{\ast} / I^{\mathord{\ast} + 1}$ of graded rings is an equivalence.
\end{corollary}

\begin{proof}
  It suffices to show that the equivalence given by
  \Cref{cor:LAdFil-symm-cot-cx} coincides with the canonical map induced by
  the multiplicative structure on $A$. For any element $\overline{x_1} \cdots
  \overline{x_n} \in \tmop{Sym}_{A / I}^n (I / I^2)$, we pick a lift $x_1,
  \ldots, x_n \in I$, which gives rise to a map $(B, J) \assign (\mathbb{Z}
  [X_1, \ldots, X_n], (X_1, \ldots, X_n)) \rightarrow (A, I)$ of pairs, which
  induces the commutative diagram
  
  \[\begin{tikzcd} {\operatorname{Sym}_{B/J}^n(J/J^2)} & {J^n/J^{n+1}} \\
  {\mathbb
  L\operatorname{Sym}_{A/I}^n(\operatorname{gr}_{\operatorname{ad}}^1(A\twoheadrightarrow
  A/I))} & {\operatorname{gr}_{\operatorname{ad}}^n(A\twoheadrightarrow A/I)}
  \arrow[from=1-1, to=1-2] \arrow[from=1-1, to=2-1] \arrow[from=2-1, to=2-2]
  \arrow[from=1-2, to=2-2] \end{tikzcd}\]
  
  \
  
  {\noindent}in the $\infty$-category $D (\mathbb{Z})_{\geq 0}$. Taking
  $\tau_{\leq 0}$ and trace the element $\overline{X_1} \cdots \overline{X_n}
  \in \tmop{Sym}_{B / J}^n (J / J^2)$, we get the result.
\end{proof}

We are unable to answer the following question in full generality:

\begin{question}
  \label{ques:qreg-ani-pd-env-static}Let $(A, I)$ be a quasiregular pair. Let
  $(B \twoheadrightarrow B'', \gamma)$ denote its animated PD-envelope. Is it
  true that $B, B''$ are static, so by
  \Cref{prop:characterize-pdpair,cor:pd-env-animated-pd-env}, it coincides
  with the classical PD-envelope?
\end{question}

However, we are able to answer it under certain flatness. First, it follows
from \Cref{lem:ani-PD-env-char-0} that when $A$ is a $\mathbb{Q}$-algebra, the
animated PD-envelope of $(A, I)$ is just $A \twoheadrightarrow A / I$, which
is also the classical PD-envelope.

Now we consider the characteristic $p > 0$ case, switching the ground ring
from $\mathbb{Z}$ to $\mathbb{F}_p$ (which is valid by
\Cref{lem:PD-env-indep-base}). We will use the notations
$\tmop{Pair}^{\tmop{st}}$ and $\tmop{Pair}^{\gamma, \tmop{st}}$ in
\Cref{subsec:ani-pairs-PD-pairs} but the occurrences of $\mathbb{Z}$ are
replaced by $\mathbb{F}_p$. We recall that the Frobenius map $A \rightarrow A,
x \mapsto x^p$ of an $\mathbb{F}_p$-algebra $A$ gives rise to an endomorphism
$\varphi \of \tmop{id}_{\tmop{Alg}_{\mathbb{F}_p}} \rightarrow
\tmop{id}_{\tmop{Alg}_{\mathbb{F}_p}}$ of the identity functor
$\tmop{id}_{\tmop{Alg}_{\mathbb{F}_p}} \of \tmop{Alg}_{\mathbb{F}_p}
\rightarrow \tmop{Alg}_{\mathbb{F}_p}$, which gives rise to an endomorphism
$\tmop{id}_{\tmop{Ani} (\tmop{Alg}_{\mathbb{F}_p})} \rightarrow
\tmop{id}_{\tmop{Ani} (\tmop{Alg}_{\mathbb{F}_p})}$ still denoted by
$\varphi$. We now introduce the {\tmdfn{conjugate filtration}} on the animated
PD-envelope of animated $\mathbb{F}_p$-pairs that we learned from
{\cite{Bhatt2012a}}.

{\construction{\label{cons:conj-fil-std-pair}Let $(A, I)$ be an
$\mathbb{F}_p$-pair such that the Frobenius $\varphi_A \of A \rightarrow A$ is
flat, and let $(B, J, \gamma)$ denote its PD-envelope. We first note that
there is a $\varphi_A^{\ast} (A / I)$-algebra structure on $B$ since $f^p = p
\gamma_p (f) = 0$ for all $f \in J$. Then we have a filtration on $B$ given by
$\tmop{Fil}_{\tmop{conj}, \tmop{cl}}^{- n} B$ for $n \geq 0$ to be the
$\varphi_A^{\ast}  (A / I)$-submodule of $B$ generated by $\left\{ \gamma_{k_1
p} (f_1) \cdots \gamma_{k_m p} (f_m) \barsuchthat k_1 + \cdots + k_m \leq n
\infixand f_1, \ldots, f_m \in I \right\}$, which gives rise to a structure of
nonpositively filtered $\varphi_A^{\ast}  (A / I)$-algebra. We note that the
filtration is exhaustive, i.e. $\tmop{Fil}_{\tmop{conj}, \tmop{cl}}^{- \infty}
B = \tmop{colim}_{n \in (\mathbb{Z}, \geq)} \tmop{Fil}^{- n} B \rightarrow B$
is an isomorphism, and we can rephrase the nonpositively filtered
$\varphi_A^{\ast}  (A / I)$-algebra structure as a map $\varphi_A^{\ast}  (A /
I) \rightarrow \tmop{Fil} B$\footnote{We will from time to time suppress the
asterisk in $\tmop{Fil}^{\ast}$ to avoid confusion with $\varphi^{\ast}$.} of
nonpositively filtered ring.}}

We need the following result:

\begin{lemma}[{\cite[Lem~3.42]{Bhatt2012a}}]
  \label{lem:bhatt-conj-fil}Let $(A, I)$ be an $\mathbb{F}_p$-pair such that
  $I / I^2$ is a flat $A / I$-module and the Frobenius $\varphi_A \of A
  \rightarrow A$ is flat, and let $(B, J, \gamma)$ denote the PD-envelope of
  $(A, I)$.
  
  Then there is a comparison map $\varphi_A^{\ast} (\Gamma_{A / I}^i (I /
  I^2)) \rightarrow \tmop{gr}_{\tmop{conj}, \tmop{cl}}^{- i} B$ (as in
  Construction~\ref{cons:conj-fil-std-pair}) of $\varphi_A^{\ast} (A /
  I)$-modules induced by the maps $(\gamma_{kp})_{k \in \mathbb{N}}$ which is
  functorial in $(A, I)$. For example, when $I / I^2$ is free, an element in
  $\Gamma_{A / I}^i (I / I^2)$ represented by $\frac{f^{\otimes i}}{i!}$ will
  be mapped to $\gamma_{ip} (f)$ for $f \in I$.
  
  Furthermore, if $I \subseteq A$ is generated by a Koszul-regular
  sequence\footnote{We only need the special case that $(A, I) \in
  \tmop{Pair}^{\tmop{st}}$.}, then the comparison map above is an isomorphism.
\end{lemma}

Now we define the conjugate filtration on the animated PD-envelope.

\begin{definition}
  \label{def:conj-fil-pd-env}The {\tmdfn{conjugate filtration functor (on the
  animated PD-envelope)}} $\tmop{Fil}_{\tmop{conj}}^{\ast} \tmop{Env}^{\gamma,
  \tmop{an}} \of \tmop{Pair}^{\tmop{an}}_{\mathbb{F}_p} \rightarrow
  \tmop{CAlg} (\tmop{DF}^{\leq 0} (\mathbb{F}_p))$ together with the
  {\tmdfn{structure map}} of functors $\tmop{Pair}^{\tmop{an}}_{\mathbb{F}_p}
  \rightrightarrows \tmop{CAlg} (\tmop{DF}^{\leq 0} (\mathbb{F}_p))$ from $(A
  \twoheadrightarrow A'', \gamma) \mapsto \varphi_A^{\ast} (A'') = A''
  \otimes_{A, \varphi_A}^{\mathbb{L}} A$ to $\tmop{Fil}_{\tmop{conj}}^{\ast}
  \tmop{Env}^{\gamma, \tmop{an}}$, or equivalently, a functor
  \begin{eqnarray*}
    \tmop{Pair}^{\tmop{an}} & \longrightarrow & \tmop{Fun} (\Delta^1,
    \tmop{CAlg} (\tmop{DF}^{\leq 0} (\mathbb{F}_p)))\\
    (A \twoheadrightarrow A'') & \longmapsto & (\varphi_A^{\ast} (A'')
    \rightarrow \tmop{Fil}_{\tmop{conj}}^{\ast} \tmop{Env}^{\gamma, \tmop{an}}
    (A \twoheadrightarrow A''))
  \end{eqnarray*}
  is defined to be the left derived functor (\Cref{prop:left-deriv-n-fun}) of
  $\tmop{Pair}^{\tmop{st}} \ni (A, I) \mapsto (\varphi_A^{\ast} (A / I)
  \rightarrow \tmop{Fil}_{\tmop{conj}, \tmop{cl}}^{\ast} D_A (I)) \in
  \tmop{Fun} (\Delta^1, \tmop{CAlg} (\tmop{DF}^{\leq 0} (\mathbb{F}_p)))$ in
  Construction~\ref{cons:conj-fil-std-pair}, where $\varphi_A^{\ast} (A / I)$
  is constantly filtered.
\end{definition}

\begin{remark}
  Informally speaking, the functor $\tmop{Pair}^{\tmop{an}} \rightarrow
  \tmop{Fun} (\Delta^1, \tmop{CAlg} (\tmop{DF}^{\leq 0} (\mathbb{F}_p)))$ in
  \Cref{def:conj-fil-pd-env} is capturing two pieces of data:
  \begin{enumerate}
    \item ahe conjugate filtration $\tmop{Fil}_{\tmop{conj}}^{\ast}$ on the
    animated PD-envelope $\tmop{Env}^{\gamma, \tmop{an}} (A \twoheadrightarrow
    A'')$;
    
    \item an $\mathbb{E}_{\infty}$-$\varphi_A^{\ast} (A'')$-algebra structure
    on the conjugate filtration,
  \end{enumerate}
  and that these data are {\tmem{functorial}} in $(A \twoheadrightarrow A'')
  \in \tmop{Pair}^{\tmop{an}}$.
\end{remark}

\begin{remark}
  We note that the conjugate filtration is exhaustive, i.e. there is a
  canonical equivalence $\tmop{Fil}_{\tmop{conj}}^{- \infty}
  \tmop{Env}^{\gamma, \tmop{an}} \rightarrow \tmop{Env}^{\gamma, \tmop{an}}$
  of functors $\tmop{Pair}_{\mathbb{F}_p}^{\tmop{an}} \rightarrow \tmop{CAlg}
  (D (\mathbb{F}_p))$, which follows either from
  \Cref{prop:left-deriv-n-fun,lem:assoc-graded-union-Lan} or the fact that
  $\tmop{Pair}^{\tmop{an}} \simeq \mathcal{P}_{\Sigma}
  (\tmop{Pair}^{\tmop{st}}) \subseteq \mathcal{P} (\tmop{Pair}^{\tmop{st}})$
  is stable under filtered colimits (\Cref{prop:Psigma-n}).
\end{remark}

It follows from \Cref{lem:bhatt-conj-fil} that

\begin{corollary}
  \label{cor:conjfil-gr}For every $(A \twoheadrightarrow A'') \in
  \tmop{Pair}^{\tmop{an}}_{\mathbb{F}_p}$, there exists an equivalence
  \[ \varphi_A^{\ast} (\Gamma_{A''}^i (\tmop{gr}_{\tmop{ad}}^1 (A
     \twoheadrightarrow A''))) \rightarrow \tmop{gr}_{\tmop{conj}}^{- i}
     \tmop{Env}^{\gamma, \tmop{an}} (A \twoheadrightarrow A'') \]
  in $D (\varphi_A^{\ast} (A''))_{\geq 0}$ for all $i \in \mathbb{N}$ which is
  functorial in $(A \rightarrow A'') \in
  \tmop{Pair}^{\tmop{an}}_{\mathbb{F}_p}$.
\end{corollary}

\begin{remark}
  \label{rem:functor-target-variable}One might wonder what precisely the
  functor is, since the target category $D (\varphi_A^{\ast} (A''))_{\geq 0}$
  depends on $(A \twoheadrightarrow A'') \in
  \tmop{Pair}^{\tmop{an}}_{\mathbb{F}_p}$. One can rigorously describe this
  $\varphi_A^{\ast} (A'')$-algebra structure in terms of structure maps (as in
  \Cref{def:conj-fil-pd-env}). However, this is cumbersome and we keep the
  current ``imprecise'' presentation.
\end{remark}

We now apply this to a quasiregular pair $(A, I) \in
\tmop{QReg}_{\mathbb{F}_p}$. We first recall that

\begin{definition}[{\cite[Def~7.2.2.10]{Lurie2017}}]
  \label{def:flat-module}Let $A$ be an $\mathbb{E}_1$-ring. We say that a
  right $A$-module spectrum $M$ is {\tmdfn{flat}} if
  \begin{enumerate}
    \item The homotopy group $\pi_0 (M)$ is a flat right $\pi_0 (A)$-module.
    
    \item For each $n \in \mathbb{Z}$, the canonical map $\pi_0 (M)
    \otimes_{\pi_0 (A)} \pi_n (A) \rightarrow \pi_n (M)$ is an isomorphism of
    abelian groups.
  \end{enumerate}
  The same concept applies to left $A$-module spectra.
\end{definition}

\begin{remark}[{\cite[Rem~7.2.2.11 \& 7.2.2.12]{Lurie2017}}]
  \label{rem:flat-connective-static}Let $R$ be an $\mathbb{E}_1$-ring and $M$
  a flat right $R$-module spectrum. By definition, if $R$ is connective (resp.
  static), then so is $R$. In particular, when $R$ is static, a flat
  $R$-module spectrum is simply a flat $R$-module, therefore we will sometimes
  refer to flat module spectra simply as flat modules since there is no
  ambiguity.
\end{remark}

\begin{lemma}
  \label{lem:flat-stable-ext}Let $A$ be a connective $\mathbb{E}_1$-ring, and
  $M' \rightarrow M \rightarrow M''$ a fiber sequence of right $A$-module
  spectra. If $M', M''$ are flat right $A$-modules, then so is $M$.
\end{lemma}

\begin{proof}
  First, $M', M''$ are connective by flatness, therefore so is $M$. For every
  static left $A$-module $N$, we have a fiber sequence $N
  \otimes_A^{\mathbb{L}} M' \rightarrow N \otimes_A^{\mathbb{L}} M \rightarrow
  N \otimes_A^{\mathbb{L}} M''$. By flatness of $M'$ and $M''$ and
  {\cite[Prop~7.2.2.13]{Lurie2017}}, the spectra $N \otimes_A^{\mathbb{L}} M'$
  and $N \otimes_A^{\mathbb{L}} M''$ are static, therefore so is $N
  \otimes_A^{\mathbb{L}} M$. The result then follows from
  {\cite[Thm~7.2.2.15]{Lurie2017}}.
\end{proof}

For future usages, we need to generalize slightly the concept of quasiregular
pairs:

\begin{definition}
  We say that an animated pair $A \twoheadrightarrow A''$ is
  {\tmdfn{quasiregular}} if the shifted cotangent complex $L_{A'' / A} [- 1]
  \in D (A'')$ is a flat $A''$-module spectrum.
\end{definition}

\begin{corollary}
  \label{cor:Fp-qreg-pd-env-flat}Let $(A \twoheadrightarrow A'')$ be a
  quasiregular animated $\mathbb{F}_p$-pair, and let $(B \twoheadrightarrow
  B'', \gamma)$ denote its animated PD-envelope. Then $B$ is a flat
  $\varphi_A^{\ast} (A'')$-module spectrum.
\end{corollary}

\begin{proof}
  It follows from \Cref{cor:conjfil-gr}, $\Gamma^{\ast}$ and base change
  preserving flatness ({\cite[Cor~25.2.3.3]{Lurie2018}} \&
  {\cite[Prop~7.2.2.16]{Lurie2017}}) that $\tmop{gr}_{\tmop{conj}}^{- i}
  \tmop{Env}^{\gamma, \tmop{an}} (A \twoheadrightarrow A'')$ is a flat
  $\varphi_A^{\ast} (A'')$-module spectrum. The result follows from the fact
  that the full subcategory spanned by flat modules over a connective
  $\mathbb{E}_1$-ring is stable under extension (\Cref{lem:flat-stable-ext})
  and under filtered colimits by {\cite[Thm~7.2.2.14(1)]{Lurie2017}}.
\end{proof}

\begin{remark}
  In fact, by \Cref{lem:crit-faithful-flat-map}, the map $\varphi_A^{\ast}
  (A'') \rightarrow B$ in \Cref{cor:Fp-qreg-pd-env-flat} is faithfully flat.
\end{remark}

It follows from
\Cref{prop:characterize-pdpair,cor:pd-env-animated-pd-env,cor:Fp-qreg-pd-env-flat,rem:flat-connective-static}
that

\begin{corollary}
  \label{cor:Fp-qreg-frob-ani-pd-env}Let $(A, I) \in
  \tmop{QReg}_{\mathbb{F}_p}$ be a quasiregular pair. Suppose that
  $\varphi_A^{\ast} (A / I)$ is static. Then the animated PD-envelope $(B
  \twoheadrightarrow B'', \gamma)$ of $(A \twoheadrightarrow A / I)$ belongs
  to $\tmop{Pair}^{\gamma}_{\mathbb{F}_p}$, therefore coincides with the
  classical PD-envelope.
\end{corollary}

We want to point out that such results for $\mathbb{F}_p$ will be used to
deduce integral results, which is based on the following lemmas, cf.
{\cite[Lem~6.1.2.4]{Lurie2018}}.

\begin{lemma}
  \label{lem:static-loc-global}Let $M \in \tmop{Sp}_{\geq 0}$ be a connective
  spectrum. Suppose that the rationalization $M
  \otimes_{\mathbb{S}}^{\mathbb{L}} \mathbb{Q}$ is static, and for every prime
  $p \in \mathbb{N}$, the homotopy groups of $M /^{\mathbb{L}} p \assign
  \tmop{cofib} \left( M \xrightarrow{p} M \right)$ are concentrated in degree
  $0, 1$. Then $M$ is static.
\end{lemma}

\begin{proof}
  Since $\mathbb{Q}$ is $\mathbb{S}$-flat, $\pi_i (M) \otimes_{\mathbb{Z}}
  \mathbb{Q} \cong \pi_i (M \otimes_{\mathbb{S}}^{\mathbb{L}} \mathbb{Q})
  \cong 0$ when $i \neq 0$. On the other hand, $\pi_{i + 1} (M /^{\mathbb{L}}
  p) \cong 0$ for $i > 0$ implies that the map $\pi_i (M) \xrightarrow{p}
  \pi_i (M)$ is injective for every prime $p \in \mathbb{N}$ and $i > 0$. It
  follows that $\pi_i (M) \cong 0$ for every $i > 0$.
\end{proof}

\begin{warning}
  \Cref{lem:static-loc-global} is false if $M$ is not assumed to be
  connective. A counterexample is given by $M = (\mathbb{Q}/\mathbb{Z}) [-
  1]$, for which $M \otimes_{\mathbb{S}}^{\mathbb{L}} \mathbb{Q} \simeq 0$ and
  $M /^{\mathbb{L}} p \simeq \mathbb{F}_p$ for every prime number $p \in
  \mathbb{N}$.
\end{warning}

\begin{lemma}[cf. {\cite[\href{https://stacks.math.columbia.edu/tag/039C}{Tag
039C}]{stacks-project}}]
  \label{lem:flat-loc-global}Let $A$ be an animated ring and $M \in D_{\geq 0}
  (A)$ a connective $A$-module spectrum. Then the following conditions are
  equivalent:
  \begin{enumerate}
    \item $M$ is a flat $A$-module.
    
    \item $M \otimes_{\mathbb{Z}}^{\mathbb{L}} \mathbb{Q}$ is a flat $A
    \otimes_{\mathbb{Z}}^{\mathbb{L}} \mathbb{Q}$-module, and for every prime
    $p \in \mathbb{N}$, $M /^{\mathbb{L}} p$ is a flat $A /^{\mathbb{L}}
    p$-module.
  \end{enumerate}
\end{lemma}

\begin{proof}
  The first implies the second by the stability of flatness under base change
  {\cite[Prop~7.2.2.16]{Lurie2017}}. We now assume the second. By
  {\cite[Thm~7.2.2.15]{Lurie2017}}, it suffices to show that for each static
  $A$-module $N$, the tensor product $M \otimes_A^{\mathbb{L}} N$ is also
  static. Indeed,
  \[ (M \otimes_A^{\mathbb{L}} N) \otimes_{\mathbb{S}}^{\mathbb{L}} \mathbb{Q}
     \simeq (M \otimes_A^{\mathbb{L}} N) \otimes_{\mathbb{Z}}^{\mathbb{L}}
     \mathbb{Q} \simeq (M \otimes_{\mathbb{Z}}^{\mathbb{L}} \mathbb{Q})
     \otimes_{A \otimes_{\mathbb{Z}}^{\mathbb{L}} \mathbb{Q}}^{\mathbb{L}} (N
     \otimes_{\mathbb{Z}}^{\mathbb{L}} \mathbb{Q}) \]
  is static by the $\mathbb{Z}$-flatness of $\mathbb{Q}$ and the flatness of
  $M \otimes_{\mathbb{Z}}^{\mathbb{L}} \mathbb{Q}$. On the other hand,
  \[ (M \otimes_A^{\mathbb{L}} N) /^{\mathbb{L}} p \simeq (M /^{\mathbb{L}} p)
     \otimes_{A /^{\mathbb{L}} p}^{\mathbb{L}} (N /^{\mathbb{L}} p) \]
  for every prime $p \in \mathbb{N}$. Since $M /^{\mathbb{L}} p$ is $A
  /^{\mathbb{L}} p$-flat,
  \[ \pi_i ((M /^{\mathbb{L}} p) \otimes_{A /^{\mathbb{L}} p}^{\mathbb{L}} (N
     /^{\mathbb{L}} p)) \simeq \pi_0 (M /^{\mathbb{L}} p) \otimes_{\pi_0 (A
     /^{\mathbb{L}} p)} \pi_i (N /^{\mathbb{L}} p) \cong 0 \]
  for all $i > 1$ by {\cite[Prop~7.2.2.13]{Lurie2017}}. It then follows from
  \Cref{lem:static-loc-global} that $M \otimes_A^{\mathbb{L}} N$ is static.
\end{proof}

We record a simple consequence (compare with {\cite[Lem~2.42]{Bhatt2019}}):

\begin{proposition}
  \label{prop:koszul-regular-ani-pd-env}Let $A$ be a ring and $I \subseteq A$
  an ideal generated by a Koszul-regular sequence. Then the animated
  PD-envelope $(B \twoheadrightarrow B'', \gamma)$ of $(A \twoheadrightarrow A
  / I)$ belongs to $\tmop{Pair}^{\gamma}$, therefore coincides with the
  classical PD-envelope.
\end{proposition}

\begin{proof}
  Note that $B'' \simeq A / I$ is static by
  \Cref{lem:ani-pd-env-preserv-target}. It follows from
  \Cref{lem:ani-PD-env-char-0} that $B \otimes_{\mathbb{Z}}^{\mathbb{L}}
  \mathbb{Q} \simeq A$ is static. Let $(f_1, \ldots, f_r)$ be a Koszul-regular
  sequence which generates $I$. Fix a prime $p \in \mathbb{N}_{> 0}$. Let
  $A_0$ denote $A /^{\mathbb{L}} p$. We follow the argument in
  {\cite[Lem~3.41]{Bhatt2012a}}:
  \begin{eqnarray*}
    \varphi_{A_0}^{\ast} ((A / I) /^{\mathbb{L}} p) & \simeq &
    \varphi_{A_0}^{\ast} (A_0 /^{\mathbb{L}} (f_1)) \otimes_{A_0}^{\mathbb{L}}
    \cdots \otimes_{A_0}^{\mathbb{L}} \varphi_{A_0}^{\ast} (A_0 /^{\mathbb{L}}
    f_r)\\
    & \simeq & (A_0 /^{\mathbb{L}} f_1^p) \otimes_{A_0}^{\mathbb{L}} \cdots
    \otimes_{A_0}^{\mathbb{L}} (A_0 /^{\mathbb{L}} f_r^p)\\
    & \simeq & A_0 /^{\mathbb{L}} (f_1^p, \ldots, f_r^p)\\
    & \simeq & (A /^{\mathbb{L}} (f_1^p, \ldots, f_r^p)) /^{\mathbb{L}} p
  \end{eqnarray*}
  Note that since $(f_1, \ldots, f_r)$ is Koszul-regular, so is $(f_1^p,
  \ldots, f_r^p)$, which implies that $\pi_i (\varphi_{A_0}^{\ast} ((A / I)
  /^{\mathbb{L}} p)) \cong 0$ for $i \neq 0, 1$. It follows from
  \Cref{cor:Fp-qreg-pd-env-flat} and the base change property
  (\Cref{lem:PD-env-compat-base-chg}) that $\pi_i (B /^{\mathbb{L}} p) \cong
  0$ for $i \neq 0, 1$. The result then follows from
  \Cref{lem:static-loc-global}.
\end{proof}

\subsection{Illusie's question}\label{subsec:Illusie-ques}Given a ring $A$ and
an ideal $I \subseteq A$ generated by a Koszul-regular sequence, let $(B, J,
\gamma)$ denote the PD-envelope of $(A, I)$. It is known that the canonical
comparison map $\Gamma_{A / I}^{\ast} (I / I^2) \rightarrow J^{[\ast]} /
J^{\left[ \mathord{\ast} + 1 \right]}$ is an isomorphism, cf.
{\cite[I.~Prop~3.4.4]{Berthelot1974}}, where $J^{[\ast]}$ are divided powers
of $J$ in $B$. In {\cite[VIII.~Ques~2.2.4.2]{Illusie1972}}, Illusie asked
whether this holds for quasiregular pairs $(A, I)$. The answer is affirmative,
and the goal of this section is to furnish a proof by our theory of animated
PD-pairs.

Our strategy is similar to \Cref{subsec:quasiregular}: both the animated
PD-envelope and the PD-envelope of a pair $(A, I)$ admit a canonical
filtration, and there is a natural comparison between the two. Although for
general quasiregular pairs $(A, I)$ we do not know whether the animated
PD-envelope coincides with the PD-envelope, the comparison map induces
equivalences on graded pieces. The associated graded of the animated
PD-envelope admits a natural structure of divided power algebra, and an
element tracing proves that the equivalence coincides with the comparison map
in Illusie's question.

We start with the {\tmdfn{PD-filtration}} on animated PD-pairs. We refer to
\Cref{subsec:graded-fil-objs} for concepts and notations about filtrations.

{\construction{Let $(A, I, \gamma) \in \tmop{Pair}^{\gamma}$ be a PD-pair and
$n \in \mathbb{N}$ a natural number. The classical divided power ideal
$I^{[n]} \subseteq A$ is the ideal generated by elements $\gamma_{i_1} (x_1)
\cdots \gamma_{i_k} (x_k)$ where $x_1, \ldots, x_k \in I$ and $(i_1, \ldots,
i_k) \in \mathbb{N}^k$ with $i_1 + \cdots + i_k \geq n$. For example, for
$(\Gamma_{\mathbb{Z}} (x) \twoheadrightarrow \mathbb{Z}) \in
\tmop{Pair}^{\gamma}$ with kernel $I$, the kernel $I^{[n]} \subseteq
\Gamma_{\mathbb{Z}} (x)$ is generated by $\{ \gamma_i (x) \barsuchthat i \geq
n \}$ (which is different from the ideal $(\gamma_n (x))$). The
{\tmdfn{classical PD-filtration}} on $A$ is given by $A \supseteq I \supseteq
I^{[2]} \supseteq \cdots$ endowing $A$ with the structure of filtered ring. A
filtered ring is naturally a (nonnegatively) filtered
$\mathbb{E}_{\infty}$-ring, and we get a functor $\tmop{Fil}_{\tmop{PD},
\tmop{cl}}^{\ast} \of \tmop{Pair}^{\gamma} \rightarrow \tmop{CAlg}
(\tmop{DF}^{\geq 0} (\mathbb{Z}))$.}}

\begin{definition}
  \label{def:pd-fil}The {\tmdfn{PD-filtration functor}}
  $\tmop{Fil}_{\tmop{PD}}^{\ast} \of \tmop{Pair}^{\gamma, \tmop{an}}
  \rightarrow \tmop{CAlg} (\tmop{DF}^{\geq 0} (\mathbb{Z}))$ is defined to be
  the left derived functor (\Cref{prop:left-deriv-n-fun}) of the composite
  functor $\tmop{Pair}^{\gamma, \tmop{st}} \hookrightarrow
  \tmop{Pair}^{\gamma} \xrightarrow{\tmop{Fil}_{\tmop{PD}, \tmop{cl}}^{\ast}}
  \tmop{CAlg} (\tmop{DF}^{\geq 0} (\mathbb{Z}))$. For an animated PD-pair $(A
  \twoheadrightarrow A'', \gamma) \in \tmop{Pair}^{\gamma, \tmop{an}}$, we
  will call the image $\tmop{Fil}_{\tmop{PD}}^{\ast}  (A \twoheadrightarrow
  A'', \gamma)$ the $\mathbb{E}_{\infty}$-$\mathbb{Z}$-algebra $A$ with
  {\tmdfn{PD-filtration}}.
\end{definition}

\begin{remark}
  By the same argument, the PD-filtration in fact gives rise to a structure of
  filtered derived ring (\Cref{rem:fil-deriv-alg}), which we will not use in
  this article.
\end{remark}

Similar to \Cref{cor:LAdFil-symm-cot-cx}, by \Cref{prop:left-deriv-n-fun},
sifted-colimit-preserving property of the (derived) divided power functor
$((A, M) \mapsto \Gamma_A M) \of \tmop{Ani} (\tmop{Mod}) \rightarrow
\tmop{CAlg} (\tmop{Gr}^{\geq 0} (D (\mathbb{Z})))$ and the concrete analysis
of $(A, I, \gamma) \in \tmop{Pair}^{\gamma, \tmop{st}}$, we get

\begin{lemma}
  \label{lem:pdfil-free-pdalg}For every $(A \twoheadrightarrow A'', \gamma)
  \in \tmop{Pair}^{\gamma, \tmop{an}}$, there exists an equivalence
  \[ \Gamma_{A''}^{\ast} (\tmop{gr}_{\tmop{PD}}^1 \nobracket (A
     \twoheadrightarrow A'', \gamma))) \rightarrow
     \tmop{gr}_{\tmop{PD}}^{\ast} (A \twoheadrightarrow A'', \gamma) \]
  of graded $\mathbb{E}_{\infty}$-$\mathbb{Z}$-algebras which is functorial in
  $(A \twoheadrightarrow A'', \gamma) \in \tmop{Pair}^{\gamma, \tmop{an}}$.
\end{lemma}

Furthermore, we can compare the adic filtration on an animated pair and the
PD-filtration on the animated PD-filtration. We first compare them on
$\tmop{Pair}^{\tmop{st}}$, then extend the comparison to
$\tmop{Pair}^{\tmop{an}}$ by \Cref{prop:left-deriv-n-fun}, obtaining

\begin{lemma}
  \label{lem:adfil-pdfil}For every $(A \twoheadrightarrow A'') \in
  \tmop{Pair}^{\tmop{an}}$, let $(B \twoheadrightarrow B'', \gamma) \in
  \tmop{Pair}^{\gamma, \tmop{an}}$ denote its animated PD-envelope. Then there
  is a canonical comparison map
  \[ \tmop{gr}_{\tmop{ad}}^{\ast} (A \twoheadrightarrow A'') \rightarrow
     \tmop{gr}_{\tmop{PD}}^{\ast} (B \twoheadrightarrow B'', \gamma) \]
  of graded $\mathbb{E}_{\infty}$-$\mathbb{Z}$-algebras which is functorial in
  $(A \twoheadrightarrow A'') \in \tmop{Pair}^{\tmop{an}}$. Furthermore, this
  map induces equivalences in $D (\mathbb{Z})$ when $\mathord{\ast} = 0, 1$.
\end{lemma}

{\construction{\label{cons:PDFil-comp}Analogous to \Cref{subsec:quasiregular},
by universal property of left Kan extensions, there exists a essentially
unique comparison map $c_{\gamma}$ from the composite functor
$\tmop{Pair}^{\tmop{an}} \xrightarrow{\tmop{Env}^{\gamma, \tmop{an}}}
\tmop{Pair}^{\gamma, \tmop{an}} \xrightarrow{\tmop{Fil}_{\tmop{PD}}^{\ast}}
\tmop{CAlg} (\tmop{DF}^{\geq 0} (\mathbb{Z}))$ to the composite functor
$\tmop{Pair}^{\tmop{an}} \rightarrow \tmop{Pair}
\xrightarrow{\tmop{Env}^{\gamma}} \tmop{Pair}^{\gamma}
\xrightarrow{\tmop{Fil}_{\tmop{PD}, \tmop{cl}}^{\ast}} \tmop{CAlg}
(\tmop{DF}^{\geq 0} (\mathbb{Z}))$, where $\tmop{Pair}^{\tmop{an}} \rightarrow
\tmop{Pair}$ is the localization in \Cref{rem:localization-pair-pdpair}.}}

The main result of this subsection is the following:

\begin{proposition}
  \label{prop:comp-PDFil-equiv-qreg}The comparison map $c_{\gamma}$ in
  Construction~\ref{cons:PDFil-comp} becomes an equivalence after composition
  $\tmop{QReg} \hookrightarrow \tmop{Pair}^{\tmop{an}} \rightrightarrows
  \tmop{CAlg} (\tmop{DF}^{\geq 0} (\mathbb{Z})) \xrightarrow{\tmop{gr}^{\ast}}
  \tmop{CAlg} (\tmop{Gr}^{\geq 0} (\mathbb{Z}))$.
\end{proposition}

\begin{remark}
  As seen in Question~\ref{ques:qreg-ani-pd-env-static}, we do not know
  whether the comparison is an equivalence when we replace $\tmop{gr}^{\ast}
  \of \tmop{CAlg} (\tmop{DF}^{\geq 0} (\mathbb{Z})) \rightarrow \tmop{CAlg}
  (\tmop{Gr}^{\geq 0} (\mathbb{Z}))$ by $\tmop{Fil}^0 \of \tmop{CAlg}
  (\tmop{DF}^{\geq 0} (\mathbb{Z})) \rightarrow \tmop{CAlg}_{\mathbb{Z}}$,
  though it is true under assumptions of \Cref{cor:Fp-qreg-frob-ani-pd-env},
  which is the only obstruction for the comparison map to be a filtered
  equivalence.
\end{remark}

We start to prove this. Unfortunately, we are unable to establish a strong
result like \Cref{lem:adfil-local-compatible} essentially due to the
complication discussed in
Warning~\ref{warn:local-anipair-pair-incompatible-forget}. Our trick is to
show that after replacing $\tmop{gr}^{\ast}$ by $\tmop{gr}^{[0, n)}$, both
functors satisfy a common universal property.

As in \Cref{subsec:quasiregular}, we can forget the
$\mathbb{E}_{\infty}$-algebra structure then replace $\tmop{gr}^{\ast}$ by
$\tmop{gr}^{[0, n)} \of \tmop{DF}^{\geq 0} (\mathbb{Z}) \rightarrow D
(\mathbb{Z}), F \mapsto \tmop{cofib} (\tmop{Fil}^n F \rightarrow \tmop{Fil}^0
F)$, i.e., it is equivalent to show that the natural comparison
$c_{\gamma}^{[0, n)}$ from the composite functor
\[ \tmop{Pair}^{\tmop{an}} \rightarrow \tmop{Pair}^{\gamma, \tmop{an}}
   \xrightarrow{\tmop{Fil}_{\tmop{PD}}^{\ast}} \tmop{DF}^{\geq 0} (D
   (\mathbb{Z})) \xrightarrow{\tmop{gr}^{[0, n)}} D (\mathbb{Z}) \]
to the composite functor
\[ \tmop{Pair}^{\tmop{an}} \rightarrow \tmop{Pair} \hookrightarrow
   \tmop{Pair}^{\gamma} \xrightarrow{\tmop{Fil}_{\tmop{PD}, \tmop{cl}}^{\ast}}
   \tmop{DF}^{\geq 0} (D (\mathbb{Z})) \xrightarrow{\tmop{gr}^{[0, n)}} D
   (\mathbb{Z}) \]
is an equivalence after restricting to the full subcategory $\tmop{QReg}
\subseteq \tmop{Pair}^{\tmop{an}}$. Note that the composite functor
$\tmop{gr}_{\tmop{PD}, \tmop{cl}}^{[0, n)} = \tmop{gr}^{[0, n)} \circ
\tmop{Fil}_{\tmop{PD}, \tmop{cl}}^{\ast}$ is concretely given by $(A, I,
\gamma) \mapsto A / I^{[n]}$, which motivates the following definition:

\begin{definition}
  We say that a PD-pair $(A, I, \gamma)$ is {\tmdfn{PD-nilpotent of height $n
  \in \mathbb{N}$}} if $I^{[n]} = 0$. We will denote by
  $\tmop{Pair}^{\gamma_n} \subseteq \tmop{Pair}^{\gamma}$ the full subcategory
  spanned by PD-nilpotent PD-pairs of height $n$.
\end{definition}

The following lemma could be checked directly, or (as $\infty$-categories) by
invoking {\cite[Prop~5.2.7.8]{Lurie2009}}:

\begin{lemma}
  Let $n \in \mathbb{N}$ be a natural number. Then the full subcategory
  $\tmop{Pair}^{\gamma_n} \hookrightarrow \tmop{Pair}^{\gamma}$ is reflective
  of which the localization $\tmop{Pair}^{\gamma} \rightarrow
  \tmop{Pair}^{\gamma_n}$, which will be denoted by $R_{\tmop{cl}}^{[n]}$, is
  given by killing the higher divided powers: $(A, I, \gamma) \mapsto (A /
  I^{[n]}, IA / I^{[n]}, \overline{\gamma})$ where $\overline{\gamma}
  (\overline{x}) = \overline{\gamma (x)}$ for all $x \in I$ and $\overline{x},
  \overline{\gamma (x)}$ are images of $x, \gamma (x)$ in $A / I^{[n]}$.
\end{lemma}

Then the composite functor $\tmop{gr}_{\tmop{PD}, \tmop{cl}}^{[0, n)} \of
\tmop{Pair}^{\gamma} \rightarrow D (\mathbb{Z}), (A, I, \gamma) \mapsto A /
I^{[n]}$ could be rewritten as the composite $\tmop{Pair}^{\gamma} \rightarrow
\tmop{Pair}^{\gamma_n} \hookrightarrow \tmop{Pair}^{\gamma, \tmop{an}}
\rightarrow D (\mathbb{Z})$ where the last functor is the functor
$\tmop{Pair}^{\gamma, \tmop{an}} \rightarrow D (\mathbb{Z}), (A
\twoheadrightarrow A'', \gamma) \mapsto A$. We now show that the composite
functor $\tmop{gr}_{\tmop{PD}}^{[0, n)} \of \tmop{Pair}^{\gamma, \tmop{an}}
\rightarrow D (\mathbb{Z})$ could also factor through $\tmop{Pair}^{\gamma,
\tmop{an}} \rightarrow D (\mathbb{Z}), (A \twoheadrightarrow A'', \gamma)
\mapsto A$. In fact, it is a ``derived'' version of the previous
factorization.

\begin{notation}
  Let $n \in \mathbb{N}$ be a natural number. Then we will denote by $R^{[n]}
  \of \tmop{Pair}^{\gamma, \tmop{an}} \rightarrow \tmop{Pair}^{\gamma,
  \tmop{an}}$ the left derived functor (\Cref{prop:left-deriv-n-fun}) of the
  composite functor $\tmop{Pair}^{\gamma, \tmop{st}} \rightarrow
  \tmop{Pair}^{\gamma_n} \hookrightarrow \tmop{Pair}^{\gamma, \tmop{an}}$
  where the first functor $\tmop{Pair}^{\gamma, \tmop{st}} \rightarrow
  \tmop{Pair}^{\gamma_n}$ is the restriction of the localization
  $R_{\tmop{cl}}^{[n]} \of \tmop{Pair}^{\gamma} \rightarrow
  \tmop{Pair}^{\gamma_n}$ to the full subcategory $\tmop{Pair}^{\gamma,
  \tmop{st}} \subseteq \tmop{Pair}^{\gamma}$.
\end{notation}

We compose $R^{[n]}$ with the functor $\tmop{Pair}^{\gamma, \tmop{an}}
\rightarrow D (\mathbb{Z}), (A \twoheadrightarrow A'', \gamma) \mapsto A$, we
get a functor $\tmop{Pair}^{\gamma, \tmop{an}} \rightarrow D (\mathbb{Z})$,
which is equivalent to the composite functor $\tmop{gr}_{\tmop{PD}}^{[0, n)}$
by \Cref{prop:left-deriv-n-fun} since both functors preserves filtered
colimits and geometric realizations and they are canonically identified on the
full subcategory $\tmop{Pair}^{\gamma, \tmop{st}} \subseteq
\tmop{Pair}^{\gamma, \tmop{an}}$.

{\construction{Let $n \in \mathbb{N}$ be a natural number. Then there is an
essentially unique comparison map $c_{\gamma}^{[n]}$ from the composite
functor
\[ \tmop{Pair}^{\tmop{an}} \rightarrow \tmop{Pair}^{\gamma, \tmop{an}}
   \xrightarrow{R^{[n]}} \tmop{Pair}^{\gamma, \tmop{an}} \]
which preserves filtered colimits and geometric realizations, to the composite
functor
\[ \tmop{Pair}^{\tmop{an}} \rightarrow \tmop{Pair} \hookrightarrow
   \tmop{Pair}^{\gamma} \xrightarrow{R_{\tmop{cl}}^{[n]}}
   \tmop{Pair}^{\gamma_n} \hookrightarrow \tmop{Pair}^{\gamma, \tmop{an}} \]
which is equivalent to $c_{\gamma}^{[0, n)}$ after composing the
sifted-colimit-preserving functor $\tmop{Pair}^{\gamma, \tmop{an}} \rightarrow
D (\mathbb{Z})$ by checking on the full subcategory $\tmop{Pair}^{\tmop{st}}
\subseteq \tmop{Pair}^{\tmop{an}}$ and the universal property of the left Kan
extension.}}

It remains to show that

\begin{lemma}
  \label{lem:comp-PDRed-equiv-qreg}The comparison map $c_{\gamma}^{[n]}$ of
  functors $\tmop{Pair}^{\tmop{an}} \rightrightarrows \tmop{Pair}^{\gamma,
  \tmop{an}}$ becomes an equivalence after restricting to the full subcategory
  $\tmop{QReg} \subseteq \tmop{Pair}^{\tmop{an}}$.
\end{lemma}

\begin{proof}
  It follows from
  \Cref{lem:pdfil-free-pdalg,lem:adfil-pdfil,cor:LAdFil-symm-cot-cx,prop:characterize-pdpair}
  and the fact that the derived divided powers $\Gamma^{\ast}$ of a flat
  module is flat therefore static, that the essential image of the composite
  functor
  \begin{equation}
    \tmop{QReg} \hookrightarrow \tmop{Pair}^{\tmop{an}} \rightarrow
    \tmop{Pair}^{\gamma, \tmop{an}} \xrightarrow{R^{[n]}} \tmop{Pair}^{\gamma,
    \tmop{an}} \label{eq:qreg-red-anipdpair}
  \end{equation}
  lies in the full subcategory $\tmop{Pair}^{\gamma} \subseteq
  \tmop{Pair}^{\gamma, \tmop{an}}$. We first show that the essential image
  further lies in the full subcategory $\tmop{Pair}^{\gamma_n} \subseteq
  \tmop{Pair}^{\gamma}$.
  
  We fix a quasiregular pair $(A, I) \in \tmop{QReg}$. Let $(C, K, \gamma) \in
  \tmop{Pair}^{\gamma}$ denote the image of $(A, I) \in \tmop{QReg}$ under the
  composite functor \eqref{eq:qreg-red-anipdpair}. Since $(A, I)$ could be
  rewritten as a sifted colimit $\tmop{colim}_{j \in \mathcal{I}}  (B_j, J_j)$
  taken in $\tmop{Pair}^{\tmop{an}}$, where $(B_j, J_j) \in
  \tmop{Pair}^{\tmop{st}}$. Let $(C_j, K_j, \gamma_j) \in \tmop{Pair}^{\gamma,
  \tmop{st}}$ be the PD-envelope of $(B_j, J_j)$. Then $(C, K, \gamma) \simeq
  \tmop{colim}_{j \in \mathcal{I}}  (C_j / K_j^{[n]}, K_j C_j / K_j^{[n]},
  \gamma_j)$ taken in $\tmop{Pair}^{\gamma, \tmop{an}}$. For every $x_1,
  \ldots, x_m \in K$ and $i_1, \ldots, i_m \in \mathbb{N}$ such that $i_1 +
  \cdots + i_m \geq n$, we need to show that $\gamma_{i_1} (x_1) \cdots
  \gamma_{i_m} (x_m) = 0$. The elements $x_1, \ldots, x_m$ define a map
  $\varphi \of (\Gamma_{\mathbb{Z}} (X_1, \ldots, X_m) \twoheadrightarrow
  \mathbb{Z}, \delta) \rightarrow (C, K, \gamma)$ in $\tmop{Pair}^{\gamma}
  \subseteq \tmop{Pair}^{\gamma, \tmop{an}}$. Since $(\Gamma_{\mathbb{Z}}
  (X_1, \ldots, X_m) \twoheadrightarrow \mathbb{Z}, \delta) \in
  \tmop{Pair}^{\gamma, \tmop{an}}$ is compact and projective and $\mathcal{I}$
  is sifted, the map $\varphi$ factors as $(\Gamma_{\mathbb{Z}} (X_1, \ldots,
  X_m) \twoheadrightarrow \mathbb{Z}, \delta) \rightarrow (C_j / K_j^{[n]},
  K_j C_j / K_j^{[n]}, \gamma_j) \rightarrow \tmop{colim}_{k \in \mathcal{I}} 
  (C_k / K_k^{[n]}, K_k C_k / K_k^{[n]}, \gamma_k)$ for some $j \in
  \mathcal{I}$. Then the element $\gamma_{i_1} (X_1) \cdots \gamma_{i_m} (X_m)
  \in \Gamma_{\mathbb{Z}} (X_1, \ldots, X_m)$ is killed by the first map,
  hence $\gamma_{i_1} (x_1) \cdots \gamma_{i_m} (x_m) = 0$.
  
  Note that the composite of left adjoints $\tmop{Pair} \rightarrow
  \tmop{Pair}^{\gamma} \xrightarrow{R_{\tmop{cl}}^{[n]}}
  \tmop{Pair}^{\gamma_n}$ preserves small colimits, $(C, K, \gamma) \in
  \tmop{Pair}^{\gamma_n}$ is isomorphic to the image of $(A, I) \in
  \tmop{QReg} \subseteq \tmop{Pair}$ under this composite functor and the map
  $(A, I) \rightarrow (C, K)$ is the unit map under this isomorphism. The
  result then follows from the uniqueness of universal objects.
\end{proof}

\begin{remark}
  In fact, there is an $\infty$-category $\tmop{Pair}^{\gamma_n, \tmop{an}}$
  of {\tmdfn{animated PD-pairs PD-nilpotent of height $n$}}, defined to be the
  nonabelian derived category of the essential image of $\tmop{Pair}^{\gamma,
  \tmop{st}} \subseteq \tmop{Pair}$ under the functor $R_{\tmop{cl}}^{[n]} \of
  \tmop{Pair}^{\gamma} \rightarrow \tmop{Pair}^{\gamma_n}$. Then there is a
  pair of adjoint functors $\tmop{Pair}^{\gamma_n, \tmop{an}} \rightleftarrows
  \tmop{Pair}^{\gamma_n, \tmop{an}}$ by
  \Cref{cor:nonab-deriv-cat-adjoint-fun}. Furthermore, by mimicking the proof
  of \Cref{lem:pdpair-Psigma1-full-faith}, the canonical functor
  $\tmop{Pair}^{\gamma_n} \rightarrow \tmop{Pair}^{\gamma_n, \tmop{an}}$ is
  fully faithful. This leads to an alternative proof of
  \Cref{lem:comp-PDRed-equiv-qreg}. Although the functor
  $\tmop{Pair}^{\gamma_n} \rightarrow \tmop{Pair}^{\gamma}$ is fully faithful,
  we conjecture that the functor $\tmop{Pair}^{\gamma_n, \tmop{an}}
  \rightarrow \tmop{Pair}^{\gamma, \tmop{an}}$ is not fully faithful, similar
  to the fact that the forgetful functor $D (\mathbb{Z}/ n\mathbb{Z})
  \rightarrow D (\mathbb{Z})$ is not fully faithful though
  $\tmop{Mod}_{\mathbb{Z}/ n\mathbb{Z}} \rightarrow \tmop{Ab}$ is so.
\end{remark}

Now we answer Illusie's question:

\begin{proposition}
  \label{prop:qreg-assoc-gr-pd}For every quasiregular pair $(A, I) \in
  \tmop{QReg}$, let $(B, J, \gamma)$ denote its PD-envelope. Then the
  canonical map $\Gamma_{A / I}^{\ast} (I / I^2) \rightarrow \bigoplus
  J^{[\ast]} / J^{\left[ \mathord{\ast} + 1 \right]}$ of graded rings induced
  by $\gamma_{\ast} \of I \rightarrow I$ is an equivalence.
\end{proposition}

\begin{proof}
  It follows from
  \Cref{cor:LAdFil-symm-cot-cx,lem:adfil-pdfil,lem:pdfil-free-pdalg,prop:comp-PDFil-equiv-qreg}
  that there is a comparison map $\Gamma_{A / I}^{\ast} (I / I^2) \rightarrow
  \bigoplus J^{[\ast]} / J^{\left[ \mathord{\ast} + 1 \right]}$ of graded
  rings. Then the result follows from element tracing, a modification of the
  proof of \Cref{cor:symm-assoc-gr-equiv-qreg} by replacing $(\mathbb{Z} [X_1,
  \ldots, X_n], (X_1, \ldots, X_n))$ with $(\Gamma_{\mathbb{Z}} (X_1, \ldots,
  X_n) \twoheadrightarrow \mathbb{Z}, \gamma)$.
\end{proof}

\section{Derived crystalline cohomology}\label{sec:deriv-crys-coh}

In this section, we define and study the {\tmdfn{Hodge-filtered derived
crystalline cohomology}}, a filtered
$\mathbb{E}_{\infty}$-$\mathbb{Z}$-algebra functorially associated to an
animated PD-pair $(A \twoheadrightarrow A'', \gamma)$ along with a map $A''
\rightarrow R$ of animated rings. To do so, we will introduce an auxiliary
construction, the {\tmdfn{Hodge-filtered derived de Rham cohomology}},
functorially associated to a map $(A \twoheadrightarrow A'', \gamma)
\rightarrow (B \twoheadrightarrow B'', \delta)$ of animated PD-pairs, which
will be proved independent of the choice of $B$, and then we define the
Hodge-filtered derived crystalline cohomology for $(A \twoheadrightarrow A'',
\gamma)$ along with $A'' \rightarrow R$ as the Hodge-filtered derived de Rham
cohomology of the map $(A \twoheadrightarrow A'', \gamma) \rightarrow
(\tmop{id}_R \of R \rightarrow R, 0)$. This is a generalization of the (usual)
derived de Rham cohomology as introduced in Illusie's thesis
{\cite{Illusie1971}}.

Furthermore, we also define the {\tmdfn{cohomology of the affine crystalline
site}} which could be endowed with Hodge-filtration, whose definition is more
similar to the classical crystalline cohomology in {\cite{Berthelot1974}}. The
Hodge-filtered derived de Rham cohomology is, roughly speaking, equivalent to
the {\tmdfn{relative animated PD-envelope}} whenever $A'' \rightarrow R$ is
surjective (\Cref{prop:crys-PD-env-equiv}), and the Hodge-filtered derived de
Rham cohomology is equivalent to the cohomology of the affine crystalline site
with Hodge filtration when $\pi_0 (R)$ is a finitely generated $\pi_0
(A'')$-algebra (\Cref{prop:deriv-crys-coh-site-comp}) or when $R$ is a
quasisyntomic $A''$-algebra (\Cref{prop:site-comp-qsyn-integral}).
Furthermore, the cohomology of the affine crystalline site is equivalent to
the classical crystalline cohomology when everything is classically given, at
least up to $p$-completion, due to the fact that our theory is non-completed
(\Cref{prop:site-static-comp}). Our results could be understood as an
extrapolation of techniques in {\cite{Bhatt2012a}}.

\begin{remark}
  Our theory is characteristic-independent. As a cost, the derived de Rham
  cohomology does not coincide with algebraic de Rham cohomology even under
  smoothness condition, although this is true when the base is of
  characteristic $p$. In particular, for a map $(A \twoheadrightarrow A'',
  \gamma) \rightarrow (B \twoheadrightarrow B'', \delta)$ of animated PD-pairs
  where $A$ is an animated $\mathbb{Q}$-algebra, the underlying
  $\mathbb{E}_{\infty}$-ring of our Hodge-filtered derived de Rham cohomology
  is constantly $A$, cf. \Cref{lem:dR-rational}. However, in this case, the
  non-completed crystalline cohomology (\Cref{def:static-aff-crys-site}) is
  also $A$, so the derived de Rham cohomology is as ``bad'' as the
  non-completed derived crystalline cohomology. On the other hand, the
  Hodge-filtration allows us to recover the ``correct'' cohomology theory in
  characteristic $0$ after taking Hodge completion by
  {\cite[Rem~2.6]{Bhatt2012a}}.
\end{remark}

As a corollary, we deduce that the (usual) derived de Rham cohomology
$\tmop{dR}_{\mathbb{Z}/\mathbb{Z} [x]}$ is, as an
$\mathbb{E}_{\infty}$-$\mathbb{Z} [x]$-algebra, equivalent to the
PD-polynomial algebra $\Gamma_{\mathbb{Z}} (x)$. Bhatt showed an $p$-completed
version of this {\cite[Thm~3.27]{Bhatt2012a}}.

\begin{remark}
  In fact, our theory stems from the observation that the $p$-completed
  derived de Rham cohomology $(\tmop{dR}_{\mathbb{Z}/\mathbb{Z}
  [x]})_p^{\wedge}$ coincides with the $p$-completed PD-polynomial ring
  $\Gamma_{\mathbb{Z}} (x)_p^{\wedge}$, and the rationalization becomes
  $\mathbb{Q} [x]$.
\end{remark}

The virtue of our Hodge-filtered derived crystalline cohomology is that it
preserves small colimits. We will show that this implies several properties of
derived crystalline cohomology, such as ``Künneth formula'' and base change
property
(\Cref{cor:Hdg-fil-crys-coh-base-chg,cor:Hdg-fil-crys-coh-symm-mon,cor:Hdg-fil-crys-coh-transitive}).

\begin{remark}
  In a forthcoming work {\cite{Magidson}}, the Hodge-filtered derived de Rham
  cohomology admits a natural enrichment to derived PD-pairs,
  \Cref{rem:deriv-PD-pairs}, and the Hodge filtration is given by the
  PD-filtration of the derived PD-pair in question.
\end{remark}

\subsection{Derived de Rham cohomology}\label{subsec:deriv-dR}In this
subsection, we define the derived de Rham cohomology for maps of animated
PD-pairs. We need the definition of {\tmdfn{modules of
PD-differentials}}\footnote{It is about differentials preserving PD-structure,
rather than a module with a PD-structure.}.

\begin{definition}[{\cite[\href{https://stacks.math.columbia.edu/tag/07HQ}{Tag
07HQ}]{stacks-project}}]
  \label{def:pd-derivation}Let $(A, I, \gamma) \rightarrow (B, J, \delta)$ be
  a map of PD-pairs and $M$ an $B$-module. A {\tmdfn{PD $A$-derivation}} into
  $M$ is a map $\theta \of B \rightarrow M$ which is additive, $\theta (a) =
  0$ for $a \in A$, satisfies the Leibniz rule $\theta (bb') = b \theta (b') +
  b' \theta (b)$ and that
  \[ \theta (\delta_n (x)) = \delta_{n - 1} (x) \theta (x) \]
  for all $n \geq 1$ and $x \in J$.
  
  In this situation, there exists a {\tmdfn{universal PD $A$-derivation}}
  \[ \mathd_{(B, J) / (A, I)} \of B \rightarrow \Omega_{(B, J) / (A, I)}^1 \]
  such that for any PD $A$-derivation $\theta \of B \rightarrow M$, there
  exists a unique $B$-linear map $\xi \of \Omega_{(B, J) / (A, I)}^1
  \rightarrow M$ such that $\theta = \xi \circ \mathd_{(B, J) / (A, I)}$. We
  also call $\Omega_{(B, J) / (A, I)}^1$ the {\tmdfn{module of
  PD-differentials}}.
\end{definition}

\begin{remark}
  In \Cref{def:pd-derivation}, the PD-structure on $A$ is irrelevant. However,
  we will soon see that the derived version of module of PD-differentials does
  depend on the PD-structure on $A$.
\end{remark}

{\construction[{{\cite[\href{https://stacks.math.columbia.edu/tag/07HZ}{Tag
07HZ}]{stacks-project}}}]{\label{cons:PD-dR-cplx-cohom}Let $(A, I, \gamma) \rightarrow (B, J, \delta)$ be a
map of PD-pairs such that $\Omega_{(B, J) / (A, I)}^1$ is a flat
$B$-module\footnote{We assume the flatness only to avoid the appearance of the
ordinary tensor product $\otimes$ and the exterior power $\bigwedge$, since
for flat modules, these coincide with the derived versions. In fact, we only
need the very special case that $((A, I, \gamma) \rightarrow (B, J, \delta))
\in \tmop{dRCon}$ defined before \Cref{def:deriv-dR-pdpair}.}. The {\tmdfn{de
Rham complex}} $(\Omega_{(B, J) / (A, I)}^{\ast}, \mathd)$ is given by
$\Omega_{(B, J) / (A, I)}^i = \bigwedgestar_B^i \Omega_{(B, J) / (A, I)}^1$
and $\mathd \of \Omega_{(B, J) / (A, I)}^i \rightarrow \Omega_{(B, J) / (A,
I)}^{i + 1}$ is the unique $A$-linear map determined by
\[ \mathd (f_0 \mathd f_1 \wedge \cdots \wedge \mathd f_i) = \mathd f_0 \wedge
   \cdots \wedge \mathd f_i . \]
We recall that a {\tmdfn{commutative differential graded $A$-algebra}}
(abbrev. {\tmdfn{$A$-CDGA}}) is a commutative algebra object in the symmetric
monoidal abelian $1$\mbox{-}category $\tmop{Ch} (\tmop{Mod}_A)$ of chain
complexes\footnote{We identify cochain complexes $(K^{\ast}, \mathd)$ with
chain complexes $\left( K_{- \mathord{\ast}}, \mathd \right)$.} in static
$A$-modules for a ring $A$. Then any nonpositively graded $A$-CDGA gives rise
to an $\mathbb{E}_{\infty}$-$A$-algebra\footnote{This is in
{\cite[Notation~3.3.12]{Raksit2020}}, and we reproduce the argument as
follows. We can identify the heart $\tmop{DF} (A)^{\heartsuit}$ with respect
to the Beilinson $t$-structure (\Cref{prop:Beilinson-t-struct}) with the
abelian $1$\mbox{-}category $\tmop{Ch} (\tmop{Mod}_A)$. Furthermore, the fully
faithful embedding $\tmop{Ch} (\tmop{Mod}_A) \hookrightarrow \tmop{DF} (A)$ is
lax symmetric monoidal (\Cref{lem:heart-symm-mon}). Thus an $A$-CDGA gives
rise to an $\mathbb{E}_{\infty}$-algebra in $\tmop{DF} (A)$. The embedding
$\tmop{Ch} (\tmop{Mod}_A) \hookrightarrow \tmop{DF} (A)$ restricts to a lax
symmetric monoidal embedding $\tmop{Ch}_{\leq 0} (\tmop{Mod}_A) \rightarrow
\tmop{DF}^{\geq 0} (A)$. Thus a nonpositively graded $A$-CDGA gives rise to an
$\mathbb{E}_{\infty}$-algebra in $\tmop{DF}^{\geq 0} (A)$, which is mapped to
an $\mathbb{E}_{\infty}$-$A$-algebra by the symmetric monoidal functor
$\tmop{DF}^{\geq 0} (A) \rightarrow D (A)$.}, and in particular, the de Rham
complex constructed above gives rise to the {\tmdfn{de Rham cohomology}} of
$(A, I, \gamma) \rightarrow (B, J, \delta)$ as an
$\mathbb{E}_{\infty}$-$A$-algebra. Furthermore, the truncation map
$(\Omega_{(B, J) / (A, I)}^{\ast}, \mathd) \rightarrow \Omega_{(B, J) / (A,
I)}^0 = B$ is a map of CDGAs, where $B$ is concentrated in degree $0$. Passing
to the cohomology, we get a map of
$\mathbb{E}_{\infty}$-$\mathbb{Z}$-algebras, called the {\tmdfn{augmentation
map}} of the de Rham cohomology of $(A, I, \gamma) \rightarrow (B, J,
\delta)$.}}

\begin{remark}
  \label{rem:ch-cx-bnd-flat-symm-mon}When restricting to the full subcategory
  $\tmop{Ch}_{\gg - \infty} (\tmop{Mod}_A^{\flat}) \subseteq \tmop{Ch}
  (\tmop{Mod}_A)$ spanned by bounded below chain complexes of flat
  $A$-modules, the fully faithful embedding $\tmop{Ch}_{\gg - \infty}
  (\tmop{Mod}_A^{\flat}) \hookrightarrow \tmop{DF} (A)$ is in fact symmetric
  monoidal. We will refer to this later.
\end{remark}

\begin{remark}
  \label{rem:ch-cx-underlying-sp}The composite functor $\tmop{Ch}_{\leq 0}
  (\tmop{Mod}_A) \hookrightarrow \tmop{DF}^{\geq 0} (A) \rightarrow D_{\leq 0}
  (A)$ maps any complex to its underlying module spectrum.
\end{remark}

Now we define the derived de Rham cohomology for PD-pairs.

{\construction{\label{cons:dR-cohom-aug-dRCon0}By \Cref{cor:fun-cat-proj-gen},
the $\infty$\mbox{-}category $\tmop{dRCon} \assign \tmop{Fun} (\Delta^1,
\tmop{Pair}^{\gamma, \tmop{an}})$ (abbrev. for de Rham context) admits a set
of compact projective generators given by maps of PD-pairs of the form
$(\Gamma_{\mathbb{Z} [X]} (Y) \twoheadrightarrow \mathbb{Z} [X]) \rightarrow
(\Gamma_{\mathbb{Z} [X, X']} (Y, Y') \twoheadrightarrow \mathbb{Z} [X, X'])$
where each of $X, Y, X', Y'$ consists of a finite set (including empty) of
variables. These objects span a full subcategory $\tmop{dRCon}^0 \subseteq
\tmop{dRCon}$ stable under finite coproducts. Then it follows from
\Cref{prop:struct-n-proj-gen-cats} that there is an equivalence
$\mathcal{P}_{\Sigma} (\tmop{dRCon}^0) \rightarrow \tmop{dRCon}$ of
$\infty$\mbox{-}categories. The de Rham cohomology, equipped with the
augmentation map in Construction~\ref{cons:PD-dR-cplx-cohom}, restricts to a
functor $\tmop{dRCon}^0 \rightarrow \tmop{Fun} (\Delta^1,
\tmop{CAlg}_{\mathbb{Z}})$ where $\tmop{CAlg}_{\mathbb{Z}}$ is the
$\infty$\mbox{-}category of $\mathbb{E}_{\infty}$-$\mathbb{Z}$-algebras.}}

\begin{definition}
  \label{def:deriv-dR-pdpair}The {\tmdfn{derived de Rham cohomology functor}}
  $\tmop{dR}_{\cdummy / \cdummy} \of \tmop{dRCon} \rightarrow
  \tmop{CAlg}_{\mathbb{Z}}$ along with a canonical map $\tmop{dR}_{(B
  \twoheadrightarrow B'', \delta) / (A \twoheadrightarrow A'', \gamma)}
  \rightarrow B$ of functors $\tmop{dRCon} \rightrightarrows
  \tmop{CAlg}_{\mathbb{Z}}$ is defined to be the left derived functor
  (\Cref{prop:left-deriv-n-fun}) of the functor $\tmop{dRCon}^0 \rightarrow
  \tmop{Fun} (\Delta^1, \tmop{CAlg}_{\mathbb{Z}})$ in
  Construction~\ref{cons:dR-cohom-aug-dRCon0}. Given a map $(A
  \twoheadrightarrow A'', \gamma) \rightarrow (B \twoheadrightarrow B'',
  \delta)$ of animated PD-pairs, its {\tmdfn{derived de Rham cohomology}},
  i.e. the image under the derived de Rham cohomology functor, is denoted by
  $\tmop{dR}_{(B \twoheadrightarrow B'', \delta) / (A \twoheadrightarrow A'',
  \gamma)}$, or simply $\tmop{dR}_{(B \twoheadrightarrow B'') / (A
  \twoheadrightarrow A'')}$ when there is no ambiguity.
\end{definition}

We first explain that this is a generalization of classical derived de Rham
cohomology.

\begin{remark}
  We recall that the functor $\tmop{CAlg}^{\tmop{an}} \rightarrow
  \tmop{Pair}^{\gamma, \tmop{an}}, A \mapsto (\tmop{id}_A \of A \rightarrow A,
  0)$ is fully faithful (\Cref{lem:ani-ring-as-pdpair}), thus so is the
  induced functor $\tmop{Fun} (\Delta^1, \tmop{CAlg}^{\tmop{an}}) \rightarrow
  \tmop{Fun} (\Delta^1, \tmop{Pair}^{\gamma, \tmop{an}}) = \tmop{dRCon}$.
\end{remark}

\begin{lemma}
  \label{lem:dR-pdpair-classical}The composite functor $\tmop{Fun} (\Delta^1,
  \tmop{CAlg}^{\tmop{an}}) \rightarrow \tmop{dRCon}
  \xrightarrow{\tmop{dR}_{\cdummy / \cdummy}} \tmop{CAlg}_{\mathbb{Z}}, (A
  \rightarrow B) \mapsto \tmop{dR}_{(\tmop{id}_B \of B \rightarrow B, 0) /
  (\tmop{id}_A \of A \rightarrow A, 0)}$ is equivalent to the classical
  derived de Rham cohomology functor $(A \rightarrow B) \mapsto \tmop{dR}_{B /
  A}$.
\end{lemma}

\begin{proof}
  The crucial point is that, $\tmop{Fun} (\Delta^1, \tmop{CAlg}^{\tmop{an}})$
  is projectively generated for which $\{ (\mathbb{Z} [X] \rightarrow
  \mathbb{Z} [X, Y]) \}$ forms a set of compact projective generators, which
  follows from \Cref{cor:fun-cat-proj-gen,lem:ani-ring-as-pdpair}. The result
  then follows from \Cref{prop:left-deriv-n-fun} and the definition of the
  classical derived de Rham cohomology functor.
\end{proof}

We compute concretely the de Rham complex on $\tmop{dRCon}^0$. Fix an object
$(\Gamma_{\mathbb{Z} [X']} (Y') \twoheadrightarrow \mathbb{Z} [X'])
\rightarrow (\Gamma_{\mathbb{Z} [X, X']} (Y, Y') \twoheadrightarrow \mathbb{Z}
[X, X']) \in \tmop{dRCon}^0$, to simplify notations, we will write $A \assign
\Gamma_{\mathbb{Z} [X']} (Y'), A'' \assign \mathbb{Z} [X']$. Then this object
could be rewritten as $(A \twoheadrightarrow A'', \gamma) \rightarrow (B
\assign \Gamma_{A [X]} (Y) \twoheadrightarrow A'' [X], \tilde{\gamma})$ where
$X = (x_1, \ldots, x_m)$ and $Y = (y_1, \ldots, y_n)$ with the module of
PD-differentials $\Omega^1 = B \mathd x_1 \oplus \cdots \oplus B \mathd x_n
\oplus B \mathd y_1 \oplus \cdots \oplus B \mathd y_n$ and the universal
PD-derivation $B \rightarrow \Omega^1$ is determined by $\mathd (X^{\alpha}
\gamma_{\beta} (Y)) = \sum_{i = 1}^m \alpha_i x_1^{\alpha_1} \cdots
x_i^{\alpha_i - 1} \cdots x_m^{\alpha_m} \gamma_{\beta} (Y) \mathd x_i +
\sum_{j = 1}^n X^{\alpha} \gamma_{\beta_1} (y_1) \cdots \gamma_{\beta_j - 1}
(y_j) \cdots \gamma_{\beta_n} (y_n) \mathd y_j$ (with multi-index product).

As we mentioned earlier, derived de Rham cohomology is considered
uninteresting in characteristic $0$. Informally, the derived de Rham
cohomology $\tmop{dR}_{(B \twoheadrightarrow B'') / (A \twoheadrightarrow
A'')}$ is functorially equivalent to $A$ after rationalization. More
precisely, we will show that

\begin{lemma}
  \label{lem:dR-rational}There is a comparison map $A \rightarrow
  \tmop{dR}_{(B \twoheadrightarrow B'') / (A \twoheadrightarrow A'')}$,
  functorial in $((A \twoheadrightarrow A'') \rightarrow (B \twoheadrightarrow
  B'')) \in \tmop{dRCon}$, as a natural transformation of two functors
  $\tmop{dRCon} \rightrightarrows \tmop{CAlg}_{\mathbb{Z}}$. This natural
  transformation becomes an equivalence after composing with the
  rationalization $(-) \otimes_{\mathbb{Z}}^{\mathbb{L}} \mathbb{Q} \of
  \tmop{CAlg}_{\mathbb{Z}} \rightarrow \tmop{CAlg}_{\mathbb{Q}}$.
\end{lemma}

\begin{proof}
  We first construct the comparison map in question. We have the composite of
  forgetful functors $\tmop{Pair}^{\gamma, \tmop{an}} \rightarrow
  \tmop{Pair}^{\tmop{an}} \rightarrow \tmop{CAlg}^{\tmop{an}} \rightarrow
  \tmop{CAlg}_{\mathbb{Z}}, (A \twoheadrightarrow A'', \gamma) \mapsto A$.
  Further composing with the evaluation map $\tmop{dRCon} \rightarrow
  \tmop{Pair}^{\gamma, \tmop{an}}$ at $[0] \in \Delta^1$, we get a functor
  $\tmop{dRCon} \rightarrow \tmop{CAlg}_{\mathbb{Z}}, ((A \twoheadrightarrow
  A'', \gamma) \rightarrow (B \twoheadrightarrow B'', \delta)) \mapsto A$. We
  restrict this functor to $\tmop{dRCon}^0$, getting a functor $\tmop{dRCon}^0
  \rightarrow \tmop{CAlg}_{\mathbb{Z}}$, which coincides with the composite
  functor $\tmop{dRCon}^0 \rightarrow \tmop{Ring} = \tmop{CAlg} (\tmop{Ab})
  \hookrightarrow \tmop{CAlg} (\tmop{Ch}_{\leq 0} (\tmop{Ab})) \rightarrow
  \tmop{CAlg}_{\mathbb{Z}}$ given by the ``same'' formula $((A
  \twoheadrightarrow A'', \gamma) \rightarrow (B \twoheadrightarrow B'',
  \delta)) \mapsto A$. Note that there is a canonical map of functors from
  $\tmop{dRCon}^0 \rightarrow \tmop{Ring} \rightarrow \tmop{CAlg}
  (\tmop{Ch}_{\leq 0} (\tmop{Ab}))$ to the de Rham complex functor
  $\tmop{dRCon}^0 \rightarrow \tmop{CAlg} (\tmop{Ch}_{\leq 0} (\tmop{Ab}))$,
  which is given by the $A$-CDGA structure on the de Rham complex. Now
  \Cref{prop:left-deriv-n-fun} gives us a comparison map of the left derived
  functors $\tmop{dRCon} \rightrightarrows \tmop{CAlg}_{\mathbb{Z}}$.
  
  It remains to see that this comparison map is an equivalence after
  rationalization. First, we note that the rationalization
  $\tmop{CAlg}_{\mathbb{Z}} \rightarrow \tmop{CAlg}_{\mathbb{Q}}$ preserves
  small colimits, and in particular, filtered colimits and geometric
  realizations, it follows from \Cref{prop:left-deriv-n-fun} that both
  functors are still left derived functors after rationalization, therefore it
  suffices to check the equivalence on $\tmop{dRCon}^0$. The Poincaré lemma
  imply that the comparison map of functors $\tmop{dRCon}^0 \rightrightarrows
  \tmop{CAlg} (\tmop{Ch}_{\leq 0} (\tmop{Ab}))$ becomes a homotopy equivalence
  after composing with $\tmop{CAlg} (\tmop{Ch}_{\leq 0} (\tmop{Ab}))
  \rightarrow \tmop{CAlg} (\tmop{Ch}_{\leq 0} (\tmop{Mod}_{\mathbb{Q}}))$,
  which implies that it becomes an equivalence after composing with
  $\tmop{CAlg} (\tmop{Ch}_{\leq 0} (\tmop{Ab})) \rightarrow
  \tmop{CAlg}_{\mathbb{Q}}$ by \Cref{rem:ch-cx-underlying-sp}.
\end{proof}

Another consequence of this computation is that the de Rham cohomology functor
$\tmop{dRCon}^0 \rightarrow \tmop{CAlg}_{\mathbb{Z}}$ preserves finite
coproducts, which follows from the fact that the de Rham cohomology functor
$\tmop{dRCon}^0 \rightarrow \tmop{CAlg} (\tmop{Ch}_{\gg - \infty}
(\tmop{FreeAb}))$ preserves finite coproducts, and that the composite functor
$\tmop{Ch}_{\gg - \infty} (\tmop{FreeAb}) \hookrightarrow \tmop{DF}^{\geq 0}
(\mathbb{Z}) \rightarrow D (\mathbb{Z})$ is symmetric monoidal, cf.
\Cref{rem:ch-cx-bnd-flat-symm-mon}. By \Cref{prop:left-deriv-n-fun}, we have

\begin{lemma}
  \label{lem:dR-preserv-colim}The derived de Rham cohomology functor
  $\tmop{dRCon} \rightarrow \tmop{CAlg}_{\mathbb{Z}}$ preserves small
  colimits.
\end{lemma}

Now we show that the derived de Rham cohomology associated to the map $(A
\twoheadrightarrow A'', \gamma) \rightarrow (B \twoheadrightarrow B'',
\delta)$ does not depend on $B$, and define the derived crystalline
cohomology. To formally define the $\infty$-category of animated PD-pairs $(A
\twoheadrightarrow A'', \gamma)$ along with a map $A'' \rightarrow R$ of
animated rings, we need the concept of comma categories in
\Cref{subsec:comma-cat}. Consider the comma category $\tmop{CrysCon} \assign
\tmop{Pair}^{\gamma, \tmop{an}} \times_{\tmop{CAlg}^{\tmop{an}}} \tmop{Fun}
(\Delta^1, \tmop{CAlg}^{\tmop{an}})$ (abbrev. for crystalline context) where
the functor $\tmop{Pair}^{\gamma, \tmop{an}} \rightarrow
\tmop{CAlg}^{\tmop{an}}$ is the composite functor $\tmop{Pair}^{\gamma,
\tmop{an}} \rightarrow \tmop{Pair}^{\tmop{an}} \rightarrow
\tmop{CAlg}^{\tmop{an}}, (A \twoheadrightarrow A'', \gamma) \mapsto A''$ and
the functor $\tmop{Fun} (\Delta^1, \tmop{CAlg}^{\tmop{an}}) \rightarrow
\tmop{CAlg}^{\tmop{an}}$ is the evaluation $(A'' \rightarrow R) \mapsto A''$
at $0 \in \Delta^1$. It follows from \Cref{cor:comma-proj-gen} that
$\tmop{CrysCon}$ admits a set of compact projective generators of the form
$((\Gamma_{\mathbb{Z} [X]} (Y) \twoheadrightarrow \mathbb{Z} [X], \gamma),
\mathbb{Z} [X] \rightarrow \mathbb{Z} [X, Z])$ where each of $X, Y, Z$
consists of a finite set of variables, which spans a full subcategory
$\tmop{CrysCon}^0 \subseteq \tmop{CrysCon}$ stable under finite coproducts.

{\construction{\label{cons:canon-dRCon-CrysCon}We note that there is a
canonical functor $\tmop{dRCon} \rightarrow \tmop{CrysCon}$ induced by the
evaluation $\tmop{dRCon} = \tmop{Fun} (\Delta^1, \tmop{Pair}^{\gamma,
\tmop{an}}) \xrightarrow{\tmop{ev}_{[0]}} \tmop{Pair}^{\gamma, \tmop{an}}$ and
the functor $\tmop{dRCon} \rightarrow \tmop{Fun} (\Delta^1,
\tmop{CAlg}^{\tmop{an}})$ which is itself induced by the composite of the
forgetful functors $\tmop{Pair}^{\gamma, \tmop{an}} \rightarrow
\tmop{Pair}^{\tmop{an}} \rightarrow \tmop{CAlg}^{\tmop{an}}, (A
\twoheadrightarrow A'', \gamma) \mapsto A''$. Concretely, the functor
$\tmop{dRCon} \rightarrow \tmop{CrysCon}$ is given by $((A \twoheadrightarrow
A'', \gamma) \rightarrow (B \twoheadrightarrow B'', \delta)) \mapsto ((A
\twoheadrightarrow A'', \gamma), A'' \rightarrow B'')$. Since both functors
preserves small colimits (we have used \Cref{prop:forget-PD-small-colim}), we
deduce that}}

\begin{lemma}
  The functor $\tmop{dRCon} \rightarrow \tmop{CrysCon}$ in
  Construction~\ref{cons:canon-dRCon-CrysCon} preserves small colimits.
\end{lemma}

It follows from \Cref{prop:left-deriv-n-fun} that $\tmop{dRCon} \rightarrow
\tmop{CrysCon}$ is the left derived functor of the composite functor
$\tmop{dRCon}^0 \rightarrow \tmop{CrysCon}^0 \rightarrow \tmop{CrysCon},
((\Gamma_{\mathbb{Z} [X]} (Y) \twoheadrightarrow \mathbb{Z} [X]) \rightarrow
(\Gamma_{\mathbb{Z} [X, X']} (Y, Y') \twoheadrightarrow \mathbb{Z} [X, X']))
\mapsto ((\Gamma_{\mathbb{Z} [X]} (Y) \twoheadrightarrow \mathbb{Z} [X]),
\mathbb{Z} [X] \rightarrow \mathbb{Z} [X, X'])$. It then follows from
\Cref{cor:nonab-deriv-cat-adjoint-fun} that

\begin{lemma}
  \label{lem:dRCon-CrysCon-adjunct}The functor $\tmop{dRCon} \rightarrow
  \tmop{CrysCon}$ in Construction~\ref{cons:canon-dRCon-CrysCon} admits a
  right adjoint $\tmop{CrysCon} \rightarrow \tmop{dRCon}$ which preserves
  sifted colimits.
\end{lemma}

One can verify (see also \Cref{lem:ani-ring-as-pdpair}) that

\begin{lemma}
  \label{lem:RA-CrysCon-dRCon}The right adjoint $\tmop{CrysCon} \rightarrow
  \tmop{dRCon}$ is concretely given by $((A \twoheadrightarrow A'', \gamma),
  A'' \rightarrow R) \mapsto ((A \twoheadrightarrow A'', \gamma) \rightarrow
  (\tmop{id}_R \of R \rightarrow R, 0))$.
\end{lemma}

The independence of $\tmop{dR}_{(B \twoheadrightarrow B'') / (A
\twoheadrightarrow A'')}$ with respect to $B$ is formally formulated as
follows:

{\construction{\label{cons:dR-crys-comp}The counit map of the adjunction
$\tmop{dRCon} \rightleftarrows \tmop{CrysCon}$ in
\Cref{lem:dRCon-CrysCon-adjunct} is an equivalence by
\Cref{lem:RA-CrysCon-dRCon}, therefore the functor $\tmop{CrysCon} \rightarrow
\tmop{dRCon}$ is fully faithful. The unit map between functors $\tmop{dRCon}
\rightrightarrows \tmop{dRCon}$ is concretely given by $((A \twoheadrightarrow
A'', \gamma) \rightarrow (B \twoheadrightarrow B'', \delta)) \rightarrow ((A
\twoheadrightarrow A'', \gamma) \rightarrow (\tmop{id}_{B''} \of B''
\rightarrow B'', 0))$. Applying the derived de Rham functor $\tmop{dRCon}
\rightarrow \tmop{CAlg}_{\mathbb{Z}}$, we get the comparison map
$\tmop{dR}_{(B \twoheadrightarrow B'', \delta) / (A \twoheadrightarrow A'',
\gamma)} \rightarrow \tmop{dR}_{(\tmop{id}_{B''}, 0) / (A \rightarrow A'',
\gamma)}$ which is functorial in $((A \twoheadrightarrow A'', \gamma)
\rightarrow (B \twoheadrightarrow B'', \delta)) \in \tmop{dRCon}$, and the
comparison map is viewed as a natural transformation between two functors
$\tmop{dRCon} \rightrightarrows \tmop{CAlg}_{\mathbb{Z}}$.}}

\begin{proposition}
  \label{prop:dR-crys-inv}The natural transformation in
  Construction~\ref{cons:dR-crys-comp}. In other words, the derived de Rham
  cohomology functor $\tmop{dRCon} \rightarrow \tmop{CAlg}_{\mathbb{Z}}$ is
  $(\tmop{dRCon} \rightarrow \tmop{CrysCon})$-invariant (\Cref{def:L-inv}).
\end{proposition}

\begin{proof}
  Both functors preserves sifted colimits, so by \Cref{prop:left-deriv-n-fun},
  it suffices to establish the equivalence for the full subcategory
  $\tmop{dRCon}^0 \subseteq \tmop{dRCon}$. For every $(\Gamma_{\mathbb{Z}
  [X']} (Y') \twoheadrightarrow \mathbb{Z} [X']) \rightarrow
  (\Gamma_{\mathbb{Z} [X, X']} (Y, Y') \twoheadrightarrow \mathbb{Z} [X, X'])
  \in \tmop{dRCon}^0$ simply denoted by $((A \twoheadrightarrow A'', \gamma)
  \rightarrow (\Gamma_{A [X]} (Y) \twoheadrightarrow A'' [X], \gamma))$, we
  need to show that the map
  \[ \tmop{dR}_{(\Gamma_{A [X]} (Y) \twoheadrightarrow A'' [X]) / (A
     \twoheadrightarrow A'')} \rightarrow \tmop{dR}_{(\tmop{id}_{A'' [X]} \of
     A'' [X] \rightarrow A'' [X], 0) / (A \twoheadrightarrow A'')} \]
  is an equivalence. Note that the constructed map
  \[ (\Gamma_{A [X]} (Y) \twoheadrightarrow A'' [X], \gamma) \rightarrow
     (\tmop{id}_{A'' [X]} \of A'' [X] \rightarrow A'' [X], 0) \]
  in $\tmop{Pair}^{\gamma, \tmop{an}}_{/ (A \twoheadrightarrow A'', \gamma)}$
  factors as
  \[ (\Gamma_{A [X]} (Y) \twoheadrightarrow A'' [X], \gamma)
     \xrightarrow{\alpha} (A [X] \twoheadrightarrow A'' [X], \gamma)
     \xrightarrow{\beta} (\tmop{id}_{A'' [X]} \of A'' [X] \rightarrow A'' [X],
     0) \]
  Thus it suffices to show that both maps $\alpha$ and $\beta$ induces
  equivalences after passing to the functor $\tmop{dR}_{\cdummy / (A
  \twoheadrightarrow A'', \gamma)} \of \tmop{Pair}^{\gamma, \tmop{an}}_{/ (A
  \twoheadrightarrow A'', \gamma)} \rightarrow \tmop{CAlg}_{\mathbb{Z}}$. Note
  that $(A [X] \twoheadrightarrow A'' [X], \gamma) \in \tmop{dRCon}^0$,
  $\tmop{dR}_{\alpha / (A \twoheadrightarrow A'', \gamma)}$ could be computed
  by de Rham complexes, which corresponds a homotopy equivalence of de Rham
  complexes by the divided power Poincaré's lemma
  {\cite[\href{https://stacks.math.columbia.edu/tag/07LC}{Tag
  07LC}]{stacks-project}}.
  
  It remains to show that $\tmop{dR}_{\beta / (A \twoheadrightarrow A'',
  \gamma)}$ is also an equivalence. For this, we need to resolve
  $(\tmop{id}_{A'' [X]} \of A'' [X] \rightarrow A'' [X], 0)$ simplicially
  under $(A [X] \twoheadrightarrow A'' [X], \gamma)$. Recall that $A =
  \Gamma_{\mathbb{Z} [X']} (Y')$ and $A'' =\mathbb{Z} [X']$. The key point is
  that we can resolve $A''$ simplicially by divided power polynomial
  $A$-algebras, in the same way as resolving $\mathbb{Z}$ simplicially by
  polynomial $\mathbb{Z} [t]$-algebras, which essentially follows from a bar
  construction of $\mathbb{N}$, see {\cite[Rem~3.31]{Bhatt2012a}}. For every
  divided power polynomial $A$-algebra $\Gamma_A (Z)$, $(\Gamma_{A [X]} (Z)
  \twoheadrightarrow A'' [X], \gamma)$ belongs to $\tmop{dRCon}^0$, and the
  map $\tmop{dR}_{(A [X] \twoheadrightarrow A'' [X]) / (A \twoheadrightarrow
  A'')} \rightarrow \tmop{dR}_{(\Gamma_{A [X]} (Z) \twoheadrightarrow A'' [X])
  / (A \twoheadrightarrow A'')}$ (functorial in $\Gamma_A (Z)$) is an
  equivalence again by the divided power Poincaré's lemma
  {\cite[\href{https://stacks.math.columbia.edu/tag/07LC}{Tag
  07LC}]{stacks-project}}. It follows that $\tmop{dR}_{\beta / (A
  \twoheadrightarrow A'', \gamma)}$ is indeed an equivalence.
\end{proof}

In view of \Cref{prop:L-inv}, we define the derived crystalline cohomology
functor which corresponds to the $(\tmop{dRCon} \rightarrow
\tmop{CrysCon})$-invariant functor $\tmop{dR}_{\cdummy / \cdummy}$:

\begin{definition}
  \label{def:deriv-crys-cohomol}The {\tmdfn{derived crystalline cohomology
  functor}} $\tmop{CrysCoh} \of \tmop{CrysCon} \rightarrow
  \tmop{CAlg}_{\mathbb{Z}}$ is defined to be the composite $\tmop{CrysCon}
  \rightarrow \tmop{dRCon} \xrightarrow{\tmop{dR}_{\cdummy / \cdummy}}
  \tmop{CAlg}_{\mathbb{Z}}$.
\end{definition}

\begin{notation}
  \label{nota:deriv-crys-coh}We will denote the derived crystalline cohomology
  of $((A \twoheadrightarrow A'', \gamma), A'' \rightarrow R) \in
  \tmop{CrysCon}$ by $\tmop{CrysCoh}_{R / (A \twoheadrightarrow A'', \gamma)}$
  (or $\tmop{CrysCoh}_{R / (A \twoheadrightarrow A'')}$ even
  $\tmop{CrysCon}_{R / A}$ when there is no ambiguity).
\end{notation}

Now we show that

\begin{proposition}
  \label{prop:deriv-crys-coh-preserv-colim}The derived crystalline cohomology
  functor $\tmop{CrysCon} \rightarrow \tmop{CAlg}_{\mathbb{Z}}$ preserves
  small colimits.
\end{proposition}

\begin{proof}
  The functor $\tmop{CrysCon} \rightarrow \tmop{dRCon}$ preserves sifted
  colimits and $\tmop{dR}_{\cdummy / \cdummy} \of \tmop{dRCon} \rightarrow
  \tmop{CAlg}_{\mathbb{Z}}$ preserves small colimits, it follows that the
  derived crystalline cohomology functor $\tmop{CrysCoh}$ preserves sifted
  colimits. By \Cref{prop:left-deriv-n-fun}, it remains to show that
  $\nobracket \tmop{CrysCoh} |_{\tmop{CrysCon}^0}$ preserves finite
  coproducts. The point is that every $(\Gamma_{\mathbb{Z} [X']} (Y')
  \twoheadrightarrow \mathbb{Z} [X'], \mathbb{Z} [X'] \rightarrow \mathbb{Z}
  [X, X']) \in \tmop{CrysCon}^0$ lifts to $(\Gamma_{\mathbb{Z} [X']} (Y')
  \twoheadrightarrow \mathbb{Z} [X']) \rightarrow (\Gamma_{\mathbb{Z} [X, X']}
  (Y, Y') \twoheadrightarrow \mathbb{Z} [X, X']) \in \tmop{dRCon}^0$, the
  functor $\tmop{dRCon}^0 \rightarrow \tmop{CrysCon}^0$ preserves finite
  coproducts, and the functor $\tmop{dR}_{\cdummy / \cdummy}$ preserves finite
  coproducts.
\end{proof}

Now we apply the discussions in \Cref{subsec:comma-cat} to deduce some formal
properties. First, by \Cref{rem:comma-base-chg}, we have

\begin{corollary}
  \label{cor:crys-coh-base-chg}The derived crystalline cohomology is
  compatible with base change. More precisely, let $((A \twoheadrightarrow
  A'', \gamma_A), A'' \rightarrow R) \in \tmop{CrysCon}$ and let $(A
  \twoheadrightarrow A'', \gamma_A) \rightarrow (B \twoheadrightarrow B'',
  \gamma_B)$ be a map of animated PD-pairs. Then the canonical map
  \[ \tmop{CrysCoh}_{R / (A \twoheadrightarrow A'', \gamma_A)}
     \otimes_A^{\mathbb{L}} B \longrightarrow \tmop{CrysCoh}_{(R
     \otimes_{A''}^{\mathbb{L}} B'') / (B \twoheadrightarrow B'', \gamma_B)}
  \]
  is an equivalence.
\end{corollary}

\begin{remark}
  Let $R$ be a (finitely generated) polynomial $\mathbb{F}_p$-algebra. Then by
  the $p$-completed derived crystalline cohomology of $R$ with respect to the
  PD-pair $(\mathbb{Z}_p, (p))$ is equivalent to the usual ($p$-completed)
  crystalline cohomology of $R$ with respect to $\mathbb{Z}_p$. In fact, later
  (\Cref{prop:site-comp-qsyn-integral,prop:site-static-comp}), we will show
  that the same $p$-completed comparison holds for quasisyntomic
  $\mathbb{F}_p$-algebras $R$ (and in particular, for smooth
  $\mathbb{F}_p$-algebras). However, our non-completed derived crystalline
  cohomology is not necessarily $p$-complete.
\end{remark}

Next, by \Cref{rem:comma-colim-fix-base}, we have

\begin{corollary}
  \label{cor:crys-coh-symm-mon}The derived crystalline cohomology is symmetric
  monoidal. More precisely, let $(A \twoheadrightarrow A'', \gamma_A) \in
  \tmop{Pair}^{\gamma, \tmop{an}}$ and let $A \rightarrow R$, $A \rightarrow
  S$ be two maps of animated rings. Then the canonical map
  \[ \tmop{CrysCoh}_{R / (A \twoheadrightarrow A'')} \otimes_A^{\mathbb{L}}
     \tmop{CrysCoh}_{S / (A \twoheadrightarrow A'')} \longrightarrow
     \tmop{CrysCoh}_{(R \otimes_{A''}^{\mathbb{L}} S) / (A \twoheadrightarrow
     A'')} \]
  is an equivalence.
\end{corollary}

Finally, by \Cref{rem:comma-transtive}, we have

\begin{corollary}
  \label{cor:crys-coh-transitive}The derived crystalline cohomology is
  transitive. More precisely, let $(A \twoheadrightarrow A'', \gamma_A)
  \rightarrow (B \twoheadrightarrow B'', \gamma_B)$ be a map of animated
  PD-pairs, and let $B'' \rightarrow R$ be a map of animated rings. Then the
  canonical map
  \[ \tmop{CrysCoh}_{R / (A \twoheadrightarrow A'')}
     \otimes_{\tmop{CrysCoh}_{B'' / (A \twoheadrightarrow A'')}}^{\mathbb{L}}
     B \longrightarrow \tmop{CrysCoh}_{R / (B \twoheadrightarrow B'')} \]
  is an equivalence, where the map $\tmop{CrysCoh}_{B'' / (A
  \twoheadrightarrow A'')} \rightarrow B$ is $\tmop{CrysCoh}_{B'' / (A
  \twoheadrightarrow A'')} \rightarrow \tmop{CrysCoh}_{B'' / (B
  \twoheadrightarrow B'')} \simeq B$.
\end{corollary}

\begin{remark}
  In particular, if we take $(A \twoheadrightarrow A'', \gamma_A) =
  (\mathbb{Z}, 0, 0)$ in \Cref{cor:crys-coh-transitive}, we see that, fix an
  animated PD-pair $(B \twoheadrightarrow B'', \gamma_B)$, any derived
  crystalline cohomology $\tmop{CrysCoh}_{R / (B \twoheadrightarrow B'')}$ is
  completely determined by the derived de Rham cohomology $\tmop{dR}_{R
  /\mathbb{Z}}$. However, without the theory of derived crystalline
  cohomology, we do not know how to construct the map $\tmop{dR}_{B''
  /\mathbb{Z}} \rightarrow B$ in terms of the PD-structure on $B
  \twoheadrightarrow B''$.
\end{remark}

\subsection{Filtrations}\label{subsec:dR-Hdg-conj-fil}In this subsection, we
will define the {\tmdfn{Hodge filtration}} on the derived de Rham cohomology
and show that most of our previous discussions are compatible with the Hodge
filtration. Furthermore, in characteristic $p$, we will define the
{\tmdfn{conjugate filtration}}, which is of technical importance to control
the cohomology. We start with the definition of the Hodge filtration.

\begin{definition}[cf.~{\cite[§6.13]{Berthelot1978}}]
  Let $(A, I, \gamma) \rightarrow (B, J, \delta)$ be a map of PD-pairs such
  that $\Omega_{(B, J) / (A, I)}^1$ is a flat $B$-module. The {\tmdfn{Hodge
  filtration}} $\tmop{Fil}_H^{\ast}$ on the de Rham complex $(\Omega_{(B, J) /
  (A, I)}^{\ast}, \mathd)$ is given by the differential graded ideals
  $\tmop{Fil}_H^m \Omega_{(B, J) / (A, I)}^{\ast} \assign J^{\left[ m -
  \mathord{\ast} \right]} \Omega_{(B, J) / (A, I)}^{\ast} \subseteq
  \Omega_{(B, J) / (A, I)}^{\ast}$.
\end{definition}

{\construction{\label{cons:Hdg-fil-dR-cohom-dRCon0}As CDGAs give rise to
$\mathbb{E}_{\infty}$-$\mathbb{Z}$-algebras, (nonnegatively) filtered CDGAs
give rise to (nonnegatively) filtered
$\mathbb{E}_{\infty}$-$\mathbb{Z}$-algebras. Moreover, the truncation map
$(\Omega_{(B, J) / (A, I)}^{\ast}, \mathd) \rightarrow B$ is a map of filtered
CDGAs, which gives rise to a map of filtered
$\mathbb{E}_{\infty}$-$\mathbb{Z}$-algebras. Thus we get a functor
$\tmop{dRCon}^0 \rightarrow \tmop{Fun} (\Delta^1, \tmop{CAlg} (\tmop{DF}^{\geq
0} (\mathbb{Z})))$.}}

\begin{definition}
  \label{def:Hdg-fil-deriv-dR}The {\tmdfn{Hodge-filtered derived de Rham
  cohomology functor}} $\tmop{Fil}_H^{\ast} \tmop{dR}_{\cdummy / \cdummy} \of
  \tmop{dRCon} \rightarrow \tmop{CAlg} (\tmop{DF}^{\geq 0} (\mathbb{Z}))$
  together with a canonical map $\tmop{Fil}_H^{\ast} \tmop{dR}_{(B
  \twoheadrightarrow B'') / (A \twoheadrightarrow A'')} \rightarrow
  \tmop{Fil}_{\tmop{PD}}^{\ast} B$ is defined to be the left derived functor
  (\Cref{prop:left-deriv-n-fun}) of the functor $\tmop{dRCon}^0 \rightarrow
  \tmop{Fun} (\Delta^1, \tmop{CAlg} (\tmop{DF}^{\geq 0} (\mathbb{Z})))$ in
  Construction~\ref{cons:Hdg-fil-dR-cohom-dRCon0}, where
  $\tmop{Fil}_{\tmop{PD}}^{\ast} B$ is an abbreviation of
  $\tmop{Fil}_{\tmop{PD}}^{\ast} (B \twoheadrightarrow B'', \gamma)$
  (\Cref{def:pd-fil}).
\end{definition}

Most of properties in \Cref{subsec:deriv-dR} hold with a similar proof:

\begin{lemma}
  The composite functor
  \[ \tmop{Fun} (\Delta^1, \tmop{CAlg}^{\tmop{an}}) \rightarrow \tmop{dRCon}
     \rightarrow \tmop{CAlg} (\tmop{DF}^{\geq 0} (\mathbb{Z})), (A \rightarrow
     B) \mapsto \tmop{Fil}_H^{\ast} \tmop{dR}_{(\tmop{id}_B \of B \rightarrow
     B, 0) / (\tmop{id}_A \of A \rightarrow A, 0)} \]
  is equivalent to the classical Hodge-filtered derived de Rham cohomology
  functor $(A \rightarrow B) \mapsto \tmop{Fil}_H^{\ast} \tmop{dR}_{B / A}$.
\end{lemma}

\begin{lemma}
  The map in \Cref{lem:dR-rational} admits a natural enrichment, that is to
  say, a map $\tmop{Fil}_{\tmop{PD}}^{\ast} A \rightarrow \tmop{Fil}_H^{\ast}
  \tmop{dR}_{(B \twoheadrightarrow B'') / (A \twoheadrightarrow A'')}$ of
  functors $\tmop{dRCon} \rightrightarrows \tmop{CAlg} (\tmop{DF}^{\geq 0}
  (\mathbb{Z}))$.\footnote{Corrected thanks to a message from Lenny
  {\tmname{Taelman}}.}
\end{lemma}

\begin{lemma}
  The Hodge-filtered derived de Rham cohomology functor $\tmop{dRCon}
  \rightarrow \tmop{CAlg} (\tmop{DF}^{\geq 0} (\mathbb{Z}))$ preserves small
  colimits.
\end{lemma}

This allows us to define the Hodge-filtration on the derived crystalline
cohomology, due to the following proposition, which follows from the proof of
\Cref{prop:dR-crys-inv} by replacing the Poincaré lemma by the filtered
Poincaré lemma, cf.~{\cite[Thm~6.13]{Berthelot1978}}:

\begin{proposition}
  The map $\tmop{Fil}_H^{\ast} \tmop{dR}_{(B \twoheadrightarrow B'', \delta) /
  (A \twoheadrightarrow A'', \gamma)} \rightarrow \tmop{Fil}_H^{\ast}
  \tmop{dR}_{(\tmop{id}_{B''}, 0) / (A \rightarrow A'', \gamma)}$ of functors
  $\tmop{dRCon} \rightrightarrows \tmop{CAlg} (\tmop{DF}^{\geq 0}
  (\mathbb{Z}))$ induced by the counit map associated to $((A
  \twoheadrightarrow A'', \gamma) \rightarrow (B \twoheadrightarrow B'',
  \delta)) \in \tmop{dRCon}$ is an equivalence. In other words, the
  Hodge-filtered de Rham cohomology functor $\tmop{dRCon} \rightarrow
  \tmop{CAlg} (\tmop{DF}^{\geq 0} (\mathbb{Z}))$ is $(\tmop{dRCon} \rightarrow
  \tmop{CrysCon})$-invariant (\Cref{def:L-inv}).
\end{proposition}

\begin{definition}
  The {\tmdfn{Hodge-filtered derived crystalline cohomology functor}}
  $\tmop{Fil}_H^{\ast} \tmop{CrysCoh} \of \tmop{CrysCon} \rightarrow
  \tmop{CAlg} (\tmop{DF}^{\geq 0} (\mathbb{Z}))$ is defined to be the
  composite $\tmop{CrysCon} \rightarrow \tmop{dRCon}
  \xrightarrow{\tmop{Fil}_H^{\ast} \tmop{dR}_{\cdummy / \cdummy}} \tmop{CAlg}
  (\tmop{DF}^{\geq 0} (\mathbb{Z}))$.
\end{definition}

\begin{proposition}
  The Hodge-filtered derived crystalline cohomology functor $\tmop{CrysCon}
  \rightarrow \tmop{CAlg} (\tmop{DF}^{\geq 0} (\mathbb{Z}))$ preserves small
  colimits.
\end{proposition}

Similar to
\Cref{cor:crys-coh-base-chg,cor:crys-coh-symm-mon,cor:crys-coh-transitive}, we
have

\begin{corollary}
  \label{cor:Hdg-fil-crys-coh-base-chg}The Hodge-filtered derived crystalline
  cohomology is compatible with base change. More precisely, let $((A
  \twoheadrightarrow A'', \gamma_A), A'' \rightarrow R) \in \tmop{CrysCon}$
  and let $(A \twoheadrightarrow A'', \gamma_A) \rightarrow (B
  \twoheadrightarrow B'', \gamma_B)$ be a map of animated PD-pairs. Then the
  canonical map
  \[ \tmop{Fil}_H \tmop{CrysCoh}_{R / (A \twoheadrightarrow A'', \gamma_A)}
     \otimes_{\tmop{Fil}_{\tmop{PD}} A}^{\mathbb{L}} \tmop{Fil}_{\tmop{PD}} B
     \longrightarrow \tmop{Fil}_H \tmop{CrysCoh}_{(R
     \otimes_{A''}^{\mathbb{L}} B'') / (B \twoheadrightarrow B'', \gamma_B)}
  \]
  is an equivalence.
\end{corollary}

\begin{corollary}
  \label{cor:Hdg-fil-crys-coh-symm-mon}The derived crystalline cohomology is
  symmetric monoidal. More precisely, let $(A \twoheadrightarrow A'',
  \gamma_A) \in \tmop{Pair}^{\gamma, \tmop{an}}$ and let $A \rightarrow R$, $A
  \rightarrow S$ be two maps of animated rings. Then the canonical map
  \[ \tmop{Fil}_H \tmop{CrysCoh}_{R / (A \twoheadrightarrow A'')}
     \otimes^{\mathbb{L}} \tmop{Fil}_H \tmop{CrysCoh}_{S / (A
     \twoheadrightarrow A'')} \rightarrow \tmop{Fil}_H \tmop{CrysCoh}_{(R
     \otimes_{A''}^{\mathbb{L}} S) / (A \twoheadrightarrow A'')} \]
  is an equivalence, where the tensor product on the left is relative to
  $\tmop{Fil}_{\tmop{PD}} A$.
\end{corollary}

\begin{corollary}
  \label{cor:Hdg-fil-crys-coh-transitive}The derived crystalline cohomology is
  transitive. More precisely, let $(A \twoheadrightarrow A'', \gamma_A)
  \rightarrow (B \twoheadrightarrow B'', \gamma_B)$ be a map of animated
  PD-pairs, and let $B'' \rightarrow R$ be a map of animated rings. Then the
  canonical map
  \[ \tmop{Fil}_H \tmop{CrysCoh}_{R / (A \twoheadrightarrow A'')}
     \otimes_{\tmop{Fil}_H \tmop{CrysCoh}_{B'' / (A \twoheadrightarrow
     A'')}}^{\mathbb{L}} \tmop{Fil}_{\tmop{PD}} B \longrightarrow \tmop{Fil}_H
     \tmop{CrysCoh}_{R / (B \twoheadrightarrow B'')} \]
  is an equivalence, where the map $\tmop{Fil}_H \tmop{CrysCoh}_{B'' / (A
  \twoheadrightarrow A'')} \rightarrow B$ is equivalent to the map
  \[ \tmop{Fil}_H \tmop{CrysCoh}_{B'' / (A \twoheadrightarrow A'')}
     \rightarrow \tmop{Fil}_H \tmop{CrysCoh}_{B'' / (B \twoheadrightarrow
     B'')} \simeq \tmop{Fil}_{\tmop{PD}} B. \]
\end{corollary}

Now we come to the characteristic $p > 0$ case. We start with an analysis of
the Frobenius map on an animated PD-$\mathbb{F}_p$-pair. Let $(A, I, \gamma)
\in \tmop{Pair}^{\gamma, \tmop{st}}_{\mathbb{F}_p}$ be an animated
PD-$\mathbb{F}_p$-pair of the form $\Gamma_{\mathbb{F}_p [X]} (Y)
\twoheadrightarrow \mathbb{F}_p [X]$. We also have similar definitions for
$\tmop{dRCon}_{\mathbb{F}_p}, \tmop{dRCon}_{\mathbb{F}_p}^0$ and
$\tmop{CrysCon}_{\mathbb{F}_p}, \tmop{CrysCon}_{\mathbb{F}_p}^0$, and a
parallel theory for $\mathbb{F}_p$-stuff. We first point out that, by
\Cref{cor:crys-coh-base-chg} along with the proof of
\Cref{lem:cot-cx-indep-base} (to compare with
\Cref{lem:PD-env-indep-base,lem:PD-env-compat-base-chg}), we have

\begin{lemma}
  \label{lem:crys-coh-indep-base-chg}The derived crystalline cohomology
  $\tmop{CrysCon}_{\mathbb{F}_p} \rightarrow \tmop{CAlg} (D (\mathbb{F}_p))$
  fits into the commutative diagram
  \[ \begin{array}{ccc}
       \tmop{CrysCon}_{\mathbb{F}_p} & \longrightarrow & \tmop{CrysCon}\\
       \longdownarrow &  & \longdownarrow\\
       \tmop{CAlg} (D (\mathbb{F}_p)) & \longrightarrow & \tmop{CAlg} (D
       (\mathbb{Z}))
     \end{array} \]
  {\noindent}of $\infty$-categories, where the horizontal arrows are forgetful
  functors. The same for the derived de Rham cohomology. Furthermore, this
  diagram is left-adjointable (roughly speaking, if we replace the horizontal
  arrows by their left adjoints, it is still a commutative diagram of
  $\infty$-categories).
\end{lemma}

Then the Frobenius map $\varphi_A \of A \rightarrow A$ factors uniquely
through the quotient map $A \twoheadrightarrow A / I$, which gives rise to a
map $A / I \rightarrow A$. It then follows from \Cref{prop:left-deriv-n-fun}
that

\begin{lemma}
  \label{lem:PD-frob}For any animated PD-$\mathbb{F}_p$-pair $(A
  \twoheadrightarrow A'', \gamma) \in \tmop{Pair}^{\gamma,
  \tmop{an}}_{\mathbb{F}_p}$, the Frobenius map $\varphi_A \of A \rightarrow
  A$ factors functorially through the map $A \twoheadrightarrow A''$, which
  gives rise to the a map $A'' \rightarrow A$, denoted by $\varphi_{(A
  \twoheadrightarrow A'', \gamma)}$ or $\varphi_{A \twoheadrightarrow A''}$
  when there is no ambiguity (when $(A \twoheadrightarrow A'', \gamma)$ comes
  from a PD-$\mathbb{F}_p$-pair $(A, I, \gamma)$, it will also be denoted by
  $\varphi_{(A, I, \gamma)}$ or $\varphi_{(A, I)}$). 
\end{lemma}

Now we point out that in the $\tmop{char} p$-case, the de Rham complex is
``Frobenius-linear'' (compare with \Cref{def:conj-fil-pd-env}):

{\construction{\label{cons:Frob-lin-dR-cohom}Let $(A, I, \gamma) \rightarrow
(B, J, \delta)$ be an object in $\tmop{dRCon}_{\mathbb{F}_p}^0$. Each graded
piece $\Omega_{(B, J, \delta) / (A, I, \gamma)}^i$ admits a natural $B$-module
structure therefore also a $\varphi_{(A, I)}^{\ast} (B / J)$-module structure
induced by the map $\varphi_{(A, I)}^{\ast} (B / J) \assign (B / J) \otimes_{A
/ I, \varphi_{(A, I)}}^{\mathbb{L}} A \rightarrow B$, the linearization of
$\varphi_{(B, J)} \of B / J \rightarrow B$. Furthermore, the differential
$\mathd$ is $\varphi_{(A, I)}^{\ast} (B / J)$-linear, which makes the de Rham
complex $(\Omega_{(B, J, \delta) / (A, I, \gamma)}^{\ast}, \mathd)$ a
$\varphi_{(A, I)}^{\ast} (B / J)$-CDGA. In other words, there is a map
$\varphi_{(A, I)}^{\ast} (B / J) \rightarrow (\Omega_{(B, J, \delta) / (A, I,
\gamma)}^{\ast}, \mathd)$ of $\mathbb{F}_p$-CDGAs, where $\varphi_{(A,
I)}^{\ast} (B / J)$ is concentrated in degree $0$.}}

{\construction{\label{cons:conj-fil-dRCon0}Let $(A, I, \gamma) \rightarrow (B,
J, \delta)$ be an object in $\tmop{dRCon}_{\mathbb{F}_p}^0$. The derived de
Rham cohomology $\tmop{dR}_{(B, J, \delta) / (A, I, \gamma)}$ is computed by
the de Rham complex $(\Omega_{(B, J, \delta) / (A, I, \gamma)}^{\ast},
\mathd)$. The Whitehead tower $(\tau_{\geq n} \tmop{dR}_{(B, J, \delta) / (A,
I, \gamma)})_{n \in (\mathbb{Z}, \geq)}$ defines a nonpositive\footnote{In the
literature, the conjugate filtration is increasing. We make it decreasing by
negating the sign.} exhaustive filtration, thus the map $\varphi_{(A,
I)}^{\ast} (B / J) \rightarrow (\Omega_{(B, J, \delta) / (A, I,
\gamma)}^{\ast}, \mathd)$ is a map of filtered $\mathbb{F}_p$-CDGAs (where
$\varphi_{(A, I)}^{\ast} (B / J)$ is trivially filtered), which gives rise to
a map of filtered $\mathbb{E}_{\infty}$-$\mathbb{F}_p$-algebras, thus a
functor $\tmop{dRCon}_{\mathbb{F}_p}^0 \rightarrow \tmop{Fun} (\Delta^1,
\tmop{CAlg} (\tmop{DF}^{\leq 0} (\mathbb{F}_p)))$.}}

\begin{definition}
  The {\tmdfn{conjugate-filtered derived de Rham cohomology functor}}
  $\tmop{Fil}_{\tmop{conj}}^{\ast} \tmop{dR}_{\cdummy / \cdummy} \of
  \tmop{dRCon}_{\mathbb{F}_p} \rightarrow \tmop{CAlg} (\tmop{DF}^{\leq 0}
  (\mathbb{F}_p))$ along with the {\tmdfn{structure map}} $\varphi_{(A
  \twoheadrightarrow A'')}^{\ast} (B'') \rightarrow
  \tmop{Fil}_{\tmop{conj}}^{\ast} \tmop{dR}_{(B \twoheadrightarrow B'',
  \delta) / (A \twoheadrightarrow A'', \gamma)}$ is defined to be the left
  derived functor (\Cref{prop:left-deriv-n-fun}) of the functor
  $\tmop{dRCon}_{\mathbb{F}_p}^0 \rightarrow \tmop{Fun} (\Delta^1, \tmop{CAlg}
  (\tmop{DF}^{\leq 0} (\mathbb{F}_p)))$ in
  Construction~\ref{cons:conj-fil-dRCon0}.
\end{definition}

It follows either from \Cref{prop:left-deriv-n-fun,lem:assoc-graded-union-Lan}
or the fact that $\tmop{Pair}^{\tmop{an}} \simeq \mathcal{P}_{\Sigma}
(\tmop{Pair}^{\tmop{st}}) \subseteq \mathcal{P} (\tmop{Pair}^{\tmop{st}})$ is
stable under filtered colimits (\Cref{prop:Psigma-n}) that

\begin{lemma}
  \label{lem:conj-fil-dR-exhaustive}The conjugate filtration on the derived de
  Rham cohomology is exhaustive.
\end{lemma}

We now prove the corresponding results of \Cref{subsec:deriv-dR} for the
conjugate filtration.

\begin{lemma}
  The conjugate-filtered derived de Rham cohomology functor
  $\tmop{dRCon}_{\mathbb{F}_p} \rightarrow \tmop{CAlg} (\tmop{DF}^{\leq 0}
  (\mathbb{F}_p))$ preserves small colimits (note that so does the functor
  $\tmop{dRCon}_{\mathbb{F}_p} \rightarrow \tmop{CAlg}_{\mathbb{F}_p}, ((A
  \twoheadrightarrow A'', \gamma) \rightarrow (B \twoheadrightarrow B'',
  \delta)) \mapsto \varphi_{(A \twoheadrightarrow A'')}^{\ast} (B'')$).
\end{lemma}

\begin{proof}
  First, we note that, for any connective $\mathbb{E}_{\infty}$-ring $A$, the
  Whitehead-tower functor $D (A) \rightarrow \tmop{DF} (A), M \mapsto
  (\tau_{\geq n} M)_{n \in (\mathbb{Z}, \geq)}$ is canonically lax symmetric
  monoidal (recall that $\tmop{DF} (A)$ is endowed with the Day convolution).
  We give an informal description: given $M, N \in D (A)$, for all $m, n \in
  \mathbb{Z}$, the canonical map $\tau_{\geq m} M \rightarrow M$ and
  $\tau_{\geq n} N \rightarrow N$ gives rise to a map $(\tau_{\geq m} M)
  \otimes_A^{\mathbb{L}} (\tau_{\geq n} N) \rightarrow M
  \otimes_A^{\mathbb{L}} N$. Since $(\tau_{\geq m} M) \otimes_A^{\mathbb{L}}
  (\tau_{\geq n} N)$ is $(m + n)$-connective, this gives rise to a map
  $(\tau_{\geq m} M) \otimes_A^{\mathbb{L}} (\tau_{\geq n} N) \rightarrow
  \tau_{\geq m + n} (M \otimes_A^{\mathbb{L}} N)$. Assembling these maps, we
  get the lax symmetric monoidal structure. Next, when $A$ is given by a
  field, in particular, $A =\mathbb{F}_p$, the structure above is in fact
  symmetric monoidal, since $(\tau_{\geq m} M) \otimes_A^{\mathbb{L}}
  (\tau_{\geq n} N) \rightarrow \tau_{\geq m + n} (M \otimes_A^{\mathbb{L}}
  N)$ is an equivalence for all $m, n \in \mathbb{Z}$.
  
  Now recall that in a symmetric monoidal $\infty$-category, finite coproducts
  of commutative algebra objects are given by tensor products. It follows from
  \Cref{lem:dR-preserv-colim} that the conjugate-filtered derived de Rham
  cohomology functor $\tmop{dRCon}_{\mathbb{F}_p} \rightarrow \tmop{CAlg}
  (\tmop{DF}^{\leq 0} (\mathbb{F}_p))$ is the left derived functor of a
  finite-coproduct-preserving functor, and then the result follows from
  \Cref{prop:left-deriv-n-fun}.
\end{proof}

Note that by the divided power Poincaré's lemma
{\cite[\href{https://stacks.math.columbia.edu/tag/07LC}{Tag
07LC}]{stacks-project}}, the conjugate filtration on the divided power
polynomial algebra is trivial. The proof of \Cref{prop:dR-crys-inv} leads to

\begin{proposition}
  The natural transformation $\tmop{Fil}_{\tmop{conj}}^{\ast} \tmop{dR}_{(B
  \twoheadrightarrow B'', \delta) / (A \twoheadrightarrow A'', \gamma)}
  \rightarrow \tmop{Fil}_{\tmop{conj}}^{\ast} \tmop{dR}_{(\tmop{id}_{B''}, 0)
  / (A \rightarrow A'', \gamma)}$ of functors $\tmop{dRCon} \rightrightarrows
  \tmop{CAlg} (\tmop{DF}^{\leq 0} (\mathbb{Z}))$ induced by the counit map
  associated to $((A \twoheadrightarrow A'', \gamma) \rightarrow (B
  \twoheadrightarrow B'', \delta)) \in \tmop{dRCon}$ is an equivalence. In
  other words, the conjugate-filtered de Rham cohomology functor $\tmop{dRCon}
  \rightarrow \tmop{CAlg} (\tmop{DF}^{\leq 0} (\mathbb{Z}))$ is $(\tmop{dRCon}
  \rightarrow \tmop{CrysCon})$-invariant (\Cref{def:L-inv}) (note that so is
  the functor $\tmop{dRCon}_{\mathbb{F}_p} \rightarrow
  \tmop{CAlg}_{\mathbb{F}_p}, ((A \twoheadrightarrow A'', \gamma) \rightarrow
  (B \twoheadrightarrow B'', \delta)) \mapsto \varphi_{(A \twoheadrightarrow
  A'')}^{\ast} (B'')$).
\end{proposition}

\begin{definition}
  \label{def:conj-fil-crys-coh}The {\tmdfn{conjugate-filtered derived
  crystalline cohomology functor}} $\tmop{Fil}_{\tmop{conj}}^{\ast}
  \tmop{CrysCoh} \of \tmop{CrysCon} \rightarrow \tmop{CAlg} (\tmop{DF}^{\leq
  0} (\mathbb{F}_p))$ along with the {\tmdfn{structure map}} $\varphi_{(A
  \twoheadrightarrow A'', \gamma)}^{\ast} (R) \rightarrow \tmop{CrysCoh}_{R /
  (A \twoheadrightarrow A'', \gamma)}$ is defined to be the composite
  $\tmop{CrysCon} \rightarrow \tmop{dRCon} \rightarrow \tmop{Fun} (\Delta^1,
  \tmop{CAlg} (\tmop{DF}^{\leq 0} (\mathbb{F}_p)))$, where the later functor
  is the conjugate-filtered derived de Rham cohomology functor combined with
  the structure map.
\end{definition}

By \Cref{lem:conj-fil-dR-exhaustive}, we have

\begin{lemma}
  \label{lem:conj-fil-deriv-crys-exhaustive}The conjugate filtration on the
  derived crystalline cohomology is exhaustive.
\end{lemma}

Similar to \Cref{prop:deriv-crys-coh-preserv-colim}, we have

\begin{proposition}
  The conjugate-filtered derived crystalline cohomology functor
  $\tmop{CrysCon} \rightarrow \tmop{CAlg} (\tmop{DF}^{\leq 0} (\mathbb{F}_p))$
  preserves small colimits.
\end{proposition}

Now we analyze the associated graded pieces of the conjugate filtration. Let
$(A, I, \gamma) \rightarrow (B, J, \delta)$ be an element in $\tmop{dRCon}^0$.
We recall the inverse\footnote{A priori, the ``inverse'' Cartier map $C^{- 1}$
is not defined to be the inverse of a map, but just defined to be a map.}
Cartier map $C^{- 1} \of \varphi_{(A, I)}^{\ast} (\Omega_{(B / J) / (A /
I)}^{\star}) \rightarrow H^{\star} (\Omega_{(B, J) / (A, I)}^{\ast}, \mathd)$
of graded $\varphi_{(A, I)}^{\ast} (B / J)$-algebras (where $\star$ is the
grading), then we deduce that this is in fact an isomorphism. Our presentation
is adapted from the proof of {\cite[Thm~7.2]{Katz1970}}.
\begin{description}
  \item[$\mathord{\star} = 0$] This is the composite map $\varphi_{(A,
  I)}^{\ast} (B / J) \rightarrow B \rightarrow H^0 (\Omega_{(B, J) / (A,
  I)}^{\ast}, \mathd)$, i.e., the $\varphi_{(A, I)}^{\ast} (B / J)$-algebra
  structure on $H^0 (\Omega_{(B, J) / (A, I)}, \mathd)$.
  
  \item[$\mathord{\star} = 1$] Consider the map $B \rightarrow H^1
  (\Omega_{(B, J) / (A, I)}^{\ast}, \mathd)$ of sets given by $f \mapsto [f^{p
  - 1} \mathd f]$. We first check that this map is additive: in
  $\Omega_{\mathbb{Z} [u, v] /\mathbb{Z}}^1$, we have
  \begin{eqnarray*}
    (u + v)^{p - 1} \mathd (u + v) - u^{p - 1} \mathd u - v^{p - 1} \mathd v &
    = & \frac{1}{p}  (\mathd ((u + v)^p) - \mathd (u^p) - \mathd (v^p))\\
    & = & \frac{1}{p} \mathd \left( \sum_{j = 1}^{p - 1} \binom{p}{j} u^j
    v^{p - 1 - j} \right)\\
    & = & \mathd \left( \sum_{j = 1}^{p - 1} \frac{1}{p}  \binom{p}{j} u^j
    v^{p - 1 - j} \right)
  \end{eqnarray*}
  We deduce the additivity by the map $\mathbb{Z} [u, v] \rightarrow B, u
  \mapsto f, v \mapsto g$.
  
  Now we note that the map $f \mapsto [f^{p - 1} \mathd f]$ satisfies Leibniz
  rule (recall that $H^1 (\Omega_{(B, J) / (A, I)}^{\ast}, \mathd)$ is a
  $\varphi_{(A, I)}^{\ast} (B / J)$-module, therefore a $B / J$-module).
  Indeed, $[(fg)^{p - 1} \mathd (fg)] = f^p  [g^{p - 1} \mathd g] + g^p  [f^{p
  - 1} \mathd f]$.
  
  Thus we get a derivation $B / J \rightarrow H^1 (\Omega_{(B, J) / (A,
  I)}^{\ast}, \mathd)$, which gives rise to a $B / J$-linear map $\Omega_{(B /
  J) / (A / I)}^1 \rightarrow H^1 (\Omega_{(B, J) / (A, I)}^{\ast}, \mathd)$
  and after linearization, we get $\varphi_{(A, I)}^{\ast} \Omega_{(B / J) /
  (A, I)}^1 \rightarrow H^1 (\Omega_{(B, J) / (A, I)}^{\ast}, \mathd)$.
  
  \item[$\mathord{\star} > 1$] Taking the exterior power of the map for
  $\mathord{\star} = 1$.
\end{description}
Now we show the Cartier isomorphism:

\begin{lemma}
  Let $(A, I, \gamma) \rightarrow (B, J, \delta)$ be an element in
  $\tmop{CrysCon}_{\mathbb{F}_p}^0$. Then the inverse Cartier map $C^{- 1} \of
  \varphi_{(A, I)}^{\ast} (\Omega_{(B / J) / (A / I)}^{\star}) \rightarrow
  H^{\star} (\Omega_{(B, J) / (A, I)}^{\ast}, \mathd)$ is an isomorphism of
  graded $\varphi_{(A, I)}^{\ast} (B / J)$-algebras.
\end{lemma}

\begin{proof}
  Recall that $(B, J, \delta)$ is of the form $(\Gamma_{A [X]} (Y)
  \twoheadrightarrow (A / I) [X], \delta)$. It is then direct to check that
  the inverse Cartier map $C^{- 1}$ factors as $\varphi_{(A, I)}^{\ast}
  (\Omega_{(B / J) / (A / I)}^{\star}) \rightarrow H^{\star} (\Omega_{(A [X],
  IA [X]) / (A, I)}^{\ast}, \mathd) \rightarrow H^{\star} (\Omega_{(B, J) /
  (A, I)}^{\ast}, \mathd)$, where the first map is the inverse Cartier map
  associated to $(A \twoheadrightarrow A / I, \gamma) \rightarrow (A [X], IA
  [X], \gamma)$, and the second map is an isomorphism by the divided power
  Poincaré's lemma
  {\cite[\href{https://stacks.math.columbia.edu/tag/07LC}{Tag
  07LC}]{stacks-project}}.
  
  Thus we can assume that $(B, J, \delta) = (A [X], IA [X], \gamma)$. In this
  case, the inverse Cartier map is base-changed from that for $(A, 0, 0)
  \rightarrow (A [X], 0, 0)$ along $(A, 0, 0) \rightarrow (A, I, \gamma)$,
  thus we can assume that $I = 0$, which is {\cite[Thm~7.2]{Katz1970}}.
\end{proof}

It then follows from \Cref{prop:left-deriv-n-fun} that

\begin{proposition}
  There exists a natural isomorphism\footnote{To avoid the ambiguity of
  symbols, we suppress the asterisk on $\tmop{Fil}^{\ast}$ to avoid confusion
  with the pullback symbol $\varphi^{\ast}$.}
  \[ C^{- 1} \of \varphi_{(A \twoheadrightarrow A'')}^{\ast} \left(
     \bigwedgestar_{B''}^{\star} L_{B'' / A''} \right) \left[ -
     \mathord{\star} \right] \rightarrow \tmop{gr}_{\tmop{conj}}^{-
     \mathord{\star}} \tmop{dR}_{(B \twoheadrightarrow B'') / (A
     \twoheadrightarrow A'')} \]
  in $\tmop{CAlg} (\tmop{Gr}^{\leq 0} (D (\varphi_{A \twoheadrightarrow
  A''}^{\ast} (B''))))$, called the {\tmdfn{derived Cartier isomorphism}} (cf.
  {\cite[Prop~3.5]{Bhatt2012a}}), which is functorial\footnote{Here we use the
  same convention as in \Cref{rem:functor-target-variable}.} in $((A
  \twoheadrightarrow A'', \gamma) \rightarrow (B \twoheadrightarrow B'',
  \delta)) \in \tmop{dRCon}_{\mathbb{F}_p}$.
\end{proposition}

Note that both functors are $(\tmop{dRCon}_{\mathbb{F}_p} \rightarrow
\tmop{CrysCon}_{\mathbb{F}_p})$-invariant (\Cref{def:L-inv}), it follows from
\Cref{prop:L-inv} that

\begin{proposition}
  \label{prop:crys-Cartier-isom}There exists a natural isomorphism
  \[ C^{- 1} \of \varphi_{(A \twoheadrightarrow A'')}^{\ast} \left(
     \bigwedgestar_R^{\star} L_{R / A''} \right) \left[ - \mathord{\star}
     \right] \rightarrow \tmop{gr}_{\tmop{conj}}^{- \mathord{\star}}
     \tmop{CrysCoh}_{R / (A \twoheadrightarrow A'')} \]
  in $\tmop{CAlg} (\tmop{Gr}^{\leq 0} (D (\varphi_{A \twoheadrightarrow
  A''}^{\ast} (R))))$, called the {\tmdfn{derived Cartier isomorphism}}, which
  is functorial in $((A \twoheadrightarrow A'', \gamma), A'' \rightarrow R)
  \in \tmop{CrysCon}_{\mathbb{F}_p}$.
\end{proposition}

\subsection{Relative animated PD-envelope}Recall that, in the classical case
{\cite[\href{https://stacks.math.columbia.edu/tag/07H9}{Tag
07H9}]{stacks-project}}, given a PD-pair $(A, I, \gamma) \in
\tmop{Pair}^{\gamma}$, and an ideal $J$ of $A$ containing $I$, there is a
relative PD-envelope $D_{(A, I, \gamma)} (J)$ which is essentially the same as
$D_A (J) \otimes_{D_A (I)} A$, where the map $D_A (I) \rightarrow A$ is
induced by the PD-structure $\gamma$ on $A$. We seek an animated version of
this construction, which is needed to study the derived crystalline
cohomology. Roughly speaking, we first construct a relative PD-envelope for
any animated PD-pair $(A \twoheadrightarrow A'', \gamma_A)$ along with a map
$(A \twoheadrightarrow A'') \rightarrow (B \twoheadrightarrow B'')$ of
animated pairs, then we restrict to the special case that the map $A
\rightarrow B$ is an equivalence. It turns out that the general case can be
recovered from the special case.

Similar to the classical case, the relative animated PD-envelope of $(A
\twoheadrightarrow A'', \gamma_A)$ along with $(A \twoheadrightarrow A'')
\rightarrow (B \twoheadrightarrow B'')$ is obtained by the base change
$\tmop{Env}^{\gamma, \tmop{an}} (B \twoheadrightarrow B'')
\amalg_{\tmop{Env}^{\gamma, \tmop{an}} (A \twoheadrightarrow A'')} (A
\twoheadrightarrow A'', \gamma_A)$ of the (absolute) animated PD-envelope,
where the map $\tmop{Env}^{\gamma, \tmop{an}} (A \twoheadrightarrow A'')
\rightarrow (A \twoheadrightarrow A'', \gamma_A)$ is induced by the
PD-structure $\gamma_A$. More formally, we start with reviewing the formal
category-theoretic fact about relative adjunctions:

\begin{lemma}[dual to {\cite[Prop~5.2.5.1]{Lurie2009}}]
  \label{lem:adjunct-undercat}Let $\mathcal{C}
  \underset{G}{\overset{F}{\longrightleftarrows}} \mathcal{D}$ be an adjoint
  pair of $\infty$-categories. Assume that the $\infty$-category $\mathcal{D}$
  admits pushouts and let $D \in \mathcal{D}$ be an object. Then
  \begin{enumerate}
    \item The induced functor $g \of \mathcal{D}_{D \mathord{/}} \rightarrow
    \mathcal{C}_{G D \mathord{/}}$ admits a left adjoint $f$.
    
    \item The functor $f$ is equivalent to the composition
    \[ \mathcal{C}_{G D \mathord{/}} \xrightarrow{f'} \mathcal{D}_{F G D
       \mathord{/}} \xrightarrow{f''} \mathcal{D}_{D \mathord{/}} \]
    where $f'$ is induced by $F$ and $f''$ is induced by the pushout along the
    counit map $F G D \rightarrow D$.
  \end{enumerate}
  We note that this construction is functorial in $D \in \mathcal{D}$.
\end{lemma}

\begin{notation}
  We denote the comma category $\tmop{Pair}^{\gamma, \tmop{an}}
  \times_{\tmop{Pair}^{\tmop{an}}, \tmop{ev}_{[0]}} \tmop{Fun} (\Delta^1,
  \tmop{Pair}^{\tmop{an}})$ by $\tmop{PDEnvCon}$, an object of which is
  denoted by $(A \twoheadrightarrow A'', \gamma) \rightarrow (B
  \twoheadrightarrow B'')$, instead of the cumbersome notation $((A
  \twoheadrightarrow A'', \gamma), (A \twoheadrightarrow A'') \rightarrow (B
  \twoheadrightarrow B''))$.
\end{notation}

\begin{definition}
  \label{def:rel-PD-env}Let $(A \twoheadrightarrow A'', \gamma) \in
  \tmop{Pair}^{\gamma, \tmop{an}}$ be an animated PD-pair. The
  {\tmdfn{(relative) animated PD-envelope}} of an animated pair in
  $\tmop{Pair}^{\tmop{an}}_{(A \twoheadrightarrow A'') \mathord{/}}$ is the
  image under the functor $\tmop{Pair}^{\tmop{an}}_{(A \twoheadrightarrow A'')
  \mathord{/}} \rightarrow \tmop{Pair}^{\gamma, \tmop{an}}_{(A
  \twoheadrightarrow A'', \gamma) \mathord{/}}$ induced by the animated
  PD-envelope functor $\tmop{Pair}^{\tmop{an}} \rightarrow
  \tmop{Pair}^{\gamma, \tmop{an}}$ by \Cref{lem:adjunct-undercat}.
  
  Concretely, let $B \twoheadrightarrow B''$ be an object in
  $\tmop{Pair}^{\tmop{an}}_{(A \twoheadrightarrow A'') \mathord{/}}$. Then the
  relative animated PD-envelope of $B \twoheadrightarrow B''$ is given by
  \[ \tmop{Env}_{(A \twoheadrightarrow A'', \gamma_A)}^{\gamma, \tmop{an}} (B
     \twoheadrightarrow B'') \assign \tmop{Env}^{\gamma, \tmop{an}} (B
     \twoheadrightarrow B'') \amalg_{\tmop{Env}^{\gamma, \tmop{an}} (A
     \twoheadrightarrow A'')} (A \twoheadrightarrow A'', \gamma_A) \]
  where the map $\tmop{Env}^{\gamma, \tmop{an}} (A \twoheadrightarrow A'')
  \rightarrow (A \twoheadrightarrow A'', \gamma)$ is the counit map associated
  to $(A \twoheadrightarrow A'', \gamma) \in \tmop{Pair}^{\gamma, \tmop{an}}$
  and the map $\tmop{Env}^{\gamma, \tmop{an}} (A \twoheadrightarrow A'')
  \rightarrow \tmop{Env}^{\gamma, \tmop{an}} (B \twoheadrightarrow B'')$ is
  the image of $(A \twoheadrightarrow A'') \rightarrow (B \twoheadrightarrow
  B'')$ under the animated PD-envelope functor.
  
  This defines the {\tmdfn{(relative) animated PD-envelope functor}}
  $\tmop{PDEnvCon} \rightarrow \tmop{Fun} (\Delta^1, \tmop{Pair}^{\gamma,
  \tmop{an}}), ((A \twoheadrightarrow A'', \gamma) \rightarrow (B
  \twoheadrightarrow B'')) \mapsto \tmop{Env}_{(A \twoheadrightarrow A'',
  \gamma)}^{\gamma, \tmop{an}} (B \twoheadrightarrow B'')$.
\end{definition}

\begin{example}
  Let $(A \twoheadrightarrow A'', \gamma) \in \tmop{Pair}^{\gamma, \tmop{an}}$
  be an animated PD-pair. Then the animated PD-envelope of $A
  \twoheadrightarrow A''$ relative to $(A \twoheadrightarrow A'', \gamma)$ is
  given by $(A \twoheadrightarrow A'', \gamma)$. This follows from the fact
  that $A \twoheadrightarrow A''$ relative to $(A \twoheadrightarrow A'',
  \gamma)$ is the base change of $\tmop{id}_{\mathbb{Z}} \of \mathbb{Z}
  \rightarrow \mathbb{Z}$ relative to $(\tmop{id}_{\mathbb{Z}} \of \mathbb{Z}
  \rightarrow \mathbb{Z}, 0)$ along the map $(\tmop{id}_{\mathbb{Z}} \of
  \mathbb{Z} \rightarrow \mathbb{Z}, 0) \rightarrow (A \twoheadrightarrow A'',
  \gamma)$ of animated PD-pairs. Compare with
  \Cref{lem:rel-PD-env-cryscon-surj}.
\end{example}

\begin{example}
  Let $(A \twoheadrightarrow A'', \gamma) \in \tmop{Pair}^{\gamma, \tmop{an}}$
  be an animated PD-pair. Then the animated PD-envelope of $\tmop{id}_{A''}
  \of A'' \twoheadrightarrow A''$ relative to $(A \twoheadrightarrow A'',
  \gamma)$ is given by $(\tmop{id}_{A''} \of A'' \rightarrow A'', 0)$. This
  follows from checking the universal property of the unit map at
  $\tmop{id}_{A''} \of A'' \twoheadrightarrow A''$ of the adjunction
  $\tmop{Pair}^{\tmop{an}}_{(A \twoheadrightarrow A'') \mathord{/}}
  \rightleftarrows \tmop{Pair}^{\gamma, \tmop{an}}_{(A \twoheadrightarrow A'',
  \gamma) \mathord{/}}$.
\end{example}

It follows immediately from \Cref{lem:ani-PD-env-char-0} that

\begin{lemma}
  \label{lem:rel-pd-env-rat}Let $(A \twoheadrightarrow A'', \gamma) \in
  \tmop{Pair}^{\gamma, \tmop{an}}$ be an animated PD-pair, $(B
  \twoheadrightarrow B'') \in \tmop{Pair}^{\tmop{an}}_{(A \twoheadrightarrow
  A'') \mathord{/}}$ an animated pair under $A \twoheadrightarrow A''$. Let
  $(C \twoheadrightarrow B'', \delta)$ denote its relative animated
  PD-envelope. Then the unit map $(B \twoheadrightarrow B'') \rightarrow (C
  \twoheadrightarrow B'')$ becomes an equivalence after rationalization.
\end{lemma}

Recall that given a PD-pair $(A, I, \gamma)$ and a map $(A, I) \rightarrow (B,
J)$ of pairs with $A \rightarrow B$ being flat, the PD-structure $\gamma$
extends to $B$, i.e, there exists a unique PD-structure $\overline{\gamma}$ on
$(B, IB)$ such that the map $(A, I) \rightarrow (B, J)$ of pairs gives rise to
a map $(A, I, \gamma) \rightarrow (B, IB, \overline{\gamma})$ of PD-pairs.
Then the PD-envelope of $(B, J)$ with respect to $(A, I, \gamma)$ is the same
as that with respect to the PD-pair $(B, IB, \overline{\gamma})$, which
corresponds to the crystalline cohomology of $B / J$ with respect to $(B, IB,
\overline{\gamma})$. We now show an animated analogue (without flatness).

\begin{notation}
  Let $\tmop{CrysCon}_{\tmop{surj}}$ denote the full subcategory
  $\tmop{Pair}^{\gamma, \tmop{an}} \times_{\tmop{CAlg}^{\tmop{an}}} \tmop{Fun}
  (\Delta^1, \tmop{CAlg}^{\tmop{an}})_{\geq 0} \subseteq \tmop{CrysCon}$
  spanned by objects $((A \twoheadrightarrow A'', \gamma), A'' \rightarrow R)$
  such that $A'' \rightarrow R$ is also surjective.
\end{notation}

{\construction{\label{cons:PD-env-con-crys-con-surj-retract}There is a
canonical functor $\tmop{PDEnvCon} \rightarrow \tmop{CrysCon}_{\tmop{surj}}$
given as follows: for every object $((A \twoheadrightarrow A'', \gamma)
\rightarrow (B \twoheadrightarrow B'')) \in \tmop{PDEnvCon}$, we get the
commutative diagram
\[ \begin{array}{ccc}
     A & \nonconverted{longtwoheadrightarrow} & A''\\
     \longdownarrow &  & \longdownarrow\\
     B & \nonconverted{longtwoheadrightarrow} & B''
   \end{array} \]
in $\tmop{CAlg}^{\tmop{an}}$, which gives rise to two surjective maps $B
\otimes_A^{\mathbb{L}} A'' \twoheadrightarrow B''$ \ and $B \twoheadrightarrow
B \otimes_A^{\mathbb{L}} A''$. Furthermore, the later admits a PD-structure:
it is the underlying animated pair of the pushout $(\tmop{id}_B \of B
\rightarrow B, 0) \amalg_{(\tmop{id}_A \of A \rightarrow A, 0)} (A
\twoheadrightarrow A'', \gamma)$ in $\tmop{Pair}^{\gamma, \tmop{an}}$. We
denote by $(B \twoheadrightarrow B \otimes_A^{\mathbb{L}} A'', \delta)$ this
pushout. Then we get an object $((B \twoheadrightarrow B
\otimes_A^{\mathbb{L}} A'', \delta), B \otimes_A^{\mathbb{L}} A''
\twoheadrightarrow B'')$ in $\tmop{CrysCon}_{\tmop{surj}}$.}}

One verifies that

\begin{lemma}
  \label{lem:PD-env-crys-con}The functor $\tmop{PDEnvCon} \rightarrow
  \tmop{CrysCon}_{\tmop{surj}}$ in
  Construction~\ref{cons:PD-env-con-crys-con-surj-retract} admits a fully
  faithful right adjoint $\tmop{CrysCon}_{\tmop{surj}} \rightarrow
  \tmop{PDEnvCon}$ given by $((A \twoheadrightarrow A'', \gamma), A''
  \twoheadrightarrow R) \mapsto ((A \twoheadrightarrow A'', \gamma)
  \rightarrow (A \twoheadrightarrow R))$.
\end{lemma}

Thus $\tmop{CrysCon}_{\tmop{surj}}$ could be seen as a reflective subcategory
(\Cref{def:reflexive-cat}) of $\tmop{PDEnvCon}$. Now we claim that

\begin{lemma}
  \label{lem:rel-PD-env-crys-con}The relative animated PD-envelope functor
  $\tmop{PDEnvCon} \rightarrow \tmop{Fun} (\Delta^1, \tmop{Pair}^{\gamma,
  \tmop{an}})$ is $(\tmop{PDEnvCon} \rightarrow
  \tmop{CrysCon}_{\tmop{surj}})$-invariant (\Cref{def:L-inv}).
\end{lemma}

\begin{proof}
  For every object $(A \twoheadrightarrow A'', \gamma) \rightarrow (B
  \twoheadrightarrow B'')$ in $\tmop{PDEnvCon}$, we have a map $(A
  \twoheadrightarrow A'', \gamma) \rightarrow (B \twoheadrightarrow A''
  \otimes_A^{\mathbb{L}} B, \delta)$ of animated PD-pairs. By the concrete
  description of the relative animated PD-envelope functor, it suffices to
  show that this map along with the counit maps forms a pushout diagram of
  animated PD-pairs. As discussed in
  Construction~\ref{cons:PD-env-con-crys-con-surj-retract}, $(B
  \twoheadrightarrow A'' \otimes_A^{\mathbb{L}} B, \delta)$ is the pushout
  $(\tmop{id}_B \of B \rightarrow B, 0) \amalg_{(\tmop{id}_A \of A \rightarrow
  A, 0)} (A \twoheadrightarrow A'', \gamma)$. The counit maps for
  $(\tmop{id}_A, 0)$ and $(\tmop{id}_B, 0)$ are identities
  (\Cref{lem:ani-ring-as-pdpair}). The result then follows from
  \Cref{prop:forget-PD-small-colim}, which implies that counit maps are
  compatible with small colimits.
\end{proof}

Consequently, in order to study the relative animated PD-envelope functor, it
suffices to study the composite $\tmop{CrysCon}_{\tmop{surj}} \rightarrow
\tmop{PDEnvCon} \rightarrow \tmop{Fun} (\Delta^1, \tmop{Pair}^{\gamma,
\tmop{an}})$. By abuse of terminology, we will simply denote this functor as
$\tmop{RelPDEnv}$ as well and call the image (or after evaluation at $1 \in
\Delta^1$) {\tmdfn{the animated PD-envelope}} of an object $((A
\twoheadrightarrow A'', \gamma), A'' \twoheadrightarrow R) \in
\tmop{CrysCon}_{\tmop{surj}}$. We remark that the functor
$\tmop{CrysCon}_{\tmop{surj}} \rightarrow \tmop{PDEnvCon}$ preserves small
colimits by \Cref{prop:forget-PD-small-colim}, therefore so does the composite
functor.

We note that $\tmop{CrysCon}_{\tmop{surj}}$ is projectively generated: let
$\tmop{CrysCon}_{\tmop{surj}}^0 \subseteq \tmop{CrysCon}_{\tmop{surj}}$ be the
full subcategory spanned by objects $((\Gamma_{\mathbb{Z} [Y, Z]} (X)
\twoheadrightarrow \mathbb{Z} [Y, Z], \gamma), \mathbb{Z} [Y, Z]
\twoheadrightarrow \mathbb{Z} [Z])$ for all finite sets $X, Y, Z$.

\begin{lemma}
  \label{lem:cryscon-comp-proj-gen}The full subcategory
  $\tmop{CrysCon}_{\tmop{surj}}^0 \subseteq \tmop{CrysCon}_{\tmop{surj}}$
  constitutes a set of compact projective generators for
  $\tmop{CrysCon}_{\tmop{surj}}$.
\end{lemma}

\begin{proof}
  We only sketch the proof, which is similar to that of
  \Cref{thm:ani-smith-eq}. The key observation is that the composite of
  forgetful functors $\tmop{CrysCon}_{\tmop{surj}}^0 \rightarrow \tmop{Fun}
  (\Delta^2, \tmop{CAlg}^{\tmop{an}})_{\tmop{surj}} \rightarrow \tmop{Fun}
  (\Delta^2, D (\mathbb{Z})_{\geq 0})_{\tmop{surj}}, ((A \twoheadrightarrow
  A'', \gamma), A'' \twoheadrightarrow R) \mapsto (A \twoheadrightarrow A''
  \twoheadrightarrow R)$, which preserves filtered colimits and geometric
  realizations by \Cref{prop:forget-PD-small-colim}, admits a left adjoint,
  where $\tmop{Fun} (\Delta^2, \mathcal{C})_{\tmop{surj}} \subseteq \tmop{Fun}
  (\Delta^2, \mathcal{C})$ is the full subcategory spanned by $(X \rightarrow
  Y \rightarrow Z) \in \tmop{Fun} (\Delta^2, \mathcal{C})$ such that $X
  \rightarrow Y$ and $Y \rightarrow Z$ are surjective, for $\mathcal{C}= D
  (\mathbb{Z})_{\geq 0}$ and $\mathcal{C}= \tmop{CAlg}^{\tmop{an}}$.
  
  The $\infty$-category $\tmop{Fun} (\Delta^2, D (\mathbb{Z})_{\geq
  0})_{\tmop{surj}}$ admits a set $\{ \mathbb{Z}X \oplus \mathbb{Z}Y \oplus
  \mathbb{Z}Z \twoheadrightarrow \mathbb{Z}Y \oplus \mathbb{Z}Z
  \twoheadrightarrow \mathbb{Z}Z \barsuchthat X, Y, Z \in \tmop{Fin} \}$ of
  compact projective generators which spans the full subcategory $\tmop{Fun}
  (\Delta^2, D (\mathbb{Z})_{\geq 0})_{\tmop{surj}}^0$, which follows from the
  fact that the left adjoint to the left derived functor $\mathcal{P}_{\Sigma}
  (\tmop{Fun} (\Delta^2, D (\mathbb{Z})_{\geq 0})_{\tmop{surj}}^0) \rightarrow
  \tmop{Fun} (\Delta^2, D (\mathbb{Z})_{\geq 0})_{\tmop{surj}}^0$ is
  conservative (cf. the proof of {\cite[Prop~25.2.1.2]{Lurie2017}}).
  
  The result then follows from \Cref{prop:adjoint-n-proj-gen}.
\end{proof}

By \Cref{prop:left-deriv-n-fun}, the functor $\tmop{CrysCon}_{\tmop{surj}}
\rightarrow \tmop{Fun} (\Delta^1, \tmop{Pair}^{\gamma, \tmop{an}})$ is the
left derived functor of the restricted functor $\tmop{CrysCon}_{\tmop{surj}}^0
\rightarrow \tmop{Fun} (\Delta^1, \tmop{Pair}^{\gamma, \tmop{an}})$, which is
concretely given as follows:

\begin{lemma}
  \label{lem:rel-PD-env-cryscon-surj}The relative animated PD-envelope of an
  object $((\Gamma_{\mathbb{Z} [Y, Z]} (X) \twoheadrightarrow \mathbb{Z} [Y,
  Z], \gamma), \mathbb{Z} [Y, Z] \twoheadrightarrow \mathbb{Z} [Z]) \in
  \tmop{CrysCon}_{\tmop{surj}}^0$ is functorially given by
  $(\Gamma_{\mathbb{Z} [Z]} (X, Y) \twoheadrightarrow \mathbb{Z} [Z],
  \tilde{\gamma}) \in \tmop{Pair}^{\gamma, \tmop{an}}_{(\Gamma_{\mathbb{Z} [Y,
  Z]} (X) \twoheadrightarrow \mathbb{Z} [Y, Z], \gamma) \mathord{/}}$, i.e.
  coincides with the classical relative PD-envelope.
\end{lemma}

\begin{proof}
  First, by the adjointness, there exists a functorial comparison map from the
  relative animated PD-envelope to $(\Gamma_{\mathbb{Z} [Z]} (X, Y)
  \twoheadrightarrow \mathbb{Z} [Z], \tilde{\gamma})$. It suffices to show
  that this is an equivalence.
  
  In this case, $((\Gamma_{\mathbb{Z} [Y, Z]} (X) \twoheadrightarrow
  \mathbb{Z} [Y, Z], \gamma), \Gamma_{\mathbb{Z} [Y, Z]} (X)
  \twoheadrightarrow \mathbb{Z} [Z]) \in \tmop{PDEnvCon}$ is the base change
  of $((\tmop{id}_{\mathbb{Z} [Y, Z]}, 0), \mathbb{Z} [Y, Z]
  \twoheadrightarrow \mathbb{Z} [Z]) \in \tmop{PDEnvCon}$ along
  $(\tmop{id}_{\mathbb{Z} [Y, Z]}, 0) \rightarrow (\Gamma_{\mathbb{Z} [Y, Z]}
  (X) \twoheadrightarrow \mathbb{Z} [Y, Z], \gamma)$. The result then follows
  from the base-change property of the relative adjunction, along with the
  simple fact that the (absolute) animated PD-envelope of $\mathbb{Z} [Y, Z]
  \twoheadrightarrow \mathbb{Z} [Z]$ is $(\Gamma_{\mathbb{Z} [Z]} (Y)
  \twoheadrightarrow \mathbb{Z} [Z], \gamma)$.
\end{proof}

As a generalization of \Cref{def:conj-fil-pd-env}, we now introduce the
{\tmdfn{conjugate filtration}} on the relative animated PD-envelope in
$\tmop{char} p$.

{\construction{\label{cons:rel-conj-fil-cl}Let $(A \twoheadrightarrow A'',
\gamma) \in \tmop{Pair}^{\gamma}_{\mathbb{F}_p}$ be a PD-pair and $I \subseteq
A''$ an ideal. We recall that there is a canonical ``Frobenius'' map
$\varphi_{A \twoheadrightarrow A''} \of A'' \rightarrow A$ by
\Cref{lem:PD-frob}. We suppose that $\varphi_{A \twoheadrightarrow A''}$ is
flat\footnote{This is satisfied when $(A \twoheadrightarrow A'', \gamma) \in
\tmop{Pair}^{\gamma, \tmop{st}}$, which is the only case that we need to
develop the theory. For more examples, see \Cref{rem:pd-env-frob-flat}.}. Let
$(B, J, \delta)$ denote the classical PD-envelope of $(A \twoheadrightarrow
A'' / I)$ relative to $(A \twoheadrightarrow A'', \gamma)$. We note that $B /
J \cong A'' / I$. As in the absolute case, due to the PD-structure $(B, J,
\delta)$, there is a canonical $\varphi_{A \twoheadrightarrow A''}^{\ast} (A''
/ I)$-algebra structure on $B$, and we consider the nonpositive filtration on
$B$ given by $\tmop{Fil}_{\tmop{conj}, \tmop{cl}}^{- n} (B)$ for $n \in
\mathbb{N}$ being the $\varphi_{A \twoheadrightarrow A''}^{\ast} (A'' /
I)$-submodule of $B$ generated by $\left\{ \gamma_{i_1 p} (f_1) \cdots
\gamma_{i_m p} (f_m) \barsuchthat i_1 + \cdots + i_m \leq n \infixand f_1,
\ldots, f_m \in I \right\}$.}}

We have the following relative version of \Cref{lem:bhatt-conj-fil}, for which
the proof of {\cite[Lem~3.42]{Bhatt2012a}} adapts:

\begin{lemma}
  \label{lem:bhatt-rel-conj-fil}Let $(A \twoheadrightarrow A'', \gamma)$ be a
  PD-$\mathbb{F}_p$-pair such that $\varphi_{A \twoheadrightarrow A''}$ is
  flat, and let $I \subseteq A''$ be an ideal such that $I / I^2$ is a flat
  $A'' / I$-module. The relative PD-envelope $(B, J, \delta)$ and the
  filtration $\tmop{Fil}_{\tmop{conj}, \tmop{cl}}^{\ast} B$ are as in
  Construction~\ref{cons:rel-conj-fil-cl}.
  
  Then there is a comparison map $\varphi_{A \twoheadrightarrow A''}^{\ast}
  (\Gamma_{A'' / I}^i (I / I^2)) \rightarrow \tmop{gr}^{- i} B$ of $\varphi_{A
  \twoheadrightarrow A''}^{\ast} (A'' / I)$-modules induced by the maps
  $(\gamma_{ip})_{i \in \mathbb{N}}$ (as in \Cref{lem:bhatt-conj-fil}) which
  is functorial in $((A \twoheadrightarrow A'', \gamma), A''
  \twoheadrightarrow A'' / I)$ in a subcategory of
  $\tmop{CrysCon}_{\mathbb{F}_p, \tmop{surj}}$. Furthermore, if $I \subseteq
  A''$ is generated by a Koszul-regular sequence\footnote{We only need the
  simple case that $((A \twoheadrightarrow A'', \gamma), A''
  \twoheadrightarrow A'' / I) \in \tmop{CrysCon}_{\mathbb{F}_p, \tmop{conj}}$,
  which ``simplifies'' the proof in the sense that a ``brute-force''
  computation suffices.}, then this comparison map is an isomorphism.
\end{lemma}

\begin{definition}
  \label{def:conj-fil-rel-pd-env}The {\tmdfn{conjugate filtration functor (on
  the animated PD-envelope)}} $\tmop{Fil}_{\tmop{conj}}
  \tmop{Env}_{(-)}^{\gamma, \tmop{an}} (-) \of \tmop{CrysCon}_{\mathbb{F}_p,
  \tmop{surj}} \rightarrow \tmop{CAlg} (\tmop{DF}^{\leq 0} (\mathbb{F}_p))$
  together with the {\tmdfn{structure map}} of functors
  $\tmop{CrysCon}_{\mathbb{F}_p, \tmop{surj}} \rightrightarrows \tmop{CAlg}
  (\tmop{DF}^{\leq 0} (\mathbb{F}_p))$ from $((A \twoheadrightarrow A'',
  \gamma), A'' \twoheadrightarrow R) \mapsto \varphi_{A \twoheadrightarrow
  A''}^{\ast} (R) = R \otimes_{A'', \varphi_{A \twoheadrightarrow
  A''}}^{\mathbb{L}} A$ to $((A \twoheadrightarrow A'', \gamma), A''
  \twoheadrightarrow R) \mapsto \tmop{Fil}_{\tmop{conj}} \tmop{Env}_{(A
  \twoheadrightarrow A'', \gamma)}^{\gamma, \tmop{an}} (R)$ is defined to be
  the left derived functor (\Cref{prop:left-deriv-n-fun}) of
  $\tmop{CrysCon}_{\mathbb{F}_p, \tmop{surj}}^0 \ni ((A \twoheadrightarrow
  A'', \gamma), A'' \twoheadrightarrow A'' / I) \mapsto (\varphi_{A
  \twoheadrightarrow A''}^{\ast} (A'' / I) \rightarrow
  \tmop{Fil}_{\tmop{conj}, \tmop{cl}}^{\ast} B) \in \tmop{Fun} (\Delta^1,
  \tmop{CAlg} (\tmop{DF}^{\leq 0} (\mathbb{F}_p)))$ as in
  Construction~\ref{cons:rel-conj-fil-cl}.
\end{definition}

As in the absolute case (including \Cref{rem:functor-target-variable}), it
follows from \Cref{lem:bhatt-rel-conj-fil} that

\begin{corollary}
  \label{cor:conjfil-gr-rel}For every $((A \twoheadrightarrow A'', \gamma),
  A'' \twoheadrightarrow R) \in \tmop{CrysCon}_{\mathbb{F}_p, \tmop{surj}}$,
  there exists an equivalence
  \[ \varphi_{A \twoheadrightarrow A''}^{\ast} (\Gamma_R^i
     (\tmop{gr}_{\tmop{ad}}^1 (A'' \twoheadrightarrow R))) \rightarrow
     \tmop{gr}_{\tmop{conj}}^{- i} \tmop{Env}_{(A \twoheadrightarrow A'',
     \gamma)}^{\gamma, \tmop{an}} (R) \]
  in $D (\varphi_{A \twoheadrightarrow A''}^{\ast} (R))_{\geq 0}$ for all $i
  \in \mathbb{N}$ which is functorial in $((A \twoheadrightarrow A'', \gamma),
  A'' \twoheadrightarrow R) \in \tmop{CrysCon}_{\mathbb{F}_p, \tmop{surj}}$.
\end{corollary}

As in the absolute case, we have

\begin{corollary}
  \label{cor:Fp-qreg-rel-pd-env-flat}For every $((A \twoheadrightarrow A'',
  \gamma), A'' \twoheadrightarrow R) \in \tmop{CrysCon}_{\mathbb{F}_p,
  \tmop{surj}}$ such that $A'' \twoheadrightarrow R$ is a quasiregular
  animated pair, let $(B \twoheadrightarrow R, \delta)$ denote the relative
  animated PD-envelope. Then $B$ is a flat $\varphi_{A \twoheadrightarrow
  A''}^{\ast} (R)$-module.
\end{corollary}

Similar to \Cref{prop:koszul-regular-ani-pd-env}, we have

\begin{proposition}
  Let $(A \twoheadrightarrow A'', \gamma) \in \tmop{Pair}^{\gamma}$ be a
  PD-pair and $I \subseteq A''$ an ideal generated by a Koszul-regular
  sequence. Let $(B \twoheadrightarrow B'', \delta)$ denote the relative
  animated PD-envelope of $((A \twoheadrightarrow A'', \gamma), A''
  \rightarrow A'' / I) \in \tmop{CrysCon}$. Then $(B \twoheadrightarrow B'',
  \delta)$ is a PD-pair, therefore coincides with the classical relative
  PD-envelope.
\end{proposition}

\begin{remark}
  \label{rem:pd-env-frob-flat}More precisely, in
  \Cref{cor:Fp-qreg-pd-env-flat}, the map $\varphi_{A \twoheadrightarrow
  A''}^{\ast} (R) \rightarrow B$ is induced by the Frobenius map $\varphi_{B
  \twoheadrightarrow R} \of R \rightarrow B$ (which could be seen by left
  deriving the special case that $((A \twoheadrightarrow A'', \gamma), A''
  \rightarrow R) \in \tmop{CrysCon}_{\mathbb{F}_p, \tmop{surj}}^0$). In
  particular, if the Frobenius map $\varphi_{A \twoheadrightarrow A''} \of A''
  \rightarrow A$ is flat, then so is the Frobenius map $\varphi_{B
  \twoheadrightarrow R} \of R \rightarrow B$.
  
  For example, when $R$ is a quasiregular semiperfect $\mathbb{F}_p$-algebra
  {\cite[Def~8.8]{Bhatt2018}}, we set $(A \twoheadrightarrow A'', \gamma) =
  (\tmop{id}_{R^{\flat}} \of R^{\flat} \rightarrow R^{\flat}, 0)$ and the map
  $A'' \rightarrow R$ to be the canonical map, by definition, $R^{\flat}$ is a
  perfect $\mathbb{F}_p$-algebra therefore $\varphi_{R^{\flat}}$ is flat. Then
  the animated PD-envelope $B \twoheadrightarrow R$ of $A'' \rightarrow R$
  satisfies the condition that the Frobenius map $\varphi_{B \rightarrow R}
  \of R \rightarrow B$ is flat and hence $B$ is static. It follows that $(R
  \twoheadrightarrow B, \delta)$ is a PD-pair
  (\Cref{prop:characterize-pdpair}).
\end{remark}

Note that the associated graded pieces of derived crystalline cohomology and
relative animated PD-envelope of a ``surjective'' crystalline context $((A
\twoheadrightarrow A'', \gamma), A'' \twoheadrightarrow R) \in
\tmop{CrysCon}_{\mathbb{F}_p, \tmop{surj}}$, with respect to conjugate
filtrations, are equivalent by
\Cref{cor:conjfil-gr-rel,prop:crys-Cartier-isom,cor:LAdFil-symm-cot-cx}. In
fact, we have

\begin{lemma}
  \label{lem:Fp-crys-PD-env-equiv}There is a canonical equivalence
  \[ \tmop{Fil}_{\tmop{conj}} \tmop{CrysCoh}_{R / (A \twoheadrightarrow A'',
     \gamma)} \rightarrow \tmop{Fil}_{\tmop{conj}} \tmop{Env}_{(A
     \twoheadrightarrow A'', \gamma)}^{\gamma, \tmop{an}} (R) \]
  in $\tmop{CAlg} (\tmop{DF}^{\leq 0} (\varphi_{A \twoheadrightarrow
  A''}^{\ast} (R)))$ which is functorial\footnote{Here we apply the same
  convention as in \Cref{rem:functor-target-variable}.} in $((A
  \twoheadrightarrow A'', \gamma), A'' \twoheadrightarrow R) \in
  \tmop{CrysCon}_{\mathbb{F}_p, \tmop{surj}}$.
\end{lemma}

\begin{proof}
  We first point out how to produce the comparison map of underlying
  $\mathbb{E}_{\infty}$-$\mathbb{F}_p$-algebras, i.e., ignoring the
  $\varphi_{A \twoheadrightarrow A''}^{\ast} (R)$-algebra structures and
  conjugate filtrations. This is logically not necessary but it benefits our
  understanding. Given $((A \twoheadrightarrow A'', \gamma), A''
  \twoheadrightarrow R)$, let $(B \twoheadrightarrow R, \delta)$ denote its
  relative animated PD-envelope. It follows from \Cref{prop:dR-crys-inv} that
  the crystalline cohomology $\tmop{CrysCoh}_{R / (A \twoheadrightarrow A'',
  \gamma)}$ is naturally equivalent to the derived de Rham cohomology
  $\tmop{dR}_{(B \twoheadrightarrow R, \delta) / (A \twoheadrightarrow A'',
  \gamma)}$, and by definition, it is equipped with a map $\tmop{dR}_{(B
  \twoheadrightarrow R, \delta) / (A \twoheadrightarrow A'', \gamma)}
  \rightarrow B$ of $\mathbb{E}_{\infty}$-$\mathbb{F}_p$-algebras, which gives
  rise to the underlying comparison map that we want.
  
  By \Cref{lem:cryscon-comp-proj-gen,prop:left-deriv-n-fun}, it suffices to
  construct the equivalence restricted to the full subcategory
  $\tmop{CrysCon}_{\mathbb{F}_p, \tmop{surj}}^0 \subseteq
  \tmop{CrysCon}_{\mathbb{F}_p, \tmop{surj}}$, i.e., to establish the
  equivalence for all $((\Gamma_{\mathbb{F}_p [Y, Z]} (X) \twoheadrightarrow
  \mathbb{F}_p [Y, Z], \gamma_0), \mathbb{F}_p [Y, Z] \twoheadrightarrow
  \mathbb{F}_p [Z]) \in \tmop{CrysCon}_{\mathbb{F}_p, \tmop{surj}}^0$. This is
  essentially {\cite[Lem~3.29 \& Thm~3.27]{Bhatt2012a}}. We will briefly
  sketch the argument. The preceding paragraph has already established a
  comparison map of underlying $\mathbb{E}_{\infty}$-$\mathbb{F}_p$-algebras.
  The key point is that both sides are static: the relative animated
  PD-envelope is static by definition, and the derived crystalline cohomology
  is static by Cartier isomorphism (\Cref{prop:crys-Cartier-isom}) and the
  fact that static modules are closed under extension and filtered colimits,
  see \Cref{cor:Fp-qreg-frob-ani-pd-env} for a similar argument. Then the
  result follows from explicit simplicial resolution.
\end{proof}

We now deduce the integral version of \Cref{lem:Fp-crys-PD-env-equiv}. We
recall that $\tmop{Fil}_{\tmop{PD}} \of \tmop{Pair}^{\gamma, \tmop{an}}
\rightarrow \tmop{CAlg} (\tmop{DF}^{\geq 0} (\mathbb{Z}))$ is the
PD-filtration functor (\Cref{def:pd-fil}), and $\tmop{Fil}_H$ is the
Hodge-filtration.

\begin{proposition}
  \label{prop:crys-PD-env-equiv}There is a canonical equivalence
  \[ \tmop{Fil}_H \tmop{CrysCoh} \rightarrow \tmop{Fil}_{\tmop{PD}} \circ
     \tmop{Env}_{(-)}^{\gamma, \tmop{an}} (-) \]
  of functors $\tmop{CrysCon}_{\tmop{surj}} \rightrightarrows \tmop{CAlg}
  (\tmop{DF}^{\geq 0} (\mathbb{Z}))$.
\end{proposition}

\begin{proof}
  The comparison map is established in the same way as in the proof of
  \Cref{lem:Fp-crys-PD-env-equiv}. It suffices to show that this is an
  equivalence.
  
  We first show that this becomes an equivalence after passing to underlying
  $\mathbb{E}_{\infty}$-$\mathbb{Z}$-algebras, i.e. ignoring the Hodge
  filtration. By conservativity of the forgetful functor
  $\tmop{CAlg}_{\mathbb{Z}} \rightarrow D (\mathbb{Z})$, it suffices to show
  the equivalence for underlying $\mathbb{Z}$-module spectra. Note that it is
  an equivalence after $(-) \otimes_{\mathbb{Z}}^{\mathbb{L}} \mathbb{F}_p$
  and rationalization by
  \Cref{lem:Fp-crys-PD-env-equiv,lem:rel-pd-env-rat,lem:dR-rational}, thus it
  is itself an equivalence.
  
  To establish the equivalence of filtered
  $\mathbb{E}_{\infty}$-$\mathbb{Z}$-algebras, it remains to show that the
  comparison map induces equivalences after passing to associated graded
  pieces, and by \Cref{lem:left-deriv-fun-adjoint}, it suffices to prove the
  result restricted to the full subcategory $\tmop{CrysCon}_{\tmop{surj}}^0
  \subseteq \tmop{CrysCon}_{\tmop{surj}}$, which is essentially due to
  {\cite[Cor~VIII.2.2.8]{Illusie1972}}, see {\cite[Rem~3.33]{Bhatt2012a}}.
\end{proof}

\subsection{Affine crystalline site}We now turn to the site-theoretic aspects
of the derived crystalline cohomology by showing that the derived crystalline
cohomology is equivalent to the cohomology of the {\tmdfn{affine crystalline
site}} under a mild smoothness condition. We warn the reader again that our
theory is non-completed. Fix a crystalline context $((A \twoheadrightarrow
A'', \gamma_A), A'' \rightarrow R) \in \tmop{CrysCon}$.

\begin{definition}
  \label{def:aff-crys-site}The {\tmdfn{affine crystalline site}} $\tmop{Cris}
  (R / (A \twoheadrightarrow A'', \gamma_A))$ is defined to be the opposite
  $\infty$-category of animated PD-pairs $(B \twoheadrightarrow B'',
  \gamma_B)$ under $(A \twoheadrightarrow A'', \gamma_A)$ along with an
  equivalence $R \xrightarrow{\simeq} B''$ of $A$-algebras, depicted by the
  diagram
  
  \[
  \xymatrix{
A\ar[rr]\ar@{->>}[d]&&B\ar@{->>}[d]\\
A''\ar[r]&R\ar^{\simeq}[r]&B''
}
  \]
  
  {\noindent}which we will simply denoted by $\left( R \xrightarrow{\simeq}
  B'' \twoheadleftarrow B \right) \in \tmop{Cris} (R / (A \twoheadrightarrow
  A'', \gamma_A))$. More formally, it is the homotopy fiber of the functor
  $\tmop{Pair}^{\gamma, \tmop{an}}_{(A \twoheadrightarrow A'') \mathord{/}}
  \rightarrow \tmop{CAlg}_{A''}^{\tmop{an}}, (B \twoheadrightarrow B'',
  \gamma_B) \mapsto B''$ at the object $R \in \tmop{CAlg}^{\tmop{an}}_{A''}$.
  The endowed Grothendieck topology is indiscrete.
  
  The {\tmdfn{structure presheaf}} $\tmop{Cris} (R / (A \twoheadrightarrow
  A'', \gamma_A))^{\tmop{op}} \rightarrow \tmop{CAlg}_A$, denoted by
  $\mathcal{O}_{\tmop{CRIS} (R / (A \twoheadrightarrow A'', \gamma_A))}$ (or
  simply $\mathcal{O}$ when there is no ambiguity), is induced by the
  evaluation $\tmop{Pair}^{\gamma, \tmop{an}}_{(A \twoheadrightarrow A'',
  \gamma_A) \mathord{/}} \rightarrow \tmop{CAlg}_A^{\tmop{an}} \rightarrow
  \tmop{CAlg} (D (A))$. Concretely, it is given by $\left( R
  \xrightarrow{\simeq} B'' \twoheadleftarrow B \right) \mapsto B$.
\end{definition}

Although the affine crystalline site is not small, the cohomology of the
structure presheaf exists in $\tmop{CAlg}_A$ by Čech--Alexander calculation
(which we will reproduce in \Cref{prop:crys-site-comp-desc}). We will simply
call it the {\tmdfn{cohomology of the crystalline site}} and denote by $R
\Gamma (\tmop{Cris} (R / (A \twoheadrightarrow A'', \gamma_A)), \mathcal{O})$.
Furthermore, the structure sheaf admits the PD-filtration (\Cref{def:pd-fil}),
which gives rise to a filtration on the cohomology of the crystalline site,
called the {\tmdfn{Hodge-filtration}} and denoted by $\tmop{Fil}_H$. We now
have a comparison between the derived crystalline cohomology and the
cohomology of the crystalline site, which becomes an equivalence after
Hodge-completion:

\begin{proposition}
  \label{prop:deriv-crys-coh-site-comp}There is a natural comparison map
  \[ \tmop{Fil}_H \tmop{CrysCoh}_{R / (A \twoheadrightarrow A'', \gamma_A)}
     \rightarrow \tmop{Fil}_H R \Gamma (\tmop{Cris} (R / (A \twoheadrightarrow
     A'', \gamma_A)), \mathcal{O}) \]
  in the $\infty$-category $\tmop{CAlg} (\tmop{DF}^{\geq 0} (A))$. After
  passing to the associated graded pieces, i.e. composition with the functor
  $\tmop{CAlg} (\tmop{DF}^{\geq 0} (A)) \rightarrow \tmop{CAlg}
  (\tmop{Gr}^{\geq 0} (A))$, this comparison map becomes an equivalence.
  Moreover, when $\pi_0 (R)$ is a finitely generated $\pi_0 (A'')$-algebra,
  then the comparison map is an equivalence.
\end{proposition}

We need some preparation about cosimplicial objects in $\infty$-categories.

\begin{definition}[{\cite[Def~6.1.2.2]{Lurie2009}}]
  Let $\mathcal{C}$ be an $\infty$-category. A {\tmdfn{cosimplicial object}}
  of $\mathcal{C}$ is a functor $X^{\bullet} \of \tmmathbf{\Delta} \rightarrow
  \mathcal{C}$. The value of this functor at $[\nu] \in \tmmathbf{\Delta}$ is
  denoted by $X^{\nu}$. A {\tmdfn{map of cosimplicial objects}} $X^{\bullet}
  \rightarrow Y^{\bullet}$ is simply a map of functors.
\end{definition}

We note that there are two inclusions $\{ 0 \} \hookrightarrow [1]
\hookleftarrow \{ 1 \}$ viewed as two maps $[0] \rightrightarrows [1]$, and a
constant map $[1] \rightarrow [0]$, which induce three functors $i_0$, $i_1$:
$\tmmathbf{\Delta} \simeq \tmmathbf{\Delta}_{/ [0]} \rightrightarrows
\tmmathbf{\Delta}_{/ [1]}$ and $\delta \of \tmmathbf{\Delta}_{/ [1]}
\rightarrow \tmmathbf{\Delta}_{/ [0]} \simeq \tmmathbf{\Delta}$. For any
$\infty$-category $\mathcal{C}$, let $i_0^{\ast}$ (resp. $i_1^{\ast}$) denote
the induced functor $\tmop{Fun} (\tmmathbf{\Delta}_{/ [1]}, \mathcal{C})
\rightarrow \tmop{Fun} (\tmmathbf{\Delta}, \mathcal{C})$, and let
$\delta^{\ast}$ denote the induced functor $\tmop{Fun} (\tmmathbf{\Delta},
\mathcal{C}) \rightarrow \tmop{Fun} (\tmmathbf{\Delta}_{/ [1]}, \mathcal{C})$.

\begin{definition}[{\cite[Def~7.2.1.6]{Lurie2017}}]
  Let $\mathcal{C}$ be an $\infty$-category and let $f$ and $g$ be two maps
  $X^{\bullet} \rightrightarrows Y^{\bullet}$ of cosimplicial objects. A
  {\tmdfn{simplicial homotopy from $f$ to $g$}} is a map $h \of \delta^{\ast}
  (X^{\bullet}) \rightarrow \delta^{\ast} (Y^{\bullet})$ of functors
  $\tmmathbf{\Delta}_{/ [1]} \rightrightarrows \mathcal{C}$ such that the map
  $i_0^{\ast} (h) \of X^{\bullet} \rightarrow Y^{\bullet}$ (resp. $i_1^{\ast}
  (h) \of X^{\bullet} \rightarrow Y^{\bullet}$), being a map of cosimplicial
  objects, is equivalent to $f$ (resp. $g$). When $X^{\bullet} = Y^{\bullet}$,
  we say that the simplicial homotopy $h \of \delta^{\ast} (X^{\bullet})
  \rightarrow \delta^{\ast} (X^{\bullet})$ is {\tmdfn{constant}} if it is
  equivalent to $\tmop{id}_{\delta^{\ast} (X^{\bullet})}$.
\end{definition}

\begin{lemma}
  \label{lem:simpl-homot-maps-eq}Let $\mathcal{C}$ be an $\infty$-category and
  $X^{\bullet}, Y^{\bullet}$ two cosimplicial objects of which the
  totalization exist in $\mathcal{C}$. Let $f$ and $g$ be two maps
  $X^{\bullet} \rightrightarrows Y^{\bullet}$ of cosimplicial objects such
  that there exists a simplicial homotopy from $f$ to $g$. Then the maps $f,
  g$ induces equivalent\footnote{Or called ``homotopic''. We avoid the
  terminology ``homotopic'' to avoid confusion with the simplicial homotopy.}
  maps $\lim_{\tmmathbf{\Delta}} X^{\bullet} \rightrightarrows
  \lim_{\tmmathbf{\Delta}} Y^{\bullet}$ of totalizations.
\end{lemma}

\begin{proof}[Denis {\tmname{Nardin}}]
  For every cosimplicial object $X^{\bullet}$ in $\mathcal{C}$, there are two
  observations:
  \begin{enumerate}
    \item The canonical map $\lim_{\tmmathbf{\Delta}_{/ [1]}} \delta^{\ast}
    (X^{\bullet}) \rightarrow \lim_{\tmmathbf{\Delta}} X^{\bullet}$ is an
    equivalence (this involves the existence of the limit as the source).
    Indeed, it suffices to show that the map $\delta \of \tmmathbf{\Delta}_{/
    [1]} \rightarrow \tmmathbf{\Delta}$ is coinitial. By Joyal's version of
    Quillen's Theorem A {\cite[Thm~4.1.3.1]{Lurie2009}}, it suffices to show
    that, for every $[n] \in \tmmathbf{\Delta}$, the category
    $\tmmathbf{\Delta}_{/ [1]} \times_{\tmmathbf{\Delta}} \tmmathbf{\Delta}_{/
    [n]}$ is weakly contractible. Its geometric realization is $\Delta^1
    \times \Delta^n$, which is known to be weakly contractible.
    
    \item The two maps $\lim_{\tmmathbf{\Delta}} X = \lim_{\tmmathbf{\Delta}}
    i_{\nu}^{\ast} \delta^{\ast} (X^{\bullet}) \rightarrow
    \lim_{\tmmathbf{\Delta}_{/ [1]}} \delta^{\ast} (X^{\bullet})$ for $\nu =
    0, 1$ are equivalences, and these two maps are equivalent. Indeed, both
    are inverses of the equivalence $\lim_{\tmmathbf{\Delta}_{/ [1]}}
    \delta^{\ast} (X^{\bullet}) \rightarrow \lim_{\tmmathbf{\Delta}}
    X^{\bullet}$ above.
  \end{enumerate}
  Note that the map $\lim_{\tmmathbf{\Delta}} f$ (resp.
  $\lim_{\tmmathbf{\Delta}} g$) could be identified with the composite
  \[ \lim_{\tmmathbf{\Delta}} X = \lim_{\tmmathbf{\Delta}} i_{\nu}^{\ast}
     \delta^{\ast} (X^{\bullet}) \longrightarrow \lim_{\tmmathbf{\Delta}_{/
     [1]}} \delta^{\ast} (X^{\bullet}) \xrightarrow{\lim_{\tmmathbf{\Delta}_{/
     [1]}} (h)} \lim_{\tmmathbf{\Delta}_{/ [1]}} \delta^{\ast} (Y^{\bullet})
     \longrightarrow \lim_{\tmmathbf{\Delta}} Y^{\bullet} \]
  for $\nu = 0$ (resp. $\nu = 1$). The result then follows.
\end{proof}

\begin{proof*}{Proof of \Cref{prop:deriv-crys-coh-site-comp}}
  There is a map from the constant presheaf $\underline{\tmop{Fil}_H
  \tmop{CrysCoh}_{R / (A \twoheadrightarrow A'', \gamma_A)}}$ on the affine
  crystalline site to the structure presheaf $\mathcal{O}$ given by the
  canonical map in \Cref{def:Hdg-fil-deriv-dR}, which induces the comparison
  map in question.
  
  Now we show that this map becomes an equivalence after passing to the
  associated graded pieces. We first note that, when the map $A'' \rightarrow
  R$ is surjective, i.e. $((A \twoheadrightarrow A'', \gamma_A), A''
  \rightarrow R) \in \tmop{CrysCon}_{\tmop{surj}}$, the result follows
  directly from \Cref{prop:crys-PD-env-equiv}. Our strategy is to reduce the
  general case to this special case via Čech--Alexander computation.
  
  We pick a polynomial $A$-algebra $P$ (of possibly infinitely many variables)
  along with a surjection $P \twoheadrightarrow R$ of $A$-algebras. Let
  $P^{\bullet} \rightarrow R$ denote the Čech conerve of the object $P
  \rightarrow R$ in the $\infty$-category $\tmop{CAlg}^{\tmop{an}}_{A
  \mathord{/}}$. Concretely, it is given by $P^{\nu} \assign
  P^{\otimes_A^{\mathbb{L}} (\nu + 1)}$, and the map $P^{\nu} \rightarrow R$
  is simply given by the composite map $P^{\nu} \rightarrow P \rightarrow R$
  which is surjective. In other words, we get a cosimplicial object
  $(P^{\bullet} \twoheadrightarrow R) \in \tmop{Fun} \left( \Delta,
  \tmop{Pair}^{\tmop{an}}_{(A \twoheadrightarrow A'') \mathord{/}} \right)$.
  Let $(D^{\bullet} \twoheadrightarrow R, \gamma_{D^{\bullet}}) \in \tmop{Fun}
  \left( \Delta, \tmop{Pair}^{\gamma, \tmop{an}}_{(A \twoheadrightarrow A'',
  \gamma_A) \mathord{/}} \right)$ denote the cosimplicial relative animated
  PD-envelope, i.e. applying the functor $\tmop{Pair}^{\tmop{an}}_{(A
  \twoheadrightarrow A'') \mathord{/}} \rightarrow \tmop{Pair}^{\gamma,
  \tmop{an}}_{(A \twoheadrightarrow A'', \gamma_A) \mathord{/}}$
  (\Cref{def:rel-PD-env}) pointwise. This effectively gives rise to a
  cosimplicial object $\Delta \rightarrow \tmop{Cris} (R / (A
  \twoheadrightarrow A'', \gamma_A))^{\tmop{op}}$. Composing with the
  Hodge-filtered presheaf $\tmop{Fil}_H \mathcal{O} \of \tmop{Cris} (R / (A
  \twoheadrightarrow A'', \gamma_A))^{\tmop{op}} \rightarrow \tmop{CAlg}
  (\tmop{DF}^{\geq 0} (A))$, we get a cosimplicial filtered
  $\mathbb{E}_{\infty}$-$A$-algebra $\Delta \rightarrow \tmop{CAlg}
  (\tmop{DF}^{\geq 0} (A))$, the limit of which computes the cohomology
  $\tmop{Fil}_H R \Gamma (\tmop{Cris} (R / (A \twoheadrightarrow A'',
  \gamma_A)), \mathcal{O})$. In plain terms, this cosimplicial filtered
  $\mathbb{E}_{\infty}$-$A$-algebra is just the PD-filtration of the
  cosimplicial animated PD-pair $(D^{\bullet} \twoheadrightarrow R,
  \gamma_{D^{\bullet}})$.
  
  For this cosimplicial object, the comparison map constructed above is
  concretely given by
  \begin{equation}
    \tmop{Fil}_H \tmop{dR}_{(D^{\bullet} \twoheadrightarrow R,
    \gamma_{D^{\bullet}}) / (A \twoheadrightarrow A'', \gamma_A)} \rightarrow
    \tmop{Fil}_{\tmop{PD}} D^{\bullet} \label{eq:cosimpl-comp}
  \end{equation}
  Now \Cref{prop:crys-PD-env-equiv,lem:PD-env-crys-con} gives us an
  equivalence
  \[ \tmop{Fil}_H \tmop{CrysCoh}_{R / (P^{\bullet} \twoheadrightarrow
     P^{\bullet} \otimes_A^{\mathbb{L}} A'', \gamma_{P^{\bullet}})}
     \rightarrow \tmop{Fil}_{\tmop{PD}} D^{\bullet} \]
  which is effectively given by
  \[ \tmop{Fil}_H \tmop{dR}_{(D^{\bullet} \twoheadrightarrow R,
     \gamma_{D^{\bullet}}) / (P^{\bullet} \twoheadrightarrow P^{\bullet}
     \otimes_A^{\mathbb{L}} A'', \gamma_{P^{\bullet}})} \rightarrow
     \tmop{Fil}_{\tmop{PD}} D^{\bullet} \]
  by chasing the proof. In other words, \eqref{eq:cosimpl-comp} could be
  rewritten as the natural map
  \[ \tmop{Fil}_H \tmop{dR}_{(D^{\bullet} \twoheadrightarrow R,
     \gamma_{D^{\bullet}}) / (A \twoheadrightarrow A'', \gamma_A)} \rightarrow
     \tmop{Fil}_H \tmop{dR}_{(D^{\bullet} \twoheadrightarrow R,
     \gamma_{D^{\bullet}}) / (P^{\bullet} \twoheadrightarrow P^{\bullet}
     \otimes_A^{\mathbb{L}} A'', \gamma_{P^{\bullet}})} \]
  or equivalently, the natural map
  \[ \tmop{Fil}_H \tmop{CrysCoh}_{R / (A \twoheadrightarrow A'', \gamma_A)}
     \rightarrow \tmop{Fil}_H \tmop{CrysCoh}_{R / (P^{\bullet}
     \twoheadrightarrow P^{\bullet} \otimes_A^{\mathbb{L}} A'',
     \gamma_{P^{\bullet}})} \]
  It remains to show that this cosimplicial map gives rise to an equivalence
  after taking the limit, i.e., the totalization, and passing to associated
  graded pieces. We isolate the remaining part into
  \Cref{lem:crys-coh-gr-desc}.
\end{proof*}

Before proving \Cref{lem:crys-coh-gr-desc}, we isolate an important
observation in the previous proof into a proposition:

\begin{proposition}
  \label{prop:crys-site-comp-desc}For every crystalline context $((A
  \twoheadrightarrow A'', \gamma_A), A'' \rightarrow R) \in \tmop{CrysCon}$,
  the followings are equivalent:
  \begin{enumerate}
    \item The comparison map in \Cref{prop:deriv-crys-coh-site-comp} is an
    equivalence.
    
    \item There exists a polynomial $A$-algebra $P$ (of possibly infinitely
    many variables) along with a surjection $P \twoheadrightarrow R$ of
    $A$-algebras, and letting $P^{\bullet} \rightarrow R$ denote the Čech
    conerve of $P \rightarrow R$ in the $\infty$-category
    $\tmop{CAlg}_A^{\tmop{an}}$ as in the proof of
    \Cref{prop:deriv-crys-coh-site-comp}, then the natural maps
    \begin{equation}
      \tmop{Fil}_H \tmop{CrysCoh}_{R / (A \twoheadrightarrow A'', \gamma_A)}
      \rightarrow \tmop{Fil}_H \tmop{CrysCoh}_{R / (P^{\bullet}
      \twoheadrightarrow P^{\bullet} \otimes_A^{\mathbb{L}} A'',
      \gamma_{P^{\bullet}})} \label{eq:Hdg-cryscoh-desc}
    \end{equation}
    form a limit diagram in $\tmop{CAlg} (\tmop{DF}^{\geq 0} (A))$.
    
    \item For all polynomial $A$-algebras $P$ (of possibly infinitely many
    variables) along with a surjection $P \twoheadrightarrow R$ of
    $A$-algebras, and letting $P^{\bullet} \rightarrow R$ denote the Čech
    conerve of $P \rightarrow R$ in the $\infty$-category
    $\tmop{CAlg}_A^{\tmop{an}}$, then the natural maps
    \eqref{eq:Hdg-cryscoh-desc} form a limit diagram in $\tmop{CAlg}
    (\tmop{DF}^{\geq 0} (A))$.
    
    \item (After proving \Cref{lem:crys-coh-gr-desc}) There exists a (or
    equivalently, for every) polynomial $A$-algebra $P$ (of possibly
    infinitely many variables) along with a surjection $P \twoheadrightarrow
    R$ of $A$-algebras, and letting $P^{\bullet} \rightarrow R$ denote the
    Čech conerve of $P \rightarrow R$ in the $\infty$-category
    $\tmop{CAlg}_A^{\tmop{an}}$ as in the proof of
    \Cref{prop:deriv-crys-coh-site-comp}, then the natural maps
    \[ \tmop{CrysCoh}_{R / (A \twoheadrightarrow A'', \gamma_A)} \rightarrow
       \tmop{CrysCoh}_{R / (P^{\bullet} \twoheadrightarrow P^{\bullet}
       \otimes_A^{\mathbb{L}} A'', \gamma_{P^{\bullet}})} \]
    form a limit diagram in $\tmop{CAlg} (D (A))$.
  \end{enumerate}
\end{proposition}

In order to deal with associated graded pieces of the Hodge filtration, we
need a variant of the {\tmdfn{Katz--Oda filtration}} in
{\cite[Cons~3.12]{Guo2020}}. We need an auxiliary construction:

\begin{definition}
  The {\tmdfn{cotangent complex functor}} $L_{\cdummy / \cdummy} \of
  \tmop{dRCon} \rightarrow \tmop{Ani} (\tmop{Mod})$ is defined to be the left
  derived functor (\Cref{prop:left-deriv-n-fun}) of the functor
  $\tmop{dRCon}^0 \rightarrow \tmop{Ani} (\tmop{Mod}), ((A, I, \gamma_A)
  \rightarrow (B, J, \gamma_B)) \mapsto (B, \Omega_{(B, J) / (A, I)}^1)$.
\end{definition}

The proof of \Cref{lem:dR-pdpair-classical} leads to

\begin{lemma}
  The composite functor $\tmop{Fun} (\Delta^1, \tmop{CAlg}^{\tmop{an}})
  \rightarrow \tmop{dRCon} \rightarrow \tmop{Ani} (\tmop{Mod})$ is equivalent
  to the classical cotangent complex functor.
\end{lemma}

We now introduce the {\tmdfn{``stupid'' filtration $\tmop{Fil}_B$}} on the
Hodge-filtered derived de Rham cohomology $\tmop{Fil}_H \tmop{dR}_{\cdummy /
\cdummy}$. For each $((A, I, \gamma_A) \rightarrow (B, J, \gamma_B)) \in
\tmop{dRCon}^0$, consider the filtration $(\Omega_{(B, J, \gamma_B) / (A, I,
\gamma_A)}^{\geq n}, \mathd)_{n \in (\mathbb{N}, \geq)}$ of the Hodge-filtered
CDGA, which gives rise to a bifiltered
$\mathbb{E}_{\infty}$-$\mathbb{Z}$-algebra. By \Cref{prop:left-deriv-n-fun},
we get a functor $\tmop{CAlg} (\tmop{Fun} ((\mathbb{N}, \geq) \times
(\mathbb{N}, \geq), D (\mathbb{Z}))), ((A \twoheadrightarrow A'', \gamma_A)
\rightarrow (B \twoheadrightarrow B'', \gamma_B)) \mapsto \tmop{Fil}_B
\tmop{Fil}_H \tmop{dR}_{(B \twoheadrightarrow B'') / (A \twoheadrightarrow
A'')}$.

\begin{warning}
  Unlike the Hodge filtration, the ``stupid'' filtration does not descend to
  $\tmop{CrysCon}$, that is to say, it depends on the choice of $B$ in
  question.
\end{warning}

We now analyze the associated graded pieces with respect to the ``stupid''
filtration:

\begin{lemma}
  \label{lem:brute-fil-gr}Let $((A \twoheadrightarrow A'', \gamma_A)
  \rightarrow (B \twoheadrightarrow B'', \gamma_B)) \in \tmop{dRCon}$ be a de
  Rham context. Then associated graded pieces $\tmop{gr}_B^i \tmop{Fil}_H
  \tmop{dR}_{(B \twoheadrightarrow B'') / (A \twoheadrightarrow A'')}$ could
  be functorially identified with $\tmop{ins}^i \left( \bigwedgestar_B^i L_{(B
  \twoheadrightarrow B'') / (A \twoheadrightarrow A'')} [- i] \right)
  \otimes_B^{\mathbb{L}} \tmop{Fil}_{\tmop{PD}} B$ as a
  $\tmop{Fil}_{\tmop{PD}} B$-module in $\tmop{DF}^{\geq 0} (B)$ (where
  $\cdummy \otimes_B^{\mathbb{L}} \tmop{Fil}_{\tmop{PD}} B$ is the base change
  from $\tmop{DF}^{\geq 0} (B)$ to the $\infty$-category of
  $\tmop{Fil}_{\tmop{PD}} B$-modules). Furthermore, $\tmop{Fil}_B^i
  \tmop{gr}_H^j \tmop{dR}_{(B \twoheadrightarrow B'') / (A \twoheadrightarrow
  A'')} \simeq 0$ when $i > j$.
\end{lemma}

\begin{proof}
  By \Cref{prop:left-deriv-n-fun}, it suffices to check the equivalences on
  $\tmop{dRCon}^0$, which follows from definitions.
\end{proof}

We are now ready to introduce the Katz--Oda filtration:

\begin{definition}[cf. {\cite[Cons~3.12]{Guo2020}}]
  \label{def:Katz-Oda-fil}Let $(A \twoheadrightarrow A'', \gamma_A)
  \rightarrow (B \twoheadrightarrow B'', \gamma_B)$ be a map of animated
  PD-pairs and $B'' \rightarrow R$ a map of animated rings. The
  {\tmdfn{Katz--Oda filtration}} on the Hodge-filtered derived crystalline
  cohomology $\tmop{Fil}_H \tmop{CrysCoh}_{R / (A \twoheadrightarrow A'')}$
  rewritten as
  \[ \tmop{Fil}_H \tmop{CrysCoh}_{R / (A \twoheadrightarrow A'')}
     \otimes_{\tmop{Fil}_H \tmop{dR}_{(B \twoheadrightarrow B'') / (A
     \twoheadrightarrow A'')}}^{\mathbb{L}} \tmop{Fil}_H \tmop{dR}_{(B
     \twoheadrightarrow B'') / (A \twoheadrightarrow A'')} \]
  is the nonnegative filtration induced by the ``stupid'' filtration on
  $\tmop{Fil}_H \tmop{dR}_{(B \twoheadrightarrow B'') / (A \twoheadrightarrow
  A'')}$.
\end{definition}

We now have

\begin{lemma}[cf. {\cite[Lem~3.13]{Guo2020}}]
  \label{lem:Katz-Oda-gr-cpl}Let $(A \twoheadrightarrow A'', \gamma_A)
  \rightarrow (B \twoheadrightarrow B'', \gamma_B)$ be a map of animated
  PD-pairs and $B'' \rightarrow R$ a map of animated rings. Then
  \begin{enumerate}
    \item The associated graded pieces $\tmop{gr}_{\tmop{KO}}^i \tmop{Fil}_H
    \tmop{CrysCoh}_{R / (A \twoheadrightarrow A'')}$ are functorially
    equivalent to
    \[ \tmop{Fil}_H \tmop{CrysCoh}_{R / (B \twoheadrightarrow B'')}
       \otimes_{\tmop{Fil}_{\tmop{PD}} B}^{\mathbb{L}} \left( \tmop{ins}^i
       \left( \bigwedgestar_B^i L_{(B \twoheadrightarrow B'') / (A
       \twoheadrightarrow A'')} [- i] \right) \otimes_B^{\mathbb{L}}
       \tmop{Fil}_{\tmop{PD}} B \right) \]
    as $\tmop{Fil}_{\tmop{PD}} B$-modules in $\tmop{DF}^{\geq 0} (\mathbb{Z})$
    for all $i \in \mathbb{N}$, where the functor $\tmop{ins}^i$ is defined in
    \Cref{subsec:graded-fil-objs}.
    
    \item The induced Katz--Oda filtration on $\tmop{gr}_H \tmop{CrysCoh}_{R /
    (A \twoheadrightarrow A'')}$ is complete. In fact, for $i > j$, we have
    \[ \tmop{Fil}_{\tmop{KO}}^i \tmop{gr}_H^j \tmop{CrysCoh}_{R / (A
       \twoheadrightarrow A'')} \simeq 0. \]
  \end{enumerate}
\end{lemma}

\begin{proof}
  We have seen (\Cref{cor:Hdg-fil-crys-coh-transitive}) that the canonical map
  \[ \tmop{Fil}_H \tmop{CrysCoh}_{R / (A \twoheadrightarrow A'')}
     \otimes_{\tmop{Fil}_H \tmop{dR}_{(B \twoheadrightarrow B'') / (A
     \twoheadrightarrow A'')}}^{\mathbb{L}} \tmop{Fil}_{\tmop{PD}} B
     \rightarrow \tmop{Fil}_H \tmop{CrysCoh}_{R / (B \twoheadrightarrow B'')}
  \]
  is an equivalence. Then both follow from \Cref{lem:brute-fil-gr}.
\end{proof}

The convergence of Katz--Oda filtration on associated graded pieces is the key
to \Cref{lem:crys-coh-gr-desc}.

\begin{lemma}
  \label{lem:crys-coh-gr-desc}In \Cref{prop:crys-site-comp-desc}, the maps
  \eqref{eq:Hdg-cryscoh-desc} form a limit diagram after passing to the
  associated graded pieces, i.e. after passing along the functor $\tmop{CAlg}
  (\tmop{DF}^{\geq 0} (A)) \rightarrow \tmop{CAlg} (\tmop{Gr}^{\geq 0} (A))$.
  Furthermore, if the $\pi_0 (A'')$-algebra $\pi_0 (R)$ is of finite type,
  then the maps \eqref{eq:Hdg-cryscoh-desc} form a limit diagram.
\end{lemma}

\begin{proof}
  For every $[\nu] \in \tmmathbf{\Delta}$, let $\tmop{Fil}_{\tmop{KO}, \nu}
  \tmop{Fil}_H \tmop{CrysCoh}_{R / (A \twoheadrightarrow A'')}$ denote the
  Katz--Oda filtration with respect to the setup $((A \twoheadrightarrow A'',
  \gamma_A) \rightarrow (P^{\nu} \rightarrow P^{\nu} \otimes_A^{\mathbb{L}}
  A'', \gamma_{P^{\nu}}), P^{\nu} \rightarrow R)$. This construction is
  canonically functorial in $[\nu] \in \tmmathbf{\Delta}$. Note that the map
  \eqref{eq:Hdg-cryscoh-desc} is the canonical map
  \begin{equation}
    \tmop{Fil}_H \tmop{CrysCoh}_{R / (A \twoheadrightarrow A'')} =
    \tmop{Fil}_{\tmop{KO}, \nu}^0 \tmop{Fil}_H \tmop{CrysCoh}_{R / (A
    \twoheadrightarrow A'')} \longrightarrow \tmop{gr}_{\tmop{KO}, \nu}^0
    \tmop{Fil}_H \tmop{CrysCoh}_{R / (A \twoheadrightarrow A'')} .
    \label{eq:Katz-Oda-gr0}
  \end{equation}
  Now we show that \eqref{eq:Katz-Oda-gr0} becomes an equivalence after
  replacing $\tmop{Fil}_H$ by $\tmop{gr}_H$ and taking limit over $[\nu] \in
  \tmmathbf{\Delta}$. That is, by the completeness in
  \Cref{lem:Katz-Oda-gr-cpl}, for every $i \in \mathbb{N}_{> 0}$, the
  totalization $\lim_{\nu \in \tmmathbf{\Delta}} \tmop{gr}_{\tmop{KO}, \nu}^i
  \tmop{gr}_H \tmop{CrysCoh}_{R / (A \twoheadrightarrow A'')}$ is
  contractible. We show the slightly stronger statement that, for every $i \in
  \mathbb{N}_{> 0}$, the totalization
  \begin{equation}
    \lim_{\nu \in \tmmathbf{\Delta}} \tmop{gr}_{\tmop{KO}, \nu}^i \tmop{Fil}_H
    \tmop{CrysCoh}_{R / (A \twoheadrightarrow A'')}
    \label{eq:Katz-Oda-gr-total}
  \end{equation}
  is contractible.
  
  The key observation is that $\tmop{ins}^i \left( \bigwedgestar^i
  L_{(P^{\bullet} \rightarrow P^{\bullet} \otimes_A^{\mathbb{L}} A'') / (A
  \twoheadrightarrow A'')} [- i] \right)$ is homotopy equivalent to $0$ as a
  cosimplicial $B^{\bullet}$-module spectrum by {\cite[Lem~2.6]{Bhatt2012}}
  when $i > 0$ (this is of course false when $i = 0$). It follows that the
  cosimplicial object $\tmop{gr}_{\tmop{KO}}^{i, (\bullet)} \tmop{Fil}_H
  \tmop{CrysCoh}_{R / (A \twoheadrightarrow A'')}$ is homotopy equivalent to
  $0$ as $\tmop{Fil}_{\tmop{PD}} B$-modules by
  {\cite[\href{https://stacks.math.columbia.edu/tag/07KQ}{Tag
  07KQ}]{stacks-project}} and \Cref{lem:Katz-Oda-gr-cpl}.
  
  Finally, when $\pi_0 (R)$ is a finitely generated $\pi_0 (A'')$-algebra, we
  pick a polynomial $A$-algebra $P$ of finitely many variables along with a
  surjection $P \twoheadrightarrow R$. For every $[\nu] \in
  \tmmathbf{\Delta}$, since the animated $A$-algebra $P$ is polynomial of
  finite type, so is the animated $A$-algebra $P^{\nu}$, thus the Katz--Oda
  filtration is finite, i.e. $\tmop{Fil}_{\tmop{KO}, \nu} \tmop{Fil}_H
  \tmop{CrysCoh}_{R / (A \twoheadrightarrow A'')}$ is finite in the
  $\tmop{Fil}_{\tmop{KO}, \nu}$-direction, and in particular, it is a complete
  filtration. Since completely filtered objects are stable under small limits,
  it follows that the object
  \[ \lim_{[\nu] \in \tmmathbf{\Delta}} \tmop{Fil}_{\tmop{KO}, \nu}
     \tmop{Fil}_H \tmop{CrysCoh}_{R / (A \twoheadrightarrow A'')} \]
  is completely filtered in the ``$\lim_{[\nu] \in \tmmathbf{\Delta}}
  \tmop{Fil}_{\tmop{KO}, \nu}$''-direction. Now for every $i \in \mathbb{N}_{>
  0}$, since the totalization \eqref{eq:Katz-Oda-gr-total} is contractible,
  the functorial map \eqref{eq:Katz-Oda-gr0} becomes an equivalence after
  taking $\lim_{[\nu] \in \tmmathbf{\Delta}}$ and the result follows.
\end{proof}

\begin{warning}
  One should be careful about homotopy equivalences. In an earlier draft of
  this article, we came up with the following ``proof'': the Hodge-filtered
  derived de Rham cohomology $\tmop{Fil}_H \tmop{CrysCoh}_{R / (A
  \twoheadrightarrow A'')}$ could be rewritten as
  \[ \tmop{Fil}_H \tmop{CrysCoh}_{R / (A \twoheadrightarrow A'')}
     \otimes_{\tmop{Fil}_H \tmop{dR}_{(P^{\bullet} \twoheadrightarrow
     P^{\bullet} \otimes_A^{\mathbb{L}} A'') / (A \twoheadrightarrow
     A'')}}^{\mathbb{L}} \tmop{Fil}_H \tmop{dR}_{(P^{\bullet}
     \twoheadrightarrow P^{\bullet} \otimes_A^{\mathbb{L}} A'') / (A
     \twoheadrightarrow A'')} \]
  and since the map $A \rightarrow P^{\bullet}$ is a homotopy equivalence as
  $A$-algebras, the map $\tmop{Fil}_H \tmop{dR}_{(P^{\bullet}
  \twoheadrightarrow P^{\bullet} \otimes_A^{\mathbb{L}} A'') / (A
  \twoheadrightarrow A'')} \rightarrow \tmop{Fil}_{\tmop{PD}} P^{\bullet}$ is
  also a homotopy equivalence ``therefore'' the constant cosimplicial algebra
  $\tmop{Fil}_H \tmop{CrysCoh}_{R / (A \twoheadrightarrow A'')}$ is homotopy
  equivalent to
  \[ \tmop{Fil}_H \tmop{CrysCoh}_{R / (A \twoheadrightarrow A'')}
     \otimes_{\tmop{Fil}_H \tmop{dR}_{(P^{\bullet} \twoheadrightarrow
     P^{\bullet} \otimes_A^{\mathbb{L}} A'') / (A \twoheadrightarrow
     A'')}}^{\mathbb{L}} \tmop{Fil}_{\tmop{PD}} P^{\bullet} \simeq
     \tmop{Fil}_H \tmop{CrysCoh}_{R / (P^{\bullet} \twoheadrightarrow
     P^{\bullet} \otimes_A^{\mathbb{L}} A'')} \]
  therefore the conditions in \Cref{prop:crys-site-comp-desc}.
  
  This argument is incorrect: when playing with homotopy equivalences, one
  cannot replace the base cosimplicial algebra by a homotopy equivalent
  algebra without justification. In fact, the last homotopy equivalence
  obtained above is also incorrect: if it {\tmem{were}} the case, we consider
  the special case that $(A \twoheadrightarrow A'', \gamma_A)$ is given by
  $(\tmop{id}_A \of A \rightarrow A, 0)$, and $\tmop{CrysCoh}_{R /
  P^{\bullet}}$ is just the animated PD-envelope of $P^{\bullet}
  \twoheadrightarrow R$ (see the proof of
  \Cref{prop:deriv-crys-coh-site-comp}). We inspect the homotopy equivalence
  of cosimplicial objects that we assumed:
  \[ \tmop{dR}_{R / A} \simeq^{\tmop{HoEq}} \tmop{dR}_{R / P^{\bullet}} \]
  when $A$ is a static $\mathbb{F}_p$-algebra and $R$ is a smooth $A$-algebra
  such that $\tmop{dR}_{R / A}$ is not static, the map $P^{\bullet}
  \twoheadrightarrow R$ is Koszul regular and the derived de Rham cohomology
  $\tmop{dR}_{R / P^{\bullet}}$ is simply the PD-envelope, therefore static.
  Applying $\pi_i$ to the homotopy equivalence, where $i \neq 0$ is so chosen
  that $\pi_i (\tmop{dR}_{R / A}) \neq 0$, we get a contradiction.
  
  In view of this warning, our proof of \Cref{lem:crys-coh-gr-desc} tells us
  that the associated graded pieces with respect to the Katz--Oda filtration
  are homotopy equivalent, but the homotopy equivalences could not be glued,
  even after forgetting all the richer structures to the underlying
  $\infty$-category $D (\mathbb{Z})$.
\end{warning}

When the $\pi_0 (A)$-algebra $\pi_0 (R)$ is not of finite type, we can still
prove that the comparison map is an equivalence with mild smoothness of $A''
\rightarrow R$ (\Cref{prop:site-comp-qsyn-integral}). We start with another
sufficient condition in characteristic $p$ which is essentially a variant of
{\cite[Prop~2.17]{Li2020}} by \Cref{prop:crys-site-comp-desc}.

\begin{lemma}
  \label{lem:crys-coh-char-p-desc}Let $((A \twoheadrightarrow A'', \gamma_A),
  A'' \rightarrow R) \in \tmop{CrysCon}_{\mathbb{F}_p}$. Suppose that
  \begin{enumerate}
    \item The cotangent complex $L_{R / A''} \in D_{\geq 0} (R)$ has
    $\tmop{Tor}$-amplitude in $[0, 1]$.
    
    \item The derived Frobenius twist $\varphi_{A \twoheadrightarrow
    A''}^{\ast} (R)$ (see \Cref{lem:PD-frob}) is bounded above, i.e. $\pi_i
    (\varphi_{A \twoheadrightarrow A''}^{\ast} (R)) \cong 0$ for $i \gg 0$.
  \end{enumerate}
  Then the comparison map in \Cref{prop:deriv-crys-coh-site-comp} is an
  equivalence.
\end{lemma}

\begin{proof}
  Our proof is also adapted from {\cite[Prop~2.17]{Li2020}}. By
  \Cref{prop:crys-site-comp-desc,lem:crys-coh-gr-desc}, it suffices to that
  the natural maps
  \begin{equation}
    \tmop{CrysCoh}_{R / (A \twoheadrightarrow A'', \gamma_A)} \longrightarrow
    \tmop{CrysCoh}_{R / (P^{\bullet} \twoheadrightarrow P^{\bullet}
    \otimes_A^{\mathbb{L}} A'', \gamma_{P^{\bullet}})} \label{eq:cryscoh-desc}
  \end{equation}
  form a limit diagram in $\tmop{CAlg} (D (A))$. We endow both sides with
  conjugate filtration (\Cref{def:conj-fil-crys-coh}), and show that this
  forms in fact a limit diagram in $\tmop{CAlg} (\tmop{DF}^{\leq 0} (A))$.
  
  We show that, after passing to associated graded pieces with respect to the
  conjugate filtration, the maps \eqref{eq:cryscoh-desc} form a limit diagram,
  which implies that the natural maps \eqref{eq:cryscoh-desc} form limit
  diagrams after passing to finite level of quotients, and then we control the
  convergence to deduce the result. To show the result for associated graded
  pieces, by \Cref{prop:crys-Cartier-isom}, it suffices to show that the maps
  \begin{equation}
    \varphi_{A \twoheadrightarrow A''}^{\ast} \left( \bigwedgestar_R^{\star}
    L_{R / A''} \right) \left[ - \mathord{\star} \right] \longrightarrow
    \varphi_{P^{\bullet} \twoheadrightarrow P^{\bullet} \otimes_A^{\mathbb{L}}
    A''}^{\ast} \left( \bigwedgestar_R^{\star} L_{R / (P^{\bullet}
    \otimes_A^{\mathbb{L}} A'')} \right) \left[ - \mathord{\star} \right]
    \label{eq:ext-cot-cx-frob-twist-desc}
  \end{equation}
  form a limit diagram in $\tmop{Gr}^{\geq 0} (D (A))$.
  
  Let $R_1 \assign \varphi_{A \twoheadrightarrow A''}^{\ast} (R)$. Note that
  the Frobenius map $\varphi_{P^{\bullet} \twoheadrightarrow P^{\bullet}
  \otimes_A^{\mathbb{L}} A''}$ factors as $P^{\bullet} \otimes_A^{\mathbb{L}}
  A'' \rightarrow \varphi_A^{\ast} (P^{\bullet}) \rightarrow P^{\bullet}$
  where the second map is the Frobenius map of $P^{\bullet}$ relative to $A$.
  Then the maps \eqref{eq:ext-cot-cx-frob-twist-desc} could be rewritten as
  the maps
  \[ \bigwedgestar_{R_1}^{\star} L_{R_1 / A} \left[ - \mathord{\star} \right]
     \longrightarrow \left( \bigwedgestar_{R_1}^{\star} L_{R_1 /
     \varphi_A^{\ast} (P^{\bullet})} \right) \left[ - \mathord{\star} \right]
     \otimes_{\varphi_A^{\ast} (P^{\bullet})}^{\mathbb{L}} P^{\bullet} \]
  or equivalently, the maps
  \[ \tmop{gr}_H^{\star} \tmop{dR}_{R_1 / A} \longrightarrow
     \tmop{gr}_H^{\star} \tmop{dR}_{R_1 / \varphi_A^{\ast} (P^{\bullet})}
     \otimes_{\varphi_A^{\ast} (P^{\bullet})}^{\mathbb{L}} P^{\bullet} \]
  by an inverse application of \Cref{lem:brute-fil-gr} (recall that for
  derived de Rham cohomology of animated rings, the ``stupid'' filtration
  coincides with the Hodge filtration). We again consider the Katz--Oda
  filtration associated to the cosimplicial system $A \rightarrow
  \varphi_A^{\ast} (P^{\bullet}) \rightarrow R_1$ (\Cref{lem:Katz-Oda-gr-cpl})
  and by completeness, we could pass to associated graded pieces for $i = 0$:
  \begin{equation}
    \tmop{gr}_H^{\star} \tmop{dR}_{R_1 / \varphi_A^{\ast} (P^{\bullet})}
    \longrightarrow \tmop{gr}_H^{\star} \tmop{dR}_{R_1 / \varphi_A^{\ast}
    (P^{\bullet})} \otimes_{\varphi_A^{\ast} (P^{\bullet})}^{\mathbb{L}}
    P^{\bullet} \label{eq:dR-frob-twist-desc}
  \end{equation}
  and $i \in \mathbb{N}_{> 0}$:
  \[ \tmop{gr}_H^{\star} \tmop{dR}_{R_1 / \varphi_A^{\ast} (P^{\bullet})}
     \otimes_{\varphi_A^{\ast} (P^{\bullet})}^{\mathbb{L}} \left(
     \bigwedgestar_{\varphi_A^{\ast} (P^{\bullet})}^i L_{\varphi_A^{\ast}
     (P^{\bullet}) / A} [- i] \right) \longrightarrow 0 \]
  As in \Cref{lem:crys-coh-gr-desc}, the later maps constitute a homotopy
  equivalence by {\cite[Lem~2.6]{Bhatt2012}} and
  {\cite[\href{https://stacks.math.columbia.edu/tag/07KQ}{Tag
  07KQ}]{stacks-project}}, therefore constitutes a limit diagram by
  \Cref{lem:simpl-homot-maps-eq}. On the other hand, by
  \Cref{lem:bhatt-fix-homotopy-eq}, the maps \eqref{eq:dR-frob-twist-desc}
  constitute a limit diagram.
  
  Now we control the convergence. Again by \Cref{lem:brute-fil-gr}, we rewrite
  the maps \eqref{eq:dR-frob-twist-desc} as the maps
  \[ \bigwedgestar_{R_1}^{\star} L_{R_1 / \varphi_A^{\ast} (P^{\bullet})}
     \left[ - \mathord{\star} \right] \longrightarrow \left(
     \bigwedgestar_{R_1}^{\star} L_{R_1 / \varphi_A^{\ast} (P^{\bullet})}
     \left[ - \mathord{\star} \right] \right) \otimes_{\varphi_A^{\ast}
     (P^{\bullet})}^{\mathbb{L}} P^{\bullet} \]
  Now consider the transitivity sequence
  \[ L_{\varphi_A^{\ast} (P^{\bullet}) / A} \otimes_{\varphi_A^{\ast}
     (P^{\bullet})}^{\mathbb{L}} R_1 \longrightarrow L_{R_1 / A}
     \longrightarrow L_{R_1 / \varphi_A^{\ast} (P^{\bullet})} \]
  For every static $R_1$-module $M$, we get the fiber sequence
  \[ L_{\varphi_A^{\ast} (P^{\bullet}) / A} \otimes_{\varphi_A^{\ast}
     (P^{\bullet})}^{\mathbb{L}} M \longrightarrow L_{R_1 / A}
     \otimes_{R_1}^{\mathbb{L}} M \longrightarrow L_{R_1 / \varphi_A^{\ast}
     (P^{\bullet})} \otimes_{R_1}^{\mathbb{L}} M \]
  Since $L_{R / A''} \in D_{\geq 0} (R)$ has $\tmop{Tor}$-amplitude in $[0,
  1]$, so does $L_{R_1 / A} \in D_{\geq 0} (R_1)$, therefore $\pi_j (L_{R_1 /
  A} \otimes_{R_1}^{\mathbb{L}} M) \cong 0$ for $j \neq 0, 1$. Note that
  $L_{\varphi_A^{\ast} (P^{\bullet}) / A}$ is a flat $\varphi_A^{\ast}
  (P^{\bullet})$-module. It follows that $\pi_j (L_{R_1 / \varphi_A^{\ast}
  (P^{\bullet})} \otimes_{R_1}^{\mathbb{L}} M) \cong 0$ for $j \neq 0, 1$.
  Furthermore, since $\varphi_A^{\ast} (P^{\bullet}) \rightarrow R_1$ is
  surjective, $\pi_0 (L_{R_1 / \varphi_A^{\ast} (P^{\bullet})}
  \otimes_{R_1}^{\mathbb{L}} M) \cong 0$. It follows that $L_{R_1 /
  \varphi_A^{\ast} (P^{\bullet})} [- 1]$ is a flat $R_1$-module, and so is
  $\bigwedgestar_{R_1}^{\star} L_{R_1 / \varphi_A^{\ast} (P^{\bullet})} \left[
  - \mathord{\star} \right] \simeq \Gamma_{R_1}^{\star} (L_{R_1 /
  \varphi_A^{\ast} (P^{\bullet})} [- 1])$. By assumption, $R_1$ is bounded
  above, therefore so is $\bigwedgestar_{R_1}^{\star} L_{R_1 /
  \varphi_A^{\ast} (P^{\bullet})} \left[ - \mathord{\star} \right]$.
  
  It remains to show that the associated graded pieces are uniformly bounded
  above, which implies that \eqref{eq:cryscoh-desc} form a limit diagram, by
  \Cref{lem:fil-colim-tot-commute} and that the conjugate filtration is
  exhaustive (\Cref{lem:conj-fil-deriv-crys-exhaustive}). Suppose that the
  homotopy groups of $R_1$ are concentrated in the range $[a, b]$, then the
  associated graded pieces of the target could be rewritten as
  $\Gamma_{R_1}^{\star} (L_{R_1 / \varphi_A^{\ast} (P^{\bullet})} [- 1])
  \otimes_{\varphi_A^{\ast} (P^{\bullet})}^{\mathbb{L}} P^{\bullet}$, where
  $\Gamma_{R_1}^{\star} (L_{R_1 / \varphi_A^{\ast} (P^{\bullet})} [- 1])$ is a
  flat $R_1$-module therefore the homotopy groups of it is also concentrated
  in the range $[a, b]$. Since the relative Frobenius $\varphi_A^{\ast}
  (P^{\bullet}) \rightarrow P^{\bullet}$ is flat, we get $\pi_j
  (\Gamma_{R_1}^{\star} (L_{R_1 / \varphi_A^{\ast} (P^{\bullet})} [- 1])
  \otimes_{\varphi_A^{\ast} (P^{\bullet})}^{\mathbb{L}} P^{\bullet}) \cong
  \pi_j (\Gamma_{R_1}^{\star} (L_{R_1 / \varphi_A^{\ast} (P^{\bullet})} [-
  1])) \otimes_{\pi_0 (\varphi_A^{\ast} (P^{\bullet}))} \pi_0 (P^{\bullet})
  \cong 0$ for $j \nin [a, b]$.
\end{proof}

We need the following lemmas:

\begin{lemma}
  \label{lem:Cech-two-maps-homotopic}Let $\mathcal{C}$ be an $\infty$-category
  which admits finite coproducts. Let $\varnothing$ denote the initial object
  of $\mathcal{C}$, and let $X, Y$ be two objects of $\mathcal{C}$. Then for
  any two maps $g_0, g_1 \in \tmop{Hom}_{\mathcal{C}} (X, Y)$, the induced
  maps $X^{\bullet} \rightrightarrows Y^{\bullet}$ of Čech conerves
  $X^{\bullet}$ of $X$ (i.e. of $\varnothing \rightarrow X$) and $Y^{\bullet}$
  of $Y$ (i.e. of $\varnothing \rightarrow Y$) are homotopic. More precisely,
  there exists a simplicial homotopy from $g_0^{\bullet}$ to $g_1^{\bullet}$
  which is functorial in $g_0$ and $g_1$. In particular, if $X = Y$ and $g_0 =
  g_1$, then the simplicial homotopy is constant.
\end{lemma}

\begin{proof}
  We start with the special case that $\mathcal{C}$ is a $1$-category. We
  define the simplicial homotopy $h \of \delta^{\ast} (X^{\bullet})
  \rightarrow \delta^{\ast} (Y^{\bullet})$ as follows: for every $(\alpha_n
  \of [n] \rightarrow [1]) \in \tmmathbf{\Delta}_{/ [1]}$, we note that
  $(X^{\bullet}) (\alpha_n) = X^n = X \amalg \cdots \amalg X$ and
  $(Y^{\bullet}) (\alpha_n) = Y^n = Y \amalg \cdots \amalg Y$, and we set
  $h_{\alpha_n} = \coprod_{i = 0}^n g_{\alpha_n (i)} \of X^n \rightarrow Y^n$.
  By construction, $i_0^{\ast} (h)^{\nu} = h_{[\nu] \xrightarrow{0} [1]} =
  \coprod_{i = 0}^n g_0 = g_0^{\nu}$ and $i_1^{\ast} (h)^{\nu} = g_1^{\nu}$.
  
  We need to check that this is a map of functors. For every map $\psi \of
  (\alpha_n \of [n] \rightarrow [1]) \rightarrow (\alpha_m \of [m] \rightarrow
  [1])$ in $\tmmathbf{\Delta}_{/ [1]}$, we need to check that the diagram
  \[ \begin{array}{ccc}
       X^n & \xrightarrow{h_{\alpha_n}} & Y^n\\
       \longdownarrow \nocomma \psi_{\ast} &  & \longdownarrow \nocomma
       \psi_{\ast}\\
       X^m & \xrightarrow{h_{\alpha_m}} & Y^m
     \end{array} \]
  commutes, where the vertical maps $\psi_{\ast} \of X^n \rightarrow X^m$ and
  $\psi_{\ast} \of Y^n \rightarrow Y^m$ are induced by $\psi$. For this end,
  let $j_i \of X \rightarrow X^n$ be the $i$-th canonical map for $0 \leq i
  \leq n$.
  
  Then the composite $h_{\alpha_n} \circ j_i \of X \rightarrow X^n \rightarrow
  Y^n$ could be rewritten as the composite $j_i \circ g_{a_n (i)} \of X
  \rightarrow Y \rightarrow Y^n$, and the composite $\psi_{\ast} \circ
  h_{\alpha_n} \circ j_i \of X \rightarrow Y^m$ is equivalent to the composite
  $j_{\psi (i)} \circ g_{\alpha_n (i)} \of X \rightarrow Y \rightarrow Y^m$.
  Similarly, the composite $\psi_{\ast} \circ j_i \of X \rightarrow X^n
  \rightarrow X^m$ is equivalent to the $\psi (i)$-th canonical map $j_{\psi
  (i)} \of X \rightarrow X^m$, and the composite $h_{\alpha_m} \circ
  \psi_{\ast} \circ j_i$ could be identified with the composite $j_{\psi (i)}
  \circ g_{\alpha_m (\psi (i))} \of X \rightarrow Y \rightarrow Y^m$.
  
  Since $\alpha_m (\psi (i)) = \alpha_n (i)$, it follows that $\psi_{\ast}
  \circ h_{\alpha_n} \circ j_i = h_{\alpha_m} \circ \psi_{\ast} \circ j_i$ for
  every $0 \leq i \leq n$. It follows that $\psi_{\ast} \circ h_{\alpha_n} =
  h_{\alpha_m} \circ \psi_{\ast}$. The other claims for the $1$-category
  $\mathcal{C}$ follow directly from the construction.
  
  Now we claim that the result for $\infty$-categories follows from that for
  $1$-categories. The point is that there exists a universal
  $1$-category\footnote{This is informed to us by Denis {\tmname{Nardin}}.}
  $\mathcal{C}_0$ along with two objects $X_0, Y_0 \in \mathcal{C}$ and two
  maps $X_0 \rightrightarrows Y_0$, which admits all finite products, such
  that for every $\infty$-category $\mathcal{C}$ as in the assumption of this
  lemma, there exists an essentially unique functor $\mathcal{C}_0 \rightarrow
  \mathcal{C}$ which preserves finite coproducts: let $K$ be the diagram
  $\bullet \rightrightarrows \bullet$, and then take the presheaf
  $\infty$-category $\mathcal{P} (K) = \tmop{Fun} (K^{\tmop{op}}, \tmop{An})$.
  Then we can take $\mathcal{C}_0$ to be the full subcategory of $\mathcal{P}
  (K)$ spanned by finite coproducts of the two vertices of $K$.
\end{proof}

\begin{corollary}
  \label{cor:retract-Cech-homotopy-eq}Let $\mathcal{C}$ be an
  $\infty$-category with finite coproducts, and two objects $X, Y$ in
  $\mathcal{C}$. Let $i \of X \rightarrow Y$ be a map which admits a left
  inverse $r \of Y \rightarrow X$. Then there is a ``strong deformation
  retract'', i.e. a simplicial homotopy from $\tmop{id}_{Y^{\bullet}}$ to
  $i^{\bullet} \circ r^{\bullet}$, which restricts to a constant simplicial
  homotopy of $X^{\bullet}$ along $i^{\bullet} \of X^{\bullet} \rightarrow
  Y^{\bullet}$, where $X^{\bullet}$ (resp. $Y^{\bullet}$) is the Čech conerve
  of $X$ (resp. $Y$), and $i^{\bullet} \of X^{\bullet} \rightarrow
  Y^{\bullet}$ and $r^{\bullet} \of Y^{\bullet} \rightarrow X^{\bullet}$ are
  induced simplicial maps.
\end{corollary}

\begin{proof}
  We apply \Cref{lem:Cech-two-maps-homotopic} to $\tmop{id}_Y, i \circ r \in
  \tmop{Hom}_{\mathcal{C}} (Y, Y)$, getting the desired simplicial homotopy.
  To see the later statement, it suffices to inspect the commutative diagram
  \[ \begin{array}{ccc}
       Y & \underset{i \circ r}{\overset{\tmop{id}_Y}{\longrightrightarrows}}
       & Y\\
       \longuparrow &  & \longuparrow\\
       X &
       \underset{\tmop{id}_X}{\overset{\tmop{id}_X}{\longrightrightarrows}} &
       X
     \end{array} \]
  and invoke the functoriality.
\end{proof}

\begin{lemma}
  \label{lem:bhatt-fix-homotopy-eq}Let $A \in \tmop{CAlg}^{\tmop{cn}}$ be a
  connective $\mathbb{E}_{\infty}$-ring and let $B \rightarrow C$ be a
  faithfully flat map of connective $\mathbb{E}_{\infty}$-$A$-algebras. Let
  $B^{\bullet}$ (resp. $C^{\bullet}$) denote the Čech conerve of the map $A
  \rightarrow B$ (resp. $A \rightarrow C$). Then for any cosimplicial
  $B^{\bullet}$-module $N^{\bullet}$, the natural cosimplicial map
  \[ N^{\bullet} \longrightarrow N^{\bullet}
     \otimes_{B^{\bullet}}^{\mathbb{L}} C^{\bullet} \]
  induces an equivalence after totalization $\lim_{\bullet \in
  \tmmathbf{\Delta}}$in $D (A)$, where the cosimplicial map $B^{\bullet}
  \rightarrow C^{\bullet}$ is induced by $B \rightarrow C$.
\end{lemma}

\begin{proof}
  Let $D^{\bullet, \bullet}$ denote the cosimplicial Čech conerve of
  $B^{\bullet} \rightarrow C^{\bullet}$ (each $D^{\nu, \bullet}$ is the Čech
  conerve of $B^{\nu} \rightarrow C^{\nu}$), which is a bicosimplicial object
  in $\tmop{CAlg}_A$. We note that there is a unique cosimplicial map
  $B^{\bullet} \rightarrow D^{\bullet, \mu}$ for all $[\mu] \in
  \tmmathbf{\Delta}$. Consider the bicosimplicial object $M^{\bullet,
  \bullet}$:
  \begin{eqnarray*}
    \tmmathbf{\Delta}^2 & \longrightarrow & D (A)\\
    ([\nu], [\mu]) & \longmapsto & N^{\nu} \otimes_{B^{\nu}}^{\mathbb{L}}
    D^{\nu, \mu}
  \end{eqnarray*}
  and its limit $I \assign \lim_{([\nu], [\mu])} M^{\nu, \mu}$. The map which
  we need to show to be an equivalence factors as $\lim_{[\nu]} N^{\nu}
  \rightarrow \lim_{([\nu], [\mu])} M^{\nu, \mu} \rightarrow \lim_{[\nu]}
  M^{\nu, 0} \simeq \lim_{[\nu]} N^{\nu} \otimes_{B^{\nu}}^{\mathbb{L}}
  C^{\nu}$. It suffices to show that both maps are equivalences.
  
  For the first map, in fact, for every $[\nu] \in \tmmathbf{\Delta}$, the map
  $N^{\nu} \rightarrow \lim_{[\mu] \in \tmmathbf{\Delta}} M^{\nu, \mu}$ is an
  equivalence by faithfully flat descent.
  
  For the second map, since $\tmmathbf{\Delta}_{\tmop{inj}}^{\tmop{op}}
  \rightarrow \tmmathbf{\Delta}^{\tmop{op}}$ is cofinal
  {\cite[Lem~6.5.3.7]{Lurie2009}} where $\tmmathbf{\Delta}_{\tmop{inj}}
  \subseteq \tmmathbf{\Delta}$ is the (non-full) subcategory with strictly
  increasing maps $[m] \rightarrow [n]$, we can replace
  $\lim_{\tmmathbf{\Delta}}  (\cdummy)$ by
  $\lim_{\tmmathbf{\Delta}_{\tmop{inj}}}  (\cdummy)$. By
  {\cite[Cor~4.4.4.10]{Lurie2009}}, it suffices to show that, for every
  injective map $[\mu_1] \rightarrow [\mu_2]$ in $\tmmathbf{\Delta}$, the
  induced map $\lim_{[\nu]} M^{\nu, \mu_1} \rightarrow \lim_{[\nu]} M^{\nu,
  \mu_2}$ is an equivalence. Every injective map $[\mu_1] \rightarrow [\mu_2]$
  admits a retract in $\tmmathbf{\Delta}$, therefore by
  \Cref{cor:retract-Cech-homotopy-eq}, the induced map $D^{\bullet, \mu_1}
  \rightarrow D^{\bullet, \mu_2}$ is a homotopy equivalence of cosimplicial
  $\mathbb{E}_{\infty}$-$B^{\bullet}$-algebras, therefore $M^{\bullet, \mu_1}
  \rightarrow M^{\bullet, \mu_2}$ is a homotopy equivalence of cosimplicial
  $A$-modules by {\cite[\href{https://stacks.math.columbia.edu/tag/07KQ}{Tag
  07KQ}]{stacks-project}}. The result then follows from
  \Cref{lem:simpl-homot-maps-eq}.
\end{proof}

\begin{remark}
  When $A$ is a static $\mathbb{F}_p$-algebra, $B$ is a polynomial $A$-algebra
  and $C = B \otimes_{A, \varphi_A}^{\mathbb{L}} A$ is the Frobenius twist of
  $B$, we recover {\cite[Lem~5.4]{Bhatt2019}}.
\end{remark}

\begin{lemma}
  \label{lem:fil-colim-tot-commute}Let $(M_i^{\bullet})_{i \in
  (\mathbb{Z}_{\leq 0}, \geq)} \in \tmop{Fun} (\tmmathbf{\Delta} \times
  (\mathbb{Z}_{\leq 0}, \geq), \tmop{Sp})$ be a cosimplicial filtered spectra.
  Suppose that it is uniformly bounded above, i.e. there exists $N \in
  \mathbb{N}$ such that for every $i \in \mathbb{Z}_{\leq 0}$ and $\nu \in
  \tmmathbf{\Delta}$, we have $\pi_j (M_i^{\nu}) \cong 0$ for all $j > N$. Let
  $M^{\bullet} \assign \tmop{colim}_{i \rightarrow - \infty} M_i^{\bullet}$.
  Then the canonical map $\tmop{colim}_{i \rightarrow - \infty} \lim_{\nu \in
  \tmmathbf{\Delta}} M_i^{\nu} \rightarrow \lim_{\nu \in \tmmathbf{\Delta}}
  M^{\nu}$ is an equivalence.
\end{lemma}

\begin{proof}
  We could rewrite $\lim_{\nu \in \tmmathbf{\Delta}}$ as $\lim_{n \rightarrow
  \infty} \lim_{\nu \in \tmmathbf{\Delta}_{\leq [n]}}$. Furthermore, the
  functor $\tmmathbf{\Delta}_{\leq [n]}^{\tmop{inj}} \rightarrow
  \tmmathbf{\Delta}_{\leq [n]}$ is right cofinal, therefore we can replace
  $\lim_{\nu \in \tmmathbf{\Delta}_{\leq [n]}}$ by $\lim_{\nu \in
  \tmmathbf{\Delta}_{\leq [n]}^{\tmop{inj}}}$ which is a finite limit,
  therefore commutes with $\tmop{colim}_{i \rightarrow - \infty}$. For any
  cosimplicial spectrum $X^{\bullet}$, there is a canonical map $\lim_{\nu \in
  \tmmathbf{\Delta}} X^{\nu} \rightarrow \lim_{\nu \in \tmmathbf{\Delta}_{\leq
  [n]}} X^{\nu}$. If $X^{\bullet}$ is assumed to be uniformly bounded above,
  then the coconnectivity of $\tmop{fib} (\lim_{\nu \in \tmmathbf{\Delta}}
  X^{\nu} \rightarrow \lim_{\nu \in \tmmathbf{\Delta}_{\leq [n]}} X^{\nu})$
  tends to $- \infty$ as $n \rightarrow \infty$ by
  {\cite[Cor~1.2.4.18]{Lurie2017}}. The result then follows.
\end{proof}

For the integral version, we need to extend quasisyntomicity to animated
rings:

\begin{definition}[cf. {\cite[Def~4.9]{Bhatt2018}}]
  \label{def:qsyn}We say that a map $R \rightarrow S$ of animated rings is
  {\tmdfn{quasisyntomic}} if it is flat and the cotangent complex $L_{S / R}$
  has $\tmop{Tor}$-amplitude in $[0, 1]$.
\end{definition}

\begin{example}
  Any smooth map, or more generally, any syntomic map of static rings is
  quasisyntomic. Quasisyntomic maps are stable under base change.
\end{example}

We now phrase the integral comparison:

\begin{proposition}
  \label{prop:site-comp-qsyn-integral}Let $((A \twoheadrightarrow A'',
  \gamma_A), A'' \rightarrow R) \in \tmop{CrysCon}$ such that $A$ is bounded
  above (that is, $\pi_n (A) \cong 0$ for $n \gg 0$) and the map $A''
  \rightarrow R$ is quasisyntomic. Then the comparison map in
  \Cref{prop:deriv-crys-coh-site-comp} is an equivalence.
\end{proposition}

\begin{proof}
  We again appeal to \Cref{prop:crys-site-comp-desc}. It suffices to show that
  the map
  \[ \tmop{CrysCoh}_{R / (A \twoheadrightarrow A'', \gamma_A)} \longrightarrow
     \lim_{\nu \in \tmmathbf{\Delta}} \tmop{CrysCoh}_{R / (P^{\nu}
     \twoheadrightarrow P^{\nu} \otimes_A^{\mathbb{L}} A'', \gamma_{P^{\nu}})}
  \]
  is an equivalence of $\mathbb{Z}$-module spectra (since the forgetful
  functor is conservative), which could be checked by base change along
  $\mathbb{Z} \rightarrow \mathbb{Z}/ p$ for all prime numbers $p \in
  \mathbb{N}_{> 0}$ and along $\mathbb{Z} \rightarrow \mathbb{Q}$. The latter
  follows from \Cref{lem:dR-rational} and that the map $A \rightarrow P$ is
  faithfully flat therefore the canonical map $A \rightarrow \lim_{\nu \in
  \tmmathbf{\Delta}} P^{\nu}$ is an equivalence (in fact, this is induced by a
  homotopy equivalence of cosimplicial objects, but we do not need this). For
  every prime number $p$, by base change property
  (\Cref{lem:crys-coh-indep-base-chg}) and \Cref{lem:crys-coh-char-p-desc},
  where the flatness of $A'' \rightarrow R$ implies the flatness of $A
  /^{\mathbb{L}} p \rightarrow \varphi_{A /^{\mathbb{L}} p \twoheadrightarrow
  A'' /^{\mathbb{L}} p}^{\ast} (R /^{\mathbb{L}} p)$, therefore the Frobenius
  twist in question is bounded above.
\end{proof}

Finally, we want to compare the cohomology of the affine crystalline site and
the classical crystalline cohomology. We first describe a non-complete variant
of the classical affine crystalline site, which we will name after
{\tmdfn{static affine crystalline site}}.

\begin{definition}
  \label{def:static-aff-crys-site}Let $(A, I, \gamma_A) \in
  \tmop{Pair}^{\gamma}$ be a PD-pair and let $A / I \rightarrow R$ be a map of
  rings. Note that $((A \twoheadrightarrow A / I, \gamma_A), A / I \rightarrow
  R) \in \tmop{CrysCon}$ is a crystalline context. The {\tmdfn{static affine
  crystalline site}} $\tmop{Cris}^{\tmop{st}} (R / (A, I, \gamma_A))$ is the
  full subcategory of $\tmop{Cris} (R / (A \twoheadrightarrow A / I,
  \gamma_A))$ spanned by those $(B \twoheadrightarrow B / J, \gamma_B) \in
  \tmop{Pair}^{\gamma}$ along with a map $R \rightarrow B / J$, i.e., the
  animated PD-pair in question is given by a PD-pair, equipped with the
  indiscrete topology.
  
  We note that the structure presheaf $\mathcal{O}$ on $\tmop{Cris} (R / (A
  \twoheadrightarrow A / I, \gamma_A))$ restricts to a presheaf
  $\tmop{Cris}^{\tmop{st}} (R / (A, I, \gamma_A))$, still called the
  {\tmdfn{structure presheaf}}, which is canonically equipped with
  PD-filtration, of which the cohomology is called the {\tmdfn{cohomology of
  the static crystalline site}} (resp. {\tmdfn{Hodge-filtered cohomology of
  the static crystalline site}}), denoted by $R \Gamma
  (\tmop{Cris}^{\tmop{st}} (R / (A, I, \gamma_A)), \mathcal{O})$ (resp.
  $\tmop{Fil}_H R \Gamma (\tmop{Cris}^{\tmop{st}} (R / (A, I, \gamma_A)),
  \mathcal{O})$.
\end{definition}

\begin{warning}
  Here the PD-filtration is that for animated PD-envelope, although we are
  considering PD-pairs. However, when $I = 0$, thanks to
  \Cref{prop:comp-PDFil-equiv-qreg}, we can consider the classical PD-envelope
  instead.
\end{warning}

Now the cohomology of the affine crystalline site coincides with the classical
version:

\begin{proposition}
  \label{prop:site-static-comp}Let $(A, I, \gamma_A) \in \tmop{Pair}^{\gamma}$
  be a PD-pair and $A / I \rightarrow R$ a quasisyntomic map of rings ($R$ is
  static by flatness). Then the comparison map
  \[ \tmop{Fil}_H R \Gamma (\tmop{Cris} (R / (A \twoheadrightarrow A / I,
     \gamma_A)), \mathcal{O}) \rightarrow \tmop{Fil}_H R \Gamma
     (\tmop{Cris}^{\tmop{st}} (R / (A, I, \gamma_A)), \mathcal{O}) \]
  of filtered $\mathbb{E}_{\infty}$-$A$-algebras induced by the inclusion
  $\tmop{Cris}^{\tmop{st}} (R / (A, I, \gamma_A)) \hookrightarrow \tmop{Cris}
  (R / (A, I, \gamma_A))$ is an equivalence.
\end{proposition}

\begin{proof}
  We adapt the Čech--Alexander computation in
  \Cref{prop:deriv-crys-coh-site-comp}. We pick a polynomial $A$-algebra $P$
  (of possibly infinitely many variables) along with a surjection $P
  \twoheadrightarrow R$. Let $P^{\bullet} \rightarrow R$ denote the Čech
  conerve of the object $P \rightarrow R$ in $(\tmop{Alg}_A)_{/ R}$.
  Concretely, $P^{\nu} = P^{\otimes_A (\nu + 1)}$. Note that since $A
  \rightarrow P$ is flat, the classical tensor product coincides with the
  derived tensor product, therefore the cosimplicial pair $P^{\bullet}
  \rightarrow R$ coincides with the cosimplicial animated pair in the proof of
  \Cref{prop:deriv-crys-coh-site-comp}, and then $\tmop{Fil}_H R \Gamma
  (\tmop{Cris}^{\tmop{st}} (R / (A, I, \gamma_A)), \mathcal{O})$ is computed
  by the classical PD-envelope of $P^{\bullet} \twoheadrightarrow R$ with
  respect to $(A, I, \gamma_A)$, equipped with the PD-filtration.
  
  Let $(D^{\bullet} \twoheadrightarrow R, \gamma_{D^{\bullet}})$ denote the
  cosimplicial animated PD-envelope of $(P^{\bullet} \twoheadrightarrow R)$
  relative to $(A, I, \gamma_A)$. It suffices to show that $(D^{\nu}
  \twoheadrightarrow R, \gamma_{D^{\bullet}})$ is given by a PD-pair for all
  $\nu \in \tmmathbf{\Delta}$, or equivalently, the underlying animated ring
  $D^{\nu}$ is static, by virtue of
  \Cref{prop:characterize-pdpair,lem:ani-pd-env-preserv-target}, which follows
  from \Cref{lem:flat-pd-env} below.
\end{proof}

\begin{lemma}
  \label{lem:flat-pd-env}Let $(A \twoheadrightarrow A'', \gamma_A)$ be an
  animated PD-pair, $A'' \rightarrow R$ a quasisyntomic map of animated rings,
  $P$ a polynomial $A$-algebra (of possibly infinitely many variables) and $P
  \twoheadrightarrow R$ a surjection of $A$-algebras. Let $(D
  \twoheadrightarrow R, \gamma_D)$ denote the animated PD-envelope of $P
  \twoheadrightarrow R$ relative to $(A \twoheadrightarrow A'', \gamma_A)$.
  Then $D$ is a flat $A$-module.
\end{lemma}

\begin{proof}
  This is a ``quasi'' variant of ``flatness of PD-envelope''
  {\cite[Lem~2.42]{Bhatt2019}}. By \Cref{lem:flat-loc-global}, it suffices to
  show that $D \otimes_{\mathbb{Z}}^{\mathbb{L}} \mathbb{Q}$ is a flat $A
  \otimes_{\mathbb{Z}}^{\mathbb{L}} \mathbb{Q}$-module, and for every prime $p
  \in \mathbb{N}$, $D /^{\mathbb{L}} p$ is a flat $A /^{\mathbb{L}} p$-module.
  
  By \Cref{lem:rel-pd-env-rat}, the map $P \otimes_{\mathbb{Z}}^{\mathbb{L}}
  \mathbb{Q} \rightarrow D \otimes_{\mathbb{Z}}^{\mathbb{L}} \mathbb{Q}$ is an
  equivalence. Since $A \rightarrow P$ is flat, so is the map $A
  \otimes_{\mathbb{Z}}^{\mathbb{L}} \mathbb{Q} \rightarrow D
  \otimes_{\mathbb{Z}}^{\mathbb{L}} \mathbb{Q}$.
  
  For every prime $p \in \mathbb{N}$, by base change property (a relative
  version of \Cref{lem:PD-env-compat-base-chg}, with a similar proof), $D_0
  \assign D /^{\mathbb{L}} p$ is the animated PD-envelope of $P /^{\mathbb{L}}
  p \twoheadrightarrow R /^{\mathbb{L}} p$ relative to the animated PD-pair
  $(A \twoheadrightarrow A'', \gamma_A)$. To simplify notations, we let $P_0
  \assign P /^{\mathbb{L}} p, R_0 \assign R /^{\mathbb{L}} p, A_0 \assign A
  /^{\mathbb{L}} p, A_0'' \assign A'' /^{\mathbb{L}} p$. Since $A \rightarrow
  R$ is quasisyntomic, so is $A_0 \rightarrow R_0$. Consider the transitivity
  sequence
  \[ L_{P_0 / A_0} \otimes_{P_0}^{\mathbb{L}} R_0 \longrightarrow L_{R_0 /
     A_0} \longrightarrow L_{R_0 / P_0} \]
  For every static $R_0$-module $M$, we get a fiber sequence
  \[ L_{P_0 / A_0} \otimes_{P_0}^{\mathbb{L}} M \longrightarrow L_{R_0 / A_0}
     \otimes_{R_0}^{\mathbb{L}} M \longrightarrow L_{R_0 / P_0}
     \otimes_{R_0}^{\mathbb{L}} M \]
  Since $P_0$ is a polynomial $A_0$-algebra, $\mathbb{L}_{P_0 / A_0}$ is a
  flat $P_0$-module. The map $A_0 \rightarrow R_0$ is quasisyntomic, therefore
  $\pi_{\ast} (L_{R_0 / A_0} \otimes_{R_0}^{\mathbb{L}} M) \cong 0$ for
  $\mathord{\ast} \neq 0, 1$. It follows that $\pi_{\ast} (L_{R_0 / P_0}
  \otimes_{R_0}^{\mathbb{L}} M) \cong 0$ for $\mathord{\ast} \neq 0, 1$.
  Furthermore, since $P_0 \rightarrow R_0$ is surjective, $\pi_0 (L_{R_0 /
  P_0} \otimes_{R_0}^{\mathbb{L}} M) \cong 0$. It follows that $P_0
  \twoheadrightarrow R_0$ is a quasiregular animated pair. By
  \Cref{cor:Fp-qreg-rel-pd-env-flat}, $D_0$ is a flat $\varphi_{P_0
  \twoheadrightarrow P_0 \otimes_{A_0}^{\mathbb{L}} A_0''}^{\ast}
  (R_0)$-module where $\varphi_{P_0 \twoheadrightarrow P_0
  \otimes_{A_0}^{\mathbb{L}} A_0''} \of P_0 \otimes_{A_0}^{\mathbb{L}} A_0''
  \rightarrow P_0$ is the Frobenius map (\Cref{lem:PD-frob}). It remains to
  see that the composite map $A_0 \rightarrow P_0 \rightarrow \varphi_{P_0
  \twoheadrightarrow P_0 \otimes_{A_0}^{\mathbb{L}} A_0''}^{\ast} (R_0) = R_0
  \otimes_{P_0 \otimes_{A_0}^{\mathbb{L}} A_0''}^{\mathbb{L}} P_0$ is flat,
  where the second map is the ``map into the second factor''.
  
  We note that the Frobenius $\varphi_{P_0 \twoheadrightarrow P_0
  \otimes_{A_0}^{\mathbb{L}} A_0''}$ factors as $P_0
  \otimes_{A_0}^{\mathbb{L}} A_0'' \rightarrow \varphi_{A_0}^{\ast} (P_0)
  \rightarrow P_0$ where the second map is the Frobenius of $P_0$ relative to
  $A_0$. Let $R_1$ denote $R_0 \otimes_{A_0'', \varphi_{A_0 \twoheadrightarrow
  A_0''}}^{\mathbb{L}} A_0$. Since $A_0'' \rightarrow R_0$ is flat, so is $A_0
  \rightarrow R_1$, and we have
  \[ \varphi_{P_0 \twoheadrightarrow P_0 \otimes_{A_0}^{\mathbb{L}}
     A_0''}^{\ast} (R_0) \simeq R_1 \otimes_{\varphi_{A_0}^{\ast}
     (P_0)}^{\mathbb{L}} P_0 \]
  as a pushout of $A_0$-algebras. The relative Frobenius $\varphi_{A_0}^{\ast}
  (P_0) \rightarrow P_0$ is flat, therefore the map $R_1 \rightarrow R_1
  \otimes_{\varphi_{A_0}^{\ast} (P_0)}^{\mathbb{L}} P_0$. The result then
  follows since flatness is stable under composition.
\end{proof}

\begin{remark}
  If we examine the proof of \Cref{lem:flat-pd-env} closely, we see that,
  instead of being a polynomial, what we really need to impose on the map $A
  \rightarrow P$ is that the map is {\tmdfn{quasismooth}} (i.e. it is flat and
  $L_{P / A}$ is a flat $P$-module), and for every prime $p \in \mathbb{N}$,
  the Frobenius of $P /^{\mathbb{L}} p$ relative to $A /^{\mathbb{L}} p$ is
  flat.
\end{remark}

\section{Animated prismatic structures}\label{sec:prism}

We fix a prime $p \in \mathbb{N}$. In this section, we will develop the theory
of animated $\delta$-rings, that of animated $\delta$-pairs, a non-complete
theory of prisms and prove a variant of the Hodge--Tate comparison, from which
we deduce a result about ``flat covers of the final object''. Almost every
ring that we will discuss is a $\mathbb{Z}_{(p)}$-algebra, we will simply
denote $\tmop{Pair}^{\tmop{an}}_{\mathbb{Z}_{(p)}}$ by
$\tmop{Pair}^{\tmop{an}}$ and $\tmop{Pair}^{\gamma,
\tmop{an}}_{\mathbb{Z}_{(p)}}$ by $\tmop{Pair}^{\gamma, \tmop{an}}$.

\subsection{Animated $\delta$-rings and $\delta$-pairs}In this section, we
will define {\tmdfn{animated $\delta$-rings}} and {\tmdfn{animated
$\delta$-pairs}} and discuss the interaction between the $\delta$-structure
and the PD-structure. Recall that

\begin{definition}[{\cite[Def~2.1]{Bhatt2019}}]
  A {\tmdfn{$\delta$-ring}} is a pair $(R, \delta)$ where $R$ is a
  $\mathbb{Z}_{(p)}$-algebra and $\delta \of R \rightarrow R$ is an
  endomorphism of the underlying set $R$ such that
  \begin{enumerate}
    \item $\delta (x + y) = \delta (x) + \delta (y) - P (x, y)$ for all $x, y
    \in R$ where $P (X, Y) \in \mathbb{Z} [X, Y]$ is the polynomial
    \[ \frac{(X + Y)^p - X^p - Y^p}{p} \assign \sum_{j = 1}^{p - 1}
       \frac{1}{p}  \binom{p}{j} X^{p - j} Y^j \]
    \item $\delta (xy) = x^p \delta (y) + y^p \delta (x) + p \delta (x) \delta
    (y)$.
    
    \item $\delta (1) = 0$.
  \end{enumerate}
  A {\tmdfn{map $f \of (R, \delta) \rightarrow (S, \delta)$ of
  $\delta$-rings}} is a map $f \of R \rightarrow S$ of rings such that $f
  \circ \delta = \delta \circ g$ as maps of sets. These form the
  {\tmdfn{$1$-category of $\delta$-rings}}, denoted by $\tmop{Ring}_{\delta}$.
\end{definition}

\begin{remark}[{\cite[Rem~2.2]{Bhatt2019}}]
  \label{rem:delta-ring-frob}Given a $\delta$-ring $(R, \delta)$, we write
  $\varphi \of R \rightarrow R$ for the map $x \mapsto x^p + p \delta (x)$.
  Then $\varphi$ is a ring endomorphism of $R$ which lifts the Frobenius map
  $R / p \rightarrow R / p$, i.e. $\varphi (x) \equiv x^p \hspace{0.2em}
  \pmod{p}$ for every $x \in R$.
\end{remark}

The $1$-category $\tmop{Ring}_{\delta}$ admits an initial object
$\mathbb{Z}_{(p)}$ {\cite[Ex~2.6]{Bhatt2019}}, and more generally, all small
colimits and small limits, and the forgetful functor $\tmop{Ring}_{\delta}
\rightarrow \tmop{Alg}_{\mathbb{Z}_{(p)}}$ preserves them
{\cite[Rem~2.7]{Bhatt2019}}. The forgetful functor $\tmop{Ring}_{\delta}
\rightarrow \tmop{Set}$ admits a left adjoint $\tmop{Set} \rightarrow
\tmop{Ring}_{\delta}$, which sends a set $S$ to the free $\delta$-ring
generated by $S$, denoted by $\mathbb{Z}_{(p)} \{ S \}$. Indeed, when $S = \{
x \}$ is a singleton, it is given by the free $\delta$-ring $\mathbb{Z}_{(p)}
\{ x \}$ of which the underlying $\mathbb{Z}_{(p)}$-algebra is isomorphic to
the polynomial $\mathbb{Z}_{(p)}$-algebra $\mathbb{Z}_{(p)} [x, \delta (x),
\delta^2 (x), \ldots]$ {\cite[Lem~2.11]{Bhatt2019}}, and the general case
follows by taking the coproduct of $S$-copies of $\mathbb{Z}_{(p)} \{ x \}$.
It then follows from \Cref{cor:ani-adjoint-funs} that

\begin{lemma}
  \label{lem:delta-ring-1-proj-gen}The $1$-category $\tmop{Ring}_{\delta}$ is
  $1$-projectively generated, therefore presentable.
\end{lemma}

\begin{definition}
  \label{def:ani-delta-ring}The {\tmdfn{$\infty$-category of animated
  $\delta$-rings}}, denoted by $\tmop{CAlg}_{\delta}^{\tmop{an}}$, is defined
  to be the animation $\tmop{Ani} (\tmop{Ring}_{\delta})$. An {\tmdfn{animated
  $\delta$-ring}} is formally denoted by $(R, \delta)$, where $R$ is the image
  of $(R, \delta)$ under the forgetful functor
  $\tmop{CAlg}_{\delta}^{\tmop{an}} \rightarrow
  \tmop{CAlg}_{\mathbb{Z}_{(p)}}^{\tmop{an}}$, or simply by $R$ when the
  $\delta$-structure is unambiguously obvious.
\end{definition}

\begin{remark}
  In {\cite[App~A]{Bhatt2022a}}, they give an alternative description of
  animated $\delta$-rings in terms of derived Frobenius lift, similar to
  \Cref{rem:delta-ring-frob}.
\end{remark}

By the adjoint functor theorem, the forgetful functor $\tmop{Ring}_{\delta}
\rightarrow \tmop{Alg}_{\mathbb{Z}_{(p)}}$ admits a left adjoint. A further
application of \Cref{cor:ani-adjoint-funs} leads to

\begin{lemma}
  \label{lem:ani-ring-delta-ring-adjoint}There is a pair
  $\tmop{CAlg}_{\mathbb{Z}_{(p)}}^{\tmop{an}} \rightleftarrows
  \tmop{CAlg}_{\delta}^{\tmop{an}}$ of adjoint functors, being the animation
  of the pair $\tmop{Alg}_{\mathbb{Z}_{(p)}} \rightleftarrows
  \tmop{Ring}_{\delta}$ of adjoint functors. We will call the functor
  $\tmop{CAlg}_{\delta}^{\tmop{an}} \rightarrow
  \tmop{CAlg}_{\mathbb{Z}_{(p)}}^{\tmop{an}}$ the {\tmdfn{free animated
  $\delta$-ring functor}}\footnote{The non-animated version was called the
  ``$\delta$-envelope'' in {\cite[Def~1.1]{Gros2020}}.}. The functor
  $\tmop{CAlg}_{\delta}^{\tmop{an}} \rightarrow
  \tmop{CAlg}_{\mathbb{Z}_{(p)}}^{\tmop{an}}$, called the {\tmdfn{forgetful
  functor}}, is conservative and preserves small colimits (and as a right
  adjoint, it preserves small limits as well).
\end{lemma}

Concretely, a set of compact projective generators for
$\tmop{CAlg}_{\delta}^{\tmop{an}}$ is given by free $\delta$-rings generated
by a finite set, which spans a full subcategory $\tmop{Ring}_{\delta}^0
\subseteq \tmop{Ring}_{\delta}$. Recall that $\tmop{Ring}_{\delta}
\hookrightarrow \tmop{CAlg}_{\delta}^{\tmop{an}}$ is a full subcategory
(\Cref{rem:ani-trunc}). Now we characterize this full subcategory in terms of
the underlying animated ring:

\begin{lemma}
  \label{lem:ani-delta-ring-trunc}Let $(R, \delta) \in
  \tmop{CAlg}_{\delta}^{\tmop{an}}$ be an animated $\delta$-ring. Then the
  followings are equivalent:
  \begin{enumerate}
    \item \label{item:delta-R-trunc}The animated $\delta$-ring $(R, \delta)
    \in \tmop{CAlg}_{\delta}^{\tmop{an}}$ is $n$-truncated.
    
    \item \label{item:underlying-R-trunc}The underlying animated ring $R \in
    \tmop{CAlg}_{\mathbb{Z}_{(p)}}^{\tmop{an}}$ is $n$-truncated.
    
    \item \label{item:homotopy-grps-vanish}For every $m \in \mathbb{N}_{> n}$,
    the homotopy group $\pi_m (R)$ vanishes.
  \end{enumerate}
\end{lemma}

\begin{proof}
  The equivalence of
  parts~\ref{item:underlying-R-trunc}~and~\ref{item:homotopy-grps-vanish} is
  {\cite[Prop~25.1.3.3]{Lurie2018}}. On the other hand,
  part~\ref{item:delta-R-trunc} is equivalent to say that, for every free
  $\delta$-ring $F$ generated by a finite set, the mapping anima
  $\tmop{Map}_{\tmop{CAlg}_{\delta}^{\tmop{an}}} (F, R)$ is $n$-truncated by
  {\cite[Rem~5.5.8.26]{Lurie2009}}. Since any such $F$ is a finite coproduct
  of $\mathbb{Z}_{(p)} \{ x \}$, it is equivalent to
  $\tmop{Map}_{\tmop{CAlg}_{\delta}^{\tmop{an}}} (\mathbb{Z}_{(p)} \{ x \},
  R)$ being $n$-truncated, which is equivalent to
  part~\ref{item:homotopy-grps-vanish} since
  \[ \tmop{Map}_{\tmop{CAlg}_{\delta}^{\tmop{an}}} (\mathbb{Z}_{(p)} \{ x \},
     R) \simeq \tmop{Map}_{\tmop{An}} (\{ x \}, R) \simeq R \]
\end{proof}

Now we define the Frobenius map on animated $\delta$-rings. We note that the
identity functor $\tmop{id} \of \tmop{CAlg}_{\delta}^{\tmop{an}} \rightarrow
\tmop{CAlg}_{\delta}^{\tmop{an}}$ is the animation of the identity functor
$\tmop{id} \of \tmop{Ring}_{\delta} \rightarrow \tmop{Ring}_{\delta}$.

\begin{definition}
  The {\tmdfn{Frobenius endomorphism}} is the endomorphism of the identity
  functor $\tmop{id} \of \tmop{CAlg}_{\delta}^{\tmop{an}} \rightarrow
  \tmop{CAlg}_{\delta}^{\tmop{an}}$ defined to be the animation of the
  Frobenius endomorphism (described in \Cref{rem:delta-ring-frob}) of the
  identity functor $\tmop{id} \of \tmop{Ring}_{\delta} \rightarrow
  \tmop{Ring}_{\delta}$.
\end{definition}

Recall that a {\tmdfn{$\delta$-pair}} is the datum $(A, I)$ of a $\delta$-ring
$A$ along with an ideal $I \subseteq A$ {\cite[Def~3.2]{Bhatt2019}}. Similar
to animated pairs, we have an ``animated version'' of $\delta$-pairs:

\begin{definition}
  \label{def:ani-delta-pair}The $\infty$-category
  $\tmop{Pair}^{\tmop{an}}_{\delta}$ of {\tmdfn{animated $\delta$-pairs}} is
  defined to be the fiber product $\tmop{CAlg}_{\delta}^{\tmop{an}}
  \times_{\tmop{CAlg}_{\mathbb{Z}_{(p)}}^{\tmop{an}}} \tmop{Pair}^{\tmop{an}}$
  where the functor $\tmop{CAlg}_{\delta}^{\tmop{an}} \rightarrow
  \tmop{CAlg}_{\mathbb{Z}_{(p)}}^{\tmop{an}}$ is the forgetful functor and the
  functor $\tmop{Pair}^{\tmop{an}} \rightarrow
  \tmop{CAlg}_{\mathbb{Z}_{(p)}}^{\tmop{an}}$ is the evaluation $(A
  \twoheadrightarrow A'') \mapsto A$. An {\tmdfn{animated $\delta$-pair}} is
  an object in $\tmop{Pair}^{\tmop{an}}_{\delta}$ which we will denote by
  $((A, \delta), A \twoheadrightarrow A'')$, or simply by $A
  \twoheadrightarrow A''$ when there is no ambiguity.
\end{definition}

It follows from \Cref{lem:ani-ring-delta-ring-adjoint} and
{\cite[Lem~5.4.5.5]{Lurie2009}} which characterizes colimits in the fiber
products, that

\begin{lemma}
  \label{lem:forget-ani-delta-pair-ani-pair}The $\infty$-category
  $\tmop{Pair}^{\tmop{an}}_{\delta}$ is cocomplete, and the forgetful functor
  $\tmop{Pair}^{\tmop{an}}_{\delta} \rightarrow \tmop{Pair}^{\tmop{an}}$ is
  conservative and preserves small colimits.
\end{lemma}

Explicitly, an {\tmdfn{animated $\delta$-pair}} is given by an animated
$\delta$-ring $(A, \delta)$ along with a surjection $A \twoheadrightarrow A''$
of animated $\mathbb{Z}_{(p)}$-algebras. Since $\tmop{Pair} \subseteq
\tmop{Pair}^{\tmop{an}}$ is a full subcategory
(\Cref{prop:infty-cat-forget-embedding}) and so is $\tmop{Ring}_{\delta}
\subseteq \tmop{CAlg}_{\delta}^{\tmop{an}}$ (\Cref{rem:ani-trunc}), the
$1$-category of $\delta$-pairs is a full subcategory of the $\infty$-category
of animated $\delta$-pairs. Similar to the $\infty$-category of animated
pairs, we have

\begin{lemma}
  \label{lem:adjunct-ani-pair-ani-delta-pair}The forgetful functor
  $\tmop{Pair}^{\tmop{an}}_{\delta} \rightarrow \tmop{Pair}^{\tmop{an}}$
  admits a left adjoint, and the $\infty$-category
  $\tmop{Pair}^{\tmop{an}}_{\delta}$ is projectively generated.
\end{lemma}

\begin{proof}
  The left adjoint $\tmop{Pair}^{\tmop{an}} \rightarrow
  \tmop{Pair}^{\tmop{an}}_{\delta}$ concretely given by $(A \twoheadrightarrow
  A'') \mapsto ((A^{\delta}, \delta), (A^{\delta} \twoheadrightarrow A''
  \otimes_A^{\mathbb{L}} A^{\delta}))$ where $A^{\delta}$ is the image of $A
  \in \tmop{CAlg}_{\mathbb{Z}_{(p)}}^{\tmop{an}}$ under the free animated
  $\delta$-ring functor $\tmop{CAlg}_{\mathbb{Z}_{(p)}}^{\tmop{an}}
  \rightarrow \tmop{CAlg}_{\delta}^{\tmop{an}}$. Now the result follows from
  \Cref{cor:ani-adjoint-funs,lem:delta-ring-1-proj-gen,lem:ani-ring-delta-ring-adjoint}.
\end{proof}

Concretely, a set of compact projective generators for
$\tmop{Pair}^{\tmop{an}}_{\delta}$ is given by the set $\{ (\mathbb{Z}_{(p)}
\{ X, Y \}, (Y)) \barsuchthat X, Y \in \tmop{Fin} \}$ of {\tmdfn{standard
$\delta$-pairs}}, which spans a full subcategory
$\tmop{Pair}^{\tmop{st}}_{\delta} \subseteq \tmop{Pair}^{\tmop{an}}_{\delta}$.
Now we turn to the PD-structure. Recall that

\begin{lemma}[{\cite[Lem~2.11]{Bhatt2019}}]
  \label{lem:frob-free-delta-flat}The Frobenius endomorphism
  $\varphi_{\mathbb{Z}_{(p)} \{ x \}} \of \mathbb{Z}_{(p)} \{ x \} \rightarrow
  \mathbb{Z}_{(p)} \{ x \}$ on the free $\delta$-ring $\mathbb{Z}_{(p)} \{ x
  \}$, which is in fact induced by $x \mapsto \varphi (x) = x^p + p \delta
  (x)$, is faithfully flat. The same holds for free $\delta$-rings generated
  by arbitrary sets (not-necessarily finite).
\end{lemma}

We remark that, thanks to \Cref{lem:flat-loc-global}, it is not necessary to
pass to the polynomial ring of finitely many variables to invoke the fiberwise
criterion of flatness.

We now relate $\delta$-structure to divided powers. Note that, for any
$p$-torsion free $\mathbb{Z}_{(p)}$-algebra $A$, any element $y \in A$ and any
$n \in \mathbb{N}$, we have
\[ \frac{y^n}{n!} \in \frac{y^n}{p^{v_p (n!)}} \tmop{GL}_1 (\mathbb{Z}_{(p)})
\]
In particular, $y^p / p!$ (resp. $y^{p^2} / (p^2) !$) differs multiplicatively
from $y^p / p$ (resp. $y^{p^2} / p^{p + 1}$) by a unit. When $A$ is a
$p$-torsion free $\delta$-ring, we have $\varphi (y) = y^p + p \delta (y)$ and
$y^p / p! \in A [p^{- 1}]$ belongs to $A$ if and only if $\varphi (y)$ is
divisible by $p$.

Now we define the animated $\delta$-ring $\mathbb{Z}_{(p)} \{ x, \varphi (x) /
p \}$ to be the pushout of the diagram
\[ \begin{array}{ccc}
     \mathbb{Z}_{(p)} \{ y \} & \xrightarrow{y \mapsto pz} & \mathbb{Z}_{(p)}
     \{ z \}\\
     \longdownarrow \nocomma y \mapsto \varphi (x) &  & \\
     \mathbb{Z}_{(p)} \{ x \} &  & 
   \end{array} \]
in the $\infty$-category $\tmop{CAlg}_{\delta}^{\tmop{an}}$. Since the
Frobenius map $\varphi \of \mathbb{Z}_{(p)} \{ y \} \rightarrow
\mathbb{Z}_{(p)} \{ x \}$ is faithfully flat, so is the map $\mathbb{Z}_{(p)}
\{ z \} \rightarrow \mathbb{Z}_{(p)} \{ x, \varphi (x) / p \}$. It follows
that $\mathbb{Z}_{(p)} \{ x, \varphi (x) / p \}$ is static and
$p$-torsion-free by \Cref{rem:flat-connective-static}, therefore it is a
$\delta$-ring by \Cref{lem:ani-delta-ring-trunc} (this is essentially
{\cite[Lem~2.36]{Bhatt2019}}). We need another characterization of the
underlying ring of $\mathbb{Z}_{(p)} \{ x, \varphi (x) / p \}$:

\begin{lemma}[{\cite[Lem~2.36]{Bhatt2019}}]
  There is a natural isomorphism
  \[ D_{\mathbb{Z}_{(p)} \{ x \}} (x) \longrightarrow \mathbb{Z}_{(p)} \{ x,
     \varphi (x) / p \} \]
  of $p$-torsion-free $\mathbb{Z}_{(p)}$-algebras.
\end{lemma}

This map transfers the surjective map $D_{\mathbb{Z}_{(p)} \{ x \}} (x)
\twoheadrightarrow \mathbb{Z}_{(p)} \{ x \} / (x)$ to a surjective map
$\mathbb{Z}_{(p)} \{ x, \varphi (x) / p \} \twoheadrightarrow \mathbb{Z}_{(p)}
\{ x \} / (x)$, the existence of which does not seem to be {\tmem{a priori}}
clear (which is implicitly involved in {\cite[Lem~2.35]{Bhatt2019}}).

Note that since $x \in D_{\mathbb{Z}_{(p)} \{ x \}} (x)$ is a
non-zero-divisor, the map from the animated PD-envelope of $(\mathbb{Z}_{(p)}
\{ x \}, (x))$ to the classical PD-envelope is an equivalence, by base change
of $(\mathbb{Z}_{(p)} [x], (x))$ along the flat map $\mathbb{Z}_{(p)} [x]
\rightarrow \mathbb{Z}_{(p)} \{ x \} \simeq \mathbb{Z}_{(p)} [x, \delta (x),
\delta^2 (x), \ldots]$, or alternatively by
\Cref{prop:koszul-regular-ani-pd-env}. We could replace $x$ by a finite number
of variables, which leads to

\begin{corollary}
  \label{cor:delta-pair-ani-PD-env-proj-gen}There exists a canonical
  $\delta$-pair structure on the animated PD-envelope of every $\delta$-pair
  $(\mathbb{Z}_{(p)} \{ X, Y \}, (Y)) \in \tmop{Pair}_{\delta}^{\tmop{st}}$.
  More formally, there exists a canonical functor
  $\tmop{Pair}_{\delta}^{\tmop{st}} \rightarrow
  \tmop{Pair}_{\delta}^{\tmop{an}}$ which fits into a commutative diagram
  
  \[\begin{tikzcd} {\operatorname{Pair}_\delta^{\operatorname{st}}} &&
  {\operatorname{Pair}_\delta^{\operatorname{an}}} \\
  {\operatorname{Pair}^{\operatorname{an}}} &
  {\operatorname{Pair}^{\gamma,\operatorname{an}}} &
  {\operatorname{Pair}^{\operatorname{an}}} \arrow[from=1-1, to=2-1]
  \arrow["{\operatorname{Env}^{\gamma,\operatorname{an}}}", from=2-1, to=2-2]
  \arrow[from=2-2, to=2-3] \arrow[from=1-3, to=2-3] \arrow[dashed, from=1-1,
  to=1-3] \end{tikzcd}\]
  
  {\noindent}of $\infty$-categories.
\end{corollary}

\begin{proof}
  The functoriality of $\tmop{Pair}_{\delta}^{\tmop{st}} \rightarrow
  \tmop{Pair}_{\delta}^{\tmop{an}}$ needs explanation: a map
  $(\mathbb{Z}_{(p)} \{ X, Y \}, (Y)) \rightarrow (\mathbb{Z}_{(p)} \{ X', Y'
  \}, (Y'))$ of $\delta$-pairs induces a map $(\mathbb{Q} \{ X, Y \}, (Y))
  \rightarrow (\mathbb{Q} \{ X', Y' \}, (Y'))$ of pairs after inverting $p$
  which is ``Frobenius''-equivariant, where $\mathbb{Q} \{ X, Y \} \assign
  \mathbb{Z}_{(p)} \{ X, Y \} [p^{- 1}]$. A careful $v_p$-analysis implies
  that this map restricts to a map $\mathbb{Z}_{(p)} \{ X, Y, \varphi (Y) / p
  \} \rightarrow \mathbb{Z}_{(p)} \{ X', Y', \varphi (Y') / p \}$ of
  $\mathbb{Z}_{(p)}$-subalgebras, which gives rise to the functoriality.
\end{proof}

It follows from
\Cref{prop:left-deriv-n-fun,prop:forget-PD-small-colim,lem:forget-ani-delta-pair-ani-pair,cor:delta-pair-ani-PD-env-proj-gen}
that

\begin{corollary}
  There exists a canonical animated $\delta$-pair structure on the animated
  PD-envelope of every animated $\delta$-pair. More formally, there exists a
  canonical functor $\tmop{Pair}_{\delta}^{\tmop{an}} \rightarrow
  \tmop{Pair}_{\delta}^{\tmop{an}}$ which fits into a commutative diagram
  
  \[\begin{tikzcd} {\operatorname{Pair}_\delta^{\operatorname{an}}} &&
  {\operatorname{Pair}_\delta^{\operatorname{an}}} \\
  {\operatorname{Pair}^{\operatorname{an}}} &
  {\operatorname{Pair}^{\gamma,\operatorname{an}}} &
  {\operatorname{Pair}^{\operatorname{an}}} \arrow[from=1-1, to=2-1]
  \arrow["{\operatorname{Env}^{\gamma,\operatorname{an}}}", from=2-1, to=2-2]
  \arrow[from=2-2, to=2-3] \arrow[from=1-3, to=2-3] \arrow[dashed, from=1-1,
  to=1-3] \end{tikzcd}\]
  
  {\noindent}of $\infty$-categories. Moreover, the functor
  $\tmop{Pair}_{\delta}^{\tmop{an}} \rightarrow
  \tmop{Pair}_{\delta}^{\tmop{an}}$ preserves small colimits.
\end{corollary}

We give an analysis of the conjugate filtration on the PD-envelope of
$(\mathbb{F}_p \{ x \}, (x))$ where $\mathbb{F}_p \{ x \} \assign
\mathbb{Z}_{(p)} \{ x \} /^{\mathbb{L}} p$, which is the base change of the
PD-envelope of $(\mathbb{Z}_{(p)} \{ x \}, (x))$ along $\mathbb{Z}_{(p)}
\rightarrow \mathbb{F}_p$. Recall that
\begin{enumerate}
  \item The (animate) PD-envelope $D_{\mathbb{F}_p [x]} (x)$ is a free
  $\mathbb{F}_p [x] / (x^p)$-module generated by the set $\{ \gamma_{kp} (x)
  \barsuchthat k \in \mathbb{N} \}$ of divided powers of $x$.
  
  \item For $i \in \mathbb{N}_{\geq 0}$, the $(- i)$-th piece of the conjugate
  filtration of $D_{\mathbb{F}_p [x]} (x)$ is generated by $\{ \gamma_{kp} (x)
  \barsuchthat k \leq i \}$ as an $\mathbb{F}_p [x] / (x^p)$-submodule.
\end{enumerate}
By the base change property (\Cref{lem:PD-env-compat-base-chg}), we have
\begin{enumerate}
  \item The (animate) PD-envelope $D_{\mathbb{F}_p \{ x \}} (x)$ is a free
  $\mathbb{F}_p \{ x \} / (x^p)$-module generated by the set $\{ \gamma_{kp}
  (x) \barsuchthat k \in \mathbb{N} \}$.
  
  \item For $i \in \mathbb{N}_{\geq 0}$, the $(- i)$-th piece of the conjugate
  filtration of $D_{\mathbb{F}_p \{ x \}} (x)$ is generated by $\{ \gamma_{kp}
  (x) \barsuchthat k \leq i \}$ as an $\mathbb{F}_p \{ x \} /
  (x^p)$-submodule.
\end{enumerate}
We follow the argument of {\cite[Lem~2.35]{Bhatt2019}}: for every $y \in
\mathbb{Z}_{(p)} \{ x \}$ with $y^p / p \in \mathbb{Z}_{(p)} \{ x \}$, we have
\begin{eqnarray}
  \delta \left( \frac{y^p}{p} \right) & = & \frac{1}{p}  \left( \frac{\varphi
  (y)^p}{p} - \left( \frac{y^p}{p} \right)^p \right) \nonumber\\
  & = & \frac{(y^p + p \delta (y))^p}{p^2} - \frac{y^{p^2}}{p^{p + 1}}
  \nonumber\\
  & = & \frac{1}{p^2}  \left( y^{p^2} + p^2 y^{p (p - 1)} \delta (y) +
  \sum_{k = 0}^{p - 2} \binom{p}{k} y^{kp}  (p \delta (y))^{p - k} \right) -
  \frac{y^{p^2}}{p^{p + 1}} \nonumber\\
  & = & \frac{p^{p - 1} - 1}{p^{p + 1}} y^{p^2} + y^{p (p - 1)} \delta (y) +
  \sum_{k = 0}^{p - 2} p^{p - 2 - k}  \binom{p}{k} y^{kp} \delta (y)^{p - k} 
\end{eqnarray}
Letting $z = x^p / p$, it follows from $p^{p - 1} - 1 \in \tmop{GL}_1
(\mathbb{Z}_{(p)})$ that
\begin{enumerate}
  \item The set $\{ z^{a_0} \delta (z)^{a_1}  (\delta^2 (z))^{a_2} \cdots
  (\delta^r (z))^{a_r} \barsuchthat r \in \mathbb{N}, 0 \leq a_0, a_1, \ldots,
  a_r < p \}$ forms a basis of the free $\mathbb{F}_p \{ x \} / (x^p)$-module
  $\mathbb{Z}_{(p)} \{ x, \varphi (x) / p \} /^{\mathbb{L}} p \simeq
  D_{\mathbb{F}_p \{ x \}} (x)$.
  
  \item For every $i \in \mathbb{N}$, the $(- i)$-th piece of the conjugate
  filtration of $D_{\mathbb{F}_p \{ x \}} (x)$ is generated by $\{ z^{a_0}
  \delta (z)^{a_1}  (\delta^2 (z))^{a_2} \cdots (\delta^r (z))^{a_r}
  \barsuchthat 0 \leq a_0, a_1, \ldots, a_r < p, a_0 + a_1 p + a_2 p^2 +
  \cdots + a_r p^r \leq i \}$.
\end{enumerate}
\begin{remark}
  \label{rem:delta-divided-powers-Frob}In a bit more imprecise terms,
  $\delta^k (z)$ differs from $\gamma_{p^k} (x)$ up to a unit, modulo ``lower
  terms''.
\end{remark}

This generalizes to multivariable case with the same argument:

\begin{lemma}
  Let $(A, I) \assign (\mathbb{Z}_{(p)} \{ X, Y \}, (Y)) \in
  \tmop{Pair}_{\delta}^{\tmop{st}}$ be a standard $\delta$-pair and let $(B,
  J, \gamma)$ be the (animated) PD-envelope of $(\mathbb{F}_p \{ X, Y \},
  (Y))$. Let $Y = \{ y_1, y_2, \ldots \}$ and $z_j \assign \varphi (y_j) / p$.
  Then
  \begin{enumerate}
    \item The $\varphi_A^{\ast} (A / I)$-module $B$ is freely generated by the
    subset $\left\{ \prod_{j, k} (\delta^k (z_j))^{a_{j, k}} \barsuchthat 0
    \leq a_{j, k} < p \right\} \subseteq B$.
    
    \item For every $i \in \mathbb{N}$, the $(- i)$-th piece of the conjugate
    filtration of $B$ is generated by
    \[ \left\{ \prod_{j, k} (\delta^k (z_j))^{a_{j, k}} \barsuchthat 0 \leq
       a_{j, k} < p, \sum_{j, k} a_{j, k} p^k \leq i \right\} \]
    as a $\varphi_A^{\ast} (A / I)$-submodule.
  \end{enumerate}
\end{lemma}

\subsection{Oriented prisms}In this subsection, we will study animated
$\delta$-rings viewed as ``non-complete oriented animated
prisms''\footnote{They are non-complete oriented analogues of
{\cite[Def~2.4]{Bhatt2022a}}.}. Recall that a {\tmdfn{orientable prism}} is a
$\delta$-pair $(A, I)$ such that the ideal $I \subseteq A$ is principal, the
$\delta$-ring $A$ is $I$-torsion free, derived $(p, I)$-complete, and $p \in I
+ \varphi (I) A$ {\cite[Def~3.2]{Bhatt2019}}. For technical reasons, we will
study the ``non-complete'' analogues where the completeness and the
torsion-freeness are dropped.

We fix a $\delta$-ring $A$ along with a chosen non-zero-divisor $d \in A$. In
practice, we are only interested in the special case that $A =\mathbb{Z}_{(p)}
\{ d \}$ and some variants like $A =\mathbb{Z}_{(p)} \{ d, \delta (d)^{- 1}
\}$. We denote by $\tmop{Ring}_{\delta, A}$ the $1$-category
$(\tmop{Ring}_{\delta})_{A \mathord{/}}$ of $\delta$-$A$-algebras. It follows
from \Cref{lem:undercat-proj-gen} that

\begin{lemma}
  The $1$-category $\tmop{Ring}_{\delta, A}$ is $1$-projectively generated,
  therefore presentable. A set of compact $1$-projective generators is given
  by $\{ A \{ X \} \assign A \otimes_{\mathbb{Z}_{(p)}}^{\mathbb{L}}
  \mathbb{Z}_{(p)} \{ X \} \barsuchthat X \in \tmop{Fin} \}$, which spans a
  full subcategory of $\tmop{Ring}_{\delta, A}$ denoted by
  $\tmop{Ring}_{\delta, A}^0$.
\end{lemma}

\begin{definition}
  Let $B$ be an animated $\delta$-ring. The {\tmdfn{$\infty$-category of
  animated $\delta$-$B$-algebras}}, denoted by $\tmop{CAlg}_{\delta,
  B}^{\tmop{an}}$, is defined to be the undercategory $\tmop{Ani}
  (\tmop{Ring}_{\delta})_{B \mathord{/}}$. When $B$ is static, it is
  equivalent to the animation $\tmop{Ani} (\tmop{Ring}_{\delta, B})$ by
  \Cref{cor:ani-undercat}.
\end{definition}

By \Cref{lem:adjunct-undercat}, we get an adjunction
$\tmop{CAlg}_A^{\heartsuit} \rightleftarrows \tmop{Ring}_{\delta, A}$, where
the forgetful functor $\tmop{Ring}_{\delta, A} \rightarrow
\tmop{CAlg}_A^{\heartsuit}$ preserves all small colimits (and as a right
adjoint, it preserves small limits as well). It follows from
\Cref{cor:ani-adjoint-funs} that

\begin{lemma}
  There is a pair $\tmop{CAlg}_A^{\tmop{an}} \rightleftarrows
  \tmop{CAlg}_{\delta, A}^{\tmop{an}}$ of adjoint functors, being the
  animation of the pair $\tmop{CAlg}_A^{\heartsuit} \rightleftarrows
  \tmop{Ring}_{\delta, A}$ of adjoint functors. We will call the functor
  $\tmop{CAlg}_A^{\tmop{an}} \rightarrow \tmop{CAlg}_{\delta, A}^{\tmop{an}}$
  the {\tmdfn{free animated $\delta$-$A$-algebra functor}}. The functor
  $\tmop{CAlg}_{\delta, A}^{\tmop{an}} \rightarrow \tmop{CAlg}_A^{\tmop{an}}$,
  called the {\tmdfn{forgetful functor}}, is conservative and preserves small
  colimits (and as a right adjoint, it preserves small limits as well).
\end{lemma}

\begin{definition}
  \label{def:ani-delta-A-pairs}The {\tmdfn{$\infty$-category of animated
  $\delta$-$A$-pairs}} $\tmop{Pair}_{\delta, A}^{\tmop{an}}$ is defined to be
  the undercategory $(\tmop{Pair}_{\delta}^{\tmop{an}})_{(\tmop{id}_A \of A
  \rightarrow A) \mathord{/}}$, which is equivalent to the fiber product
  $\tmop{CAlg}_{\delta, A}^{\tmop{an}} \times_{\tmop{CAlg}_A^{\tmop{an}}}
  \tmop{Pair}^{\tmop{an}}_A$ by {\cite[Lem~5.4.5.4]{Lurie2009}}.
\end{definition}

\begin{remark}
  By {\cite[Cor~2.10]{Bhatt2022a}}, our $\infty$-category
  $\tmop{Pair}_{\delta, A}^{\tmop{an}}$ is a non-complete version of the
  $\infty$-category of animated prisms over $(A, d)$. Their setup is more
  general in the sense that the base prism could be both animated and
  non-orientable.
\end{remark}

The set $\{ (A \{ X, Y \}, (Y)) \barsuchthat X, Y \in \tmop{Fin} \}$ form a
set of compact projective generators for $\tmop{Pair}_{\delta, A}^{\tmop{an}}$
by by \Cref{lem:undercat-proj-gen}, which spans a full subcategory
$\tmop{Pair}_{\delta, A}^{\tmop{st}} \subseteq \tmop{Pair}_{\delta,
A}^{\tmop{an}}$. It follows from
\Cref{lem:adjunct-undercat,lem:adjunct-ani-pair-ani-delta-pair} that

\begin{lemma}
  \label{lem:adjunct-ani-A-pair-ani-delta-A-pair}The forgetful functor
  $\tmop{Pair}_{\delta, A}^{\tmop{an}} \rightarrow \tmop{Pair}^{\tmop{an}}_A$
  admits a left adjoint.
\end{lemma}

There is a canonical functor $\tmop{CAlg}_{\delta, A}^{\tmop{an}} \rightarrow
\tmop{Pair}_{\delta, A}^{\tmop{an}}$ given by $B \mapsto (B \twoheadrightarrow
B /^{\mathbb{L}} d)$. We observe that

\begin{lemma}
  \label{lem:prism-env}The functor $\tmop{CAlg}_{\delta, A}^{\tmop{an}}
  \rightarrow \tmop{Pair}_{\delta, A}^{\tmop{an}}, B \mapsto (B
  \twoheadrightarrow B /^{\mathbb{L}} d)$ admits a left adjoint
  $\tmop{Pair}_{\delta, A}^{\tmop{an}} \rightarrow \tmop{CAlg}_{\delta,
  A}^{\tmop{an}}$, given by the left derived functor
  (\Cref{prop:left-deriv-n-fun}) of $\tmop{Pair}_{\delta, A}^{\tmop{st}}
  \rightarrow \tmop{CAlg}_{\delta, A}^{\tmop{an}}, (A \{ X, Y \}, (Y)) \mapsto
  A \{ X, Y / d \}$ where $A \{ X, Y / d \}$ is an abbreviation for the free
  $\delta$-$A$-algebra $A \{ x_1, x_2, \ldots, y_1 / d, y_2 / d, \ldots
  \}$\footnote{These generators $y_i / d$ are in fact formal variables $z_i$.
  This notation indicates that the counit map $(A \{ X, Y \}, (Y)) \rightarrow
  (A \{ X, Y / d \} \twoheadrightarrow A \{ X, Y / d \} /^{\mathbb{L}} d)$ is
  induced by $x_i \mapsto x_i$ and $y_i \mapsto z_i d$.}.
\end{lemma}

\begin{proof}
  Let $G$ denote the functor $\tmop{CAlg}_{\delta, A}^{\tmop{an}} \rightarrow
  \tmop{Pair}_{\delta, A}^{\tmop{an}}, B \mapsto (B \twoheadrightarrow B
  /^{\mathbb{L}} d)$. Then we have a functor $F \of
  \tmop{Pair}_{\delta}^{\tmop{an}} \rightarrow \tmop{Fun}
  (\tmop{CAlg}_{\delta}^{\tmop{an}}, \tmop{An})^{\tmop{op}}$ which preserves
  small colimits and sends $(B \twoheadrightarrow B'') \in
  \tmop{Pair}_{\delta, A}^{\tmop{st}}$ to the functor
  $\tmop{Map}_{\tmop{Pair}_{\delta, A}^{\tmop{an}}} (B \twoheadrightarrow B'',
  G (\cdummy))$. By \Cref{prop:left-deriv-n-fun}, it is the left derived
  functor of its restriction to the full subcategory $\tmop{Pair}_{\delta,
  A}^{\tmop{st}} \subseteq \tmop{Pair}_{\delta, A}^{\tmop{an}}$.
  
  We now show that, for every $(A \{ X, Y \}, (Y)) \in \tmop{Pair}_{\delta,
  A}^{\tmop{st}}$, the functor $F ((A \{ X, Y \}, (Y)))$ is equivalent to the
  functor $\tmop{Map}_{\tmop{CAlg}_{\delta, A}^{\tmop{an}}} (A \{ X, Y / d \},
  \cdummy)$. In other words, the essential image of $\nobracket F
  |_{\tmop{Pair}_{\delta, A}^{\tmop{st}}}$ lies in the full subcategory
  $\tmop{Ring}_{\delta, A}^0 \hookrightarrow \tmop{CAlg}_{\delta,
  A}^{\tmop{an}} \hookrightarrow \tmop{Fun} (\tmop{CAlg}_{\delta,
  A}^{\tmop{an}}, \tmop{An})^{\tmop{op}}$. By adjunctions $\tmop{Fun}
  ((\Delta^1)^{\tmop{op}}, D (\mathbb{Z})_{\geq 0}) \rightleftarrows
  \tmop{Pair}^{\tmop{an}}_A \rightleftarrows \tmop{Pair}_{\delta,
  A}^{\tmop{an}}$
  (\Cref{def:anipair-anipdpair,lem:adjunct-ani-A-pair-ani-delta-A-pair}), we
  have
  \begin{eqnarray*}
    F (A \{ X, Y \}, (Y)) (B) & \simeq &
    \tmop{Map}_{\tmop{Pair}^{\tmop{an}}_A} (A [X, Y] \twoheadrightarrow A [X],
    B \twoheadrightarrow B /^{\mathbb{L}} d)\\
    & \simeq & \tmop{Map}_{\tmop{Fun} ((\Delta^1)^{\tmop{op}}, D
    (\mathbb{Z})_{\geq 0})} \left( X\mathbb{Z} \oplus Y\mathbb{Z} \leftarrow
    Y\mathbb{Z}, B \xleftarrow{d} B \right)\\
    & \simeq & B^{\tmop{Card} (Y)} \times B^{\tmop{Card} (X)}\\
    & \simeq & \tmop{Map}_{\tmop{CAlg}_{\delta, A}^{\tmop{an}}} (A \{ X, Y /
    d \}, B)
  \end{eqnarray*}
  which are functorial in $B \in \tmop{CAlg}_{\delta, A}^{\tmop{an}}$ (note
  that naively speaking, the ``values'' of $Y / d$ correspond to the
  ``preimages'' of $Y$ under the map $B \xleftarrow{d} B$, therefore the
  formal notation $Y / d$).
  
  Since the Yoneda embedding $\tmop{CAlg}_{\delta, A}^{\tmop{an}}
  \hookrightarrow \tmop{Fun} (\tmop{CAlg}_{\delta, A}^{\tmop{an}},
  \tmop{An})^{\tmop{op}}$ is stable under small colimits, it follows that the
  essential image of $F$ lies in $\tmop{CAlg}_{\delta, A}^{\tmop{an}}$, which
  proves that $G$ admits a left adjoint $L \of \tmop{Pair}_{\delta,
  A}^{\tmop{an}} \rightarrow \tmop{CAlg}_{\delta, A}^{\tmop{an}}$, and that $L
  (A \{ X, Y / d \}) \simeq A \{ X, Y / d \}$.
  
  We still need to show that $\nobracket L |_{\tmop{Pair}_{\delta,
  A}^{\tmop{st}}}$ coincides with the functor defined in the obvious way. We
  have shown this objectwise, and since $L (\tmop{Pair}_{\delta,
  A}^{\tmop{st}})$ lies in the full subcategory $\tmop{Ring}_{\delta, A}^0
  \hookrightarrow \tmop{CAlg}_{\delta, A}^{\tmop{an}}$ which is a
  $1$-category, we only need to show that the image of morphisms coincide with
  the ``obvious'' choice, i.e. without higher categorical complication. This
  can be checked putting different $d$-torsion-free $\delta$-$A$-algebras $B
  \in \tmop{Ring}_{\delta, A}$ into the functorial isomorphism
  \[ \tmop{Hom}_{\tmop{Pair}_{\delta, A}} ((A \{ X, Y \}, (Y)), (B, (d)))
     \cong \tmop{Hom}_{\tmop{Ring}_{\delta, A}} (A \{ X, Y / d \}, B) \]
  given by the adjunction.
\end{proof}

Now we introduce a variant of \Cref{def:ani-delta-A-pairs}:

\begin{definition}
  The {\tmdfn{$\infty$-category of animated $\delta$-$(A, d)$-pairs}}
  $\tmop{Pair}_{\delta, (A, d)}^{\tmop{an}}$ is defined to be the
  undercategory $(\tmop{Pair}_{\delta, A}^{\tmop{an}})_{(A \twoheadrightarrow
  A /^{\mathbb{L}} d) \mathord{/}}$, which is equivalent to the fiber product
  $\tmop{CAlg}_{\delta, A}^{\tmop{an}} \times_{\tmop{CAlg}_A^{\tmop{an}}}
  \tmop{Pair}^{\tmop{an}}_{(A \twoheadrightarrow A /^{\mathbb{L}} d)
  \mathord{/}}$ by {\cite[Lem~5.4.5.4]{Lurie2009}}.
\end{definition}

By \Cref{lem:undercat-proj-gen}, we have

\begin{lemma}
  The $\infty$-category $\tmop{Pair}_{\delta, (A, d)}^{\tmop{an}}$ is
  projectively generated. A set of compact projective generators is given by
  $\{ (A \{ X, Y \}, (d, Y)) \barsuchthat X, Y \in \tmop{Fin} \}$, which spans
  a full subcategory of $\tmop{Pair}_{\delta, (A, d)}^{\tmop{an}}$ denoted by
  $\tmop{Pair}_{\delta, (A, d)}^{\tmop{st}}$.
\end{lemma}

Note that $A$ is initial in $\tmop{Pair}_{\delta, A}^{\tmop{an}}$. It follows
from \Cref{lem:adjunct-undercat,lem:prism-env} that

\begin{corollary}
  \label{cor:prism-env}The functor $\tmop{CAlg}_{\delta, A}^{\tmop{an}}
  \rightarrow \tmop{Pair}_{\delta, (A, d)}^{\tmop{an}}, B \mapsto (B
  \twoheadrightarrow B /^{\mathbb{L}} d)$ admits a left adjoint
  $\tmop{Pair}_{\delta, (A, d)}^{\tmop{an}} \rightarrow \tmop{CAlg}_{\delta,
  A}^{\tmop{an}}$, which will be denoted by
  $\tmop{Env}^{\Prism}$\footnote{This is understood as a ``non-complete''
  prismatic envelope when $p$ lies in the Jacobson radical $\tmop{Rad} (A)$
  and $d \in A$ is weakly distinguished (\Cref{def:weak-dist}).}, given by the
  left derived functor (\Cref{prop:left-deriv-n-fun}) of $\tmop{Pair}_{\delta,
  (A, d)}^{\tmop{st}} \rightarrow \tmop{CAlg}_{\delta, A}^{\tmop{an}}, (A \{
  X, Y \}, (d, Y)) \mapsto A \{ X, Y / d \}$.
\end{corollary}

Furthermore, for every $B \in \tmop{CAlg}_{\delta, A}^{\tmop{an}}$, by
unrolling the definitions, the counit map $\tmop{Env}^{\Prism} (B
\twoheadrightarrow B /^{\mathbb{L}} d) \rightarrow B$ is an equivalence,
therefore

\begin{lemma}
  \label{lem:prism-full-faithful-delta-pair}The functor $\tmop{CAlg}_{\delta,
  A}^{\tmop{an}} \rightarrow \tmop{Pair}_{\delta, (A, d)}^{\tmop{an}}$ is in
  fact fully faithful, the image of which is a reflective subcategory
  (\Cref{def:reflexive-cat}).
\end{lemma}

The following concept is not strictly necessary, but it would help us to
understand when we need to ``divide by $d$'':

\begin{definition}
  Let $A$ be a $\delta$-ring and $d \in A$ a non-zero-divisor. Let $M \in D (A
  /^{\mathbb{L}} d)$ be a $A /^{\mathbb{L}} d$-module spectrum. For every $n
  \in \mathbb{Z}$, the {\tmdfn{$n$-th Breuil--Kisin twist}} of $M$ with
  respect to $(A, d)$, denoted by $M \{ n \}$, is defined to be $M \otimes_{A
  /^{\mathbb{L}} d}^{\mathbb{L}} (dA / d^2 A)^{\otimes_{A /^{\mathbb{L}}
  d}^{\mathbb{L}} n}$.
\end{definition}

Note that when $d \in A$ is a non-zero-divisor, the $A /^{\mathbb{L}}
d$-module $d^n A / d^{n + 1} A$ is a free of rank $1$, therefore equivalent to
$A /^{\mathbb{L}} d$. The Breuil--Kisin twists are strictly necessary when we
want to generalize to non-orientable prisms. In our case, we understand $M \{
1 \}$ ``formally multiplied by $d$'' and $M \{ - 1 \}$ ``formally divided by
$d$'', just as the formal notations $y_i / d$ in \Cref{lem:prism-env}.

Finally, we introduce a variant of the concept of {\tmdfn{distinguished
elements}} {\cite[Def~2.19]{Bhatt2019}}:

\begin{definition}
  \label{def:weak-dist}Let $A$ be a $\delta$-ring. We say that an element $d
  \in A$ is {\tmdfn{weakly distinguished}} if the ideal $(d, \delta (d))$ is
  the unital ideal $A$, or equivalently, $\delta (d)$ is invertible in $A /
  d$.
\end{definition}

\begin{remark}
  Let $A$ be a $\delta$-ring and $d \in \tmop{Rad} (A)$ an element in the
  Jacobson radical. Then $d$ is weakly distinguished if and only if it is
  distinguished.
\end{remark}

The following lemma is a motivation for the introduction of weakly
distinguished elements:

\begin{lemma}[cf. {\cite[Lem~2.23]{Bhatt2019}}]
  \label{lem:delta-mod-I-indep-gen}Let $A$ be a $\delta$-ring, $I = (d)
  \subseteq A$ a principal ideal. Then for any invertible element $u \in
  \tmop{GL}_1 (A)$, the principal ideals $\delta (d)  (A / I)$ and $\delta
  (ud)  (A / I)$ are the same. In particular, when $I$ is generated by a
  non-zero-divisor, the principal ideal $\delta (d)  (A / I)$ does not depend
  on the choice of the generator $d \in I$.
\end{lemma}

\begin{proof}
  We have $\delta (ud) = \varphi (u) \delta (d) + \delta (u) d^p \equiv
  \varphi (u) \delta (d) \hspace{0.2em} \pmod{ud}$. Since $u$ is invertible,
  so is $\varphi (u)$, and the result follows.
\end{proof}

\begin{corollary}
  \label{cor:weak-dist-gen-indep}Let $A$ be a $\delta$-ring, $I \subseteq A$ a
  principal ideal generated by a non-zero-divisor. Then the followings are
  equivalent:
  \begin{enumerate}
    \item There exists a weakly distinguished generator $d$ of $I$.
    
    \item Every generator $d$ of $I$ is weakly distinguished.
  \end{enumerate}
\end{corollary}

\begin{remark}[Bhatt]
  We have a variant of \Cref{cor:weak-dist-gen-indep} which does not involve
  non-zero-divisors, by replacing a principal ideal $I$ by the equivalence
  classes of maps $A \rightarrow A$ of $A$-modules, and the proof of
  \Cref{lem:delta-mod-I-indep-gen} implies that the concept of ``weakly
  distinguished'' is invariant under this equivalence. More generally, we can
  consider the equivalence classes of an invertible $A$-module $I$ along with
  a map $I \rightarrow A$, and define the concept of such a map $I \rightarrow
  A$ being weakly distinguished when $p \in \tmop{Rad} (A)$. This generalizes
  to animated $\delta$-rings.
\end{remark}

Recall that an animated ring $A$ is {\tmdfn{$p$-local}} if the element $p \in
\pi_0 (A)$ lies in the Jacobson radical\footnote{The {\tmdfn{Jacobson
radical}} $\tmop{Rad} (A)$ of a ring $A$ is defined to be the subset (and
{\tmem{a fortiori}}, the ideal) of elements $x \in A$ such that for every $a
\in A$, the element $1 + ax$ is invertible in $A$.} $\tmop{Rad} (\pi_0 (A))$.

\begin{lemma}
  Let $A$ be a $p$-local $\delta$-ring and $d \in A$ a weakly distinguished
  element. Then for every $n \in \mathbb{N}$, $\varphi^n (\delta (d))$ is
  invertible in $A / d$.
\end{lemma}

\begin{proof}
  By induction, it suffices to show that, for every $u \in A$ of which the
  image in $A / d$ is invertible, then so is the image of $\varphi (u)$ in $A
  / d$. It follows from the identity $\varphi (u) = u^p + p \delta (u)$, since
  the image of $u^p$ in $A / d$ invertible, and $p \in \tmop{Rad} (A / d)$.
\end{proof}

\subsection{Conjugate filtration}\label{subsec:prism-conj-fil}In this
subsection, we will introduce the {\tmdfn{conjugate filtration}} on
``non-complete prismatic envelopes'', which plays a similar role as the
conjugate filtrations on animated PD-envelopes and derived crystalline
cohomology. Let $A$ be a $p$-local $\delta$-ring, $d \in A$ a weakly
distinguished non-zero-divisor. To simplify the presentation, we mostly
concentrate on the ``single variable'' case: $\tmop{Env}^{\Prism} (A \{ y \},
(d, y)) /^{\mathbb{L}} d \simeq A \{ y / d \} /^{\mathbb{L}} d$ as an $A \{ y
\} /^{\mathbb{L}} (d, y)$-algebra (or module).

First, note that the identity
\begin{eqnarray}
  \delta (u^p) & = & (\varphi (u^p) - u^{p^2}) / p \nonumber\\
  & = & ((u^p + p \delta (u))^p - u^{p^2}) / p \nonumber\\
  & = & \sum_{k = 1}^p \binom{p}{k} u^{p (p - k)} p^{k - 1} \delta (u)^k 
  \label{eq:delta-power}
\end{eqnarray}
holds in the free $\delta$-ring $\mathbb{Z}_{(p)} \{ u \}$, therefore it is an
identity in any $\delta$-ring.

We now compute $\delta^n (y)$ in terms of $\delta^n (z)$ where $y = zd$ in the
free $\delta$-$A$-algebra $A \{ y \}$:
\begin{eqnarray*}
  \delta (y) & = & \delta (zd)\\
  & = & \delta (z) \varphi (d) + z^p \delta (d)\\
  \delta^2 (y) & = & \delta (\delta (z) \varphi (d) + z^p \delta (d))\\
  & = & \delta (\delta (z) \varphi (d)) + \delta (z^p \delta (d)) -
  \underbrace{\sum_{k = 1}^{p - 1} \frac{1}{p}  \binom{p}{k}  (\delta (z)
  \varphi (d))^{p - k}  (z^p \delta (d))^k}_{\backassign R_2}\\
  & = & \delta^2 (z) \varphi^2 (d) + \delta (z)^p \delta (\varphi (d)) +
  \delta (z^p) \delta (\varphi (d)) + z^{p^2} \delta^2 (d) - R_2\\
  & = & \delta^2 (z) \varphi^2 (d) + \varphi (\delta (d))  (1 + p^{p - 1})
  \delta (z)^p + \sum_{k = 1}^{p - 1} \cdots + z^{p^2} \delta^2 (d) - R_2
\end{eqnarray*}
where we used the fact that $\varphi \circ \delta = \delta \circ \varphi$ and
\eqref{eq:delta-power} (which leads to the summand $\sum_{k = 1}^{p - 1}
\cdots$), and in general, we have

\begin{lemma}
  \label{lem:expand-delta-n}Let $A \{ z \}$ be the free $\delta$-$A$-algebra
  and $y \assign zd$. For every $n \in \mathbb{N}$, there exists a unique
  polynomial $P_n \in A [X_0, \ldots, X_{n - 1}]$ with $\deg_{X_{n - 1}} P_n
  \leq p$ such that
  \[ \delta^n (y) = \delta^n (z) \varphi^n (d) + P_n (z, \delta (z), \ldots,
     \delta^{n - 1} (z)) \]
  Moreover, there exists a unique $Q_n \in A [X_0, \ldots, X_{n - 1}]$ with
  $\deg_{X_{n - 1}} Q_n < p$ such that $P_n = a_n \varphi^{n - 1} (\delta (d))
  X_{n - 1}^p + Q_n$ where $a_n$ are partial sums $\sum_{k = 0}^{n - 1} p^{k
  (p - 1)}$ of the geometric progression $(p^{k (p - 1)})_{k \in \mathbb{N}}$.
  Note that $a_n \in \tmop{GL}_1 (\mathbb{Z}_{(p)})$ for $n > 0$. On the other
  hand, if we endow $X_i$ with degree $p^i$, then $P_n$ is homogeneous of
  degree $p^n$.
\end{lemma}

\begin{proof}
  The uniqueness follows from the freeness. We prove the existence inductively
  on $n \in \mathbb{N}$. When $n = 0$, this is obvious. Now let $n \in
  \mathbb{N}_{> 0}$, and assume that this is true for every $m < n$, Now we
  have
  \begin{eqnarray*}
    \delta^n (y) & = & \delta (\delta^{n - 1} (y))\\
    & = & \delta (\delta^{n - 1} (z) \varphi^{n - 1} (d) + P_{n - 1} (z,
    \delta (z), \ldots, \delta^{n - 2} (z)))\\
    & = & \delta (\delta^{n - 1} (z) \varphi^{n - 1} (d)) + \delta (P_{n - 1}
    (z, \delta (z), \ldots, \delta^{n - 2} (z))) - R_n
  \end{eqnarray*}
  where $\delta (\delta^{n - 1} (z) \varphi^{n - 1} (d)) = \delta^n (z)
  \varphi^n (d) + (\delta^{n - 1} (z))^p \varphi^{n - 1} (\delta (d))$ and
  \[ R_n \assign \sum_{k = 1}^{p - 1} \frac{1}{p}  \binom{p}{k}  (\delta^{n -
     1} (z) \varphi^{n - 1} (d))^{p - k}  (P_{n - 1} (z, \delta (z), \ldots,
     \delta^{n - 2} (z)))^k \]
  Note that the ``degree'' of $\delta^{n - 1} (z)$ in $R_n$ is strictly less
  than $p$. Let $b_{n - 1} = a_{n - 1} \varphi^{n - 2} (\delta (d))$, we have
  \begin{eqnarray*}
    \delta (P_{n - 1} (z, \delta (z), \ldots)) & = & \delta (b_{n - 1} 
    (\delta^{n - 2} (z))^p + Q_{n - 1} (z, \delta (z), \ldots))\\
    & = & \delta (b_{n - 1}  (\delta^{n - 2} (z))^p) + \delta (Q_{n - 1} (z,
    \delta (z), \ldots)) - \underbrace{\sum_{k = 1}^{p - 1}
    \cdots}_{\backassign R_n'}\\
    & = & \varphi (b_{n - 1}) \delta ((\delta^{n - 2} (z))^p) + \delta (Q_{n
    - 1} (\cdots)) + \delta (b_{n - 1})  (\delta^{n - 2} (z))^{p^2}\\
    &  & - R_n'
  \end{eqnarray*}
  and only $\varphi (b_{n - 1}) \delta ((\delta^{n - 2} (z))^p)$ has
  contribution on $\delta^{n - 1} (z)^p$, and
  \[ \delta ((\delta^{n - 2} (z))^p) = \sum_{k = 1}^p \binom{p}{k}  (\delta^{n
     - 2} (z))^{p (p - k)} p^{k - 1}  (\delta^{n - 1} (z))^k \]
  has contribution on $\delta^{n - 1} (z)^p$ only at $k = p$, i.e. $p^{p - 1}
  \delta^{n - 1} (z)^p$. Note that $\varphi (b_{n - 1}) = a_{n - 1} \varphi^{n
  - 1} (\delta (d))$, the result then follows.
\end{proof}

{\construction{\label{cons:conj-fil-prism-env}We now rewrite $A \{ y \}
\rightarrow A \{ z \}, y \mapsto zd$ as the sequential composite (i.e. the $A
\{ y \}$-algebra $A \{ z \}$ is equivalent to the sequential colimit of)
\begin{equation}
  A \{ y \} \longrightarrow A \{ y \} \otimes_{B_0}^{\mathbb{L}} C_0
  \longrightarrow A \{ y \} \otimes_{B_1}^{\mathbb{L}} C_1 \longrightarrow
  \cdots \label{eq:Ay-in-Az}
\end{equation}
where $A_n \assign A [z, \ldots, \delta^{n - 1} (z)], B_n \assign A_n
[\delta^n (y)]$ and $C_n \assign A_n [\delta^n (z)]$ are polynomial algebras,
and the map $B_n \rightarrow C_n$ is given by the evaluation $\delta^n (y)
\mapsto \delta^n (z) \varphi^n (d) + P_n (z, \delta (z), \ldots, \delta^{n -
1} (z))$ by \Cref{lem:expand-delta-n}. Thus $B_n \rightarrow C_n$ could be
written as the composite (where we replace $\delta^n (y)$ by $u$ and $\delta^n
(z)$ by $v$)
\begin{equation}
  B_n = A_n [u] \rightarrow A_n [u, v] / (u - \varphi^n (d) v - P_n (z, \delta
  (z), \ldots, \delta^{n - 1} (z))) \cong A_n [v] = C_n \label{eq:Bn-to-Cn}
\end{equation}
In other words, $B_n \rightarrow C_n$ is essentially formally
adjoining\footnote{Note that this is true although $\varphi^n (d)$ is not
necessarily a non-zero-divisor.} $(\delta^n (y) - P_n (z, \delta (z), \ldots,
\delta^{n - 1} (z))) / \varphi^n (d)$ to $B_n$ as an (animated) $A$-algebra,
and the $A \{ y \}$-algebra $A \{ z \}$ is obtained by formally adjoining
$(\delta^n (y) - P_n (z, \delta (z), \ldots, \delta^{n - 1} (z))) / \varphi^n
(d)$ iteratively from $A \{ y \}$. The {\tmdfn{conjugate filtration}} on $A \{
y / d \} /^{\mathbb{L}} d$ is given by $\tmop{Fil}^{- i}_{\tmop{conj}} (A \{ y
/ d \} /^{\mathbb{L}} d)$ being the $A \{ y \} /^{\mathbb{L}} (y,
d)$-submodule of $A \{ y / d \} /^{\mathbb{L}} d$ spanned by $\{ (y / d)^{a_0}
\delta (y / d)^{a_1}  (\delta^2 (y / d))^{a_2} \cdots (\delta^r (y / d))^{a_r}
\barsuchthat r \in \mathbb{N}, 0 \leq a_0, a_1, \ldots, a_r < p \}$.}}

Passing to the multivariable version, we get:

\begin{lemma}
  \label{lem:conj-fil-prism-env}Let $A$ be a $p$-local $\delta$-ring and $d
  \in A$ a weakly distinguished non-zero-divisor. Then there exists a
  canonical functor $\tmop{Fil}_{\tmop{conj}}^{\ast} \left(
  \tmop{Env}^{\Prism} (\cdummy) /^{\mathbb{L}} d \right) \of
  \tmop{Pair}_{\delta, (A, d)}^{\tmop{st}} \rightarrow \tmop{CAlg}
  (\tmop{DF}^{\leq 0} (A /^{\mathbb{L}} d))$ which preserves finite
  coproducts, along with a functorial map $\tmop{Fil}_{\tmop{conj}}^{\ast}
  \left( \tmop{Env}^{\Prism} (B, J) /^{\mathbb{L}} d \right) \rightarrow
  \tmop{Env}^{\Prism} (B, J) /^{\mathbb{L}} d$, understood as the
  {\tmdfn{conjugate filtration}} on $\tmop{Env}^{\Prism} (B, J) /^{\mathbb{L}}
  d$, such that
  \begin{enumerate}
    \item The conjugate filtration is exhaustive, that is to say, the induced
    map $\tmop{Fil}_{\tmop{conj}}^{- \infty} \left( \tmop{Env}^{\Prism} (B, J)
    /^{\mathbb{L}} d \right) \rightarrow \tmop{Env}^{\Prism} (B, J)
    /^{\mathbb{L}} d$ is an equivalence in $D (A)$.
    
    \item The filtration $\tmop{Fil}_{\tmop{conj}}^{\ast} \left(
    \tmop{Env}^{\Prism} (A \{ y \} \twoheadrightarrow A \{ y \} /^{\mathbb{L}}
    (y, d)) /^{\mathbb{L}} d \right)$ coincides with the filtration
    $\tmop{Fil}_{\tmop{conj}}^{\ast} (A \{ y / d \} /^{\mathbb{L}} d)$ in
    Construction~\ref{cons:conj-fil-prism-env}.
    
    \item The maps $\tmop{Fil}_{\tmop{conj}}^{- i} \left( \tmop{Env}^{\Prism}
    (A \{ x \} \twoheadrightarrow A \{ x \} /^{\mathbb{L}} d) /^{\mathbb{L}} d
    \right) \rightarrow \tmop{Env}^{\Prism} (A \{ x \} \twoheadrightarrow A \{
    x \} /^{\mathbb{L}} d) /^{\mathbb{L}} d \simeq A \{ x \} /^{\mathbb{L}} d$
    are equivalences for all $i \in \mathbb{N}$, that is to say, the conjugate
    filtration on $(A /^{\mathbb{L}} d) \{ x \}$ is ``constant\footnote{More
    precisely, it is constant after restriction to $\mathbb{Z}_{\leq 0}$, but
    this restriction is expected as the conjugate filtration is
    non-positive.}''.
  \end{enumerate}
\end{lemma}

\begin{proof}
  The conjugate filtration on each object $\tmop{Env}^{\Prism} (B, J)$ for
  $(B, J) \in \tmop{Pair}_{\delta, (A, d)}^{\tmop{st}}$ is completely
  determined by these properties and that the functor preserves finite
  coproducts, since every $(B, J)$ could be written as a coproduct of $A \{ x
  \} \twoheadrightarrow A \{ x \} /^{\mathbb{L}} d$ and $A \{ y \}
  \twoheadrightarrow A \{ y \} /^{\mathbb{L}} (y, d)$. Concretely,
  $\tmop{Fil}_{\tmop{conj}}^{- i} (A \{ X, Y / d \} /^{\mathbb{L}} d)$ are
  generated, as an $A \{ X, Y \} / (d, Y)$-submodule, by ``standard
  monomials'' $\prod_{(r, y) \in E} \delta^r (y / d)$ ``of total
  degree$\nosymbol \leq i$'' where $E \subseteq \mathbb{N} \times Y$ is a
  finite subset and the element $\delta^r (y / d)$ is of degree $p^r$ for $r
  \in \mathbb{N}$ and $y \in Y$. One verifies that this indeed gives rise to a
  functor.
  
  Alternatively, if we further assume that $d$ is weakly transversal
  (\Cref{def:weakly-transversal}), then we can invoke
  \Cref{lem:prism-conj-fil-mod-p-rat} to reduce significantly the
  computations.
\end{proof}

\begin{definition}
  Let $A$ be a $p$-local $\delta$-ring and $d \in A$ a weakly distinguished
  non-zero-divisor. Then the {\tmdfn{conjugate filtration}} on
  $\tmop{Env}^{\Prism} (B \twoheadrightarrow B'') /^{\mathbb{L}} d$ for $(B
  \twoheadrightarrow B'') \in \tmop{Pair}_{\delta, (A, d)}^{\tmop{an}}$ is
  given by the left derived functor (\Cref{prop:left-deriv-n-fun})
  $\tmop{Pair}_{\delta, (A, d)}^{\tmop{an}} \rightarrow \tmop{CAlg}
  (\tmop{DF}^{\leq 0} (A /^{\mathbb{L}} d))$ of the functor
  $\tmop{Pair}_{\delta, (A, d)}^{\tmop{st}} \rightarrow \tmop{CAlg}
  (\tmop{DF}^{\leq 0} (A /^{\mathbb{L}} d))$ in \Cref{lem:conj-fil-prism-env}.
\end{definition}

It follows from \Cref{lem:assoc-graded-union-Lan} that

\begin{lemma}
  Let $A$ be a $p$-local $\delta$-ring and $d \in A$ a weakly distinguished
  non-zero-divisor. Then the {\tmdfn{conjugate filtration}} on
  $\tmop{Env}^{\Prism} (B \twoheadrightarrow B'') /^{\mathbb{L}} d$ for every
  $(B \twoheadrightarrow B'') \in \tmop{Pair}_{\delta, (A, d)}^{\tmop{an}}$ is
  exhaustive, i.e. $\tmop{Fil}^{- \infty} \tmop{Env}^{\Prism} (B
  \twoheadrightarrow B'') /^{\mathbb{L}} d \rightarrow \tmop{Env}^{\Prism} (B
  \twoheadrightarrow B'') /^{\mathbb{L}} d$ is an equivalence.
\end{lemma}

We now analyze the ``denominators'' $\varphi^n (d)$ when $A$ is $p$-local and
$d$ is weakly distinguished:

\begin{lemma}[cf. {\cite[Lem~3.5]{Anschuetz2019}}]
  \label{lem:phi-d-p-local}Let $A$ be a $p$-local $\delta$-ring and $d \in A$
  a weakly distinguished element. Then for every $n \in \mathbb{N}_{> 0}$,
  there exists a unit $u \in \tmop{GL}_1 (A / d)$ such that $\varphi^n (d)
  \equiv pu \hspace{0.2em} \pmod{d}$.
\end{lemma}

\begin{proof}
  We will construct inductively on $n \in \mathbb{N}_{> 0}$ a sequence
  $(u_n)_n \in A^{\mathbb{N}_{> 0}}$ such that for every $n \in \mathbb{N}_{>
  0}$, the image of $u_n$ in $A / d$ is invertible, and $\varphi^n (d) -
  d^{p^n} = pu_n$. We take $u_1 = \delta (d)$, and suppose that $u_m$ are
  already constructed for $1 \leq m < n$, then
  \begin{eqnarray*}
    \varphi^n (d) & = & \varphi^{n - 1} (\varphi (d))\\
    & = & \varphi^{n - 1} (d^p + p \delta (d))\\
    & = & (\varphi^{n - 1} (d))^p + p \varphi (\delta (d))\\
    & = & (d^{p^{n - 1}} + pu_{n - 1})^p + p \varphi (\delta (d))\\
    & = & d^{p^n} + p \left( \varphi (\delta (d)) + \sum_{k = 1}^p
    \binom{p}{k} d^{p^{n - 1}  (p - k)} p^{k - 1} u_{n - 1}^k \right)
  \end{eqnarray*}
  We pick $u_n = \delta (d) + \sum_{k = 1}^p \binom{p}{k} d^{p^{n - 1}  (p -
  k)} p^{k - 1} u_{n - 1}^k$. Note that the second summand $\sum_{k = 1}^p
  \cdots$ is canonically divisible by $p$ (separating the cases $k = 0$ and $k
  > 1$), thus $u_n \equiv \delta (d) \hspace{0.2em} \pmod{p}$ of which the
  image in $A / (p, d)$ is invertible. The result then follows from the fact
  that $p \in \tmop{Rad} (A / d)$.
\end{proof}

We introduce the following temporary terminology:

\begin{definition}
  \label{def:weakly-transversal}Let $A$ be a $\delta$-ring. We say that an
  element $d \in A$ is {\tmdfn{weakly transversal}} if it is weakly
  distinguished and the sequence $(d, p)$ is regular in $A$, that is to say,
  $d$ is a non-zero-divisor and $A / d$ is $p$-torsion-free.
\end{definition}

Recall that for a ring $A$, the {\tmdfn{Zariski localization}} of $A$ along an
ideal $I \subseteq A$ is defined to be the localization of $A$ at the
multiplicative set $1 + I$. The image of $I$ in $(1 + I)^{- 1} A$ lies in the
Jacobson radical.

\begin{example}
  The element $d$ in the $p$-local $\delta$-ring $\mathbb{Z}_{(p)} \{ d,
  \delta (d)^{- 1} \}_{(p)}$ is weakly transversal. In fact, this special case
  suffices for our applications.
\end{example}

Now we assume that $d \in A$ is weakly transversal. In the ``single variable''
case $A \{ y \} /^{\mathbb{L}} (d, y) \rightarrow A \{ y / d \} /^{\mathbb{L}}
d$, by \Cref{lem:phi-d-p-local,lem:expand-delta-n}, the sequence $(z, \delta
(z), \delta^2 (z), \ldots)$ forms a system similar to that of divided
$p^r$-powers $(\gamma_{p^r})_{r \in \mathbb{N}}$ up to a multiplication of a
unit after modulo $d$:
\begin{eqnarray*}
  p \delta (z) & \equiv & - a_1 \delta (d) z^p \hspace{0.2em} \pmod{B}\\
  p \delta^2 (z) & \equiv & - a_2 \varphi (\delta (d)) \delta (z)^p
  \hspace{0.2em} \pmod{B [\delta (z)]}\\
  p \delta^3 (z) & \equiv & - a_3 \varphi^2 (\delta (d)) \delta^2 (z)^p
  \hspace{0.2em} \pmod{B [\delta (z), \delta^2 (z)]}
\end{eqnarray*}
where $B \assign A \{ y \} /^{\mathbb{L}} (d, y)$ and $a_n \varphi^{n - 1}
(\delta (d)) \in \tmop{GL}_1 (A / d)$ (cf.
\Cref{rem:delta-divided-powers-Frob}). We now translate this observation to an
analysis of the conjugate filtration, which seems hard to attack directly. We
look at the maps $B_0 /^{\mathbb{L}} (d, y) \rightarrow C_0 /^{\mathbb{L}} d$
and $B_n /^{\mathbb{L}} d \rightarrow C_n /^{\mathbb{L}} d$ for $n \in
\mathbb{N}_{> 0}$ induced by the map \eqref{eq:Bn-to-Cn}. We first note that
the map $B_0 /^{\mathbb{L}} (d, y) \rightarrow C_0 /^{\mathbb{L}} d$ is the
polynomial algebra in single variable $z$.

If we further (derived) modulo $p$, we see that $B_n /^{\mathbb{L}} (d, p)
\rightarrow C_n /^{\mathbb{L}} (d, p)$ for $n \in \mathbb{N}_{> 0}$ is killing
a polynomial $\delta^n (y) - P_n (z, \delta (z), \ldots, \delta^{n - 1} (z))$
monic in $\delta^{n - 1} (z)$ of degree $p$, and then adjoining a formal
variable $\delta^n (z)$. In view of \eqref{eq:Ay-in-Az}, we see that the map
$A \{ y \} /^{\mathbb{L}} (d, p, y) \xrightarrow{y \mapsto zd} A \{ z \}
/^{\mathbb{L}} (d, p)$ is the composition of consecutively adjoining a root of
a monic polynomial of degree $p$, and consequently, as a $A \{ y \}
/^{\mathbb{L}} (d, p, y)$-module, $A \{ z \} /^{\mathbb{L}} (d, p)$ is freely
generated by
\[ \{ z^{a_0} \delta (z)^{a_1}  (\delta^2 (z))^{a_2} \cdots (\delta^r
   (z))^{a_r} \barsuchthat r \in \mathbb{N}, 0 \leq a_0, a_1, \ldots, a_r < p
   \} . \]
On the other hand, if we invert $p$, we see that, for every $n \in
\mathbb{N}_{> 0}$, the maps $(B_n /^{\mathbb{L}} d) [p^{- 1}] \rightarrow (C_n
/^{\mathbb{L}} d) [p^{- 1}]$ are equivalences, therefore $(A \{ y \}
/^{\mathbb{L}} (d, y)) [p^{- 1}] \rightarrow (A \{ z \} /^{\mathbb{L}} d)
[p^{- 1}]$ is the polynomial algebra in one variable $z$.

The $\tmop{mod} p$ conjugate filtration $\tmop{Fil}_{\tmop{conj}}^{- i} (A \{
z \} /^{\mathbb{L}} (d, p)) /^{\mathbb{L}} p$ is then freely generated by
\[ \{ z^{a_0} \delta (z)^{a_1}  (\delta^2 (z))^{a_2} \cdots (\delta^r
   (z))^{a_r} \barsuchthat r \in \mathbb{N}, 0 \leq a_0, a_1, \ldots, a_r < p,
   a_0 + pa_1 + \cdots + p^r a_r \leq i \} . \]
On the other hand, the rationalized conjugate filtration
$\tmop{Fil}_{\tmop{conj}}^{- i} (A \{ z \} /^{\mathbb{L}} (d, p)) [p^{- 1}]$
is given by the $(A \{ y \} /^{\mathbb{L}} (d, y)) [p^{- 1}]$-polynomials in
$z$ of degree$\nosymbol \leq i$. This follows from the following lemma, which
can be established by induction on $n$:

\begin{lemma}
  \label{lem:delta-n-rat}In the rationalized free $\delta$-ring
  $\mathbb{Z}_{(p)} \{ x \} [p^{- 1}] \cong \mathbb{Q} [x, \varphi (x),
  \varphi^2 (x), \ldots]$, for every $n \in \mathbb{N}$, the image of
  $\delta^n (x) \in \mathbb{Z}_{(p)} \{ x \}$ in $\mathbb{Q} [x, \varphi (x),
  \varphi^2 (x), \ldots]$ is given by a polynomial $D_n (x, \varphi (x),
  \ldots, \varphi^n (x))$ such that $\deg_x D_n = p^n$ with leading term $(-
  p^{- 1})^{1 + p + \cdots + p^n} x^{p^n}$ for all $n \in \mathbb{N}$.
\end{lemma}

We summarize the ``multi-variable'' version as follows:

\begin{lemma}
  \label{lem:prism-conj-fil-mod-p-rat}Let $A$ be a $p$-local $\delta$-ring and
  $d \in A$ a weakly transversal element. Let $(A \{ X, Y \}, (d, Y)) \in
  \tmop{Pair}_{\delta, (A, d)}^{\tmop{st}}$. Then
  \begin{enumerate}
    \item The generator $\left\{ \prod_{(r, y) \in E} \delta^r (y / d)
    \right\}_E$ for $\tmop{Fil}_{\tmop{conj}}^{- i} (A \{ X, Y / d \}
    /^{\mathbb{L}} d)$ as an $A \{ X, Y \} /^{\mathbb{L}} (d, Y)$-submodule,
    ``of total degree$\nosymbol \leq i$'' where $E \subseteq \mathbb{N} \times
    Y$ is a finite subset and the element $\delta^r (y / d)$ is of degree
    $p^r$, becomes an basis after (derived) modulo $p$. This also holds for $i
    = + \infty$.
    
    \item The $(- i)$-th piece of the rationalized conjugate filtration
    $\tmop{Fil}_{\tmop{conj}}^{- i} (A \{ X, Y / d \} /^{\mathbb{L}} d) [p^{-
    1}] \subseteq (A \{ X, Y / d \} /^{\mathbb{L}} d) [p^{- 1}]$ is given by
    the $A \{ X, Y \} /^{\mathbb{L}} (d, Y)$-polynomials in variables $Y / d$
    of total degree$\nosymbol \leq i$. This also holds for $i = + \infty$.
  \end{enumerate}
  Furthermore, an element $x \in A \{ X, Y / d \} /^{\mathbb{L}} d$ belongs to
  the $(- i)$-th piece of the conjugate filtration
  $\tmop{Fil}_{\tmop{conj}}^{- i} (A \{ X, Y / d \} /^{\mathbb{L}} d)$ if and
  only if so does it after (derived) modulo $p$ and after rationalization.
\end{lemma}

\begin{remark}
  In some vague terms, in \Cref{lem:prism-conj-fil-mod-p-rat}, the derived
  modulo $p$ is about ``controlling the denominators'', and the
  rationalization is about ``controlling the degree''.
\end{remark}

Recall that for every $(B, J) \in \tmop{Pair}_{\delta, (A, d)}^{\tmop{st}}$,
there exists a canonical map $B /^{\mathbb{L}} d \simeq B
\otimes_A^{\mathbb{L}} (A /^{\mathbb{L}} d) \rightarrow B / J$ which is in
fact surjective. Then we have the following ``multivariable'' version:

\begin{lemma}
  \label{lem:conj-fil-prism-comp}Let $A$ be a $p$-local $\delta$-ring and $d
  \in A$ a weakly transversal element. For every $(A \{ X, Y \}, (d, Y))
  \backassign (B, J) \in \tmop{Pair}_{\delta, (A, d)}^{\tmop{st}}$, let $K
  \assign \ker (B /^{\mathbb{L}} d \rightarrow B / J)$. Note that $K / K^2$ is
  naturally a $B / J$-module. Then there exists a comparison map
  \[ \Gamma_{B / J}^{\ast} ((K / K^2) \{ - 1 \}) \longrightarrow
     \tmop{gr}_{\tmop{conj}}^{- \mathord{\ast}} \left( \tmop{Env}^{\Prism} (B,
     J) /^{\mathbb{L}} d \right) \]
  of graded $B / J$-algebras induced by $[\gamma_n (z_i)] \mapsto \prod_{j =
  0}^r \left( \frac{\delta^j (z_i)}{- a_j \varphi^j (\delta (d))}
  \right)^{n_j}$ where $Y = \{ y_1, \ldots \}$, $z_i = y_i / p$ and $n =
  \sum_{j = 0}^r n_j p^j$ is the $p$-adic expansion of $n$. The comparison map
  is functorial in $(B, J) \in \tmop{Pair}_{\delta, (A, d)}^{\tmop{st}}$.
\end{lemma}

\begin{proof}
  The comparison map is induced by $\left[ \gamma_n \left( \frac{y}{d} \right)
  \right] \mapsto \prod_{j = 0}^r \left( \frac{\delta^j (y / d)}{- a_j
  \varphi^j (\delta (d))} \right)^{n_j}$ for every $y$ in the ideal $(d, Y)$.
  To see that this is well-defined, the most nontrivial part is to show that
  this vanishes when $y \in (d, Y^2)$. By the multiplicity of the conjugate
  filtration, we can assume that $n_r = 1$ and $n_j = 0$ for $j \neq r$, and
  it suffices to analyze $\delta^r (y / d)$ when $y \in (d, Y^2)$, which can
  be reduced to the special case that $y = y_1 y_2$ where $y_1, y_2 \in Y$.
  
  By \Cref{lem:delta-n-rat}, the element $\delta^r (y_1 y_2 / d) \in A [p^{-
  1}] [X, Y / d, \varphi (Y), \varphi^2 (Y / d), \ldots]$ is a polynomial in
  $y_1 y_2 / d = y_1 z_2, \varphi (y_1 y_2), \ldots, \varphi^r (y_1 y_2)$. The
  crucial point is that $y_1 y_2 / d = (y_1 / d)  (y_2 / d) d = 0$ in $A \{ X,
  Y / d \} /^{\mathbb{L}} d$, therefore after rationalization, $\delta^r (y_1
  y_2 / d)$ lies in $\tmop{Fil}_{\tmop{conj}}^0 \left( \tmop{Env}^{\Prism} (B,
  J) /^{\mathbb{L}} d \right) [p^{- 1}]$.
  
  By \Cref{lem:expand-delta-n}, $\delta^r (y_1 z_2) = \delta^r (z_2) \varphi^r
  (y_1) + P_r (z_2, \ldots, \delta^{r - 1} (z_2))$ where $P_r$ is an $A \{ y_1
  \}$-polynomial. Note that $\varphi^r (y_1) = \varphi^r (z_1 d) = \varphi^r
  (z_1) \varphi^r (d) \equiv 0 \hspace{0.2em} \pmod{(d, p)}$ by
  \Cref{lem:phi-d-p-local}. Since $P_r$ is homogeneous of degree $p^r$ when
  $\deg (\delta^j (z_2)) = p^j$, it follows that for every monomial $\prod_j
  T_j^{n_j}$ of $P_r$, there exists a $j$ such that $n_j \geq p$, but then
  $\delta^j (z_2)^{n_j}$ is a linear combination of basis elements in
  \Cref{lem:prism-conj-fil-mod-p-rat} which shows that $\prod_j (\delta^j
  (z_2))^{n_j} \in \tmop{Fil}_{\tmop{conj}}^{- (p^r - 1)} \left(
  \tmop{Env}^{\Prism} (B, J) /^{\mathbb{L}} d \right) /^{\mathbb{L}} p$. The
  result then follows from the last part of
  \Cref{lem:prism-conj-fil-mod-p-rat}.
\end{proof}

It again follows from \Cref{lem:prism-conj-fil-mod-p-rat}, via derived modulo
$p$ and rationalization, that

\begin{lemma}
  \label{lem:conj-fil-prism-comp-eq}Let $A$ be a $p$-local $\delta$-ring and
  $d \in A$ a weakly transversal element. For every $(B, J) \in
  \tmop{Pair}_{\delta, (A, d)}^{\tmop{st}}$, the comparison map in
  \Cref{lem:conj-fil-prism-comp} is an equivalence.
\end{lemma}

After such a long march, let us harvest the Hodge--Tate comparison, which is a
prismatic analogue of \Cref{cor:conjfil-gr}. Note that for every $(B
\twoheadrightarrow B'') \in \tmop{Pair}_{\delta, (A, d)}^{\tmop{an}}$, note
that the commutative diagram
\[ \begin{array}{ccc}
     A & \longrightarrow & B\\
     \longdownarrow &  & \longdownarrow\\
     A /^{\mathbb{L}} d & \longrightarrow & B''
   \end{array} \]
induces a natural map $B /^{\mathbb{L}} d \simeq B \otimes_A^{\mathbb{L}} (A
/^{\mathbb{L}} d) \rightarrow B''$ which is surjective, that is to say, $B
/^{\mathbb{L}} d \twoheadrightarrow B''$ is an animated pair. It then follows
from \Cref{lem:conj-fil-prism-comp-eq,prop:left-deriv-n-fun} that

\begin{theorem}[Hodge--Tate]
  \label{thm:Hdg-Tate}Let $A$ be a $p$-local $\delta$-ring and $d \in A$ a
  weakly transversal element. Then for every animated $\delta$-$(A, d)$-pair
  $(B \twoheadrightarrow B'') \in \tmop{Pair}_{\delta, (A, d)}^{\tmop{an}}$,
  there exists a canonical equivalence
  \[ \Gamma_{B''}^i (\tmop{gr}_{\tmop{ad}}^1 (B /^{\mathbb{L}} d
     \twoheadrightarrow B'') \{ - 1 \}) \longrightarrow
     \tmop{gr}_{\tmop{conj}}^{- i} \left( \tmop{Env}^{\Prism} (B
     \twoheadrightarrow B'') /^{\mathbb{L}} d \right) \]
  which is functorial in $(B \twoheadrightarrow B'') \in \tmop{Pair}_{\delta,
  (A, d)}^{\tmop{an}}$, where $\tmop{Fil}_{\tmop{ad}}$ is the adic filtration
  functor defined in Construction~\ref{cons:adic-fil}.
\end{theorem}

Let $R$ be an $\mathbb{E}_1$-ring. Recall that a right $R$-module $M$ is
{\tmdfn{faithfully flat}} if it is flat (\Cref{def:flat-module}) and $\pi_0
(M)$ is a faithfully flat right $\pi_0 (R)$-module. A map $R \rightarrow S$ of
$\mathbb{E}_{\infty}$-rings is {\tmdfn{faithfully flat}} if $S$ is faithfully
flat as an $R$-module. There is a useful characterization of faithfully flat
algebras:

\begin{lemma}[{\cite[Lemma~5.5]{Lurie2004}}]
  \label{lem:crit-static-faithful-flat-map}Let $f \of R \rightarrow S$ be a
  map of static (commutative) rings. Then $f$ is faithfully flat if and only
  if $f$ is flat, injective and that $\tmop{coker} (f)$ taken in the category
  of $R$-modules is flat.
\end{lemma}

\begin{lemma}
  \label{lem:crit-faithful-flat-map}Let $f \of R \rightarrow S$ a map of
  $\mathbb{E}_{\infty}$-rings. If $f$ is faithfully flat, then $\tmop{cofib}
  (f)$ taken in the $\infty$-category of $R$-module spectra is flat. The
  converse is true if $R$ is supposed to be connective.
\end{lemma}

\begin{proof}
  Assume first that $f$ is faithfully flat. Let $M \assign \tmop{coker} (\pi_0
  (R) \rightarrow \pi_0 (S))$. By \Cref{lem:crit-static-faithful-flat-map},
  the map $\pi_0 (R) \rightarrow \pi_0 (S)$ is injective and the $\pi_0
  (R)$-module $M$ is flat, Then for every $n \in \mathbb{Z}$, we have the
  exact sequence
  \[ \tmop{Tor}_1^{\pi_0 (R)} (\pi_n (R), M) \rightarrow \pi_n (R) \rightarrow
     \pi_n (R) \otimes_{\pi_0 (R)} \pi_0 (S) \rightarrow \pi_n (R)
     \otimes_{\pi_0 (R)} M \rightarrow 0 \]
  which implies that the map $\pi_n (R) \rightarrow \pi_n (R) \otimes_{\pi_0
  (R)} \pi_0 (S)$ is injective. Since $f$ is flat, the canonical map $\pi_n
  (R) \otimes_{\pi_0 (R)} \pi_0 (S) \rightarrow \pi_n (S)$ is an isomorphism,
  therefore the map $\pi_n (R) \rightarrow \pi_n (S)$ is injective. Then the
  long exact sequence associated to the fiber sequence $R \rightarrow S
  \rightarrow \tmop{cofib} (f)$ splits into short exact sequences
  \[ 0 \longrightarrow \pi_n (R) \longrightarrow \pi_n (S) \longrightarrow
     \pi_n (\tmop{cofib} (f)) \longrightarrow 0 \]
  which implies that the canonical map $M \rightarrow \pi_0 (\tmop{cofib}
  (f))$ is an isomorphism. Furthermore, we have a morphism of short exact
  sequences
  
  \[
  \xymatrix{
0\ar[r]&\pi_n(R)\ar[r]\ar^{\sim}[d]
&\pi_n(R)\otimes_{\pi_0(R)}\pi_0(S)\ar[r]\ar[d]^{\sim}
&\pi_n(R)\otimes_{\pi_0(R)}M\ar[r]\ar[d]&0\\
0\ar[r]&\pi_n(R)\ar[r]&\pi_n(S)\ar[r]
&\pi_n(\operatorname{cofib}(f))\ar[r]&0
}
  \]
  
  {\noindent}By the short five lemma, the map $\pi_n (R) \otimes_{\pi_0 (R)} M
  \rightarrow \pi_n (\tmop{cofib} (f))$ is an isomorphism, therefore
  $\tmop{cofib} (f)$ is flat.
  
  Now we assume that $R$ is connective and that $\tmop{cofib} (f)$ is flat. By
  definition, $\tmop{cofib} (f)$ is connective and so is $S$ by the fiber
  sequence $R \rightarrow S \rightarrow \tmop{cofib} (f)$. For every static
  $R$-module $M$, we have the fiber sequence
  \[ M \longrightarrow M \otimes_R^{\mathbb{L}} S \longrightarrow M
     \otimes_R^{\mathbb{L}} \tmop{cofib} (f) \]
  By flatness of $\tmop{cofib} (f)$ and {\cite[Prop~7.2.2.13]{Lurie2017}}, $M
  \otimes_R^{\mathbb{L}} \tmop{cofib} (f)$ is static, therefore so is $M
  \otimes_R^{\mathbb{L}} S$. It then follows from
  {\cite[Thm~7.2.2.15]{Lurie2017}} that $S$ is a flat $R$-module. It remains
  to show that the map $\pi_0 (R) \rightarrow \pi_0 (S)$ is faithfully flat.
  By \Cref{lem:crit-static-faithful-flat-map}, it suffices to show that $\pi_0
  (R) \rightarrow \pi_0 (S)$ is injective and $\tmop{coker} (\pi_0 (R)
  \rightarrow \pi_0 (S))$ is flat. The first follows from the connectivity of
  $\tmop{cofib} (f)$, and the later follows from the isomorphism $\tmop{coker}
  (\pi_0 (R) \rightarrow \pi_0 (S)) \cong \pi_0 (\tmop{cofib} (f))$ and the
  flatness of $\tmop{cofib} (f)$.
\end{proof}

Now we have a prismatic analogue of \Cref{cor:Fp-qreg-pd-env-flat}, with a
similar argument:

\begin{proposition}
  \label{prop:quasireg-prism-env}Let $A$ be a $p$-local $\delta$-ring and $d
  \in A$ a weakly transversal element. Let $(B \twoheadrightarrow B'') \in
  \tmop{Pair}_{\delta, (A, d)}^{\tmop{an}}$ be an animated $\delta$-$(A,
  d)$-pair such that the canonical animated pair $B /^{\mathbb{L}} d
  \twoheadrightarrow B''$ is quasiregular. Then the unit map $B'' \rightarrow
  \tmop{Env}^{\Prism} (B \twoheadrightarrow B'') /^{\mathbb{L}} d$ is
  faithfully flat.
\end{proposition}

\begin{proof}
  By \Cref{thm:Hdg-Tate} and the quasiregularity of $B /^{\mathbb{L}} d
  \twoheadrightarrow B''$, for every $i \in \mathbb{N}$, the $B''$-module
  $\tmop{gr}_{\tmop{conj}}^{- i} \left( \tmop{Env}^{\Prism} (B
  \twoheadrightarrow B'') /^{\mathbb{L}} d \right)$ is flat. By
  \Cref{lem:flat-stable-ext}, for every $i \in \mathbb{N}_{> 0}$,
  $\tmop{cofib} \left( B'' \rightarrow \tmop{Fil}_{\tmop{conj}}^{- i} \left(
  \tmop{Env}^{\Prism} (B \twoheadrightarrow B'') /^{\mathbb{L}} d \right)
  \right)$ is flat. Since the conjugate filtration is exhaustive
  (\Cref{lem:conj-fil-prism-env}) and the collection of flat modules is stable
  under filtered colimits {\cite[Lem~7.2.2.14(1)]{Lurie2017}}, we get
  $\tmop{cofib} \left( B'' \rightarrow \tmop{Env}^{\Prism} (B
  \twoheadrightarrow B'') /^{\mathbb{L}} d \right)$ is a flat $B''$-module.
  Then the result follows from \Cref{lem:crit-faithful-flat-map}.
\end{proof}

\begin{remark}
  In \Cref{prop:quasireg-prism-env}, if we further assume that $B''$ is
  static, then so is $\tmop{Env}^{\Prism} (B \twoheadrightarrow B'')
  /^{\mathbb{L}} d$. This does not imply that $\tmop{Env}^{\Prism} (B
  \twoheadrightarrow B'')$ is static. However, it implies that, after taking
  $d$-completion, $\tmop{Env}^{\Prism} (B \twoheadrightarrow B'')$ becomes
  static which should be understood as a ``static $d$-completed envelope''.
\end{remark}

\begin{remark}
  \label{rem:p-compl-prism-env}There is a $p$-completed analogue of
  \Cref{prop:quasireg-prism-env}: suppose that the animated pair $B
  /^{\mathbb{L}} d \twoheadrightarrow B''$ is
  $p$-completely\footnote{``$p$-complete'' concepts are usually applied to
  $p$-complete objects. However, this is not necessary because we can always
  derived $p$-complete a non-complete object.} quasiregular, that is to say,
  the shifted cotangent complex $L_{B'' / (B /^{\mathbb{L}} d)} [- 1]$ is a
  $p$-completely flat $B''$-module, then the same proof shows that the unit
  map $B'' \rightarrow \tmop{Env}^{\Prism} (B \twoheadrightarrow B'')
  /^{\mathbb{L}} d$ is $p$-completely faithfully flat (i.e., it becomes
  faithfully flat after derived modulo $p$).
  
  In particular, if if $(B, d)$ is a bounded oriented prism
  {\cite[Def~3.2]{Bhatt2019}} and that $B''$ is static and has bounded
  $p$-power torsion, then the $p$-completion of $\tmop{Env}^{\Prism} (B
  \twoheadrightarrow B'') /^{\mathbb{L}} d$ is static. Moreover, by
  {\cite[Lem~3.7(2,3)]{Bhatt2019}}, the $(p, d)$-completion of $C \assign
  \tmop{Env}^{\Prism} (B \twoheadrightarrow B'')$ is static and thus it
  follows from a formal argument that $(C_{(p, d)}^{\wedge}, d)$ is the
  prismatic envelope of the $\delta$-pair $B \twoheadrightarrow B''$ as long
  as it is $d$-torsion free. In other words, we generalize
  {\cite[Prop~3.13]{Bhatt2019}} by weakening regularity to quasiregularity.
\end{remark}

We record a simple corollary which furnishes a quite general class of ``flat
covers of the final object'' in the {\tmdfn{affine prismatic site}} (similar
to \Cref{def:aff-crys-site}). For this, we need the following definition:

\begin{definition}
  Let $A$ be a $\delta$-ring, $d \in A$ an element and $B$ an animated
  $\delta$-$A$-algebra. The {\tmdfn{$\infty$-category of $\delta$-$(B,
  d)$-pairs}}, denoted by $\tmop{Pair}_{\delta, (B, d)}^{\tmop{an}}$, is
  defined to be the undercategory $(\tmop{Pair}_{\delta, (A,
  d)}^{\tmop{an}})_{(B \twoheadrightarrow B /^{\mathbb{L}} d) \mathord{/}}$.
\end{definition}

Let $A$ be a $p$-local $\delta$-ring and $d \in A$ a weakly distinguished
non-zero-divisor. Let $B$ be an animated $\delta$-$A$-algebra, and $R$ an
animated $B /^{\mathbb{L}} d$-algebra. Similar to \Cref{def:aff-crys-site}, we
can consider the category of animated $\delta$-$B$-algebras $C$ along with a
map\footnote{Unlike the crystalline case, here we do not assume that the map
$R \rightarrow C /^{\mathbb{L}} d$ is an equivalence.} $R \rightarrow C
/^{\mathbb{L}} d$ of animated $B /^{\mathbb{L}} d$-algebras, which we will
denoted by $R \rightarrow C /^{\mathbb{L}} d \twoheadleftarrow C$, depicted by
the commutative diagram

\[ \xymatrix{
B\ar[rr]\ar@{->>}[d]&&C\ar@{->>}[d]\\
B/^{\mathbb
L}d\ar[r]&R\ar[r]&
C/^{\mathbb L}d
} \]

{\noindent}More formally, this is the fiber product $\tmop{Ani}
(\tmop{Ring}_{\delta})_{B \mathord{/}} \times_{\tmop{Ani} (\tmop{Ring})_{(B
/^{\mathbb{L}} d) \mathord{/}}} \tmop{Ani} (\tmop{Ring})_{R \mathord{/}}$ of
$\infty$-categories, the opposite category of which will be denoted by $\Prism
(R / (B, d))$\footnote{In {\cite{Bhatt2019}}, they used the notation $(R /
A)_{\Prism}$. However, this notation is usually devoted to topoi (such as
$X_{\tmop{et}}$ and $X_{\tmop{cris}}$). We therefore adopt the traditional
notation for sites.}.

Now let $P$ be an animated $\delta$-$B$-algebra along with a surjection $P
\twoheadrightarrow R$ of animated $B$-algebras such that the cotangent complex
$L_{P / B} /^{\mathbb{L}} d$ is a flat $P /^{\mathbb{L}} d$-module.

\begin{remark}
  \label{rem:qsmooth-maps-enough}We note that such $P$ exists in abundance.
  For example, this happens when $R$ is a smooth $B /^{\mathbb{L}} d$-algebra
  which admits a smooth $B$-lift $P$ with a $\delta$-structure compatible with
  that on $B$, or $P$ is a polynomial $B$-algebra $B [x_i]$ (of possibly
  infinitely many variables) with $\delta (x_i) = 0$ along with a surjection
  $P \twoheadrightarrow R$ of animated $B$-algebras.
\end{remark}

Then the animated pair $P \twoheadrightarrow R$ admits a canonical animated
$\delta$-$(B, d)$-pair structure, and thus the animated $\delta$-ring
$\tmop{Env}^{\Prism} (P \twoheadrightarrow R)$ gives rise to an object of
$\Prism (R / (B, d))$ (by abuse of notation, we will still denote by
$\tmop{Env}^{\Prism} (P \twoheadrightarrow R)$ the object of $\Prism (R / (B,
d))$).

\begin{remark}
  \label{rem:already-prism}By \Cref{lem:prism-full-faithful-delta-pair}, when
  $P \twoheadrightarrow R$ is ``already'' a non-completed prism in the sense
  that the induced map $P /^{\mathbb{L}} d \rightarrow R$ is an equivalence,
  the non-completed prismatic envelope $\tmop{Env}^{\Prism} (P
  \twoheadrightarrow R)$ is equivalent to $P$ itself.
\end{remark}

For any object $(R \rightarrow C /^{\mathbb{L}} d \twoheadleftarrow C) \in
\Prism (R / (B, d))$, by unrolling the definitions, the product of $(R
\rightarrow C /^{\mathbb{L}} d \twoheadleftarrow C)$ and $\tmop{Env}^{\Prism}
(P \twoheadrightarrow R)$ in $\Prism (R / (B, d))$ is given by
$\tmop{Env}^{\Prism} (P \otimes_B^{\mathbb{L}} C \twoheadrightarrow R)$. We
have therefore a map $C /^{\mathbb{L}} d \rightarrow \tmop{Env}^{\Prism} (P
\otimes_B^{\mathbb{L}} C \twoheadrightarrow C /^{\mathbb{L}} d) /^{\mathbb{L}}
d$ of animated $R$-algebras. The following proposition is essentially
equivalent to the ``flat cover of the final object'', cf.
{\cite[Prop~1.1.2]{Chatzistamatiou2020}}\footnote{This characterization was
already implicit in the Faltings's proof of ``independence of the choice of
the framing''.}.

\begin{proposition}
  \label{prop:prism-flat-cov-final-obj}Let $A$ be a $p$-local $\delta$-ring
  and $d \in A$ a weakly transversal element. Let $B$ be an animated
  $\delta$-$A$-algebra and $P \twoheadrightarrow R$ an animated $\delta$-$(B,
  d)$-pair such that the cotangent complex $L_{P / B} /^{\mathbb{L}} d$ is a
  flat $P /^{\mathbb{L}} d$-module. Then for every $(R \rightarrow C
  /^{\mathbb{L}} d \twoheadleftarrow C) \in \Prism (R / (B, d))$, the map $C
  /^{\mathbb{L}} d \rightarrow \tmop{Env}^{\Prism} (P \otimes_B^{\mathbb{L}} C
  \twoheadrightarrow C /^{\mathbb{L}} d) /^{\mathbb{L}} d$ is faithfully flat.
\end{proposition}

\begin{proof}
  By \Cref{prop:quasireg-prism-env}, it suffices to show that the map $(P
  /^{\mathbb{L}} d) \otimes_{B /^{\mathbb{L}} d}^{\mathbb{L}} (C
  /^{\mathbb{L}} d) \rightarrow C /^{\mathbb{L}} d$ is quasiregular. To
  simplify the notations, let $P'' \assign P /^{\mathbb{L}} d$, $B'' \assign B
  /^{\mathbb{L}} d$ and $C'' \assign C /^{\mathbb{L}} d$. We have the
  transitivity sequence
  \[ L_{(P'' \otimes_{B''}^{\mathbb{L}} C'') / C''} \otimes_{P''
     \otimes_{B''}^{\mathbb{L}} C''}^{\mathbb{L}} C'' \rightarrow L_{C'' /
     C''} \simeq 0 \rightarrow L_{C'' / (P'' \otimes_{B''}^{\mathbb{L}} C'')}
  \]
  associated to the maps $C'' \rightarrow P'' \otimes_{B''}^{\mathbb{L}} C''
  \rightarrow C''$ whose composite is $\tmop{id}_{C''}$. Note that $L_{(P''
  \otimes_{B''}^{\mathbb{L}} C'') / C''} \simeq L_{P'' / B''}
  \otimes_{B''}^{\mathbb{L}} C''$ is a flat $P'' \otimes_{B''}^{\mathbb{L}}
  C''$-module. It follows that $L_{C'' / (P'' \otimes_{B''}^{\mathbb{L}} C'')}
  [- 1]$ is a flat $C''$-module.
\end{proof}

We first learned the possibility of such kind of result from
{\cite[Prop~3.4]{Morrow2020}} (which is closely related to
{\cite[Prop~1.1.2]{Chatzistamatiou2020}}). Later we came up with an argument
which is essentially equivalent to the proof of
\Cref{prop:prism-flat-cov-final-obj}, but the foundation was lacking then,
therefore the current article could be understood as paving the way to this
proof. Now we want to point out that, with minor modifications, this proof
would imply {\cite[Prop~3.4]{Morrow2020}} and the relevant technical lemmas in
the recent works by Y.~Tian and by A.~Ogus {\cite{Ogus2021}} announced in
Illusie conference. Furthermore, when the proper foundation is laid, the same
proof would lead to a flat cover of the final object in the {\tmdfn{absolute
prismatic site}}, and in particular, it would recover
{\cite[Lem~5.2.8]{Anschuetz2019a}}. We now show this implication.

As in \Cref{rem:p-compl-prism-env}, we assume that $(B, d)$ is a bounded
oriented prism, $R$ is derived $p$-complete and the map $B /^{\mathbb{L}} d
\rightarrow R$ is a $p$-completely quasisyntomic (i.e. the map $B
/^{\mathbb{L}} d \rightarrow R$ is $p$-completely flat and the cotangent
complex $L_{R / (B /^{\mathbb{L}} d)}$ has $p$-complete $\tmop{Tor}$-amplitude
in $[0, 1]$ as an $R$-module spectrum). Then by {\cite[Lem~4.7]{Bhatt2018}},
$R$ is static and has bounded $p$-power torsion. Let $P$ be a derived $(p,
d)$-complete animated $\delta$-$B$-algebra which is $(p, d)$-completely
quasismooth (i.e. the map $B \rightarrow P$ is $(p, d)$-completely flat and
the cotangent complex $L_{P / B}$ is a $(p, d)$-completely flat $B$-module).
Then by {\cite[Lem~3.7(2,3)]{Bhatt2019}}, $P$ is static and for every $n \in
\mathbb{N}$, the multiplication map $d^n \of P \rightarrow P$ is injective and
$P / d^n$ has bounded $p$-power torsion.

Now suppose that we are given a surjection $P \twoheadrightarrow R$ of
$B$-algebras. Then by \Cref{rem:p-compl-prism-env}, the derived $(p,
d)$-completion of $\tmop{Env}^{\Prism} (P \twoheadrightarrow R)$ is static and
the prism defined by this $(p, d)$-completed algebra is the prismatic envelope
in the sense of {\cite[Prop~3.13]{Bhatt2019}}, where the $d$-torsion-freeness
follows from the complete flatness of $B \rightarrow P$ and
{\cite[Lem~3.7(2)]{Bhatt2019}}. Moreover, since both $B /^{\mathbb{L}} d
\rightarrow R$ and $R \rightarrow \tmop{Env}^{\Prism} (P \twoheadrightarrow R)
/^{\mathbb{L}} d$ are $p$-completely flat, the map $B \rightarrow
\tmop{Env}^{\Prism} (P \twoheadrightarrow R)$ is $(p, d)$-completely flat
(this in fact generalizes the flatness in {\cite[Prop~3.13]{Bhatt2019}}). The
proof of \Cref{prop:prism-flat-cov-final-obj} shows that

\begin{proposition}
  \label{prop:compl-prism-flat-cover}Let $(B, d)$ be a bounded oriented prism,
  $R$ a derived $p$-complete and $p$-completely quasisyntomic $B / d$-algebra.
  Let $P$ be a derived $(p, d)$-complete animated $\delta$-$B$-algebra which
  is $(p, d)$-completely quasismooth over $B$, equipped with a surjection $P
  \twoheadrightarrow R$ of $B$-algebras. Then the $(p, d)$-completion of
  $\tmop{Env}^{\Prism} (P \twoheadrightarrow R)$ is static which gives rise to
  a bounded prism $(C, d)$ in the prismatic site\footnote{It is the
  non-animated but $(p, d)$-completed version of our $\Prism (R / (B, d))$.}
  defined in {\cite[Def~4.1]{Bhatt2019}} of $R$ relative to the base prism
  $(B, d)$. Furthermore, $(C, d)$ is a flat cover of the final object in this
  site.
\end{proposition}

This implies virtually all the similar technical cover results for relative
prismatic site mentioned above, cf.~\Cref{rem:already-prism}.

\begin{remark}
  For the absolute prismatic site, the proof also works in the special case of
  {\cite[Lem~5.2.8]{Anschuetz2019a}}, but we are not aware of a statement as
  general as \Cref{prop:compl-prism-flat-cover}.
\end{remark}

\

\appendix\section{Animations and projectively generated
categories}\label{app:animation}

In this appendix, we recollect basic category-theoretic facts about animations
{\cite{Cesnavicius2019}} and projectively generated categories
{\cite[§5.5.8]{Lurie2009}} needed in the text.

\subsection{Projectively generated $n$\mbox{-}categories}In this subsection,
we will briefly recollect basic facts about projectively generated
$n$-categories. We will denote by $\tmop{An}$ the $\infty$-category of animæ
(see \Cref{sec:intro}), and by $\widehat{\tmop{An}}$ the $\infty$-category of
large animæ. We say that an anima $X$ is {\tmdfn{$n$\mbox{-}truncated}} for
$n \in \mathbb{N}_{\geq 0}$ if the homotopy groups $\pi_i (X, x) = 0$ for
every point $x \in X$ and every $i \in \mathbb{N}_{> n}$, and {\tmdfn{$(-
1)$-truncated}} if $X$ is either empty or contractible, and {\tmdfn{$(-
2)$-truncated}} if $X = \varnothing$. An $\infty$\mbox{-}category
$\mathcal{C}$ is an {\tmdfn{$n$\mbox{-}category}}
{\cite[Prop~2.3.4.18]{Lurie2009}} if for every pair $(X, Y) \in \mathcal{C}
\times \mathcal{C}$ of objects, the mapping anima $\tmop{Map}_{\mathcal{C}}
(X, Y)$ is $(n - 1)$-truncated. We will denote by $\tmop{An}_{\leq n}$ the
$\infty$\mbox{-}category of $n$-truncated animae, and by
$\widehat{\tmop{An}}_{\leq n}$ the $\infty$\mbox{-}category of large
$n$\mbox{-}truncated animae.

\begin{remark}
  $1$\mbox{-}categories are just categories in the classical category theory.
  If we define $\infty$\mbox{-}categories as quasicategories as in
  {\cite{Lurie2009}}, this identification is given by the nerve construction.
  Since in our texts, categories often mean $\infty$\mbox{-}categories, we
  usually add ``$1$-'' to avoid possible ambiguities.
\end{remark}

In fact, for the text, we only need results for $n = 1$ and $n = \infty$, but
the generalization to general $n \in \mathbb{N}_{> 0} \cup \{ \infty \}$ is
quite cost-free.

\begin{proposition}[{\cite[Cor~2.3.4.8]{Lurie2009}}]
  Let $\mathcal{C}$ be an $n$\mbox{-}category and $K$ a simplicial set. Then
  $\tmop{Fun} (K, \mathcal{C})$ is an $n$\mbox{-}category.
\end{proposition}

\begin{definition}[{\cite[ Rem~5.5.8.20]{Lurie2009}}]
  Let $\mathcal{C}$ be a cocomplete $n$\mbox{-}category and $X \in
  \mathcal{C}$ an object. We say that $X$ is {\tmdfn{compact and
  $n$\mbox{-}projective}}, or that $X$ is a {\tmdfn{compact
  $n$\mbox{-}projective object}}, if the functor $\mathcal{C} \rightarrow
  \tmop{An}_{\leq n - 1}, Y \mapsto \tmop{Map}_{\mathcal{C}} (X, Y)$
  corepresented by $X$ commutes with filtered colimits and geometric
  realizations.
\end{definition}

\begin{remark}
  Here we need $\widehat{\tmop{An}}$ in lieu of $\tmop{An}$ because the
  $\infty$\mbox{-}category $\mathcal{C}$ is not necessarily locally small. In
  practice, the $\infty$\mbox{-}categories that we encounter, e.g.
  projectively generated $\infty$\mbox{-}categories, are {\tmem{a fortiori}}
  locally small, but not necessarily {\tmem{{\tmem{a priori}}}} locally small.
\end{remark}

\begin{remark}
  In fact, an object $X \in \mathcal{C}$ is called $n$\mbox{-}projective if
  and only if the functor $\mathcal{C} \rightarrow \tmop{An}_{\leq n - 1}, Y
  \mapsto \tmop{Map}_{\mathcal{C}} (X, Y)$ corepresented by $X$ commutes with
  geometric realizations. In particular, when $\mathcal{C}$ is an abelian
  $1$\mbox{-}category, an object $X \in \mathcal{C}$ is $1$\mbox{-}projective
  if and only if it is a ``projective object'' of the abelian
  $1$\mbox{-}category $\mathcal{C}$.
\end{remark}

\begin{remark}
  Let $\mathcal{C}$ be a cocomplete $n$\mbox{-}category and $X \in
  \mathcal{C}$ a compact $n$\mbox{-}projective object. In general, $X$ is not
  a compact projective object of $\mathcal{C}$ as an $\infty$\mbox{-}category.
  In fact, the inclusion $\widehat{\tmop{An}}_{\leq n - 1} \rightarrow
  \tmop{An}$ does {\tmem{not}} commute with geometric realizations. That is to
  say, for general simplicial objects $Y_{\bullet} \of
  \tmmathbf{\Delta}^{\tmop{op}} \rightarrow \mathcal{C}$, the geometric
  realization $| \tmop{Map}_{\mathcal{C}} (X, Y_{\bullet}) |_{\bullet \in
  \tmmathbf{\Delta}^{\tmop{op}}}$ is not in general $(n - 1)$-truncated.
\end{remark}

\begin{remark}
  \label{rem:n-geom-real}There is another way to characterize geometric
  realizations in an $n$\mbox{-}category $\mathcal{C}$. In fact, the fully
  faithful embedding $\tmmathbf{\Delta}_{\leq [n]}^{\tmop{op}} \hookrightarrow
  \tmmathbf{\Delta}^{\tmop{op}}$ is ``$n$-cofinal'', therefore the geometric
  realization of a simplicial object $\tmmathbf{\Delta}^{\tmop{op}}
  \rightarrow \mathcal{C}$ exists if and only if colimit of the composite
  functor $\tmmathbf{\Delta}_{\leq [n]}^{\tmop{op}} \hookrightarrow
  \tmmathbf{\Delta}^{\tmop{op}} \rightarrow \mathcal{C}$ exists, and the two
  colimits are equivalent. Furthermore, for any diagram
  $\tmmathbf{\Delta}_{\leq [n]}^{\tmop{op}} \rightarrow \mathcal{C}$, the left
  Kan extension along $\tmmathbf{\Delta}_{\leq [n]}^{\tmop{op}}
  \hookrightarrow \tmmathbf{\Delta}^{\tmop{op}}$ always exists. Thus for a
  cocomplete $n$\mbox{-}category $\mathcal{C}$, an object $X \in \mathcal{C}$
  is $n$\mbox{-}projective if and only if the functor
  $\tmop{Map}_{\mathcal{C}} (X, \cdummy)$ corepresented by $X$ preserves
  $\tmmathbf{\Delta}_{\leq [n]}^{\tmop{op}}$-indexed colimits. See
  {\cite{Nardin2016}} and the proof of {\cite[Lem~1.3.3.10]{Lurie2017}}.
\end{remark}

\begin{definition}[{\cite[Def~5.5.8.23]{Lurie2009}}]
  \label{def:n-proj-gen}Let $\mathcal{C}$ be a cocomplete $n$\mbox{-}category
  and $S \subseteq \mathcal{C}$ a (small) collection of objects of
  $\mathcal{C}$. We say that $S$ is a {\tmdfn{set of compact
  $n$\mbox{-}projective generators}} for $\mathcal{C}$ if the following
  conditions are satisfied:
  \begin{enumerate}
    \item Each element of $S$ is a compact $n$\mbox{-}projective object of
    $\mathcal{C}$.
    
    \item The full subcategory of $\mathcal{C}$ spanned by finite coproducts
    of elements of $S$ is essentially small.
    
    \item The set $S$ generates $\mathcal{C}$ under small colimits.
  \end{enumerate}
  We say that an $n$\mbox{-}category $\mathcal{C}$ is
  {\tmdfn{$n$\mbox{-}projectively generated}} if it is cocomplete and there
  exists a set $S$ of compact $n$\mbox{-}projective generators for
  $\mathcal{C}$.
\end{definition}

\begin{remark}
  Let $\mathcal{C}$ be a cocomplete $n$\mbox{-}category and $\mathcal{C}_0
  \subseteq \mathcal{C}$ an essentially small full subcategory. Then we will
  abuse the terminology by saying that $\mathcal{C}_0$ is a set of compact
  $n$-projective generators for $\mathcal{C}$ if a skeleton of $\mathcal{C}_0$
  is a set of compact $n$-projective generators for $\mathcal{C}$.
\end{remark}

\begin{notation}
  \label{nota:Psigma-n}Let $\mathcal{C}$ be a small $n$\mbox{-}category which
  admits finite coproducts. We let $\mathcal{P}_{\Sigma, n} (\mathcal{C})$
  denote the full subcategory of $\mathcal{P}_n (\mathcal{C}) \assign
  \tmop{Fun} (\mathcal{C}^{\tmop{op}}, \tmop{An}_{\leq n - 1})$ spanned by
  those functors $\mathcal{C}^{\tmop{op}} \rightarrow \tmop{An}_{\leq n - 1}$
  which preserves finite products. When $n = \infty$, we will omit the
  subscript $n$.
\end{notation}

\begin{proposition}
  \label{prop:Psigma-n}Let $\mathcal{C}$ be a small $n$\mbox{-}category which
  admits finite coproducts. Then
  \begin{enumerate}
    \item The $\infty$\mbox{-}category $\mathcal{P}_{\Sigma, n} (\mathcal{C})$
    is an accessible localization of $\mathcal{P}_n (\mathcal{C})$, therefore
    presentable.
    
    \item The Yoneda embedding $j \of \mathcal{C} \rightarrow \mathcal{P}_n
    (\mathcal{C})$ factors through $\mathcal{P}_{\Sigma, n} (\mathcal{C})$.
    Moreover, the induced functor $\mathcal{C} \rightarrow
    \mathcal{P}_{\Sigma, n} (\mathcal{C})$ preserves finite coproducts.
    
    \item Let $\mathcal{D}$ be a presentable $n$\mbox{-}category and let
    $\mathcal{P} (\mathcal{C}) \underset{G}{\overset{F}{\longrightleftarrows}}
    \mathcal{D}$ be a pair of adjoint functors. Then $G$ factors through
    $\mathcal{P}_{\Sigma, n} (\mathcal{C})$ if and only if $f = F \circ j \of
    \mathcal{C} \rightarrow \mathcal{D}$ preserves finite coproducts.
    
    \item The full subcategory $\mathcal{P}_{\Sigma, n} (\mathcal{C})
    \subseteq \mathcal{P}_n (\mathcal{C})$ is stable under sifted
    colimits\footnote{We do not introduce $n$-sifted diagrams, so {\tmem{a
    priori}} it is a sifted diagram defined in
    {\cite[Def~5.5.8.1]{Lurie2009}}. However, here one can replace sifted
    diagrams by $n$-sifted diagram. See \Cref{rem:n-geom-real}. }.
  \end{enumerate}
\end{proposition}

We recall that, for a small $\infty$\mbox{-}category $\mathcal{C}$,
$\tmop{Ind} (\mathcal{C}) \subseteq \mathcal{P} (\mathcal{C})$ is the full
subcategory generated under filtered colimits by the essential image of the
Yoneda embedding $\mathcal{C} \rightarrow \mathcal{P} (\mathcal{C})$,
{\cite[Prop~5.3.5.3 \& Cor~5.3.5.4]{Lurie2009}}. It follows from
{\cite[Prop~5.3.5.11]{Lurie2009}} that

\begin{lemma}
  \label{lem:ind-P-sigma-n}Let $\mathcal{C}$ be a small $n$\mbox{-}category
  which admits finite coproducts. Then the fully faithful embedding
  $\mathcal{C} \hookrightarrow \mathcal{P}_{\Sigma, n} (\mathcal{C})$ extends
  uniquely to a functor $\tmop{Ind} (\mathcal{C}) \rightarrow
  \mathcal{P}_{\Sigma, n} (\mathcal{C})$ which preserves filtered colimit.
  This functor $\tmop{Ind} (\mathcal{C}) \rightarrow \mathcal{P}_{\Sigma, n}
  (\mathcal{C})$ is fully faithful.
\end{lemma}

\begin{lemma}
  \label{lem:nonab-deriv-cat-n-proj-gen}Let $\mathcal{C}$ be a small
  $n$\mbox{-}category which admits finite coproducts. Then the
  $n$\mbox{-}category $\mathcal{P}_{\Sigma, n} (\mathcal{C})$ is
  $n$\mbox{-}projectively generated for which $\mathcal{C} \subseteq
  \mathcal{P}_{\Sigma, n} (\mathcal{C})$ is a set of $n$\mbox{-}projective
  generators. In fact, for any $X \in \mathcal{P}_{\Sigma, n} (\mathcal{C})$,
  there exists a simplicial object $U_{\bullet} \of
  \tmmathbf{\Delta}^{\tmop{op}} \rightarrow \tmop{Ind} (\mathcal{C})$ (or
  equivalently, a diagram $\tmmathbf{\Delta}_{\leq n}^{\tmop{op}} \rightarrow
  \tmop{Ind} (\mathcal{C})$ by \Cref{rem:n-geom-real}) whose colimit is $X$.
\end{lemma}

\begin{proof}
  First, since $\mathcal{P}_{\Sigma, n} (\mathcal{C}) \subseteq \mathcal{P}_n
  (\mathcal{C})$ is a accessible localization, $\mathcal{P}_{\Sigma, n}
  (\mathcal{C})$ is presentable {\cite[Rem~5.5.1.6]{Lurie2009}} therefore
  cocomplete. Since $\mathcal{P}_{\Sigma, n} (\mathcal{C}) \subseteq
  \mathcal{P}_n (\mathcal{C})$ is stable under sifted colimits
  (\Cref{prop:Psigma-n}), the objects of $\mathcal{C}$ are compact and
  $n$-projective. The last statement then follows from
  {\cite[Lem~5.5.8.14]{Lurie2009}}.
\end{proof}

\begin{proposition}
  \label{prop:left-deriv-n-fun}Let $\mathcal{C}$ be a small
  $n$\mbox{-}category which admits finite coproducts and let $\mathcal{D}$ be
  an $n$\mbox{-}category which admits filtered colimits and geometric
  realizations. Let $\tmop{Fun}_{\Sigma} (\mathcal{P}_{\Sigma, n}
  (\mathcal{C}), \mathcal{D})$ denote the full subcategory spanned by those
  functors $\mathcal{P}_{\Sigma, n} (\mathcal{C}) \rightarrow \mathcal{D}$
  which preserve filtered colimits and geometric realizations. Then
  \begin{enumerate}
    \item Composition with the Yoneda embedding $j \of \mathcal{C} \rightarrow
    \mathcal{P}_{\Sigma, n} (\mathcal{C})$ induces an equivalence $\theta \of
    \tmop{Fun}_{\Sigma} (\mathcal{P}_{\Sigma, n} (\mathcal{C}), \mathcal{D})
    \rightarrow \tmop{Fun} (\mathcal{C}, \mathcal{D})$ of categories. The
    inverse $\theta^{- 1}$ is given by the left Kan extension along $j$. In
    this case, we will call $\theta^{- 1} (f)$ the {\tmdfn{left derived
    functor}} of $f \in \tmop{Fun} (\mathcal{C}, \mathcal{D})$.
    
    \item Any functor $g \in \tmop{Fun}_{\Sigma} (\mathcal{P}_{\Sigma, n}
    (\mathcal{C}), \mathcal{D})$ preserves sifted colimits.
    
    \item Assume that $\mathcal{D}$ admits finite coproducts. A functor $g \in
    \tmop{Fun}_{\Sigma} (\mathcal{P}_{\Sigma, n} (\mathcal{C}), \mathcal{D})$
    preserves small colimits if and only if $g \circ j$ preserves finite
    coproducts.
  \end{enumerate}
\end{proposition}

\begin{proposition}
  \label{prop:left-deriv-n-full}Let $\mathcal{C}$ be a small
  $n$\mbox{-}category which admits finite coproducts, $\mathcal{D}$ an
  $n$\mbox{-}category which admits filtered colimits and geometric
  realizations, and $F \of \mathcal{P}_{\Sigma, n} (\mathcal{C}) \rightarrow
  \mathcal{D}$ a left derived functor of $f = F \circ j \of \mathcal{C}
  \rightarrow \mathcal{D}$, where $j \of \mathcal{C} \rightarrow
  \mathcal{P}_{\Sigma, n} (\mathcal{D})$ denotes the Yoneda embedding.
  Consider the following conditions:
  \begin{enumerate}
    \item \label{item:f-fully-faithful-n}The functor $f$ is fully faithful.
    
    \item \label{item:ess-img-f-cpct-n-proj}The essential image of $f$
    consists of compact $n$\mbox{-}projective objects of $\mathcal{D}$.
    
    \item \label{item:ess-img-gen-n-cat}The $n$\mbox{-}category $\mathcal{D}$
    is generated by the essential image of $f$ under filtered colimits and
    geometric realizations.
  \end{enumerate}
  If \ref{item:f-fully-faithful-n} and \ref{item:ess-img-f-cpct-n-proj} are
  satisfied, then $F$ is fully faithful. Moreover, $F$ is an equivalence if
  and only if \ref{item:f-fully-faithful-n}, \ref{item:ess-img-f-cpct-n-proj}
  and \ref{item:ess-img-gen-n-cat} are satisfied.
\end{proposition}

\begin{proposition}
  \label{prop:struct-n-proj-gen-cats}Let $\mathcal{C}$ be a
  $n$\mbox{-}projectively generated $n$\mbox{-}category with a set $S$ of
  compact $n$\mbox{-}projective generators for $\mathcal{C}$. Then
  \begin{enumerate}
    \item Let $\mathcal{C}^0 \subseteq \mathcal{C}$ be the full subcategory
    spanned by finite coproducts of the objects in $S$. Then $\mathcal{C}^0$
    is essentially small, and the left derived functor $F \of
    \mathcal{P}_{\Sigma, n} (\mathcal{C}^0) \rightarrow \mathcal{C}$ is an
    equivalence of $n$\mbox{-}categories. In particular, $\mathcal{C}$ is a
    compactly generated presentable $n$\mbox{-}category.
    
    \item Let $C \in \mathcal{C}$ be an object. The following conditions are
    equivalent:
    \begin{enumerate}
      \item The object $C \in \mathcal{C}$ is compact and
      $n$\mbox{-}projective.
      
      \item The functor $\mathcal{C} \rightarrow \tmop{An}_{\leq n - 1}$
      corepresented by $C$ preserves sifted colimits.
      
      \item There exists an object $C' \in \mathcal{C}^0$ such that $C$ is a
      retract of $C'$.
    \end{enumerate}
  \end{enumerate}
\end{proposition}

\begin{proof}
  We explain more details of the first point than
  {\cite[Prop~5.5.8.25]{Lurie2009}}. It follows from definitions that
  $\mathcal{C}^0$ is essentially small. Then it follows from
  \Cref{prop:left-deriv-n-full} that the left derived functor $F \of
  \mathcal{P}_{\Sigma, n} (\mathcal{C}^0) \rightarrow \mathcal{C}$ is fully
  faithful. Since $\mathcal{C}^0 \subseteq \mathcal{C}$ is stable under finite
  coproducts taken in $\mathcal{C}$, the embedding $\mathcal{C}^0
  \hookrightarrow \mathcal{C}$ preserves finite coproducts. It follows from
  \Cref{prop:left-deriv-n-fun} that $F$ preserves small colimits, thus the
  essential image of $F$ is stable under small colimits. By assumption, $S$
  generates $\mathcal{C}$ under small colimits, therefore $F$ is essentially
  surjective.
\end{proof}

\begin{corollary}
  \label{cor:proj-gen-n-sifted}Let $\mathcal{C}$ be a projectively generated
  $n$\mbox{-}category and let $\mathcal{D}$ be an $n$\mbox{-}category which
  admits filtered colimits and geometric realizations. If a functor
  $\mathcal{C} \rightarrow \mathcal{D}$ preserves filtered colimits and
  geometric realizations, then it also preserve sifted colimits.
\end{corollary}

The following proposition is extremely useful to detect projectively generated
$n$\mbox{-}categories:

\begin{proposition}[{\cite[Cor~4.7.3.18]{Lurie2017}}]
  \label{prop:adjoint-n-proj-gen}Given a pair $\mathcal{C}
  \underset{G}{\overset{F}{\longrightleftarrows}} \mathcal{D}$ of adjoint
  functors between $n$\mbox{-}categories. Assume that
  \begin{enumerate}
    \item The $n$\mbox{-}category $\mathcal{D}$ admits filtered colimits and
    geometric realizations, and the functor $G$ preserves filtered colimits
    and geometric realizations.
    
    \item The $n$\mbox{-}category $\mathcal{C}$ is $n$\mbox{-}projectively
    generated.
    
    \item The functor $G$ is conservative.
  \end{enumerate}
  Then
  \begin{enumerate}
    \item The $n$\mbox{-}category $\mathcal{D}$ is $n$\mbox{-}projectively
    generated.
    
    \item An object $D \in \mathcal{D}$ is compact and $n$\mbox{-}projective
    if and only if there exists a compact $n$\mbox{-}projective object $C \in
    \mathcal{C}$ such that $D$ is a retract of $F (C)$.
    
    \item The functor $G$ preserves all sifted colimits.
  \end{enumerate}
\end{proposition}

\subsection{Animation of $n$\mbox{-}projectively generated
$n$\mbox{-}categories}\label{subsec:animation}In this subsection, we describe
a procedure, called {\tmdfn{animation}}, introduced in
{\cite[§5.1]{Cesnavicius2019}}, to produce a projectively generated
$\infty$\mbox{-}category from an $n$\mbox{-}projectively generated
$n$\mbox{-}category. Roughly speaking, this projectively generated
$\infty$\mbox{-}category is determined by a set of compact
$n$\mbox{-}projective generators for the $n$\mbox{-}category in question.

\begin{definition}
  \label{def:animation-cats}Let $\mathcal{C}$ be an $n$\mbox{-}projectively
  generated $n$\mbox{-}category. We choose a set $S \subseteq \mathcal{C}$ of
  compact $n$\mbox{-}projective generators for $\mathcal{C}$. Let
  $\mathcal{C}^0 \subseteq \mathcal{C}$ be the full subcategory spanned by
  finite coproducts of the objects in $S$. Then the {\tmdfn{animation}} of
  $\mathcal{C}$, denoted by $\tmop{Ani} (\mathcal{C})$, is defined to be the
  projectively generated $\infty$\mbox{-}category $\mathcal{P}_{\Sigma}
  (\mathcal{C}^0)$.
\end{definition}

\begin{remark}
  The definition of the animation does not depend on the choice of the set of
  compact $n$\mbox{-}projective generators. The key is that if $S'$ is another
  compact $n$\mbox{-}projective generators, then it follows from
  \Cref{prop:struct-n-proj-gen-cats} that every object $X' \in S'$ is a
  retract of an object $X \in \mathcal{C}^0$ in \Cref{def:animation-cats}. The
  same applies to the discussions below.
\end{remark}

\begin{example}
  Let $\tmop{Ab}$ be the abelian category of abelian groups. Then $\tmop{Ani}
  (\tmop{Ab})$ coincides with the (connective) derived category $D_{\geq 0}
  (\tmop{Ab})$.
\end{example}

\begin{remark}
  \label{rem:ani-trunc}In the context of \Cref{def:animation-cats}, we have
  $\mathcal{C} \simeq \mathcal{P}_{\Sigma, n} (\mathcal{C}^0)$ by
  \Cref{prop:struct-n-proj-gen-cats} and $\tmop{Ani} (\mathcal{C}) \simeq
  \mathcal{P}_{\Sigma} (\mathcal{C}^0)$. It follows that the
  $n$\mbox{-}category $\mathcal{C}$ could be identified with $n$-truncated
  objects in $\tmop{Ani} (\mathcal{C})$. In particular, there exists a left
  adjoint $\tau_{\leq n - 1} \of \tmop{Ani} (\mathcal{C}) \rightarrow
  \mathcal{C}$ to the fully faithful embedding $\mathcal{C} \hookrightarrow
  \tmop{Ani} (\mathcal{C})$, cf. {\cite[Rem~5.5.8.26]{Lurie2009}}.
\end{remark}

We now discuss the animation of functors.

\begin{definition}[{\cite[§5.1.4]{Cesnavicius2019}}]
  \label{def:animation-funs}Let $\mathcal{C}, \mathcal{D}$ be two
  $n$\mbox{-}projectively generated $n$\mbox{-}categories and $F \of
  \mathcal{C} \rightarrow \mathcal{D}$ a functor. Then the {\tmdfn{animation}}
  of the functor $F$, denoted by $\tmop{Ani} (F) \of \tmop{Ani} (\mathcal{C})
  \rightarrow \tmop{Ani} (\mathcal{D})$, is defined as follows:
  
  We choose a set $S \subseteq \mathcal{C}$ of compact $n$\mbox{-}projective
  generators for $\mathcal{C}$. Let $\mathcal{C}^0 \subseteq \mathcal{C}$ be
  the full subcategory spanned by finite coproducts of the objects in $S$.
  Then the functor $F \of \mathcal{C} \rightarrow \mathcal{D}$ gives rise to
  the composite $\mathcal{C}^0 \rightarrow \mathcal{C} \rightarrow \mathcal{D}
  \rightarrow \tmop{Ani} (\mathcal{D})$. We define $\tmop{Ani} (F) \of
  \tmop{Ani} (\mathcal{C}) \rightarrow \tmop{Ani} (\mathcal{D})$ to be the
  left derived functor (in \Cref{prop:left-deriv-n-fun}) of $\mathcal{C}^0
  \rightarrow \tmop{Ani} (\mathcal{D})$.
\end{definition}

\begin{example}
  Let $F \of \tmop{Ab} \rightarrow \tmop{Ab}$ be an additive functor. Then the
  animation $\tmop{Ani} (F) \of \tmop{Ani} (\tmop{Ab}) \rightarrow \tmop{Ani}
  (\tmop{Ab})$ coincides with the left derived functor $\mathbb{L}F \of
  D_{\geq 0} (\tmop{Ab}) \rightarrow D_{\geq 0} (\tmop{Ab})$ in homological
  algebra.
\end{example}

It follows from \Cref{prop:left-deriv-n-fun,prop:struct-n-proj-gen-cats} that

\begin{corollary}
  \label{cor:ani-preserve-colim}In \Cref{def:animation-funs}, if $F$ preserves
  sifted colimits (cf. \Cref{cor:proj-gen-n-sifted}), then so does $\tmop{Ani}
  (F)$. Furthermore, if $F$ preserves small colimits, then so does $\tmop{Ani}
  (F)$.
\end{corollary}

In homological algebra, there is a natural comparison map $H_0 \circ
\mathbb{L}F \rightarrow F \circ H_0$, which becomes an equivalence when $F$ is
assumed to be right exact. Now we study the animated analogue. In the context
of \Cref{def:animation-funs}, the composite functor $\tmop{Ani} (\mathcal{C})
\xrightarrow{\tau_{\leq n - 1}} \mathcal{C} \xrightarrow{F} \mathcal{D}
\overset{j_{\mathcal{D}}}{\hookrightarrow} \tmop{Ani} (\mathcal{D})$ is an
extension of the composite functor $\mathcal{C} \xrightarrow{F} \mathcal{D}
\hookrightarrow \tmop{Ani} (\mathcal{D})$. Since $\tmop{Ani} (F) \of
\tmop{Ani} (\mathcal{C}) \rightarrow \tmop{Ani} (\mathcal{D})$ is the left Kan
extension, there exists an essentially unique map $\tmop{Ani} (F) \rightarrow
j_{\mathcal{D}} \circ F \circ \tau_{\leq n - 1}$ of functors $\tmop{Ani}
(\mathcal{C}) \rightrightarrows \tmop{Ani} (\mathcal{D})$. By adjunction, we
get a canonical map $\tau_{\leq n - 1} \circ \tmop{Ani} (F) \rightarrow F
\circ \tau_{\leq n - 1}$ of functors $\tmop{Ani} (\mathcal{C})
\rightrightarrows \mathcal{D}$.

\begin{lemma}[{\cite[§5.1.4]{Cesnavicius2019}}]
  \label{lem:ani-trunc}In \Cref{def:animation-funs}, suppose that the functor
  $F \of \mathcal{C} \rightarrow \mathcal{D}$ (between $n$\mbox{-}categories)
  preserves sifted colimits. Then the map $\tau_{\leq n - 1} \circ \tmop{Ani}
  (F) \rightarrow F \circ \tau_{\leq n - 1}$ of functors constructed above is
  an equivalence of functors.
\end{lemma}

\begin{proof}
  First, note that the map $\tau_{\leq n - 1} \circ \tmop{Ani} (F) \rightarrow
  F \circ \tau_{\leq n - 1}$ of functors $\tmop{Ani} (\mathcal{C})
  \rightrightarrows \mathcal{D}$ is an equivalence of functors after composing
  with the inclusion $\mathcal{C}^0 \hookrightarrow \tmop{Ani} (\mathcal{C})$.
  We claim that both functors $\tau_{\leq n - 1} \circ \tmop{Ani} (F)$ and $F
  \circ \tau_{\leq n - 1}$ preserve sifted colimits, thus belonging to
  $\tmop{Fun}_{\Sigma} (\tmop{Ani} (\mathcal{C}), \mathcal{D})$ which becomes
  an equivalence after mapped along $\tmop{Fun}_{\Sigma} (\tmop{Ani}
  (\mathcal{C}), \mathcal{D}) \rightarrow \tmop{Fun} (\mathcal{C},
  \mathcal{D})$, and hence by \Cref{prop:left-deriv-n-fun}, the constructed
  map of functors is an equivalence.
  
  In fact, since $\tau_{\leq n - 1}$ is a left adjoint, therefore commutes
  with small colimits, which implies that $\tau_{\leq n - 1} \circ \tmop{Ani}
  (F)$ commutes with sifted colimits. On the other hand, $F \of \mathcal{C}
  \rightarrow \mathcal{D}$ is a functor which preserves sifted colimits,
  therefore also preserves sifted colimits since $\mathcal{C}, \mathcal{D}$
  are $n$\mbox{-}categories. Thus $F \circ \tau_{\leq n - 1}$ also preserves
  sifted colimits.
\end{proof}

In homological algebra, leftly deriving functors is not compatible with
compositions, therefore neither is animation of functors in general. However,
recall that with some acyclicity conditions
{\cite[\href{https://stacks.math.columbia.edu/tag/015M}{Tag
015M}]{stacks-project}}, there is a compatibility of leftly deriving functors
and compositions. Here is such a condition in the world of animations:

\begin{proposition}[{\cite[Prop~5.1.5]{Cesnavicius2019}}]
  \label{prop:ani-composite}Let $\mathcal{C}, \mathcal{D}, \mathcal{E}$ be
  three $n$\mbox{-}projectively generated $n$\mbox{-}categories and $F \of
  \mathcal{C} \rightarrow \mathcal{D}, G \of \mathcal{D} \rightarrow
  \mathcal{E}$ two functors preserving sifted colimits (cf.
  \Cref{cor:proj-gen-n-sifted}). Then
  \begin{enumerate}
    \item There is a natural transformation from the composite $\tmop{Ani} (G)
    \circ \tmop{Ani} (F)$ to $\tmop{Ani} (G \circ F)$ (In fact, for this, we
    only need that $G$ preserves sifted colimits).
    
    \item Let $\mathcal{C}^0 \subseteq \mathcal{C}$ and $\mathcal{D}^0
    \subseteq \mathcal{D}$ be full subcategories determined by a choice of set
    of compact $n$\mbox{-}projective generators as in
    \Cref{def:animation-cats}. If either $F (\mathcal{C}^0) \subseteq
    \tmop{Ind} (\mathcal{D}^0)$ in $\mathcal{D}$ or $(\tmop{Ani} (G)) (F
    (\mathcal{C}^0)) \subseteq \mathcal{E}$ in $\tmop{Ani} (\mathcal{E})$,
    then the natural transformation $\tmop{Ani} (G) \circ \tmop{Ani} (F)
    \rightarrow \tmop{Ani} (G \circ F)$ is an equivalence.
  \end{enumerate}
\end{proposition}

\

\end{document}